\documentclass[oneside, 11pt]{article}
\usepackage{setspace} 
\usepackage[utf8]{inputenc}
\usepackage[T1]{fontenc}
\usepackage[english]{babel}
\usepackage{lmodern}
\usepackage{amsmath}
\usepackage{amsthm}
\usepackage{amssymb}
\usepackage{bm}
\usepackage{units}
\usepackage{amscd}
\usepackage{geometry}
\usepackage{graphicx}
\usepackage{tikz}
\usetikzlibrary{positioning,calc,shapes.geometric}
\usepackage{tikz-cd}
\usepackage{pst-plot}
\usepackage{enumitem}
\usetikzlibrary{arrows,automata}
\usepackage{graphicx}
\usepackage{url}
\usepackage{mathtools}
\usepackage{hyperref}
\usepackage{comment}
\usepackage[capitalize]{cleveref}
\usepackage{xspace}
\usepackage{bbm}
\usepackage{xparse}
\usepackage[Tol,OkabeIto]{colorblind}
\usepackage{fontawesome}
\usepackage{authblk}
\graphicspath{{.},{Figures/}}

\newcommand{\Ran}{\mathrm{Ran}}
\newcommand{\Comp}{\mathcal{P}_\mathrm{compact}}
\newcommand{\Uni}{\mathrm{Uni}}
\DeclareMathOperator{\card}{card}
\DeclareMathOperator{\merg}{rad}
\DeclareMathOperator{\diam}{diam}
\DeclareMathOperator{\e}{e}
\newcommand{\R}{\mathbb{R}}
\newcommand{\N}{\mathbb{N}}
\newcommand{\B}{\mathcal{B}}

\newcommand{\dH}{\ensuremath{d_\mathrm{H}}\xspace}
\newcommand{\dm}{d}
\newcommand{\dsum}{d_\Sigma}
\newcommand{\Conf}{\mathrm{Conf}}
\newcommand{\Conford}{\mathrm{Conf}^\mathrm{ord}}
\newcommand{\norm}[1]{\left\lVert#1\right\rVert}
\newcommand{\abs}[1]{\left\lvert#1\right\rvert}
\newcommand{\Nsum}[1]{\norm{#1}_{\Sigma}}

\newcommand{\emb}{\mathrm{emb}}
\newcommand{\con}{\mathrm{con}}
\newcommand{\Map}{\mathrm{Map}}

\newcommand{\tH}{\ensuremath{\tau_\mathrm{H}}\xspace}
\newcommand{\tom}{\ensuremath{\tau^{\omega}\xspace}}
\newcommand{\tchi}{\ensuremath{\tau^{\chi}\xspace}}
\newcommand{\tpi}{\ensuremath{\tau_\mathrm{fin}}\xspace}
\newcommand{\BH}{B_\mathrm{H}}
\newcommand{\BHc}{\overline{B}_\mathrm{H}}

\newcommand{\Bom}{B^\omega}
\newcommand{\Som}{S^\omega}
\newcommand{\Bchi}{B^{\chi}}
\newcommand{\Schi}{S^{\chi}}
\newcommand{\Bsum}{B_\Sigma}

\newcommand{\W}{\mathbb{W}}
\newcommand{\op}{\mathrm{op}}
\newcommand{\Top}{\mathrm{Top}}
\newcommand{\Id}{\mathrm{Id}}
\newcommand{\elo}{\ell^{\omega}}
\newcommand{\dom}{d^{\omega}}
\newcommand{\dchi}{d^{\chi}}

\newcommand{\cP}{\mathcal{P}}
\newcommand{\Exit}{\mathrm{Exit}}
\DeclareMathOperator{\Shift}{Shift}

\DeclareMathOperator{\Scale}{Scale}
\DeclareMathOperator{\ScaleX}{\Scale_{\mathnormal{X}}}
\DeclareMathOperator{\Scalex}{\Scale_{\underline{\mathnormal{x}}}}
\DeclareMathOperator{\Mid}{Mid}
\DeclareMathOperator{\MidX}{\Mid_{\underline{\mathnormal{x}}}}
\DeclareMathOperator{\Clust}{Clust}
\DeclareMathOperator{\ClustX}{\Clust_{\underline{\mathnormal{x}}}}
\DeclareMathOperator{\scale}{\times}
\DeclareMathOperator{\pr}{pr}
\DeclareMathOperator{\Ima}{Im}

{\theoremstyle{definition}\newtheorem{defin}{Definition}[section]}
{\newtheorem{prop}[defin]{Proposition}\newtheorem{corol}[defin]{Corollary}\newtheorem{lem}[defin]{Lemma}}
{\theoremstyle{remark}\newtheorem{rem}[defin]{Remark}\newtheorem{exa}[defin]{Example}}
{\newtheorem{theo}[defin]{Theorem}
\newtheorem{question}{Question}}
\newtheorem{theo*}{Theorem}
\newtheorem*{statement}{Statement}

\makeatletter
\def\cref@thmoptarg[#1]#2#3#4{%
    \ifhmode\unskip\unskip\par\fi%
    \normalfont%
    \trivlist%
    \let\thmheadnl\relax%
    \let\thm@swap\@gobble%
    \thm@notefont{\fontseries\mddefault\upshape}%
    \thm@headpunct{.}
    \thm@headsep 5\p@ plus\p@ minus\p@\relax%
    \thm@space@setup%
    #2
    \@topsep \thm@preskip               
    \@topsepadd \thm@postskip           
    \def\@tempa{#3}\ifx\@empty\@tempa%
      \def\@tempa{\@oparg{\@begintheorem{#4}{}}[]}%
    \else%
      \refstepcounter[#1]{#3}
      \@namedef{cref@#3@alias}{#1}
      \def\@tempa{\@oparg{\@begintheorem{#4}{\csname the#3\endcsname}}[]}%
    \fi%
    \@tempa}%
\makeatother

\crefname{defin}{Definition}{Definitions}
\Crefname{defin}{Definition}{Definitions}
\crefname{prop}{Proposition}{Propositions}
\Crefname{prop}{Proposition}{Propositions}
\crefname{corol}{Corollary}{Corollaries}
\Crefname{corol}{Corollary}{Corollaries}
\crefname{lem}{Lemma}{Lemma}
\Crefname{lem}{Lemma}{Lemma}
\crefname{rem}{Remark}{Remarks}
\Crefname{rem}{Remark}{Remarks}
\crefname{theo}{Theorem}{Theorems}
\Crefname{theo}{Theorem}{Theorems}
\crefname{exa}{Example}{Examples}
\Crefname{exa}{Example}{Examples}
\crefname{question}{Question}{Questions}

\crefname{enumi}{}{}
\setlist[enumerate,1]{label=(\alph*)}
\setlist[enumerate,2]{label=(\roman*),ref=\theenumi.(\roman*)}

\title{Old and new structures on Ran spaces\\
\large Length structures, completeness, and conicality}

\author[1]{Sylvain Douteau}
\affil[1]{IRIF, Universit\'e Paris Cit\'e}
\author[2]{Marie Labeye}
\affil[2]{CPCV, Ecole Normale Sup\'erieure, PSL University, Sorbonne Universit\'e, CNRS}

\begin{document}

\maketitle
\begin{abstract}
    We study topologies on Ran spaces. In the literature, two distinct topologies frequently appear: the Hausdorff topology, and a finer one constructed as a colimit, that we call the final topology. In this work, given a metric space $M$, we construct new metric topologies on $\Ran(M)$, called weighted topologies. They interpolate between the Hausdorff and final topologies, the later being recovered as a limit in the category of spaces. This structure equips the final topology with a uniformity, which we show to be complete. Finally we study the Ran spaces as stratified spaces. We show that, whenever $M$ is a Riemannian manifold, the weighted topologies are conically stratified, while the final topology is only so in a weak sense.
\end{abstract}
\tableofcontents

\clearpage

\section*{Introduction}
\addcontentsline{toc}{section}{Introduction}
Given a topological space, the spaces formed by considering collections of its finite subsets have been studied for a very long time. Already, in the 1930's Borsuk and Ulam \cite{Borsuk-Ulam-symmetric-product}, studied the spaces of subsets of bounded cardinality, under the names symmetric products \footnote{Today, symmetric products refer to the related, but quite distinct, notion where $SP_n(M)=M^n/\Sigma_n$. However in the original paper of Borsuk and Ulam, there is quite a bit of confusion between $SP_n(M)$ and the truncations $\Ran_{\leq n}(M)$, and both intuitions seem to coexist.}, and later symmetric potency \cite{borsuk-circle}.  From then on, the study of those spaces continued in algebraic topology, under many names, though they were commonly called plainly "spaces of finite subsets", see for example \cite{TanreSpacesOfFiniteSubsetes,HandelTruncationsHomotopy}. When the cardinality of the subsets is fixed, this yield the (unordered) configurations spaces of $M$ which have also been thoroughly studied (see e.g. \cite{IdrissiBook}).  More recently, the space of all finite subsets has played a central role in the theory of factorization algebras. In their foundational book \cite{ChiralAlgebras}, Beilinson and Drinfeld attribute the introduction of this space in algebraic geometry to Ran \cite{ZivRan}, and thus coined the name Ran space. As a set, the Ran space of a topological space $M$ is given by
\begin{equation*}
    \Ran(M)=\{X\subset M\mid 0< \card (X)<\infty\}
\end{equation*}
and, as we have seen it has also been studied through its truncations:
\begin{equation*}
    \Ran_{\leq n}(M)=\{X\subset M\mid 0< \card(X)\leq n\}.
\end{equation*}

Given its long history, it should come as no surprise that the Ran space appears in various context. First, in mathematical physics, factorization algebras can be understood as certain algebraic objects built on Ran spaces, as explained in \cite[Chapter 3.4]{ChiralAlgebras}. The Ran space has also been studied through the lens of its natural stratification, $\card\colon\Ran(M)\to \N_{\geq 1}$, by studying constructible sheaves \cite{Guglielmo}, and its infinity category of exit paths \cite{CepekExit}. Additionally, a surprising application to topological data analysis was outlined by Lazovskis in \cite{JanisStratificationRan}, in which he constructs a poset stratification $\varphi\colon\Ran(M)\times \R_{\geq 0}\to P$, where $P$ is a poset encoding all homotopy types of finite simplicial complexes, and the map $\varphi$ sends a cluster $X\subset M$ and a radius $r>0$ to the associated \v{C}ech complex.

\subsection*{Ambiguous topologies}
\addcontentsline{toc}{subsection}{Ambiguous topologies}
While all of those examples are discussing "the" Ran space, they are in fact considering varying (and sometimes, ambiguous) topologies. The article \cite{CepekLejayExponential} shines some light on the situation. In their paper \v{C}epek and Lejay study three distinct topologies on $\Ran(M)$ (or rather $\mathrm{exp(M)}$ the set of possibly empty finite subsets) and point out that each of them is used for the study of factorization algebras by one of the key reference in the field. First, Beilinson and Drinfeld, in their book Chiral algebra \cite{ChiralAlgebras}, consider on $\Ran(M)$ the final topology with respect to the collection of maps $M^n\to \Ran(M)$ sending a $n$-tuple $(x^{(1)},\dots,x^{(n)})$ to the set $X=\{x^{(1)},\dots,x^{(n)}\}$, in $\Ran(M)$. We will call this topology the \textbf{final} topology on $\Ran(M)$, see \cref{subsec:FinalTopology}. Second, in Lurie's Higher algebra \cite{HigherAlgebra} the topology considered on $\Ran(M)$ is a strictly coarser one called the Hausdorff topology. This is the topology induced by the Hausdorff distance whenever the underlying space $M$ is equipped with a metric (see \cref{def:Hausdorff}). Third, an even coarser topology, that we won't study in this text, is considered by Costello and Gwilliam in \cite{CostelloGwilliam}.

This ambiguity between topologies was already present in the work of Borsuk and Ulam. Indeed, in \cite{Borsuk-Ulam-symmetric-product}, they define $\Ran(M)$ (implicitly) as the increasing union of the $\Ran_{\leq n}(M)$, and claim that its closure is the space of compact subsets of $M$. This suggests a confusion between the set-theoretic definition ($\Ran(M)$ is indeed the increasing union of its truncations as a set) and the topological definition (which would be needed to provide a precise statement). And while the statement $\overline{\Ran(M)}=\Comp(M)$ is meaningful for the Hausdorff topology (see \cref{prop:compact sets completion of ran hausdorff}), $\Ran(M)$ equipped with the final topology, i.e. the topology induced by the increasing union, has the topology of a complete uniform space, see (\cref{sec:completeness}), and cannot be identified with a subspace of $\Comp(M)$. 

This is not the only mistake in the history of Ran spaces. Notably, in 1949 \cite{borsuk-circle}, proves that $\Ran_{\leq 3}(S^1)$ is homeomorphic to $S^1\times S^2$, only for Bott to correct his mistake a few years later and prove that $\Ran_{\leq 3}(S^1)\simeq S^3$ \cite{Bott-circle}. This particular computation has been revisited many times since, with improvements such as the identification of the chain of inclusions $S^1=\Ran_{\leq 1}(S^1)\hookrightarrow \Ran_{\leq 2}(S^1)\hookrightarrow \Ran_{\leq 3}(S^1)$ with the inclusion of $S^1$ as the boundary of a (triple twisted) Möbius band into $S^3$, whose image is a trefoil knot \cite{Mostovoy-lattices}.

What those results should serve to illustrate is that computations on Ran spaces are difficult, and rich. Whether they are computation in low cardinality, or on the overall space.

\subsection*{Differences at infinity}
\addcontentsline{toc}{subsection}{Differences at infinity}
It is now time to turn to the heart of this paper, which is the study of topologies on $\Ran(M)$. From now on, we will always assume that the space $M$ is equipped with a metric $d$. Any reasonable topology on $\Ran(M)$ should allow for the merging of points. That is, if $x,y\in M$ are two points that are close in $M$, i.e. if $d(x,y)$ is small, then the configurations $\{x,y\}$ and $\{x\}$ should be close in $\Ran(M)$, and a similar claim should hold for configurations of larger cardinality. This will be a feature of all the topologies we consider, and this intuition is sufficient to arrive at an unambiguous topology on $\Ran_{\leq n}(M)$ for all $n\geq 1$. The topology we will obtain on $\Ran_{\leq n}(M)$ will be that of the quotient $M^n\to \Ran_{\leq n}(M)$. The ambiguity then arises "at infinity", and can be summarized by the following question.

\begin{question}
\label{quest:Converge_or_not}
If $(X_n)_{n\geq 0}$ is a sequence of configurations, containing increasingly many points, which are increasingly closer to some point $x\in M$. Should the sequence converge to $\{x\}$, because of proximity, or diverge, because of cardinality? Or should it depend on the sequence?
\end{question}

A choice of topology on $\Ran(M)$ is then, in particular, a choice of answer to this question. We already mentioned two possible choices of topology, let us recall some of their properties here.

A topology on $\Ran(M)$ is that of the increasing union, $\Ran(M)=\cup_{n\geq 1}\Ran_{\leq n}(M)$. We called it the final topology, and we will denote it $\tpi$. For this topology, the answer to \cref{quest:Converge_or_not} is: the sequence will diverge (see \cref{prop:unbounded cardinal diverges in final}).

On the other hand, there is a very direct way of equipping $\Ran(M)$ with a metric, which is to consider the Hausdorff distance, defined directly on $\Comp(M)$ as follows:
\begin{equation*}
    \dH(X,Y)=\max\{\min_{x\in X}\{d(x,Y)\},\min_{y\in Y}\{d(X,y)\}\}
\end{equation*}
this turns $\Ran(M)\subset\Comp(M)$ into a metric space, and we recover the quotient topology on the subspace $\Ran_{\leq n}(M)$. The answer to \cref{quest:Converge_or_not} is then: yes, the sequence converges to $\{x\}$.

We now have two topologies on $\Ran(M)$, the final topology, $\tpi$, and the topology $\tH$ induced by the Hausdorff distance $\dH$. Let us summarize in which way they are different, we omit the varying hypothesis on $M$, for clarity.

\begin{itemize}
    \item $(\Ran_{\leq n}(M),\tpi)=(\Ran_{\leq n}(M),\tH)$, $\forall n\geq 1$.
    \item Convergence in $\tpi$ is very fine, while it is very coarse in $\tH$.
    \item $(\Ran(M),\dH)$ is a metric space, while $(\Ran(M),\tpi)$ is non-metrizable (\cref{prop:final not metrizable}), though it admits the structure of an uniform space (\cref{sec:completeness}).
    \item $(\Ran(M),\dH)$ is not complete, while $(\Ran(M),\tpi)$, equipped with the uniformity defined in \cref{sec:completeness}, is a complete uniform space (\cref{cor:ran finale complet}).
    \item $(\Ran(M),\dH)$ is defined as a subspace $\Ran(M)\subset \Comp(M)$, while $(\Ran(M),\tpi)$ can be defined both as a colimit of the $\Ran_{\leq n}(M)$, and as a limit of metric spaces (\cref{theo:topologie limite}).
    \item The topology $\tpi$ refines $\tH$ (\cref{prop:Final_refines_Hausdorff}).
\end{itemize}

Once again, the first point should remind us that those two topologies are only distinct "at infinity",  
however this distinction should not be taken lightly. In particular, while checking the continuity of maps $(\Ran(M),\tpi)\to N$ is easy (one just needs to check that each restriction $\Ran_{\leq n}(M)\to N$ is continuous) checking that maps into $(\Ran(M),\tpi)$ are continuous is highly non-trivial. On the other hand, $(\Ran(M),\dH)$ being a metric space provides a systematic way of checking continuity of maps to and from it. While one may consider this observation as a good reason \textbf{not} to work with the final topology, we argue that, on the contrary, this rigidity is a good reason to \textbf{prefer it} to the Hausdorff topology. Indeed, if one wants to treat the space $\Ran(M)$ as an invariant of $M$, then one should be interested in \textbf{all} maps $\Ran(M)\to \Ran(N)$. Now, the answer to the question, "what are those maps?" will depend, a priori, on whether one considers the final or the Hausdorff topology. Intuitively, the fact that the final topology is very rigid at infinity will force those maps to be closer to maps coming from $M\to N$ than could be the case for the Hausdorff topology. 

\subsection*{Maps between Ran spaces and spaces of embeddings}
One reason to be interested in the systematic study of maps between Ran spaces, is their relation to spaces of embeddings. Given two manifolds $M$ and $N$, let us denote $\emb(M,N)$ the space of embeddings from $M$ to $N$. The study of the homotopy type of $\emb(M,N)$ has a long history, that we will not attempt to retell here. However, a modern development in this history is the study of configuration categories.  
Those were introduced by Andrade in his thesis \cite{AndradeThesis}, and studied more systematically by Boavida and Weiss who showed \cite{BoavidaWeissConfigurationCategories}  that under some conditions on $M$ and $N$, the homotopy type of $\emb(M,N)$ could be recovered from the homotopy type of some derived mapping space between $\con(M)$ and $\con(N)$, the configuration categories associated to $M$ and $N$ (see also \cite{BoavidaWeissTorusTrick}, which illustrates how metric techniques can be leveraged to study configurations categories).

The configuration category associated to a manifold $M$ is an $\infty$-category, of which Boavida and Weiss give several models. One of which, the "particle model" \cite[section 3.1]{BoavidaWeissConfigurationCategories}, is closely related to (the opposite of) the $\infty$-category of exit paths associated to $\Ran(M)$, which was studied by Lurie in \cite{HigherAlgebra} in the context of factorization algebras, and that we will discuss in a moment.

From this relation, we see that a possible topological translation for the derived mapping space $\Map(\con(M),\con(N))$ would be the (derived) mapping space $\Map(\Ran(M),\Ran(N))$, in the category of stratified spaces. And while we may expect that the latter derived mapping space does not depend on the choice of topology on $\Ran(M)$, it is to be expected that the strict mapping space does, and for that reason, we would prefer to consider a topology as fine as possible on $\Ran$.

Thus, we are left trying to understand continuous maps to $(\Ran(M),\tpi)$, which is a very non-trivial endeavor. One way to alleviate this problem would be to find a basis for the topology $\tpi$. The search for such a basis was our initial motivation in introducing weighted distances and topologies on $\Ran(M)$.

\subsection*{Weighted topologies as a technical tool}
\addcontentsline{toc}{subsection}{Weighted topologies as a technical tool}

Weighted topologies arise when your choice of answer for \cref{quest:Converge_or_not} is "it depends". It should be immediately clear that there is more than a single way to make the answer depend on the sequence. Weights, which are non-decreasing sequences $\omega\colon \N_{\geq 1}\to [1,\infty)$, each corresponds to a criterion to distinguish between those sequences that will converge and those that will not. 

Intuitively, the criterion should be that $X_n$ converges if and only if points in $X_n$ are getting close to $x$ faster than they are being added into $X_n$ (or faster than $\card(X_n)$ grows), and how much faster is given by $\omega$. A very crude approximation of this idea would be that $X_n$ converges to $\{x\}$ if and only if $\card(X_n)\dH(X_n,\{x\})=o\left(\frac{1}{\omega(\card(X_n))}\right)$.

What we do is slightly more subtle, but the previous intuition should be a reasonable guide. Given two configurations $X$ and $Y$, we may consider a way of going from one to the other as a relation $R\subset X\times Y$. A point $x\in X$ may be split into several points in $Y$, say $y$ and $y'$, which will translate to $(x,y)$ and $(x,y')$ both being in $R$, and symmetrically, several points in $X$ may merge in $Y$. We should ask that points do not vanish nor appear out of thin air, and thus every point in $X$ and $Y$ should be in at least one of the pairs. To rephrase the previous condition, we may say that $R$ must be the graph of a surjective multivalued function. One may think of it as an intermediary state, where some points in $X$ have split, and some of them will merge in $Y$ at the end of their trip. We may then define the length of $R$, with respect to $\omega$ as
\begin{equation*}
    \elo(R)=\omega(\card(R))\sum_{(x,y)\in R}d(x,y)
\end{equation*}
where the term $\omega(\card(R))$ encodes the weight, together with the fact that, in between $X$ and $Y$, $\card(R)$ many "intermediary points" exist. We may then consider concatenations of such jumps which start at $X$ and terminate at $Y$ and define a distance $\dom(X,Y)$ by taking it to be the minimum total length among such chains (see \cref{def:relation length omega,def:weighted-distance}).

This approach gives rise to infinitely many topologies on $\Ran(M)$, that we denote $\tom$. They are not all distinct, for example if $\omega$ and $\chi$ are two weights, such that $\omega(n)=\chi(n)$ for sufficiently large $n$, then $\tom=\tchi$ (see \cref{lem:Topology_Weights_Bounded}). Nevertheless, the set of weighted topologies is uncountable, (see \cref{prop:omega_strictement_plus_fine,rem:Uncountably_many_topologies})

This infinite collection of topologies (and their associated metrics) then allows us to answer our initial question, at least when $M$ is locally compact (see \cref{theo:topologie limite}).
\begin{theo*}
\label{theointro:limit}
Let $(M,d)$ be a locally compact metric space. A basis for the final topology, $\tpi$, on $\Ran(M)$ is given by
\begin{equation*}
    \{\Bom(X,r)\mid X\in \Ran(X),\  r>0,\ \text{$\omega$ is a weight}\}
\end{equation*}
\end{theo*}
It should be noted that we are not aware of any counter-example to this theorem when $M$ fails to be locally compact, and we expect that a simpler proof of \cref{theo:topologie limite} should exist, without this hypothesis. 

This theorem does two things. First, it gives us access to an explicit description of the open sets of the final topology. And then, we may reinterpret it as a statement that $(\Ran(M),\tpi)$ is a limit in the category of topological spaces (see \cref{theo:topologie limite}). This second statement gives a very explicit condition for checking that maps $N\to (\Ran(M),\tpi)$ are continuous: we just need to check the continuity of $N\to (\Ran(M),\tom)$ for all weights $\omega$. This is made easier by the fact that $(\Ran(M),\tom)$ is a metric space.

Besides, this theorem gives for free extra structure on $(\Ran(M),\tpi)$. Indeed, provided $M$ satisfies the hypothesis, then $(\Ran(M),\tpi)$ can be equipped with a uniformity, just by computing this limit in the category of uniform spaces (see \cref{sec:completeness}). Reasoning on the completions of each of the $(\Ran(M),\tom)$ we can then prove the following (see (\cref{cor:ran finale complet}).
\begin{theo*}
Let $M$ be a locally compact and complete metric space. Then $\Ran(M)$ equipped with the uniformity defined in \cref{def:entourages Ran} is a complete uniform space.
\end{theo*}
This somehow gives another reason to prefer working with the final topology rather than the Hausdorff topology. While the latter corresponds to a dense subset of the space of compact subspaces of $M$, the former gives an intrinsic description of the space of finite subsets, with no missing points.

\subsection*{Weighted distances as an object of study}
\addcontentsline{toc}{subsection}{Weighted distances as an object of study}
The weighted distances carry some interesting structure of their own. First, the kind of reasoning that appears when computing distances between configurations involves thinking in terms of which points in a configuration are close to each other and to another target configuration. Leading to a kind of reasoning which merges points in order, which is reminiscent of what happens in topological data analysis, when studying persistent complexes. The computation of the distance in itself is very hard, and leads to optimization problems whose complexity quickly gets out of hand (see \cref{sec:Examples_Computations} for a few examples where we do know how to compute the distance).

While the exact computation of the distances can be unachievable in some cases, we do understand it well enough to provide meta computations. Indeed, the distance $\dom(X,Y)$ is defined as an infimum of length over chains between $X$ and $Y$, but we prove that one only needs to consider chains where the intermediate configurations $X=X_0,X_1,\dots X_n=Y$ appearing are such that the sequence $\card(X_0),\card(X_1),\dots,\card(X_n)$ is at first strictly decreasing, and then strictly increasing (see \cref{theo:MBS} and \cref{def:MBS} for a precise statement). In particular, if $X$ and $Y$ are configurations in $\Ran_{\leq n}(M)$, then the distance can be computed inside $\Ran_{\leq n}(M)$, which considerably simplifies the computation. 

Furthermore, if the metric space $(M,d)$ satisfies the Heine-Borel property, that is, if the closed balls are compact, then we know that the infimum is reached (see \cref{theo:Geodesic_Chain}).

\begin{theo*}
Assume that $M$ satisfies the Heine-Borel property. 
    Then, for any $X,Y\in \Ran(M)$ and any weight $\omega$, there exists a chain of relations $(R_0,\dots,R_{n-1})$ between $X$ and $Y$ such that
    \begin{equation*}
        \dom(X,Y)=\elo(R_0,\dots,R_{n-1})
    \end{equation*}
\end{theo*}
It should be noted that unlike \cref{theointro:limit}, the hypothesis of $M$ satisfying the Heine-Borel property does not appear to be entirely superfluous. In particular, in \cref{exa:peace_and_love}, we show that, for $M=\R^2\setminus\{*\}$, some geodesic chains do not exist even though $M$ is locally compact. On the other hand, satisfying the Heine-Borel property is certainly not a necessary condition in general, as can be seen by hand by reasoning on the open interval $M=(0,1)$. This suggest that there exist much sharper versions of this theorem, where the Heine-Borel property is replaced by some form of convexity condition. 

Chains realizing this infimum are of the form discussed above, and deserve the name \textbf{geodesic chains}. They are a combinatorial version of the idea of a shortest path between two configurations. Their existence suggests the study of actual continuous paths which may embody this idea. This leads us to the study of length structures.

\subsection*{Length structures on $\Ran(M)$}
\addcontentsline{toc}{subsection}{Length structures on $\Ran(M)$}

Let us now assume that the metric space $(M,d)$ is equipped with a length structure. That is, a collection of continuous paths, $\mathcal{P}\subset \mathcal{C}^0([a,b],M)$ which we will call \textbf{admissible paths}, together with a length function $\ell\colon \cP\to \R_+$, satisfying some conditions, recalled in \cref{sec:PathMetricSpaces}. The most notable of those condition is of course that this notion of length recovers the distance in the following way.
\begin{equation*}
    d(x,y)=\inf\{\ell(\gamma)\mid \gamma\in \cP, \gamma(a)=x,\gamma(b)=y\}
\end{equation*}

In this case we define a set of admissible path in $\Ran(M)$, $\cP^{\Ran}$, as well as a length $\elo\colon \cP^{\Ran}\to \R_+$, for any weight $\omega$ (see \cref{def:relevements,def:Chemin_C1_Par_Morceaux,def:longueur_rel,def:longueur_chemin}). We then prove (see \cref{theo:PathDistanceEqualCombDistance,cor:ran length space}). 

\begin{theo*}
    If $(M,d)$ is a length space, then the metric on $\Ran(M)$ associated to the length structure given by $\cP^{\Ran}$ and $\elo$ is $\dom$.
\end{theo*}

Moreover, we may now ask if geodesics exist in the usual sense of paths realizing the distance. We prove that they do under additional hypothesis on $M$ (\cref{cor:existence geodesic path}).

\begin{theo*}
    Assume that the metric space $(M,d)$ is a length space, that it satisfies the Heine-Borel property, and that any two points in $M$ are connected by a geodesic. Then, for any two configurations $X,Y\in \Ran(M)$ and any weight $\omega$, there exists a path $\Gamma\colon [a,b]\to \Ran(M)$ such that $\Gamma(a)=X$, $\Gamma(b)=Y$ and
    \begin{equation*}
        \elo(\Gamma)=\dom(X,Y).
    \end{equation*}
\end{theo*}

Note that, crucially, the path $\Gamma$ realizing the distance may depend on the choice of $\omega$. And such a path is geodesic for a single choice of values for $\omega$ up to some $n$ in general, see \cref{exa:peace_and_love}. We can prove once more, however, that their "shape" is always the same. That is, a geodesic path in $\Ran(M)$ for any weight $\omega$, must be a path $\Gamma\colon [a,b]\to \Ran(M)$, such that there exists some $c\in [a,b]$ such that, when restricted to $[a,c]$, $\card(\Gamma)$ is a non-increasing function and when restricted to $[c,b]$, $\card(\Gamma)$ is a non-decreasing function. This is deeply related to the stratification of $\Ran(M)$ given by $\card\colon \Ran(M)\to \N_{\geq 1}$, and its invariants.  

\subsection*{The stratification of $\Ran(M)$}
\addcontentsline{toc}{subsection}{The stratification of $\Ran(M)$}
Indeed, the cardinality map $\card\colon \Ran(M)\to \N_{\geq 1}$ is a stratification on $\Ran(M)$, whether one considers the Hausdorff, weighted or final topology (See \cref{sec:Prelim_Strat}). As such, one may want to look at $\Ran(M)$ not as a space, but as a stratified space, which is a noticeably distinct point of view, at least when doing homotopy theory. 

In such a context, it is important to know if stratified spaces are geometrically well-behaved. Specifically, it is natural to ask of such spaces if they are conically stratified.  We recall the precise definition at the beginning of section \cref{sec:conicity}, but its core idea can be summarized as every point in the stratified space admitting a neighborhood homeomorphic to the product of a cone and a neighborhood in the strata. If they are, then a theorem of Lurie \cite[Theorem A.6.4]{HigherAlgebra} states that a particular invariant, called the $\infty$-category of exit paths can be defined. 

Whether $\Ran(M)$ is a conically stratified space when equipped with the Hausdorff or the final topology has already been investigated. For the former, a proof that it is conical, when $M$ is a Hausdorff space locally homeomorphic to some (possibly infinite dimensional) topological vector spaces is given in \cite[Theorem 4.12]{CepekLejayExponential}. On the other hand, for the final topology, it is easy to show that when $M$ is a manifold of dimension at least $1$, then $\Ran(M)$ cannot be conically stratified. Indeed, if it were, then the topology $\tpi$ would automatically be first countable, yet it is not. A more detailed argument is given in the proof of \cref{prop:Ran_final_Not_Conical}.

While true, the previous statement is somewhat misleading. First of all, even though $(\Ran(M),\tpi)$ is not conically stratified, it still admits an $\infty$-category of exit paths, for the reason that this invariant can be computed only from the data of the truncations (which coincide with those of $(\Ran(M),\tH)$). And secondly, as we show at the very end of \cref{sec:conicity}, when $M$ is a manifold, $(\Ran(M),\tpi)$ fails to be conically stratified on a technicality. Specifically, given a topological space $L$, there are two, a priori distinct, topologies one can equip the quotient $L\times [0,1]/L\times\{0\}$ with. 
The first is the quotient topology. The second, called the teardrop topology, is coarser in general.  Whenever $L$ is a compact space,  those two topologies coincide. Since most geometrically occurring stratified spaces carry compact links, there are few occasions to be made aware of this subtlety. But it turns out that for the definition of conically stratified space to behave correctly, (in particular, for the proof of \cite[Theorem A.6.4]{HigherAlgebra} to go through), one needs to use the teardrop topology. On the other hand, if one were to naively replace the teardrop topology by the quotient topology in the definition of a conically stratified space, then one gets a naive definition of conicality. We show that $(\Ran(M),\tpi)$ satisfies this naive definition whenever $M$ is a manifold (\cref{cor:RanFinalManifoldNaivementConique}).

We also investigate whether the weighted topologies that we introduce turn $\Ran(M)$ into a conically stratified space. We answer in \cref{Corol:Ran_Riemannian_Conical}.
\begin{theo*}
Let $(M,d)$ be a Riemannian manifold, and $\omega$ be a weight. Then $(\Ran(M),\tom)$ is conically stratified.
\end{theo*}
It can be argued that the result holds in much more generality. Due to a result of Sullivan about Lipschitz manifolds \cite{SullivanLipschitzManifold}, provided $M$ is a (topological) manifold of dimension other than $4$, we can meaningfully say that $(\Ran(M),\tom)$ is conically stratified without specifying a metric on $M$. This is a very subtle point and a more detailed discussion can be found in \cref{rem:Lipschitz_manifold_Conical}. 

The above theorem is really a simple corollary of the observation that $(\Ran(M),\tom)$ is conically stratified when $M=\R^p$ is a Euclidean space, \cref{Corol:Ran_Euclidean_Conical}, which is what we spend most of \cref{sec:conicity} proving. We do so in a very explicit way, providing concrete descriptions of the links, and of the local homeomorphisms in terms of geometric constructions in $\Ran(\R^p)$.

Surprisingly, some topological arguments are still needed to prove the naive conicality of $\tpi$, and it cannot be formally deduced from the conicality of the $\tom$. We discuss in \cref{rem:naive_not_limit} why abstract nonsense alone is not sufficient to conclude.

\subsection*{Geodesics and exit paths}
\addcontentsline{toc}{subsection}{Geodesics and exit paths}
For reasons discussed in the previous section, $(\Ran(M),\tpi)$, $(\Ran(M),\tH)$ and $(\Ran(M),\tom)$ all admit an ($\infty$)-category of exit paths (and in fact, \textbf{the same one}). Let us briefly recall what this invariant is.

The motivation behind this invariant stems from the well-known result that, given a reasonable topological space $M$, locally constant sheaves on it are equivalent to functors from its fundamental groupoid, $\Pi_1(M)$. The category $\Pi_1(M)$ can be succinctly described as the category whose objects are points in $M$, and whose morphisms between two points $x,y\in M$ are continuous paths $\gamma\colon [0,1]\to M$ up to boundary preserving homotopies.

The category $\Pi_1(M)$ is an homotopy invariant of $M$, but one cannot recover $M$ from the data of $\Pi_1(M)$ in general. However, one can recover $M$ up to weak equivalence, from the data of $\Pi(M)$, the infinity groupoid of $M$. It is an infinity category in which objects are points of $M$, morphisms are continuous paths $\gamma\colon [0,1]\to M$, but instead of quotienting by homotopies, one considers higher morphisms to be boundary preserving homotopies between morphisms of lower dimension.

The idea of following this construction for stratified spaces seems to date back to MacPherson, though the first published steps in defining the exit path infinity category are due to Treumann \cite{TreumannExit2Cat} and Woolf \cite{WoolfFundamentalCategory}. Given reasonable stratified spaces $\varphi\colon T\to P$, they define $2$-categorical and $1$-categorical versions of the category of exit paths, whose object are points of $T$ and whose morphisms are (equivalence classes of) continuous paths $\gamma\colon [0,1]\to T$ which are exit paths. That is, there must be some $p\leq q\in P$ such that $\varphi(\gamma(0))=p$ and $\varphi(\gamma(t))=q$ for all $0<t\leq 1$. In particular, such a path must visit at most two strata.

Later, Lurie introduced an infinity categorical version of this category \cite[Appendix A]{HigherAlgebra}. Several variations of homotopy theories for stratified spaces have since been introduced \cite{NandLalThese,DouteauCatStratifiedSpace,HaineHomotopyTheoryStrat}. They all satisfy some version of the following statement 
\begin{statement}
    Let $T\to P$ be a stratified space. Provided $\Exit(T)$ is well-defined, then one can recover $T\to P$ up to stratified weak equivalence from the data of $\Exit(T)$.
\end{statement}
Thus, understanding the homotopy type of $\Ran(M)$ as a stratified space amounts to understanding the $\infty$-category $\Exit(\Ran(M))$. The first step being understanding its $1$-morphisms, which are the exit paths in $\Ran(M)$. A first step in this direction can be deduced from results we have already discussed.
\begin{theo*}
    Let $(M,d)$ be a complete Riemannian manifold, $X,Y\in \Ran(M)$ and $\omega$ be a weight. Then there exists a geodesic $\Gamma\colon [0,1]\to \Ran(M)$ between $X$ and $Y$, and $c\in [0,1]$ such that
    \begin{itemize}
        \item $\Gamma_{|[c,1]}$ is a concatenation of exit paths.
        \item $\Gamma_{|[0,c]}$ is a concatenation of opposites of exit paths.
    \end{itemize}
\end{theo*}
Thus geodesics can be decomposed into exit paths, which are themselves shorter geodesics. We may attempt to study $\Exit(\Ran(M))$ through subcategories formed by only considering exit paths that are geodesic for some weight $\omega$. While we leave this part of the study for future works, this was a motivation in attempting to prove the existence and characterization of geodesics in $\Ran(M)$.

\subsection*{Invariance properties of $\Ran(M)$ as a space}
\addcontentsline{toc}{subsection}{Invariance properties of $\Ran(M)$ as a space}

Studying Ran spaces, their truncations $\Ran_{\leq n}(M)$, or their strata $\Conf_n(M)$, as invariants of the underlying space $M$, is certainly not a new idea. Though, there are numerous immediate obstructions to doing so. First, and most obvious, if $f\colon M\to N$ is an arbitrary continuous map between two topological spaces, then $f$ does not induce a continuous map $\Conf_n(M)\to \Conf_n(N)$. It does induce a continuous map $\Ran(M)\to \Ran(N)$, sending $X\subset M$ to $f(X)\subset N$, which restricts properly to the truncations, but this map will not be a stratum preserving map in general. 
Indeed, if $f$ fails to be injective, then there are some $X\subset M$ such that $\card(f(X))<\card(X)$.

Still, there have been attempts to interpret those objects as some sort of homotopy invariants of $M$. For configurations spaces in particular, it is immediately obvious that those are not homotopy invariants on the nose. Consider for example that $\Conf_2(\R^0)$ is empty yet $\Conf_2(\R^1)$ is non empty but contractible, and $\Conf_2(\R^2)$ is non contractible.

Yet, in \cite{IdrissiRealHomotopyType}, Idrissi proved that provided $M$ is a simply connected closed manifold, the real homotopy type of $\Conf^{\text{ord}}_n(M)$ can be recovered from the real homotopy type of $M$. Here $\Conf^{\text{ord}}_n(M)$ corresponds to the \textbf{ordered} configuration space, i.e. the open subset of $M^n$ which covers $\Conf_n(M)$, and the real homotopy type roughly refers to the part of the homotopical information about $M$ which survives when tensoring the homotopy groups with $\R$. Which is to say, part of the homotopical information of $\Conf^{\text{ord}}_n(M)$ can be deduced only from homotopical information about $M$ (of the same nature).

On the other hand, the case of the homotopical invariance of $\Ran(M)$ is trivial. Indeed, when $\Ran(M)$ is equipped with the final topology, $\Ran(M)$ is weakly contractible as soon as $M$ is Hausdorff and path connected. (see \cite[Corollary 4.3]{HandelTruncationsHomotopy}). This follows from the observation that any continuous map $S^k\to \Ran_{\leq n}(M)$ can be homotoped to a constant map inside $\Ran_{\leq n'}(M)$ for some $n'\geq n$. Since any map $S^k\to \Ran(M)$ which is continuous for the final topology must factor through some truncations $\Ran_{\leq n}(M)$, this implies that the homotopy groups of $(\Ran(M),\tpi)$ are all trivial.

The same result is known for the Hausdorff topology, and can be obtained by the method outlined in \cite[section 3.4.1]{ChiralAlgebras}, though the authors present their proof for the final topology (see \cite[Section 5.5.1]{HigherAlgebra} for a version of this proof applied to the Hausdorff topology). It should be expected that the same holds for the weighted topologies, and that $(\Ran(M),\tom)$ is also contractible, though we do not investigate this here. One expects that a proof should follow from the observation made above that all of those topologies coincide on truncations, hence they correspond to weakly equivalent stratified spaces.

The last purely topological invariant one could hope for are the truncations of the Ran spaces, which are both homotopy invariant and non-trivial in general. Indeed, as we have seen $\Ran_{\leq 3}(S^1) \simeq S^3$ is non-contractible, and an immediate computation shows that homotopic maps $f,g\colon M\to N$ induce homotopic maps between truncations $\Ran_{\leq n}(M)\to \Ran_{\leq n}(N)$. However, if one tries to assemble each of those invariants for $n\geq 1$, that is, if one takes the increasing union, then one recovers $\Ran(M)$ which is trivial.
Thus, to produce a meaningful invariant out of all the truncations, one needs to consider $\Ran(M)$ as a stratified space\footnote{One may instead choose to consider $\Ran(M)$ as a filtered space. That is, equipped with the filtration given by the increasing union. This was, for a long time, the preferred point of view when discussing objects with singularity. See for example the case of Whitney stratifications. Nowadays, this point of view has been mainly replaced by that of poset stratifications, which we presented here. It should be emphasized that those two points of view do not lend themselves to the study of the same kinds of maps, and as such would not produce the same homotopy theory (if one were to construct one for filtered spaces). Thus, the reason we choose to see $\Ran(M)$ as a stratified space instead of a filtered space is not because the points of view are equivalent (they are not), nor because the poset-stratified point of view is inherently better (it may or may not be) but because only the former homotopy theory has already been developed.}.

\subsection*{Properties of $\Ran(M)$ as a stratified space}

The space $\Ran(M)$ equipped with its stratification $\card\colon \Ran(M)\to \N_{\geq 1}$ has some very peculiar properties as a stratified space, especially when $(M,d)$ is a Riemannian manifold. Indeed, it is an example coming from geometry - in particular not a purpose-built counter example - which is conically stratified (naively so for the final topology) yet does not admit a regular stratum. This is not entirely unique, since any construction of an infinite dimensional object out of finite dimensional strata will fail to have a regular stratum, but is in stark contrast with the classical examples encoding the idea of "manifolds with singularities".

The fact that several topologies on the underlying space $\Ran(M)$ are routinely used, all with the same geometrical intuition - that collision of points should be continuous - and sometimes for the same goal is also striking. Of course, this is in part due to the fact that a lot of focus lies on the truncations, which coincide from one topology to another. Another related explanation of this phenomenon is that, from the point of view of homotopy theory, those topologies are not different since they yield the same (weak) stratified homotopy type. Meaning, from the point of view of stratified homotopy theory, it is unambiguous to speak about the Ran space, without specifying the topology under consideration.

On the other hand, as we have seen, the final topology fails to be conically stratified, it is only naively so. Yet, we have proven that it is a limit of conically stratified space, through a reasonably mild functor (it is indexed by a poset, and all the maps are just refinements of topology). To our knowledge, this is the first example of this kind. Still, the final topology admits an $\infty$-category of exit paths (by virtue of being weakly equivalent to conically stratified spaces).

The stratified homotopy type of $\Ran(M)$ itself is interesting. It is non-trivial in general, even when $M$ is a Euclidean space, yet it is known that the underlying space is contractible. This is absolutely not the first example of this kind, not even the first geometrically occuring, since any stratifications of the infinite dimensional sphere by increasing inclusions of finite dimensional ones shares this property. Yet, it is notable that non-trivial information about any space $M$ can be stored in the contractible space $\Ran(M)$ by recalling the stratification of the latter\footnote{Straying a bit from the philosophy of this example the stratified cone on any space $M$ also shares those properties, in a somewhat tautological way. The stratified homotopy type of the cone is a strict invariant of the homotopy type of $M$ however.}.

And, as discussed, if $M$ is equipped with a length structure, we may hope to recover some description of the homotopical information contained in the stratified space $\Ran(M)$ through the study of geodesics in $(\Ran(M),\dom)$.

Overall, this begs the following question.

\begin{question}
Given a topological space $M$, how much information about its homeomorphism type can be recovered from the (weak) stratified homotopy type of $\Ran(M)$ ?
\end{question}

Clearly, the stratified homotopy type of $\Ran(M)$ fully recovers the homotopy type of $M$, since stratified homotopy equivalences preserve the homotopy type of strata. But since the stratified homotopy type of $\Ran(M)$ is not itself an homotopy invariant of $M$ we may hope to recover quantifiably more than that.

Now, if we look specifically at the spaces $(\Ran(M),\tom)$ for a given weight $\omega$, we may ask what it is an invariant of. The situation is even more drastic than it is for configuration spaces. Indeed, for a continuous map $f\colon (M,d)\to (N,d')$ to induce a continuous stratified map between the Ran spaces of $M$ and $N$ equipped with the topologies associated to $\omega$, it is not even enough to ask that $f$ be injective (see \cref{exa:racine pas relevable sur ran omega}).  One should ask in addition that $f$ be locally Lipschitz, just to get a continuous map between the topologies $\tom$. We prove specifically \cref{corol:Ran_omega_functor}.

\begin{theo*}
    Let $\mathcal{C}$ be the category of metric spaces and locally Lipschitz maps, then for any weight $\omega$, the following functor is well-defined
    \begin{align*}
        \Ran^{\omega}\colon \mathcal{C}&\to \mathcal{C}\\
        (M,d)&\mapsto (\Ran(M),\dom)\\
        f\colon (M,d)\to (N,d')&\mapsto \Ran(f)\colon \Ran^{\omega}(M,d)\to \Ran^{\omega}(N,d')
    \end{align*}
    Furthermore, if $f\colon (M,d)\to (N,d')$ is a map in $\mathcal{C}$ which is injective, then $\Ran^{\omega}(f)$ is stratum preserving.
\end{theo*}

Which now suggests that $(\Ran(M),\tom)$ recalls not only information about the space $M$, but also about its metric, even without recalling the induced metric $\dom$, otherwise this is tautological. Since existence of maps between weighted topologies is not a given for maps that would otherwise lift to the final topology, we may also ask:

\begin{question}
Given a metric space $(M,d)$, how much information about its metric can be recovered from the properties of $(\Ran(M),\tom)$ and the (strict) stratified homotopy type of $(\Ran(M),\tom)$ for $\omega$ a weight ?
\end{question}

In particular, there is no reason to believe \textit{a priori} that a Whitehead type theorem should be true for the stratified spaces $(\Ran(M),\tom)$. Hence the strict stratified homotopy type of $(\Ran(M),\tom)$ may contain extra information about the metric, compared to the weak stratified homotopy type of $\Ran(M)$ (which coincide whether one considers the weighted or the final topologies). While we do not attempt to address those questions here, we expect that the tool we developed will help make explicit computations easier for future work.

\section*{Organization of the document}
\addcontentsline{toc}{section}{Organization of the document}
Immediately after this introduction, a very brief summary of notations can be found. Since some of the proofs can be notationally heavy, it may be useful to refer back to \cref{tab:notations} to disambiguate between notations. Outside of possible mistakes, every object appearing in the document should be written in a way consistent with this table.

The rest of the document is organized as follows.

\cref{sec:Preliminaries} contains known results and constructions about Ran spaces that we will need throughout the text, together with classical definitions, such as that of length structure on a metric space, or of stratified spaces. Any reader somewhat familiar with those subjects may choose to skip it, and possibly refer back to it later.

The weighted distances are introduced in \cref{sec:Combinatorial_Distance}. After working through the definition in section \cref{sec:Comb_Distance_def}, the next two sections are spent proving \cref{theo:MBS}, which states that the weighted distance between two configurations $X$ and $Y$ can be computed by only considering chains which first merge points two at a time, and then, after a bijection, splits them two at a time. We leverage this structure in \cref{sec:Geodesic_chains} to prove that in any metric space satisfying the Heine-Borel property, the weighted distance, which is defined as an infimum, is in fact the length of a certain geodesic chain. This is \cref{theo:Geodesic_Chain}.

\cref{sec:PathMetricsOnRan} then focuses on the particular case in which $(M,d)$ carries a length structure. This is the case for example if $(M,d)$ is a Riemannian manifold. In \cref{sec:Lifts}, we construct a length structure on $\Ran(M)$, given any weight $\omega$, and show that the distance it induces is $\dom$, the weighted distance introduced in \cref{sec:Combinatorial_Distance} (\cref{theo:PathDistanceEqualCombDistance}). Then, in \cref{sec:GeodesicPaths}, we show \cref{cor:existence geodesic path}, a continuous analogue of \cref{theo:Geodesic_Chain}. It states that, if we assume in addition to the hypothesis of the latter theorem that any two points in $M$ are connected by a geodesic, then any two configurations in $\Ran(M)$ are also connected by a geodesic, for any distance $\dom$.

\cref{sec:Examples_Computations} is then focused on explicit computations of weighted distances. In \cref{sec:Examples_Examples}, we first go through several illustrated examples, to underline which intuitions are correct, and dissipate some incorrect assumptions. It should be fairly accessible, and readers may find it helpful to read those examples while working through \cref{sec:Combinatorial_Distance}. Examples should make it clear why some apparently trivial statements about the weighted distances do need non-trivial proofs.
\cref{subsec:locality} is then concerned with several lemmas related to computation of distances in the neighborhood of a configuration. Those results will be needed in later sections, and help shed some light on the techniques that are available to compute weighted distances. 
Finally, \cref{subsec:Functoriality} is concerned with the invariance of $(\Ran(M),\tom)$ with respect to maps of metric spaces $f\colon (M,d)\to (N,d')$.

The characterization of the final topology is then the main focus of \cref{sec:Distances_Topologies}. After examining how the weighted topologies compare to each other in \cref{subsec:comparing_weighted_topologies}, we prove \cref{theo:topologie limite} in \cref{subsec:Final_limite}. This states that the final topology on $\Ran(M)$ is obtained as the limit of the weighted topologies in the category of topological spaces, provided that the space $M$ is locally compact.

When the previous theorem holds, $(\Ran(M),\tpi)$ then inherits a uniform structure from being a limit of uniform spaces (\cref{prop:Uniformity_Ran_finale}), which we investigate in \cref{sec:completeness}. Specifically, this section is devoted to proving that whenever $M$ is a complete and locally compact metric space, then $(\Ran(M),\tpi)$ is a complete uniform space (\cref{cor:ran finale complet}). To work through the proof, we first extend the definitions of \cref{sec:Combinatorial_Distance} to also apply to infinite chains in \cref{subsec:Infinite_Chains}. We use this in \cref{subsec:Cauchy_Ran_omega}, to characterize non-converging Cauchy sequences in $(\Ran(M),\dom)$. This allows us to identify the completion of $(\Ran(M),\dom)$, as a set, with a subset of $\Comp(M)$, which we leverage in \cref{subsec:Completeness_proof} to prove that $(\Ran(M),\tpi)$ is complete, by expressing it as a limit of complete spaces.

Finally, \cref{sec:conicity} focuses on conicality. We first prove that $(\Ran(M),\dom)$ is a conically stratified space, whenever $(M,d)$ is a Riemannian manifold (\cref{Corol:Ran_Riemannian_Conical}), and then that $(\Ran(M),\tpi)$ is conically stratified in the naive sense, (\cref{Ran faible conique}). The latter follows easily from the techniques introduced to prove the former, and is done at the very end in \cref{subsec:naive_conical}. The proof of the former is very long, in part due to the fact that we do not prove the mere existence of a local homeomorphism, but provide an explicit construction in terms of geometric maps which we define for this purpose. Given the length and technicality of this proof, we provide a summary in \cref{subsec:sketch}, which follows a reminder about conical stratifications in \cref{subsec:Prelim_conical}.

\section*{Notations and terminology}
\addcontentsline{toc}{section}{Notations and terminology}

\subsection*{About the name "final topology"}
\addcontentsline{toc}{subsection}{About the name "final topology"}

The topology of the increasing union $\Ran(M)=\cup_{n\geq 1}\Ran_{\leq n}(M)$ could be given many names. The most natural thing, for any algebraic topologist, or anyone familiar with CW complexes would be to call it the "weak topology". However, this topology is actually the \textbf{finest} topology among all we will consider (and possibly, the finest topology one may reasonably want to consider on $\Ran(M)$), thus the name would be very misleading. Another possibility would be to call it the "colimit topology", which is factually correct since an increasing union is nothing more than a colimit in the category of topological spaces. However, we will prove that $(\Ran(M),\tpi)$ is also a limit in the category of uniform spaces (see \cref{theo:topologie limite} and \cref{sec:completeness}) we thus fear that this name would lead to confusion. In the end, the name "final topology" refers to the properties of the maps $\pi_n\colon M^n\to\Ran(M)$ which play a central role throughout this paper, and has the added advantage that the "fin" in $\tpi$ stands for both "final topology" and "finest topology".

\subsection*{Notational conventions}
\addcontentsline{toc}{subsection}{Notational conventions}
This brief section aims to help the reader navigate through the notations of this work. Most of the notations are summarized in \cref{tab:notations}.

\renewcommand{\arraystretch}{1.5}
\begin{table}[h]
    \centering
    \begin{tabular}{|l| l | l | l | l | l | l | l |}
         \hline 
         sets & elements & paths & distances & balls &topologies \\
         \hline
         $M$ & $x,y,z,w$ & $\gamma,\lambda,\mu\colon[a,b]\to M$ & $d$ & $B$&\\
         $M^n$ & $\underline{x}=(x^{(1)},\dots,x^{(n)})$ & $\underline{\gamma},\underline{\lambda},\underline{\mu} \colon [a,b]\to M^n$ & $\dsum$ & $\Bsum$& \\
         $\Ran(M)$ & $X,Y,Z,W$ & $\Gamma\colon[a,b]\to \Ran(M)$ & $\dom,\dH$ & $\Bom,\BH$& $\tpi,\tom,\tH$ \\
         $\Ran_{\leq n}(M)$ & & &$\dom,\dH$ & $\Bom_{\leq n}$ &unique\\
         \hline
    \end{tabular}
    \caption{Summary of notations}
    \label{tab:notations}
\end{table}
\renewcommand{\arraystretch}{1}

We will always consider a metric space $(M,d)$. In rare cases, the index $M$ will be specified on the distance $(M,d_M)$ and balls $B_M(x,r)$ to avoid confusion. In the case where $M=\R^p$ for some $p\in\N$, we will write $x=(x_1,\dots,x_p)$ the coordinates of its elements in the canonical basis, while the elements of the Cartesian product will be written $\underline{x}=(x^{(1)},\dots,x^{(n)})\in M^n$.

We will consider the following topologies on $\Ran(M)$, from finest to coarsest:
\begin{itemize}
    \item The final topology $\tpi$ which is the final topology for the quotient applications $\pi_n\colon M^n\to\Ran(M)$, see \cref{def:final topology}
    \item The weighted topologies $\tom,\tchi$ induced by the weighted distances $\dom,\dchi$ for various weights $\omega,\chi\in\W$, introduced in \cref{sec:Combinatorial_Distance}
    \item The Hausdorff topology $\tH$ induced by the Hausdorff distance $\dH$, see \cref{def:Hausdorff}
\end{itemize}
All these topologies coincide on the truncations $\Ran_{\leq n}(M)$ of the Ran space (see \cref{prop:topologie_troncation}) and thus when discussing those, their will be no need to specify the topology.

\section{Preliminaries}
\label{sec:Preliminaries}
This work aims at defining and studying new distances and topologies on Ran spaces. In this section we recall some known facts about Ran spaces and some useful notions of general topology. We first define configuration spaces (\cref{subsec:Conf_Spaces}) before introducing the Ran space together with two of its topologies. The final topology in \cref{subsec:FinalTopology}, which is induced by the description of $\Ran(M)$ as an increasing union, and the Hausdorff topology in \cref{subsec:HausdorffTopology}, which is induced by the Hausdorff distance. We recall the notion of a stratified space in \cref{sec:Prelim_Strat}, and describe the stratified structure of $\Ran(M)$.
In \cref{subsec:prelim_truncations} we recall some properties of the truncations of the Ran space that will be useful throughout this document. Finally, \cref{sec:PathMetricSpaces} discusses length structures in preparation of \cref{sec:PathMetricsOnRan} which will discuss how such structures are inherited by $\Ran(M)$ whenever $M$ is itself a length space.

Unless otherwise stated $(M,d)$ will always be a metric space, and, for any $n\in\N$, we always equip $M^n$ with the following distance, which is compatible with the product topology, 
\begin{align*}
    \dsum\colon M^n \times M^n &\to \R_+\\
    ((x^{(1)},\dots,x^{(n)}),(y^{(1)},\dots,y^{(n)}) )&\mapsto \sum_{i=1}^n d(x^{(i)},y^{(i)})
\end{align*}

\subsection{Configuration spaces}
\label{subsec:Conf_Spaces}
We begin by defining the configuration spaces, both the ordered and unordered ones.
\begin{defin}
    Let $M$ be a topological space, and $n\in\N^*$. The $n^\text{th}$ ordered configuration space of $M$, denoted $\Conford_n(M)$, is the set of $n$-tuples of $M^n$ whose elements are pairwise distinct, i.e.:
    \begin{equation*}
        \Conford_n(M)=\{(x^{(1)},\dots,x^{(n)})\mid x^{(i)}\neq x^{(j)} ,\forall i\neq j\in\{1,\dots,n\}\}\subset M^n,
    \end{equation*}
    equipped with the topology induced by the inclusion.
\end{defin}

\begin{rem}
    If $M$ is a topological manifold of dimension $\geq 1$, one can check that $\Conford_n(M)$ is a dense subset of $M^n$ for any $n\geq 1$.
\end{rem}

\begin{defin}
    Let $M$ be a topological space, and let $m,n\in\N^*$ and $f\colon \{1,\dots,m\}\to\{1,\dots,n\}$. Then, for all $(x^{(1)},\dots,x^{(n)})\in M^n$, write $f\cdot (x^{(1)},\dots,x^{(n)}) = (x^{(f(1))},\dots,x^{(f(m))}))\in M^m$.
\end{defin}

\begin{rem}
    If $n=m$, this defines an action of the symmetric group $\Sigma_n=\{f\colon\{1,\dots,n\}\to\{1,\dots,n\} \mid f\text{ bijective}\} $ on $M^n$ which restricts to an action on $\Conford_n(M)$.
\end{rem}

\begin{defin}
    The $n^\text{th}$ configuration space of $M$, denoted $\Conf_n(M)$ is the quotient of $\Conford_n(M)$ by the action of the permutation group $\Sigma_n$, equipped with the quotient topology. As a set, it is explicitly described as
    \begin{equation*}
        \Conf_n(M)=\{X\subset M\mid \card(X)=n\}.
    \end{equation*}
\end{defin}

We will need the following proposition in \cref{sec:PathMetricsOnRan} to understand paths in $\Ran(M)$ by studying lift through $M^n$.
\begin{prop}
\label{prop:quotient conford to conf covering}
    If $M$ is Hausdorff, then the quotient application $\Conford_n(M)\to\Conf_n(M)$ is a covering. 
\end{prop}
\begin{proof}
    Since $M$ is Hausdorff, for any $\underline{x}=(x^{(1)},\dots,x^{(n)})\in\Conford_n(M)$ we have a set $(U_i)_{i\in\{1,\dots,n\}}$ of disjoint neighborhoods of the $x^{(i)}$ in $M$. The product $U=\prod_{i=1}^n U_i$ is a neighborhood of $\underline{x}$ in $\Conford_n(M)$ such that, for any $\sigma\in\Sigma_n$, we have $\sigma\cdot U \cap U =\emptyset$. In particular, the action of $\Sigma_n$ on $\Conford_n(M)$ is properly discontinuous. Since $\Sigma_n$ acts by homeomorphism, the quotient application is a covering.
\end{proof}

\subsection{Ran space and final topology}
\label{subsec:FinalTopology}
We are now able to define the underlying set of Ran spaces as the disjoint union of the configuration spaces. But we have to define a topology on this set $\Ran(M)$. We recall two definitions commonly found in the literature: the final topology in this subsection, then the Hausdorff distance and its associated topology in the following one. 

\begin{defin}
    Let $M$ be a topological space, the \textbf{Ran space} of $M$, denoted $\Ran(M)$, has an underlying set composed of the union of the configuration spaces of $M$:
    \begin{align*}
        \Ran(M)=\bigsqcup_{i\geq 1} \Conf_i(M)=\{X\subset M\mid 0< \card (X)<\infty\} .
    \end{align*}
    Let $n\in\N^*$, the $n^\text{th}$ truncation of $\Ran(M)$, denoted $\Ran_{\leq n}(M)$ is the subset defined as
    \begin{align*}
        \Ran_{\leq n}(M)=\bigsqcup_{i= 1}^n \Conf_i(M)=\{X\subset M\mid 0< \card (X)\leq n\}.
    \end{align*}
\end{defin}

\begin{defin}
\label{def:final topology}
    Let $M$ be a topological space, and $n\in\N^*$. We define
    \begin{align*}
        \pi_n\colon M^n &\to \Ran(M)\\
        (x^{(1)},\dots,x^{(n)})&\mapsto \{x^{(1)},\dots,x^{(n)}\}.
    \end{align*}
   The final topology for the $\pi_n$ is called the \textbf{final topology} on $\Ran(M)$, and written $\tpi$.
\end{defin}

\begin{rem}
    We can give another definition of the final topology. For any $n\in\N^*$, we equip $\Ran_{\leq n}(M)$, the $n^\text{th}$ truncation of $\Ran(M)$, with the quotient topology associated to the quotient map $\pi_n^{|\Ran_{\leq n}(M)}\colon M^n\to \Ran_{\leq n}(M)$. We then remark that, for these topologies, the inclusion $\Ran_{\leq n}(M) \to \Ran_{\leq n+1}(M)$ is the inclusion of a closed subspace. The final topology is then that of the increasing union
    \begin{align}
    \label{eq:final topology increasing union of truncation}
        \left(\Ran(M),\tpi \right) = \bigcup_{n\geq 1}\Ran_{\leq n}(M). 
    \end{align}
    
    We easily check that the two definitions are equivalent. If $U\subset\Ran(M)$ is a subset of $\Ran(M)$, it is open for the final topology if and only if $\pi_n^{-1}(U)$ is open in $M^n$ for all $n\geq 1$. Besides $U\cap \Ran_{\leq n}(M)$ is open if and only if $\left(\pi_n^{|\Ran_{\leq n}(M)}\right)^{-1}(U\cap \Ran_{\leq n}(M))$ is open in $M^n$. Since $\Ran_{\leq n}(M)$ is the image of $\pi_n$, the two conditions are identical, and the topologies coincide. In particular, $(\Ran(M),\tpi)$ is directly obtained as the colimit of its truncations.
\end{rem}

\begin{rem}
\label{rem:symmetric space}
    The restriction of $\pi_n$ coincides with the quotient application $\Conford_n(M)\to\Conf_n(M)$ which corresponds to the quotient by the action of the symmetric group $\Sigma_n$ on $\Conford_n(M)$. We emphasize however that the application $\pi_n$ do \textbf{not} correspond to the quotient by the action of $\Sigma_n$ on $M^n$. Indeed, for $n=3$ and for any $x\neq y\in M$, $\pi_3(x,x,y)=\pi_3(x,y,y)$ but $(x,x,y)$ and $(x,y,y)$ are \textbf{not equivalent} under the action of $\Sigma_3$. The quotient of $M^n$ with respect to the action of $\Sigma_n$ is actually the $n^\text{th}$ symmetric product of $M$, and is not, in general, isomorphic to $\Ran_{\leq n}(M)$. See \cite{ChinenKoyamaCircle} for a discussion of the difference between those two spaces in the case where $M=S^1$. Note that Chinen and Koyama use the terminology of symmetric hyperspaces for what we call the truncations of the Ran space. We will just mention here the illustrative example that all the $n^\text{th}$~symmetric products of $S^1$ are homotopy equivalent to $S^1$, and so is their colimit. In particular, the $3^\text{rd}$ symmetric product of $S^1$ is a solid torus \cite{MortonCircleTorus}. On the contrary $\Ran_{\leq 3}(S^1)$ is homeomorphic to $S^3$, and $\Ran(S^1)$ is a contractible space. 
\end{rem}

We recall here a few known properties on $(\Ran(M),\tpi)$, that we will show can be reinterpreted as direct consequences of \cref{theo:topologie limite}. 
\begin{lem}
\label{lem:complementary of sequence open}
Let $(X_k)_{k\in\N}$ be a sequence in $\Ran(M)$. If the sequence $\card(X_k)$ is strictly increasing, then $\Ran(M)\setminus \{X_k\mid k\in\N\}$ is open in {\tpi}.
\end{lem}
\begin{proof}
    Let $K= \{X_k\mid k\in\N\}$ and $U=\Ran(M)\setminus K$. For $n\geq 1$, we have $\pi_n^{-1}(U)=M^n\setminus\pi_n^{-1}(K)=M^n\setminus\pi_n^{-1}(K\cap \Ran_{\leq n}(M))$. Since the cardinality of $X_k$ is increasing, $K$ contains finitely many configurations of cardinality less than $n$. Since $\pi_n$ has finite fibers, $\pi_n^{-1}(K)$ is a finite subset of $M^n$, and thus $\pi_n^{-1}(U)$ is open for all $n\geq 1$. This proves that $U$ is an open set for {\tpi}.
\end{proof}

\begin{prop}
\label{prop:unbounded cardinal diverges in final}
    Let $(X_k)_{n\in \N}$ be a sequence in $\Ran(M)$. Assume that $\card(X_k)$ is unbounded for $k\in \N$, then the sequence $(X_k)$ does not converge for the final topology.
\end{prop}

\begin{proof}
    By contradiction, assume that $(X_k)_{k\in \N}$ converges toward some limit $X$. We may assume that $\card(X_k)_{k\in \N}$ is a strictly increasing sequence by extracting an appropriate subsequence. Similarly, we may assume that $X\not\in K=\{X_k\mid k\in \N\}$. Then, by \cref{lem:complementary of sequence open} $U=\Ran(M)\setminus K$ is an open set for the final topology.  In particular, $U$ is an open neighborhood of $X$, but $X_k\not\in U$ for all $k\in \N$ which contradicts the assumption that $X_k$ converges to $X$.
\end{proof}

\begin{prop}
\label{prop:final not metrizable}
    Let $M$ be a metric space. Let us suppose that there exists $x\in M$ such that for any $\epsilon>0$ the ball $B(x,\epsilon)$ contains infinitely many points, i.e. $x$ is an accumulation point. Then, the space $(\Ran(M),\tpi)$ is not first countable. In particular, $(\Ran(M),\tpi)$ is not metrizable.
\end{prop}
\begin{proof}
    Let us fix $(U_k)_{k\in \N}$ a countable family of neighborhood of $\{x\}$ in $(\Ran(M),\tpi)$ and let us show that it cannot be a basis of neighborhood. For any $k\in\N$, the preimage $\pi_{k+2}^{-1}(U_k)\subset M^{k+2}$ is an open neighborhood of $\underline{x_k}=(x,\dots,x)$ in $M^{k+2}$. Therefore, there exists $\epsilon>0$ such that $\Bsum(\underline{x_k},\epsilon)\subset\pi_{k+2}^{-1}(U_k)$. By hypothesis, we can find $k+2$ distinct elements $y^{(1)},\dots,y^{(k+2)}$ in $B(x,\frac{\epsilon}{k+2})\subset M$, so that the $k+2$-tuple $(y^{(1)},\dots,y^{(k+2)})\in \Bsum(\underline{x_k},\epsilon)\subset\pi_{k+2}^{-1}(U_k)$. We thus have constructed a configuration $X_k=\pi_{k+2}(y^{(1)},\dots,y^{(k+2)})\in U_k$ of cardinality $\card(X_k)=k+2$. By \cref{lem:complementary of sequence open} the set $U=\Ran(M)\setminus \{X_k\mid k\in\N\}$ is open in $(\Ran(M),\tpi)$. Moreover, since $\card(X_k)=k+2>1$ for any $k\in\N$, we have $\{x\}\in U$, so that $U$ is an open neighborhood of $\{x\}$. But since, for any $k\in\N$, $X_k\not\in U$ then $U_k\not\subset U$. Therefore, $(U_k)_{k\in \N}$ cannot be a basis of neighborhood.
\end{proof}

Finally, the following property will be useful to demonstrate the continuity of maps defined on $\Ran(M)$.
\begin{prop}
\label{prop:contiunuite sur les troncations}
    Let $M,B$ be two spaces, $A\subset\Ran(M)$ a clopen of $(\Ran(M),\tpi)$ equipped with the subspace topology, and  $f\colon A\to B$ an application. Then $f$ is continuous if and only if its restrictions $f_{\leq n}\colon A_{\leq n}\to B$ to the truncations $A_{\leq n}=\{X\in A\mid \card(X)\leq n\}$ are continuous for any $n\in\N$.
\end{prop}
\begin{proof}
    There is nothing to prove for the "only if" case. Let us suppose that the restrictions $f_{\leq n}$ are continuous for all $n\in\N^*$. First, let us write $A=U\cap F$, with $U\subset \Ran(M)$ open and $F\subset\Ran(M)$ closed. Let $V\subset B$ and let us show that $f^{-1}(V)\subset A$ is open. Let $n\geq 1$. Then $f_{n}^{-1}(V)\subset A_{\leq n}$ is open in $A_{\leq n}$. There exists $W\subset \Ran(M)$ open such that $f_{n}^{-1}(V)=W\cap A_{\leq n}=W\cap U\cap F \cap \Ran_{\leq n}(M)$, so that $f_{n}^{-1}(V)$ is open in $F \cap \Ran_{\leq n}(M)$. Therefore $F \cap \Ran_{\leq n}(M)\setminus f_{n}^{-1}(V)$ is closed in $F \cap \Ran_{\leq n}(M)$ and thus in $\Ran_{\leq n}(M)$. But since $F \cap \Ran_{\leq n}(M)\setminus f_{n}^{-1}(V) = \left(F\setminus f^{-1}(V)\right) \cap \Ran_{\leq n}(M)$, this proves that $\left(F\setminus f^{-1}(V)\right) \cap \Ran_{\leq n}(M)$ is closed in $\Ran_{\leq n}(M)$ for all $n\geq 1$. Since the final topology is also induced by the increasing union in \cref{eq:final topology increasing union of truncation}, this proves that $F\setminus f^{-1}(V)$ is closed in $\Ran(M)$ and thus it is also closed in $F$. So that finally $f^{-1}(V)$ is open in $F$ and therefore it is also open in $A$.
\end{proof}

\subsection{Ran spaces and Hausdorff topology}
\label{subsec:HausdorffTopology}
The Hausdorff topology on $\Ran(M)$ can be defined for any topological space $M$ \cite[section 5.5.1]{HigherAlgebra}. Nevertheless, when $M$ is a metric space, the Hausdorff topology is induced by a distance, $\dH$, called the Hausdorff distance. Since we mainly focus on distances, we will only give the definition in that case, which is considerably simpler. 

\begin{defin}
\label{def:Hausdorff}
    The Hausdorff distance on $\Ran(M)$, denoted $\dH$ is defined as
    \begin{align*}
        \dH\colon \Ran(M)\times\Ran(M)&\to \R_+\\
        (X,Y)&\mapsto \max\left\{ \max_{x\in X}\{\dm(x,Y)\} , \max_{y\in Y}\{\dm(X,y)\} \right\},
    \end{align*}
    where $\dm(x,Y)=\min_{y\in Y}\{\dm(x,y)\}$. We call Hausdorff topology, denoted $\tH$, the topology induced by the Hausdorff distance on $\Ran(M)$.
    We will write, for any $X\in\Ran(M)$, $r>0$
    \begin{align*}
        \BH(X,r)=\left\{Y\in\Ran(M)\mid \dH(X,Y)<r \right\}.
    \end{align*}
\end{defin}

\begin{rem}
\label{rem:haudorff on PM}
    The Hausdorff distance can be defined more generally on the power set $\cP(M)$, although it is only a pseudo-distance in this context. We get a distance by restraining to the compact sets of $M$. We will write $\Comp(M)$ the set of non-empty compact subsets of $M$. Note that, if $M$ is a complete metric space, then $(\Comp(M),\dH)$ is a complete metric space \cite[Theorem 2.5.2, p72]{Edgar-fractal-geometry}. This is the subject of the following proposition.
\end{rem}

\begin{prop}
\label{prop:compact sets completion of ran hausdorff}
    If $M$ is a complete metric space, then $(\Comp(M),\dH)$ is the completion of $(\Ran(M),\dH)$.
\end{prop}
\begin{proof}
    Let $(X_k)_{k\in\N}$ be a Cauchy sequence in $(\Ran(M),\dH)$. 
    As shown in \cite[Theorem 2.5.2, p72]{Edgar-fractal-geometry}, the set $X=\{ x : \text{ there is a sequence } (x_k) \text{ with } x_k \in X_k \text{ and } x_k \to x \}$ is a non-empty compact subset of $M$ such that $\dH(X_k,X)$ goes to zero. Which shows that the completion of $(\Ran(M),\dH)$ is a subspace of $(\Comp(M),\dH)$.
    
    Let $K\in\Comp(M)$. Since it is compact, it is precompact, so that for any $k\geq 1$ there exists $X_k\in\Ran(M)$ such that $K\subset \bigcup_{x\in X_k} B(x,\frac{1}{k})$, and $K\cap B(x,\frac{1}{k}) \neq \emptyset$ for any $x\in X_k$. We thus have $\dH(X_k,K)\leq \frac{1}{k}$, so that $(X_k)$ converges to $K$, and $(\Comp(M),\dH)$ is a subspace of the completion of $(\Ran(M),\dH)$.
\end{proof}

Let us compare the two topologies that we just defined. 
\begin{prop}
\label{prop:Final_refines_Hausdorff}
    The final topology is finer that the Hausdorff topology.
\end{prop}
\begin{proof}
Let $X\in\Ran(M)$ and $r>0$. Let us show that $\BH(X,r)$ is open in $\tpi$. To this end, let us fix $n\in\N^*$ and let us show that $\pi^{-1}_n(\BH(X,r))$ is open in $M^n$. Let $\underline{y}=(y^{(1)},\dots,y^{(n)})\in \pi^{-1}_n(\BH(X,r))$, and $Y=\pi_n(\underline{y})$. Let us define $\epsilon = r-\dH(X,Y)>0$ and let us show that $B(y^{(1)},\epsilon)\times\dots\times B(y^{(n)},\epsilon)$ is an open neighborhood of $\underline{y}$ in $\pi^{-1}_n(\BH(X,r))$. Let $\underline{z}=(z^{(1)},\dots,z^{(n)})\in B(y^{(1)},\epsilon)\times\dots\times B(y^{(n)},\epsilon)$. For any $i\in\{1,\dots,n\}$, we have 
\begin{align*}
    \dm(y^{(i)},Z)\leq \dm(y^{(i)},z^{(i)})<\epsilon,
\end{align*}
and
\begin{align*}
    \dm(z^{(i)},Y)\leq \dm(z^{(i)},y^{(i)})<\epsilon.
\end{align*}
So that 
\begin{align*}
    \dH(Y,Z)=\max\{\max_{y\in Y}\{\dm(y,Z)\}, \max_{z\in Z}\{\dm(Y,z)\}\} <\epsilon
\end{align*}
and finally 
\begin{align*}
    \dH(X,Z)\leq \dH(X,Y)+\dH(Y,Z) <\dH(X,Y)+\epsilon <r,
\end{align*}
which proves that $\underline{z}\in \pi^{-1}_n(\BH(X,r))$. We thus found an open neighborhood of $\underline{y}$ in $\pi^{-1}_n(\BH(X,r))$, i.e. $\pi^{-1}_n(\BH(X,r))$ is open in $M^n$ for any $n\in\N^*$, and thus $\BH(X,r)$ is open in \tpi.

In particular, we showed that the identity $\Id\colon (\Ran(M),\tpi)\to(\Ran(M),\tH)$ is continuous.

\end{proof}
\begin{rem}
\label{rem:hausdorff and final differ}
As soon as $M$ has an accumulation point, we showed in \cref{prop:final not metrizable} that the final topology is not first countable, in particular it is not metrizable. Therefore, the two topologies differ, and the final topology is strictly finer that the Hausdorff topology. Besides, note that the proof of \cref{prop:final not metrizable} defines a sequence of increasing cardinality in $\Ran(M)$, which cannot converge for the final topology, by \cref{prop:unbounded cardinal diverges in final}. Nevertheless, it does converge to a singleton for the Hausdorff topology. This important distinction between the final topology and the metric topologies that we can define on $\Ran(M)$ will be at the heart of \cref{sec:completeness}.
\end{rem}

\subsection{The canonical stratification}
\label{sec:Prelim_Strat}
Now that we have topologies on $\Ran(M)$ we can discuss continuous map to and from Ran spaces. The first one we are interested in is the canonical stratification on $\Ran(M)$ given by the cardinality function. Let us first recall what we mean by a stratification.

\begin{defin}
\label{defin:Strat_Space}
Let $P$ be a partially ordered set. It can be equipped with a topology, usually called the Alexandrov topology, whose opens are exactly those sets $U\subset P$ satisfying the property
\begin{equation*}
    \left(p\in U \text{ and } q\geq p\right) \Rightarrow q\in U
\end{equation*}
A stratification on a topological space $T$ is the data of a poset $P$ together with a map $f\colon T\to P$ which is continuous for the Alexandrov topology. Given such a structure, we may say that $T$ is a stratified space, or that $T$ is stratified over the poset $P$.
Given any $p\in P$, the preimage $f^{-1}(p)\subset T$ will be called the $p$ stratum of $T$.
\end{defin}

In the case of $\Ran(M)$ the stratification we will be interested in is given by the poset of positive integers, $\N_{\geq 1}$, together with the map $\card\colon \Ran(M)\to \N_{\geq 1}$ sending a configuration $X\subset M$ to its cardinality. In that example, the open sets of $\N_{\geq 1}$ can be seen to be exactly the sets of the form $\N_{\geq p}=\{n\in \N\mid n\geq p\}$, for all $p\in \N_{\geq 1}$. 

We may now precisely state in what way $\Ran(M)$ is a stratified space.

\begin{prop}
    The cardinality map $\card\colon \Ran(M)\to \N_{\geq 1}$ is continuous on $(\Ran(M),\tH)$ and thus also on $(\Ran(M),\tpi)$. In particular, they both are stratified spaces.
\end{prop}
\begin{proof}
    Let $n\in\N^*$ and let us show that $\card^{-1}(\N_{\geq n})=\Ran(M)\setminus\Ran_{\leq n-1}(M)$ is open for $\tH$. If $n=1$, there is nothing to prove, otherwise, let $X\in\Ran(M)\setminus\Ran_{\leq n-1}(M)$. We have $\card(X)\geq n >1$, so that we can define $\epsilon=\frac{1}{2}\min_{x\neq y \in X}\{\dm(x,y)\}>0$. Let us show that $\BH(X,\epsilon)\subset \Ran(M)\setminus\Ran_{\leq n-1}(M)$. Let $Y\in\BH(X,\epsilon)$. For any $x\in X$, there exists $y_x\in Y$ such that
    \begin{align*}
        \dm(x,y_x)=\dm(x,Y)\leq \dH(X,Y) <\epsilon.
    \end{align*}
    Moreover, if $x'\neq x \in X$, we have
    \begin{align*}
        \dm(x',y_x)\geq \dm(x',x) - \dm(x,y_x)> \min_{z\neq z'\in X}\{\dm(z,z')\} - \epsilon =\epsilon,
    \end{align*}
    and thus $y_x\neq y_{x'}$. As a result $\card(Y)\geq \card(X)$, so that $Y\in\Ran(M)\setminus\Ran_{\leq n-1}(M)$ and thus $\BH(X,\epsilon)\subset \Ran(M)\setminus\Ran_{\leq n-1}(M)$, which concludes this proof.
\end{proof}

\begin{rem}
    Since the weighted topologies we will introduce in \cref{sec:Combinatorial_Distance} also refine the Hausdorff topology, it will also be the case that the corresponding $(\Ran(M),\tom)$ are stratified spaces.
\end{rem}

\begin{rem}
    The configuration spaces $\Conf_n(M)$ are the strata of the stratified space $\Ran(M)\to\N_{\geq 1}$. Note that the topology on these strata is the same whether we consider the topology $\tH$ or $\tpi$ on $\Ran(M)$.
\end{rem}

We also deduce from the previous proposition the following result
\begin{corol}
\label{cor:truncations closed}
For any $n\geq 1$, the truncation $\Ran_{\leq n}(M)$ is closed in $\Ran(M)$.
\end{corol}

We will discuss in much more detail the properties of those stratifications in \cref{sec:conicity}, but for now, let us detail some properties of the truncations.

\subsection{Truncations of the Ran space}
\label{subsec:prelim_truncations}
Even though the Hausdorff and final topologies are quite different in general, they happen to coincide on the truncations of $\Ran(M)$, which is given by the following proposition.

\begin{prop}
\label{prop:topologie_troncation}
    Let $n\in\N^*$. The subspace topologies induced by $\tH$ and $\tpi$ on $\Ran_{\leq n}(M)$ coincide.
\end{prop}

\begin{proof}
    Since $\tpi$ refines $\tH$, it is enough to show that open subsets of $\tpi$ restricts to open subsets of the subspace topology induced by $\tH$ on any truncation.
    Let $U\subset \Ran(M)$ be an open in $\tpi$. Let $n\in\N^*$, and let us show that $U_{\leq n}=U\cap \Ran_{\leq n}(M)$ is open for the topology induced by $\tH$ on $\Ran_{\leq n}(M)$. Let $X\in U_{\leq n}$, and $r=\min(\{\frac{d(x,x')}{2} \mid x\neq x'\in X\}\cup \{1\})$. By definition of {\tpi}, $\pi_n^{-1}(U)$ is open in $M^n$, so that, for any $\underline{x}\in \pi_n^{-1}(X)$, there exists $0<\epsilon_{\underline{x}}\leq r$ such that $\Bsum(\underline{x},\epsilon_{\underline{x}})\subset \pi_n^{-1}(U)$. Now let us define $\epsilon=\min\{\epsilon_{\underline{x}}\mid  \underline{x}\in \pi_n^{-1}(X)\}$ and let us show that $\BH(X,\frac{\epsilon}{n})\cap\Ran_{\leq n}(M)\subset U_{\leq n}$.
    
    Let $Y\in \BH(X,\frac{\epsilon}{n})\cap\Ran_{\leq n}(M)$, and let $\underline{y}=(y^{(1)},\dots,y^{(n)})\in\pi_n^{-1}(Y)$. For all $1,\leq i \leq n$ there exists $x^{(i)}\in X$ such that $d(x^{(i)},y^{(i)})=d(X,y^{(i)})\leq \dH(X,Y)$. Now observe that, since $\dH(X,Y)<r$, $x^{(i)}$ must be unique. Indeed if $x\neq x^{(i)}\in X$ then
    \begin{align*}
        d(x,y^{(i)})\geq d(x,x^{(i)})-d(x^{(i)},y^{(i)})\geq d(x,x^{(i)}) - \dH(X,Y) > r > \dH(X,Y).
    \end{align*}
    Besides, remark that, for any $x\in X$, there exists $y\in Y$ such that $d(x,y)\leq \dH(X,Y)$, so that $x=x^{(i)}$ for some $1\leq i \leq n$, and thus $\underline{x}=(x^{(1)},\dots,x^{(n)})\in\pi_n^{-1}(X)$. Finally, we have
    \begin{align*}
        \dsum(\underline{x},\underline{y})=\sum_{i=1}^n d(x^{(i)},y^{(i)}) \leq n\dH(X,Y) < \epsilon,
    \end{align*}
    so that $\underline{y}\in \Bsum(\underline{x},\epsilon)\subset \pi_n^{-1}(U)$ and thus $Y=\pi_n(\underline{y})\in U_{\leq n}$ which proves that $\BH(X,\frac{\epsilon}{n})\cap\Ran_{\leq n}(M)\subset U_{\leq n}$ and thus that $U_{\leq n}$ is open for the topology induced by $\tH$ on $\Ran_{\leq n}(M)$.
\end{proof}
\begin{rem}
This result illustrates that most of the interesting subtleties appear only when the cardinality of configurations tends to infinity. In particular, as seen in \cref{prop:unbounded cardinal diverges in final} sequences of unbounded cardinality cannot converge for the final topology but may do so for the Hausdorff topology. This intuitively explains why the comparison of different topologies will often involve sequences of increasing cardinality in the following.  
\end{rem}

Finally, we will need, in \cref{sec:completeness}, the following result to discuss the completeness of the different topologies on $\Ran(M)$. 

\begin{prop}
\label{prop:hausdorff_complete_troncation}
    Let $M$ be a complete metric space, and let $n\geq 1$. The truncation $\Ran_{\leq n}(M)$ equipped with the Hausdorff distance is a complete metric space.
\end{prop}

\begin{proof}
Let $(X_k)_{k\in\N}$ be a Cauchy sequence in $(\Ran_{\leq n}(M),\tH)\subset \Comp(M)$. 
Since $(\Comp(M),\dH)$ is complete \cite[Theorem 2.5.2, p72]{Edgar-fractal-geometry}, there exists a compact subset $X\subset M$ such that $\dH(X_k,X)$ goes to zero.
Let us now show that $\card{X}\leq n$. By contradiction, suppose that $\card(X)\geq n+1$. But then, there exists $Y\subset X$ such that $\card(Y)=n+1$, and, by \cref{lem:minoration Hausdorff strate}, there exists $\delta_Y>0$ such that, for any $k\geq 0$ 
\begin{align*}
    \dH(X_k,X)&\geq \sup_{y\in X}\{\dm(X_k,y)\}\\
    &\geq \max_{y\in Y}\{\dm(X_k,y)\}\\
    &\geq \delta_Y
\end{align*}
which is absurd since $\dH(X_k,X)$ converges to zero. We finally deduce that $(X_k)_{k\in\N}$ converges to $X$ in $(\Ran_{\leq n}(M),\tH)$.
\end{proof}

\begin{lem}
\label{lem:minoration Hausdorff strate}
Let $n\geq 1$ and let $X\in\Ran(M)$ such that $\card(X)=n+1$. Then there exists $\delta_X$ such that, for any $Y\in\Ran_{\leq n}(M)$, we have
\begin{align*}
    \max_{x\in X} \{\dm(Y,x)\}\geq \delta_X.
\end{align*}
\end{lem}
\begin{proof}
    We set $\delta_X=\frac{1}{2}\min_{x\neq x'\in X}\{\dm(x,x')\}>0$. Let $Y\in\Ran_{\leq n}(M)$. For any $x\in X$ there exists $y_x\in Y$ such that $\dm(y_x,x)=\dm(Y,x)$. Since $\card(Y)<\card(X)$ there exists $x\neq z\in X$ such that $y_x=y_z$. But then $\dm(x,y_x)+\dm(y_z,z)\geq \dm(x,z)\geq 2\delta_X$. Hence
    \begin{align*}
        \delta_X\leq \max\{\dm(x,y_x),\dm(z,y_z)\}\leq \max_{w\in X}\{\dm(Y,w)\}.
    \end{align*}
\end{proof}

\begin{rem}
    Even though all its truncations are complete, $\Ran(M)$ equipped with the Hausdorff distance is not, in general, complete. We give later in \cref{exa:Cantor} an example of a sequence in $\Ran(M)$ which is Cauchy for the Hausdorff distance, and that converges to a Cantor set in $\cP(M)$.
\end{rem}

\subsection{Length structures and length spaces}
\label{sec:PathMetricSpaces}
We recall the notions of length structures and length spaces from metric geometry. We follow \cite[chapter 2]{CourseMetricGeometry}, in which numerous examples can be found.

\begin{defin}
\label{def:Length_structure}
    Given a topological space $M$, a \textbf{length structure} on $M$ is the data of 
    \begin{itemize}
        \item a set of admissible paths $\cP$, whose elements are continuous maps $\gamma\colon [a,b]\to M$, with $a,b\in \R$,
        \item a map $\ell\colon \cP\to \R_+$
    \end{itemize} 
    such that, for any continuous path $\gamma\colon [a,b]\to M$, we have 
    \begin{itemize}
        \item     given any $c\in [a,b]$, we have
            \begin{equation*}
        \gamma\in \cP\Leftrightarrow \gamma_{|[a,c]}\in \cP\text{ and }\gamma_{|[c,b]}\in \cP,
    \end{equation*}
    and in that case, 
    \begin{equation*}
        \ell(\gamma)=\ell(\gamma_{|[a,c]})+\ell(\gamma_{|[c,b]})
    \end{equation*}
    \item given any affine homeomorphism $\varphi\colon [c,d]\to [a,b]$, we have
    \begin{equation*}
        \gamma\in \cP\Leftrightarrow \gamma\circ \varphi \in \cP
    \end{equation*}
    and in that case $\ell(\gamma)=\ell(\gamma\circ \varphi)$
    \item Given $\gamma\colon [a,b]\to \R$ in $\cP$, the following map is continuous
    \begin{align*}
        [a,b]&\to \R_+\\
        t&\mapsto \ell(\gamma_{|[a,t]})
    \end{align*}
    \item Given $x\in X$, and $U\subset X$ a neighborhood of $X$, there exists some $\epsilon>0$ such that, for any $\gamma\colon [a,b]\to X$ in $\cP$ such that $\gamma(a)=x$ and $\gamma(b)\not \in U$, $\ell(\gamma)>\epsilon$.
    \end{itemize}
\end{defin}

\begin{exa}
If $M$ is a Riemmanian manifold, then $\cP$ can be taken to be the set of piecewise $C^1$ paths, and $\ell(\gamma)$ is the integral of the norm of the derivative along the path (i.e. its length in the usual sense).
\end{exa}

\begin{defin}
    Let $M$ be a Hausdorff topological space equipped with a length structure $\cP$, $\ell$, then we define the distance
    \begin{equation*}
        d^{\ell}(x,y)=\inf\{\ell(\gamma)\mid \gamma\in \cP, \ x,y\in \Ima(\gamma)\}
    \end{equation*}
\end{defin}

\begin{rem}\label{rem:checking_length_structure}
    The fact that $d^{\ell}$ is a distance, instead of a pseudo-distance is guaranteed by the last condition in \cref{def:Length_structure}, which will further impose that the metric topology refines the topology we started with. Conversely, if we drop this last condition, and assume that $d^{\ell}$ is indeed a distance, and that it refines the topology on $M$ we started with, we can recover this condition as follows. Observe that by hypothesis, any neighborhood of $x$ must contain an open ball of some radius $\epsilon>0$ around $x$ for the distance $d^{\ell}$, but this immediately implies the condition, since any path leaving such a ball must have length at least $\epsilon$. We will use this observation in \cref{sec:PathMetricsOnRan} to define a length structure on $\Ran(M)$.
\end{rem} 

\begin{defin}\label{defin:PathMetricSpaces}
    A \textbf{length space} is a topological space $M$ equipped with a length structure $\cP$, $\ell$, such that $d^{\ell}$ is a distance which induces the topology of $M$. In that case, we say that $d^{\ell}$ is a length metric. 
\end{defin}

\begin{exa}
    Any Riemannian manifold is a length space for the length structure outlined in the previous example, since the topology induced by the Riemannian metric is indeed the underlying topology of the manifold.
\end{exa}

\begin{defin}
    Let $M$ be a topological space equipped with a length structure, $\cP$, $\ell$. For $n\geq 0$, write $\cP ^n=\{(\gamma^{(1)},\dots,\gamma^{(n)})\colon [a,b]\to M^n\mid \gamma_i\in \cP,\ 1\leq i\leq n\} $ and $\ell^n(\gamma^{(1)},\dots,\gamma^{(n)})=\sum_{i=1}^n\ell(\gamma^{(i)})$.
\end{defin}

We immediately get, from the definition, the following proposition, which will be the starting point of \cref{sec:PathMetricsOnRan}.
\begin{prop}\label{prop:Length_structure_Mn}
    Let $M$ be a length space, with length structure $\cP$ and $\ell$. Let $n\geq 1$, then $M^n$ is a length space for the length structure $\cP^n$ and $\ell^n$. We write $\dsum=d^{\ell^n}$ and remark that
    \begin{equation*}
        \dsum((x^{(1)},\dots,x^{(n)}),(y^{(1)},\dots,y^{(n)}))=\sum_{i=1}^nd^\ell(x^{(i)},y^{(i)})
    \end{equation*}
\end{prop}

\section{Combinatorial distance}
\label{sec:Combinatorial_Distance}
In this section we define new distances on $\Ran(M)$ indexed by some non-decreasing sequence $\omega\colon \N\to [1,\infty)$ that we call a weight. First we define surjective relations between two configurations in $\Ran(M)$, to which we attribute a weighted length $\elo$. Then we define combinatorial chains between two configurations as the concatenation of surjective relations, to which we also attribute a weighted length by summing the weighted length of the relations. Finally the weighted distance between two configurations is defined as the infimum of the lengths of combinatorial chains between the two configurations. We show that this indeed defines a distance on $\Ran(M)$, and then, the remaining of the section is dedicated to refining the computation of the distance. Eventually we show in \cref{theo:MBS} that we obtain the same infimum by restricting to chains of a specific shape, that we call MBS (merge-bijection-splitting), see \cref{def:MBS}, where we first merge points together, then possibly move them around and finally split the points. This result allows us to show in \cref{theo:Geodesic_Chain} that, if $M$ satisfies the Heine-Borel property, then this infimum is computed on a compact space, and is thus attained for a given chain that we call geodesic.

Throughout, $(M,d)$ is a metric space.

\subsection{Definitions}
\label{sec:Comb_Distance_def}
In this section, we define the weighted distance on $\Ran(M)$ and show that it is a distance. First we define the combinatorial notions that will be central in the following of this work: surjective relations and combinatorial chains.
\begin{defin}
    Let $X,Y\in \Ran(M)$ be two configurations. A \textbf{relation} between $X$ and $Y$ is the data of a subset $R\subset X\times Y$. We will write $xRy$ for $(x,y)\in R$. It is said to be \textbf{surjective}, if 
    \begin{itemize}
        \item $\forall x\in X$, $\exists y\in Y$, such that $xRy$, and
        \item $\forall y\in Y$, $\exists x\in X$, such that $xRy$.
    \end{itemize}
    The \textbf{cardinality} of the relation $R$, denoted $\card(R)$, is simply the cardinality of the set $R\subset X\times Y$. We will write $R^{\op}\subset Y\times X$ for the relation defined by $yR^{\op}x\Leftrightarrow xRy$. 
\end{defin}
\begin{defin}
\label{def:combinatorial chain}
     Given an integer $k\geq 0$, a \textbf{combinatorial chain} in $M$ (with $k$ terms) is the data of
    \begin{itemize}
        \item for all $0\leq i\leq k$ a configuration $X_i\in \Ran(M)$,
        \item for all $0\leq j\leq k-1$, a surjective relation $R_j$ between $X_j$ and $X_{j+1}$.
    \end{itemize}
    We will write such a chain $((X_0,\dots,X_k),(R_0,\dots,R_{k-1}))$, or simply $(R_0,\dots,R_{k-1})$. And we will say that it is a chain between $X_0$ and $X_k$, or from $X_0$ to $X_k$.
\end{defin}
We can now assign very naturally a \textbf{length} to relations and combinatorial chains, as if they were some sort of generalized "paths", but with their length computed "as the crow flies".
\begin{defin}
\label{def:relation length}
    Given a relation $R$ between $X$ and $Y$ in $M$, we define the \textbf{length} of the relation $R$ as
    \begin{equation*}
        \ell(R)=\sum\limits_{(x,y)\in R}d(x,y).
    \end{equation*}
    Given  $(R_0,\dots,R_{k-1})$, a combinatorial chain on $M$, its length is defined as
    \begin{equation*}
        \ell(R_0,\dots,R_{k-1})=\sum_{j=0}^{k-1}\ell(R_j)
    \end{equation*}
\end{defin}
Recall that, as explained in the introduction when discussing possible answers to \cref{quest:Converge_or_not}, we want some way to modulate our distances depending on the cardinality of the configurations. This is achieved by weighting the length by a cardinality-dependent weight. 
\begin{defin}
    A \textbf{weight} is a sequence $\omega\colon \N^*\to [1,\infty)$, such that
    \begin{itemize}
        \item $\omega(1)=1$,
        \item $\omega$ is non-decreasing.
    \end{itemize}
\end{defin}
\begin{defin}
\label{def:relation length omega}
    Given a relation $R$ between $X$ and $Y$ in $M$, and a weight $\omega$, we define the \textbf{length} of the relation $R$, \textbf{relative to $\omega$} as
    \begin{equation*}
        \elo(R)=\omega(\card(R))\ell(R).
    \end{equation*}
    Given  $(R_0,\dots,R_{k-1})$, a combinatorial chain on $M$, its length relative to $\omega$ is defined as
    \begin{equation*}
        \elo(R_0,\dots,R_{k-1})=\sum_{j=0}^{k-1}\elo(R_j)
    \end{equation*}
\end{defin}
\begin{rem}
    If a relation $R\subset X\times Y$ is the graph of a map $f\colon X\to Y$, then \cref{def:relation length omega} gives
    \begin{equation*}
        \elo(R)=\elo(f)=\omega(\card(X))\sum_{x\in X}\dm(x,f(x)).
    \end{equation*}
\end{rem}
Finally, the weighted distance is computed, quite intuitively, as the weighted length of the "shortest generalized path" between configurations. 
\begin{defin}
\label{def:weighted-distance}
    Let $X,Y\in \Ran(M)$ be two configurations, and $\omega$ some weight. We define the \textbf{weighted distance} between $X$ and $Y$ as the following infimum
    \begin{equation*}
        \dom(X,Y)=\inf\{\elo(R_0,\dots,R_{k-1})\}
    \end{equation*}
    Where the $(R_0,\dots,R_{k-1})$ ranges over all combinatorial chains between $X$ and $Y$.
\end{defin}
There remains to be proven that this indeed defines a distance on $\Ran(M)$.
\begin{prop}
    Given some weight $\omega$, the combinatorial distance $\dom$ on $\Ran(M)$ is a distance.
\end{prop}
\begin{proof}
    Let $X$, $Y$ and $Z$ be three configurations in $\Ran(M)$. \begin{itemize}
    \item  Given $(R_0,\dots,R_{k-1})$, a combinatorial chain between $X$ and $Y$ , $(R_{k-1},\dots,R_0)$ is a combinatorial chain between $Y$ and $X$, with $\elo(R_0,\dots,R_{k-1})=\elo(R_{k-1},\dots,R_0)$. Hence, $\dom(X,Y)=\dom(Y,X)$. 
        \item The identity relation on $X$ gives $\dom(X,X)=0$. Conversely, assume that $X\not=Y$. By symmetry, assume that there exists $x\in X$ such that $x\not\in Y$, let $y\in Y$ be such that $d(x,y)\leq d(x,y')$ for any $y'\in Y$. Consider a combinatorial chain between $X$ and $Y$, $((X_0,\dots,X_k),(R_0,\dots,R_{k-1}))$. Surjectivity of each of the $R_j$ implies the existence of a sequence $x_0,\dots,x_k$ such that $x=x_0$,  $x_i\in X_i$, and $x_{i-1}R_{i-1}x_i$, for $1\leq i\leq k$. But then, $\elo(R_0,\dots,R_{k-1})\geq \sum_{i=0}^{k-1}d(x_i,x_{i+1})\geq d(x,x_k)\geq d(x,y)$. Hence, $\dom(X,Y)\geq d(x,y)>0$, which proves that $\dom(X,Y)=0\Leftrightarrow X=Y$.
        \item Finally, given two combinatorial chains $(R_0,\dots,R_{k-1})$ and $(P_0,\dots,P_{l-1})$ between $X$ and $Y$ and between $Y$ and $Z$ respectively, the concatenation $(R_0,\dots,R_{k-1},P_0,\dots,P_{l-1})$ is a chain between $X$ and $Z$ of length
    \begin{equation*}
\elo(R_0,\dots,R_{k-1},P_0,\dots,P_{l-1})=\elo(R_0,\dots,R_{k-1})+\elo(P_0,\dots,P_{l-1})
    \end{equation*}
    Hence, $\dom(X,Z)\leq \dom(Y,Z)+\dom(X,Y)$. Which concludes the proof.
    \end{itemize}
\end{proof}

\subsection{Refining the computation}

We now have a definition for the weighted distance in $\Ran(M)$. However, its definition is, at the same time, very general, but very impractical, so that the distance is difficult to evaluate, or manipulate. Since it is defined as an infimum, finding a "relevant" upper bound is as easy as finding a "relevant" combinatorial chain. On the contrary, finding a relevant lower bound can be much harder, since the set of combinatorial chains between two configurations is quite large. Over the next two subsections, we will show that one may actually restrict the computation of the infimum to very specific combinatorial chains, (see \cref{theo:MBS}). Since we want to reduce the set of combinatorial chains that we need to consider, the strategy is always to construct, from a surjective relation or a combinatorial chain between two configurations, another, shorter, surjective relation or combinatorial chain between the same configurations, and with better properties. In particular, we want the "good" chains to be decomposed into relations that either merge points together, and reduce cardinality, or that do the opposite. These relations are either the graph of a surjective map, or their opposite is the graph of a surjective map. In addition, we also want to control the cardinality of the relations, this is thus another central aspect in the following.

For the rest of this section, we fix $\omega$ a weight.

We first show a quite intuitive lemma that will be useful throughout this article.
\begin{lem}\label{lem:included relation shorter}
Let $R\subset X\times Y$ be a surjective relation between two configurations $X,Y\in \Ran(M)$. If $R'\subset R$ is a surjective relation between $X'\subset X$ and $Y'\subset Y$ then $\elo(R')\leq\elo(R)$. 
\end{lem}
\begin{proof}
Since $R'\subset R$, we have $\card(R')\leq\card(R)$. So that we have
    \begin{align*}
            \elo(R')&=\omega(\card(R'))\sum_{(x,y)\in R'}d(x,y)\\
            &\leq\omega(\card(R))\sum_{(x,y)\in R}d(x,y)\\
            &\leq \elo(R).
        \end{align*}
\end{proof}
\begin{rem}
This result being very intuitive, we will most often omit the explicit call to this lemma.
\end{rem}

We then show that we can get an upper bound for the cardinality of the relations of the considered chains.
\begin{lem}\label{lem:Relation_Cardinal_Bounded}
    Let $R\subset X\times Y$ be a surjective relation between two configurations $X,Y\in \Ran(M)$. If $\card(R)>\max\{\card(X),\card(Y)\}$, then there exists a combinatorial chain between $X$ and $Y$, $((Y_0,\dots,Y_k),(P_0,\dots,P_{k-1}))$, such that
    \begin{itemize}
        \item $\card(P_j)<\card(R)$, $0\leq j\leq k-1$,
        \item $\elo(P_0,\dots,P_{k-1})\leq \elo(R)$,
        \item $\card(Y_i)\leq \max(\card(X),\card(Y))$, for all $0\leq i\leq k$.
    \end{itemize}
\end{lem}

\begin{proof}
    By assumption, since $\card(R)>\card (X)$,
    there must exist $x_0\in X$ and $y_1\not=y_2\in Y$ such that $x_0 R y_1$ and $x_0 R y_2$.
    Similarly, since $\card(R)>\card(Y)$,
    there must also exist $x_1\not=x_2\in X$ and $y_0\in Y$
    such that $x_1Ry_0$ and $x_2Ry_0$. By symmetry, we distinguish three cases :
    \begin{itemize}
        \item If $x_0=x_1$ and $y_0=y_1$, then $P=R\setminus\{(x_0,y_0)\}$ is still a surjective relation, since $x_1Py_2$ and $x_2Py_1$, and we can interpret it as a combinatorial chain with a single relation. Furthermore, by construction $\card(P)=\card(R)-1<\card(R)$, and, since $P\subset R$, we have, by \cref{lem:included relation shorter}, $\elo(P)\leq \elo(R)$.
        \item $x_0\not=x_1,x_2$ and $y_0=y_1$. This case is similar to the previous one, and $P=R\setminus\{(x_0,y_0)\}$ fits.
        \item $x_0\not=x_1,x_2$, and $y_0\not=y_1,y_2$. We first consider the subcase where $x_0\in Y\setminus\{y_1,y_2\}$. In that case, surjectivity of $R$ implies the existence of some $x_3\in X$ such that $x_3Rx_0$. Define $Q\subset X\times Y$ as the relation $Q=R\setminus\{(x_3,x_0),(x_0,y_2)\}\cup\{(x_0,x_0),(x_3,y_2)\}$. Note that $Q$ is still a surjective relation, by construction, and that we have $\card(Q)\leq\card(R)$, hence
        \begin{align*}
            \elo(Q)&=\omega(\card(Q))\sum_{(x,y)\in Q}d(x,y)\\
            &\leq \omega(\card(Q))\left(\left(\sum_{(x,y)\in R}d(x,y)\right)+d(x_3,y_2)-d(x_3,x_0)-d(x_0,y_2)\right) \\
            &\leq \omega(\card(Q))\left(\sum_{(x,y)\in R}d(x,y)\right)\\
            &\leq \omega(\card(R))\left(\sum_{(x,y)\in R}d(x,y)\right)\\
            &\leq \elo(R)
        \end{align*}
        Now, either $\card(Q)<\card(R)$ and we are done, or we reduced to the case where $x_0\not\in Y\setminus\{y_1,y_2\}$ by considering $Q$ instead of $R$ and $y_1,x_0$ instead of $y_1,y_2$. 
        
        In the latter case, where $x_0\not\in Y\setminus\{y_1,y_2\}$, we may freely assume that there exists no $x_3\not = x_0\in X$ such that $x_3Ry_1$ or $x_3Ry_2$. Indeed, if such an element of $X$ exists, then up to relabelling the corresponding $y_i$ by $y_0$ and $x_0$ and $x_3$ as $x_1$ and $x_2$, the situation is covered by the first case. So, we can set $Z=Y\setminus\{y_1,y_2\}\cup \{x_0\}$, and define two relations, $P_0\subset X\times Z$ and $P_1\subset Z\times Y$ by :
        \begin{align*}
xP_0z&\Leftrightarrow (z\in Y\setminus\{y_1,y_2\} \text{ and } xRz) \text{ or } (x=z=x_0)\\
zP_1y&\Leftrightarrow (z=y)\text{ or } \left((z=x_0) \text{ and } (y=y_1 \text{ or }y=y_2)\right)
        \end{align*}
        Note that, by construction, $\card(Z)\leq \card(Y)$, 
        and that $\card(P_0)\leq \card(R)-1$, and $\card(P_1)=\card(Y)<\card(R)$. By the above remark, $P_0$ is still a surjective relation, and we compute
        \begin{align*}
            \elo(P_0,P_1)&=\omega(\card(P_0))\left(\sum_{(x,z)\in P_0}d(x,z)\right)+\omega(\card(P_1))\left(\sum_{(z,y)\in P_1}d(z,y)\right)\\
            &\leq \omega(\card(R))\left(\sum_{(x,z)\in P_0}d(x,z)\right)+\omega(\card(R))\left(d(x_0,y_1)+d(x_0,y_2)\right)\\
            &\leq \omega(\card(R))\left(\sum_{(x,y)\in R}d(x,y)\right)\\
            &\leq \elo(R).
        \end{align*}
        This concludes the proof. See \cref{fig:Proof_Relation_Cardinal_Bounded}, for an illustration.  
    \end{itemize}
\end{proof}

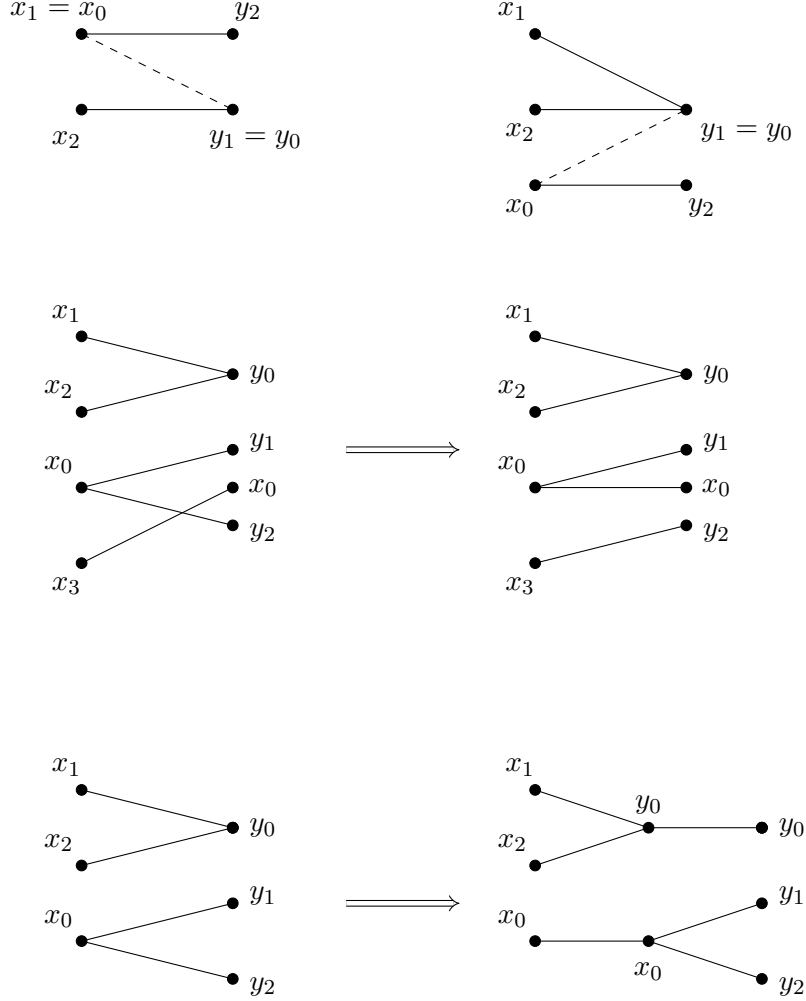
\begin{figure}
    \centering
\begin{tikzpicture}
    \filldraw  (0,1) circle (2 pt);
    \filldraw (0,0) circle (2pt);
    \filldraw (2,0) circle (2pt);
    \filldraw (2,1) circle (2pt);
    \draw (0,1)--(2,1);
    \draw (0,0)--(2,0);
    \draw[dashed] (0,1)--(2,0);
    \node at (-0.3,1.3) {$x_1=x_0$};
    \node at (-0.2,-0.4) {$x_2$};
    \node at (2.2,1.3) {$y_2$};
    \node at (2.3,-0.4) {$y_1=y_0$};

    \filldraw[shift={(6,0)}]  (0,1) circle (2 pt);
    \filldraw[shift={(6,0)}] (0,0) circle (2pt);
    \filldraw[shift={(6,0)}] (2,0) circle (2pt);
    \filldraw[shift={(6,0)}] (0,-1) circle (2pt);
    \filldraw[shift={(6,0)}] (2,-1) circle (2pt);
    \draw[shift={(6,0)},dashed] (0,-1)--(2,0);
    \draw[shift={(6,0)}](0,-1)--(2,-1);
    \draw[shift={(6,0)}] (0,0)--(2,0);
    \draw[shift={(6,0)}] (0,1)--(2,0);
    \node[shift={(6,0)}] at (-0.3,1.3) {$x_1$};
    \node[shift={(6,0)}] at (-0.2,-0.3) {$x_2$};
    \node[shift={(6,0)}] at (2.2,-1.3) {$y_2$};
    \node[shift={(6,0)}] at (2.8,-0.3) {$y_1=y_0$};
    \node[shift={(6,0)}] at (-0.2,-1.3) {$x_0$};

    \filldraw[shift={(0,-6)}]  (0,0) circle (2 pt);
    \filldraw[shift={(0,-6)}] (0,1) circle (2pt);
    \filldraw[shift={(0,-6)}] (0,2) circle (2pt);
    \filldraw[shift={(0,-6)}] (0,3) circle (2pt);
    \filldraw[shift={(0,-6)}] (2,1) circle (2pt);
    \filldraw[shift={(0,-6)}] (2,0.5) circle (2pt);
    \filldraw[shift={(0,-6)}] (2,1.5) circle (2pt);
    \filldraw[shift={(0,-6)}] (2,2.5) circle (2pt);
    \draw[shift={(0,-6)}] (0,3)--(2,2.5);
    \draw[shift={(0,-6)}](0,2)--(2,2.5);
    \draw[shift={(0,-6)}] (0,1)--(2,1.5); 
    \draw[shift={(0,-6)}] (0,1)--(2,0.5);
    \draw[shift={(0,-6)}] (0,0)--(2,1);
    \node[shift={(0,-6)}] at (-0.3,1.3) {$x_0$};
    \node[shift={(0,-6)}] at (-0.2,-0.3) {$x_3$};
      \node[shift={(0,-6)}] at (-0.3,2.3) {$x_2$};
    \node[shift={(0,-6)}] at (-0.2,3.3) {$x_1$};
    \node[shift={(0,-6)}] at (2.4,1) {$x_0$};
    \node[shift={(0,-6)}] at (2.4,0.4) {$y_2$};
      \node[shift={(0,-6)}] at (2.4,1.6) {$y_1$};
    \node[shift={(0,-6)}] at (2.4,2.5) {$y_0$};

    \draw[shift={(0,-6)},-implies,double equal sign distance](3.5,1.5)--(5,1.5);

    \filldraw[shift={(6,-6)}]  (0,0) circle (2 pt);
    \filldraw[shift={(6,-6)}] (0,1) circle (2pt);
    \filldraw[shift={(6,-6)}] (0,2) circle (2pt);
    \filldraw[shift={(6,-6)}] (0,3) circle (2pt);
    \filldraw[shift={(6,-6)}] (2,1) circle (2pt);
    \filldraw[shift={(6,-6)}] (2,0.5) circle (2pt);
    \filldraw[shift={(6,-6)}] (2,1.5) circle (2pt);
    \filldraw[shift={(6,-6)}] (2,2.5) circle (2pt);
    \draw[shift={(6,-6)}] (0,3)--(2,2.5);
    \draw[shift={(6,-6)}](0,2)--(2,2.5);
    \draw[shift={(6,-6)}] (0,1)--(2,1.5);
    \draw[shift={(6,-6)}] (0,1)--(2,1);
    \draw[shift={(6,-6)}] (0,0)--(2,0.5) ;
    \node[shift={(6,-6)}] at (-0.3,1.3) {$x_0$};
    \node[shift={(6,-6)}] at (-0.2,-0.3) {$x_3$};
      \node[shift={(6,-6)}] at (-0.3,2.3) {$x_2$};
    \node[shift={(6,-6)}] at (-0.2,3.3) {$x_1$};
    \node[shift={(6,-6)}] at (2.4,1) {$x_0$};
    \node[shift={(6,-6)}] at (2.4,0.4) {$y_2$};
      \node[shift={(6,-6)}] at (2.4,1.6) {$y_1$};
    \node[shift={(6,-6)}] at (2.4,2.5) {$y_0$};

\filldraw[shift={(0,-12)}] (0,1) circle (2pt);
    \filldraw[shift={(0,-12)}] (0,2) circle (2pt);
    \filldraw[shift={(0,-12)}] (0,3) circle (2pt);
    \filldraw[shift={(0,-12)}] (2,0.5) circle (2pt);
    \filldraw[shift={(0,-12)}] (2,1.5) circle (2pt);
    \filldraw[shift={(0,-12)}] (2,2.5) circle (2pt);
    \filldraw[shift={(0,-12)}] (2,2.5) circle (2pt);
    \filldraw[shift={(0,-12)}] (2,2.5) circle (2pt);
    
    \draw[shift={(0,-12)}] (0,3)--(2,2.5);
    \draw[shift={(0,-12)}](0,2)--(2,2.5);
    \draw[shift={(0,-12)}] (0,1)--(2,1.5);
    \draw[shift={(0,-12)}] (0,1)--(2,0.5);
    \node[shift={(0,-12)}] at (-0.3,1.3) {$x_0$};
      \node[shift={(0,-12)}] at (-0.3,2.3) {$x_2$};
    \node[shift={(0,-12)}] at (-0.2,3.3) {$x_1$};
    \node[shift={(0,-12)}] at (2.4,0.4) {$y_2$};
      \node[shift={(0,-12)}] at (2.4,1.6) {$y_1$};
    \node[shift={(0,-12)}] at (2.4,2.5) {$y_0$};

    \draw[shift={(0,-12)},-implies,double equal sign distance](3.5,1.5)--(5,1.5);

    \filldraw[shift={(6,-12)}] (0,1) circle (2pt);
    \filldraw[shift={(6,-12)}] (0,2) circle (2pt);
    \filldraw[shift={(6,-12)}] (0,3) circle (2pt);
    \filldraw[shift={(6,-12)}] (3,0.5) circle (2pt);
    \filldraw[shift={(6,-12)}] (3,1.5) circle (2pt);
    \filldraw[shift={(6,-12)}] (3,2.5) circle (2pt);
    
    \filldraw[shift={(6,-12)}] (1.5,2.5) circle (2pt);
    \filldraw[shift={(6,-12)}] (1.5,1) circle (2pt);
    \filldraw[shift={(6,-12)}] (3,2.5) circle (2pt);
    \filldraw[shift={(6,-12)}] (3,2.5) circle (2pt);
    
    \draw[shift={(6,-12)}] (0,3)--(1.5,2.5);
    \draw[shift={(6,-12)}](0,2)--(1.5,2.5);
    \draw[shift={(6,-12)}] (0,1)--(1.5,1);
    \draw[shift={(6,-12)}] (1.5,1)--(3,1.5);
    \draw[shift={(6,-12)}] (1.5,2.5)--(3,2.5);
    \draw[shift={(6,-12)}] (1.5,1)--(3,0.5);
    \node[shift={(6,-12)}] at (-0.3,1.3) {$x_0$};
      \node[shift={(6,-12)}] at (-0.3,2.3) {$x_2$};
    \node[shift={(6,-12)}] at (-0.2,3.3) {$x_1$};
    \node[shift={(6,-12)}] at (1.5,0.6) {$x_0$};
    \node[shift={(6,-12)}] at (3.4,0.4) {$y_2$};
      \node[shift={(6,-12)}] at (3.4,1.6) {$y_1$};
    \node[shift={(6,-12)}] at (3.4,2.5) {$y_0$};
    \node[shift={(6,-12)}] at (1.5,2.8) {$y_0$};

\end{tikzpicture}
    \caption{An illustration of the proof of Lemma \ref{lem:Relation_Cardinal_Bounded}. The top two pictures represent the first two cases, where the dashed line is the pair $(x,y)$ to be deleted in $R$ to obtain $P$. The second line represent the passage from $R$ to $Q$ in the particular case where $x_0\in Y$. On the last line, the left part is the relation $R$, and the right part is the concatenation of $P_0$ and $P_1$. }
    \label{fig:Proof_Relation_Cardinal_Bounded}
\end{figure}

We deduce that we can indeed decompose the considered chains into relations that either merge points together, or that split some points, while keeping a control over the cardinality of the intermediate configurations of the chain.
\begin{corol}\label{cor:Combinatorial_Chain_Graph}
    Let $((X_0,\dots, X_{k}),(R_0,\dots,R_{k-1}))$ be a combinatorial chain, and $\omega$ be a weight. There exists a combinatorial chain $((Y_0,\dots,Y_l),(P_0,\dots,P_{l-1}))$ such that:
    \begin{itemize}
         \item $Y_0=X_0$ and $Y_l=X_k$,
        \item For all $0\leq i \leq l-1$, either $P_i$ or $P_i^{\op}$ is the graph of a surjective map,
        \item $\elo(P_0,\dots,P_{l-1})\leq \elo(R_0,\dots,R_{k-1})$,
        \item $\card(Y_i)\leq \max\{\card(X_j)\mid 0\leq j\leq k\}$ for all $0\leq i\leq l-1$.
    \end{itemize}
\end{corol}

\begin{proof}
    It is sufficient to prove the case where $k=1$,
    since the general result will be obtained by concatenation, so let $R\subset X\times Y$
    be a surjective relation, $n=\max(\card(X),\card(Y))$. We proceed by induction on $n\geq 1$. If $n=1$, then $X$, $Y$ and $R$ are all singletons, and $R$ must be already be the graph of a bijection and there is nothing to prove.
    
    Let us suppose the result known for all value $\leq n_0$ and let suppose $n=n_0+1$.
    We prove the induction step on $n$ by induction on $m=\card(R)$. Since $R$ must be surjective, we have $\card(R)=m\geq n$. If $m=n$, then either $\card(R)=\card(X)$ or $\card(R)=\card(Y)$. In the first case, $R$ is the graph of a map (which is necessarily surjective), and in the later $R^{\op}$ is. In any case, there is nothing to prove.
    Now let us assume that the result is known for all integers $\leq m_0$, and that $m=m_0+1>n$. 
    Then by \cref{lem:Relation_Cardinal_Bounded}, there exists a combinatorial chain $((Y_0,\dots,Y_l),(P_0,\dots,P_{l-1}))$ such that
    \begin{itemize}
        \item $\card(P_j)<\card(R)$ for all $0\leq j\leq l-1$,
        \item $\elo(P_0,\dots,P_{l-1})\leq \elo(R)$,
        \item $Y_0=X$, $Y_l=Y$,
        \item $\card(Y_i)\leq\max(\card(X),\card(Y))$
    \end{itemize}
    But then, since each $P_i$ has cardinality strictly smaller than $m$, and each $Y_i$ has cardinality smaller that $n$ applying both induction hypotheses gives the desired result.  
\end{proof}

\subsection{Merges, bijections and splitting}
In this subsection, we continue to restrict the set of combinatorial chains over which we compute the weighted distance, to arrive at the result of \cref{theo:MBS}. We start by formalizing properly the relations that we want in our "good" combinatorial chains: merges, splittings and bijections.
\begin{defin}
    Let $X,Y\in \Ran(M)$ be two configurations, and $R\subset X\times Y$ a surjective relation between $X$ and $Y$. The relation $R$ is said to be :
    \begin{itemize}
        \item a \textbf{merge} (from $X$ to $Y$) if $R$ is the graph of a map $f\colon X\to Y$ such that
        \begin{itemize}
            \item there exists $y_0\in Y$ such that $\card(f^{-1}(y_0))\geq 2$,
            \item for all $y\not=y_0$, $f^{-1}(y)=\{y\}\subset X$.
        \end{itemize}
        The merge is said to be \textbf{simple} if $\card(f^{-1}(y_0))=2$. In that case, we will say that $R$ is the merge of $f^{-1}(y_0)$ at $y_0$.
        \item a \textbf{bijection} (from $X$ to $Y$) if $R$ is the graph of a bijection $\phi\colon X\to Y$.
        \item a \textbf{splitting} (from $X$ to $Y$) if $R^{\op}$ is a merge (from $Y$ to $X$). It is said to be a \textbf{simple} splitting, if $R^{\op}$ is a simple merge.
    \end{itemize}
\end{defin}
We can finally define the specific shape that we want for the combinatorial chains.
\begin{defin}
\label{def:MBS}
    Let $(R_0,\dots,R_{k-1})$ be a combinatorial chain. We say that it is a \textbf{Merge-Bijection-Splitting-chain} (or \textbf{MBS-chain}) if there exists $0\leq l\leq k-1$ such that
    \begin{itemize}
        \item $R_i$ is a merge, for $0\leq i\leq l-1$,
        \item $R_l$ is a bijection,
        \item $R_j$ is a splitting, for $l+1\leq j \leq k-1$
    \end{itemize}
    or if there exists $0\leq l\leq k$ such that
    \begin{itemize}
        \item $R_i$ is a merge, for $0\leq i\leq l-1$,
        \item $R_j$ is a splitting, for $l\leq j \leq k-1$
    \end{itemize}
    It is said to be a \textbf{simple} MBS-chain if in addition each of the merges and splittings are simple.
\end{defin}

\begin{rem}
    We may always assume that a MBS-chain is of the first form, by adding the identity bijection to a chain of the second form, since this will not change its $\omega$-length. Thus, in the following proofs, we may implicitly assume that we are in the first case without loss of generality.
\end{rem}

We will now spend the rest of the section to prove the following theorem.
\begin{theo}\label{theo:MBS}
    Let $X,Y\in \Ran(M)$ be two configurations, and $\omega$ a weight. Then, the combinatorial distance can be computed as
    \begin{equation*}
        \dom(X,Y)=\inf\{\elo(R_0,\dots,R_{k-1})\}
    \end{equation*}
    where the infimum is computed on simple MBS-chains between $X$ and $Y$.
\end{theo}
We still need some extra lemmas in order to complete the proof. Note that all of them are symmetric in terms of merge/splitting, so that we only prove them for merges. First we prove that we can decompose a relation that "merges" points together into an chain of actual merges, followed by a bijection.
\begin{lem}\label{lem:Surjections_Are_Merges}
    Let $X,Y\in \Ran(M)$ be two configurations, and $R$ be the graph of a surjective map $f\colon X\to Y$. Then, there exists a combinatorial chain $((X_0,\dots,X_k),(R_0,\dots,R_{k-1}))$ from $X$ to $Y$ such that
    \begin{itemize}
        \item $\elo(R_0,\dots,R_{k-1})\leq \elo(R)$,
        \item $R_i$ is a merge, whenever $0\leq i\leq k-2$,
        \item $R_{k-1}$ is a bijection.
    \end{itemize}
\end{lem}

\begin{proof}
    Consider the following partitions of $X$ and $Y$ :
    \begin{align*}
        Y_{-}=\{y\in Y\mid \card(f^{-1}(y))=1\},&\quad&  X_{-}=f^{-1}(Y_{-}),\\
        Y_{+}=\{y\in Y\mid \card(f^{-1}(y))>1\},&& X_{+}=f^{-1}(Y_{+}).
    \end{align*}
    The map $f$ restricts to a (possibly empty) bijection $f_-\colon X_-\to Y_-$ and a nowhere injective, surjective map $f_+\colon X_+\to Y_+$. We will show the claim by induction on $\card(Y_+)$. 

    If $\card(Y_+)=0$, then $f$ is a bijection, and there is nothing to prove. Now, let $l\geq 0$, be such that the result is known whenever $\card(Y_+)\leq l$. Assume $\card(Y_+)=l+1$, and fix some $y\in Y_+$. We distinguish two cases.
    \begin{itemize}
        \item If $y\not \in X\setminus f^{-1}(y)$, define $\widetilde{X}=X\setminus f^{-1}(y)\cup \{y\}$, together with a map $f_0\colon X\to \widetilde{X}$ such that $f_0(x)=x$ for all $x\not\in f^{-1}(y)$ and $f_0(x)=y$ for all $x\in f^{-1}(y)$. By construction, $R_0$, the graph of $f_0$ is a merge, and we have
        \begin{equation*}
            \elo(R_0)=\omega(\card(X))\sum_{x\in f^{-1}(y)}\dm(x,y)
        \end{equation*}
        On the other hand, we define $\widetilde{f}\colon \widetilde{X}\to \widetilde{Y}=Y$ via $\widetilde{f}(x)=f(x)$ for $x\not = y$, and $\widetilde{f}(y)=y$, and we have for $\widetilde{R}$ the graph of $\widetilde{f}$,
        \begin{equation*}
            \elo(\widetilde{R})=\omega(\card(\widetilde{X}))\sum_{x\in X\setminus f^{-1}(y)}\dm(x,f(x))
        \end{equation*}
        Now, just as before, define $\widetilde{Y}_+\subset \widetilde{Y}$ to be the subset of points at which $\widetilde{f}$ is non-injective, and observe that $\widetilde{Y}_+=Y_+\setminus \{y\}$, thus $\card(\widetilde{Y}_+)\leq l$, and we can apply the induction hypothesis to produce a combinatorial chain $(R_1,\dots,R_{k-1})$ of the desired form, and satisfying
        \begin{equation*}
            \elo(R_1,\dots,R_{k-1})\leq \elo(\widetilde{R})
        \end{equation*}
        Now, the concatenation $(R_0,\dots,R_{k-1})$ is of the desired form, and we have
        \begin{align*}
            \elo(R_0,\dots,R_{k-1})&=\elo(R_0)+\elo(R_1,\dots,R_{k-1})\\
            &\leq \elo(R_0)+\elo(\widetilde{R})\\
            &\leq \omega(\card(X))\sum_{x\in f^{-1}(y)}\dm(x,y)+\omega(\card(\widetilde{X}))\sum_{x\in X\setminus f^{-1}(y)}\dm(x,f(x))\\
            &\leq \omega(\card(X))\sum_{x\in X}\dm(x,f(x))\\
            &\leq \elo(R)
        \end{align*}
        Which concludes the proof in this case.
        \item If $y\in X\setminus f^{-1}(y)$, choose some $z\in f^{-1}(y)$, and define $\widetilde{X}=X\setminus f^{-1}(y)\cup \{z\}$, together with a map $f_0\colon X\to \widetilde{X}$ which satisfies $f_0(x)=y$ if $x\in f^{-1}(y)\setminus z$ and $f_0(x)=x$ otherwise. Observe that $f_0$ is a merge, since $f_0^{-1}(y)=f^{-1}(y)\cup \{y\}\setminus \{z\}$ is of cardinality at least $2$. Now, let $\widetilde{f}\colon \widetilde{X}\to \widetilde{Y}=Y$ be the map defined by $\widetilde{f}(x)=f(x)$ for all $x\in \widetilde{X}$. Observe that $\widetilde{Y}_+=Y_+\setminus\{y\}$, since $\widetilde{f}$ is injective at $y$, and we can thus apply the induction hypothesis to $\widetilde{f}$. Furthermore, observe that we have
        \begin{align*}
            \elo(f_0)+\elo(\widetilde{f}) &=\omega(\card(X))\left(\sum_{x\in f^{-1}(y)\setminus \{z\}}d(x,y)\right)
            +\omega(\card(\widetilde{X}))\left(\sum_{x\in X\setminus f^{-1}(y)\cup \{z\}}\dm(x,f(x))\right)\\
            &\leq \omega(\card(X))\sum_{x\in X}\dm(x,f(x))\\
            &\leq \elo(f)
        \end{align*}
        This gives the induction step.
    \end{itemize}
\end{proof}

Then we prove that we can essentially change the order of merges and bijections to always start with the merges, and end with a bijection.
\begin{lem}\label{lem:Merge_Bijections_Exchange}
    Let $((X,Y,Z),(R_0,R_1))$ be a combinatorial chain where $R_0$ is a bijection and $R_1$ is a merge, then there exists a combinatorial chain $(P_0,\dots,P_{k-1})$ between $X$ and $Z$ such that
    \begin{itemize}
        \item $P_i$ is a merge, for all $0\leq i<k-1$,
        \item $P_{k-1}$ is a bijection,
        \item $\elo(P_0,\dots,P_{k-1})\leq \elo(R_0,R_1)$
    \end{itemize}    
\end{lem}

\begin{proof}
    Since $R_0$ is the graph of a bijective map $f_0\colon X\to Y$ and $R_1$ is the graph of a surjective map $f_1\colon Y\to Z$, we may consider the composition $f_1\circ f_0$, whose graph we will denote $R$. It is a surjection, by construction, and since $f_0$ is a bijection, we have $\card(R)=\card(X)=\card(R_0)=\card(Y)=\card(R_1)$. Thus, we have
    \begin{align*}
        \elo(R)&=\omega(\card(X))\sum_{(x,z)\in R}\dm(x,z)\\
        &=\omega(\card(X))\sum_{(x,y)\in R_0}\dm(x,f_1(y))\\
        &\leq\omega(\card(X))\sum_{(x,y)\in R_0}\left(\dm(x,y)+\dm(y,f_1(y))\right)\\
        &\leq\omega(\card(X))\left(\sum_{(x,y)\in R_0}\dm(x,y)+ \sum_{(y,z)\in R_1}\dm(y,z)\right)\\
        &\leq \elo(R_0,R_1)
    \end{align*}
    We conclude by applying \cref{lem:Surjections_Are_Merges} to $R$.
\end{proof}

By iterating this lemma, we can get the following corollary for an arbitrary number of merges and bijections.
\begin{corol}\label{cor:multiple_Merge_Bijections_Exchange}
    Let $X,Y\in\Ran(M)$ be two configurations and $(R_0,\dots,R_{k-1})$ be a combinatorial chain between $X$ and $Y$, such that $R_i$ is a bijection or a merge for all $0\leq i\leq k-1$. Then there exists a chain $(P_0,\dots,P_{l-1})$ between $X$ and $Y$ such that
    \begin{itemize}
        \item $P_i$ is a merge, for all $0\leq i<l-1$,
        \item $P_{l-1}$ is a bijection,
        \item $\elo(P_0,\dots,P_{l-1})\leq \elo(R_0,\dots,R_{k-1})$
    \end{itemize}    
\end{corol}

\begin{proof}
    We just need to iterate \cref{lem:Merge_Bijections_Exchange} to get a shorter combinatorial chain between $X$ and $Y$ that starts with merges, and ends with bijections. But then, composing all the bijections into one gives the desired result.
\end{proof}

Now we also need to switch the order of merges and splitting, this is the object of the following lemma.
\begin{lem}
\label{lem:switch split merge}
Let $((X_0,X_1,X_2),(R_0,R_1))$ be a combinatorial chain where $R_0$ is a splitting, and $R_1$ is a merge, then there exists a combinatorial chain $((X_0,Y_1,X_2),(P_0,P_1))$ such that $P_0$ is either a merge or a bijection, $P_1$ is either a bijection or a splitting, and $\elo(P_0,P_1)\leq\elo(R_0,R_1)$.
\end{lem}
\begin{proof}
    Since $R_0$ is a splitting, its opposite $R_0^{\op}$ is the graph of a map $f_0\colon X_1\to X_0$, such that there exists $x_0 \in X_0$ such that $X_1'=f_0^{-1}(\{x_0\})$ has a cardinality $>1$ and $f_0$ restricts to the identity on $X_1\setminus X_1'$. Symmetrically, $R_1$ is the graph of a map $f_1\colon X_1\to X_2$, and there is a point $x_2\in X_2$ such that $X_1''=f_1^{-1}(\{x_2\})$ has a cardinality $>1$ and $f_1$ restricts to the identity on $X_1\setminus X_1''$. Remark that we thus have $\card(R_0)=\card(R_1)=\card(X_1)\geq \card(X_0),\card(X_2)$
    
    Now, we distinguish three cases
    \begin{itemize}
        \item If $X_1''\subset X_1'$, then observe that we can define a surjective relation $P_1\subset X_0\times X_2$ as follows
    \begin{equation*}
        xP_1y \iff (x=x_0 \text{ and } (y \in X_1'\setminus X_1'' \text{ or } y=x_2)) \text{ or } (x\in X_0\setminus \{x_0\} \text{ and } x=y)
    \end{equation*}
    and that it is either a bijection or a splitting, so that $\card(P_1)=\card(X_2)\leq \card(X_1)$. We can thus define $P_0$ as the graph of $\Id_{X_0}$ and $((X_0,X_0,X_2),(P_0,P_1))$ is of the desired form. Moreover
    \begin{align*}
        \elo(P_0,P_1)=\elo(P_1)= &\card(P_1)\sum_{y\in X_1'\setminus X_1''} d(x_0,y) + \card(P_1)d(x_0,x_2)\\
        \leq& \card(X_1) \sum_{y\in X_1'\setminus X_1''} d(x_0,y) + \card(X_1)\sum_{y\in X_1''} d(x_0,x_2) \\
        \leq & \card(X_1) \sum_{y\in X_1'\setminus X_1''} d(x_0,y) + \card(X_1)\sum_{y\in X_1''} (d(x_0,y)+ d(y,x_2))\\
        \leq & \card(X_1) \sum_{y\in X_1'} d(x_0,y) + \card(X_1)\sum_{y\in X_1''}d(y,x_2)\\
        \leq &\elo(R_0,R_1)
    \end{align*}
    \item If $X_1'\subset X_1''$ the case is the symmetric of the preceding one.
    \item Otherwise there exists $y_0\in X_1''\setminus X_1'\subset X_0\cap X_1$ and $y_2\in X_1'\setminus X_1''\subset X_1\cap X_2$. Remark that we must have $y_0\neq x_0$ because if $x_0\in X_1$ then we must have $x_0\in X_1'$. Symmetrically we have $y_2\neq x_2$. We can thus define $Y_1=X_1\setminus \left(X_1' \cup X_1''\right) \cup \{x_0,x_2\}$ and $P_0$ as the graph of 
    \begin{align*}
        g_0\colon X_0 &\to Y_1\\
        x & \mapsto \begin{cases}
        x_2 \quad &\text{if } x\in X_1''\setminus \{x_0\}\\
        x & \text{otherwise}
        \end{cases}
    \end{align*}
    Observe that $P_0$ is surjective: if $y\in Y_1$, then either $y=x_0$ and $g_0(y)=y$, or $y=x_2$ and $g_0(y_0)=y$ or $y\in Y_1\setminus \{x_0,x_2\}\subset X_1\setminus X_1'$, so that $y=f_0(y)\in X_0$. In the latter case, we have furthermore,  $y\in Y_1\setminus \{x_0,x_2\}\subset X_1\setminus X_1''$ so that $y=f_0(y)\in X_0\setminus X''_1$ and finally $g_0(y)=y$. It is indeed a merge or a bijection, so that $\card(P_0)=\card(X_0)$. Symmetrically, we can define $P_1$ such that $P_1^{\op}$ is the graph of
    \begin{align*}
        g_1\colon X_2 &\to Y_1\\
        x & \mapsto \begin{cases}
        x_0 \quad& \text{if } x\in X_1'\setminus \{x_2\}\\
        x & \text{otherwise}
        \end{cases}
    \end{align*}
    so that $P_1$ is a splitting or a bijection and $\card(P_1)=\card(X_2)$. Finally
    \begin{align*}
        \elo(P_0,P_1)&=\card(X_0)\sum_{y\in X_1''\setminus \{x_0\}} d(y,x_2) + \card(X_2)\sum_{y\in X_1'\setminus \{x_2\}} d(x_0,y)\\
        &\leq \card(X_1)\sum_{y\in X_1''} d(y,x_2) + \card(X_1)\sum_{y\in X_1'} d(x_0,y) \\
        &\leq \elo(R_0,R_1)
    \end{align*}
    \end{itemize}
\end{proof}

We can now iterate \cref{lem:Merge_Bijections_Exchange} and \cref{lem:switch split merge} to get the following.
\begin{corol}
\label{cor:exchange_merge_chain}
Let $X,Y\in\Ran(M)$ be two configurations and  $(R_0,\dots,R_{k-1})$ be a combinatorial chain between $X$ and $Y$, such that $R_i$ is a bijection or a splitting for all $0\leq i\leq k-2$ and $R_{k-1}$ is a merge. There exists $0\leq m \leq l$ and a combinatorial chain $(P_0,\dots,P_{l-1})$ between $X$ and $Y$ such that 
\begin{itemize}
    \item $P_i$ is a merge for all $0\leq i \leq m-1$
    \item $P_i$ is a bijection or a splitting for all $m\leq i\leq l-1$
    \item $\elo(P_0,\dots,P_{l-1})\leq\elo(R_0,\dots,R_{k-1})$ 
\end{itemize}
\end{corol}
Then again, we can iterate \cref{cor:exchange_merge_chain} to get the following corollary.
\begin{corol}\label{cor:all merges first}
Let $X,Y\in\Ran(M)$ be two configurations and  $(R_0,\dots,R_{k-1})$ be a combinatorial chain between $X$ and $Y$ such that $R_i$ is either a merge, a splitting or a bijection for all $0\leq i \leq k-1$. Then there exists $0\leq m \leq l$ and a combinatorial chain $(P_0,\dots,P_{l-1})$ between $X$ and $Y$ such that
\begin{itemize}
    \item $P_i$ is a merge for all $0\leq i\leq m-1$
    \item $P_i$ is a bijection or a splitting for all $m\leq i\leq l-1$
    \item $\elo(P_0,\dots,P_{l-1})\leq\elo(R_0,\dots,R_{k-1})$
\end{itemize}
\end{corol}

Finally we show that we can decompose any merge into a chain of simple merges.
\begin{lem}\label{lem:Simple_Merges}
    Let $X,Y\in \Ran(M)$ be two configurations and $R$ a merge from $X$ to $Y$. Then, there exists a combinatorial chain from $X$ to $Y$, $(R_0,\dots,R_{k-1})$, such that
    \begin{itemize}
        \item $R_i$ is a simple merge for all $0\leq i\leq k-1$,
        \item $\elo(R_0,\dots,R_{k-1})\leq \elo(R)$.
    \end{itemize}
\end{lem}

\begin{proof}
    We proceed by induction on $\card(X)-\card(Y)$. Since $R$ is a merge, we must have $\card(X)>\card(Y)$. Now, if $\card(X)=\card(Y)+1$, then $R$ is a simple merge, and there is nothing to prove. Otherwise, let $f\colon X\to Y$ be the map whose graph is $R$, and let $y\in Y$ be the unique element at which $f$ is non-injective. $f^{-1}(y)$ is of cardinality $\card(X)-\card(Y)+1\geq 3$. We distinguish two cases.
    \begin{itemize}
        \item If $y\not\in f^{-1}(y)$, let $x_1,x_2\in f^{-1}(y)$, and define $\widetilde{X}=X\setminus\{x_1,x_2\}\cup \{y\}$, together with a map $f_0\colon X\to \widetilde{X}$ which satisfies $f_0(x_1)=f_0(x_2)=y$ and $f_0(x)=x$ otherwise. In particular, the graph of $f_0$, $R_0$, is a simple merge, and satisfies
        \begin{equation*}
            \elo(R_0)=\omega(\card(X))(\dm(x_1,y)+\dm(x_2,y)).
        \end{equation*}
        On the other hand, define $\widetilde{f}\colon \widetilde{X}\to Y$ which satisfies $\widetilde{f}(x)=f(x)$ whenever $x\not=y$, and $\widetilde{f}(y)=y$. Then, $\widetilde{R}$, the graph of $\widetilde{f}$ is again a merge and satisfies, $\card(\widetilde{X})-\card(Y)=\card(X)-\card(Y)-1$, we may thus apply the induction hypothesis, to produce a chain of simple merges $(R_1,\dots,R_{k-1})$. Finally, we have a chain of simple merges, $(R_0,\dots,R_{k-1})$, and we compute its length :
        \begin{align*}
            \elo(R_0,\dots,R_{k-1})&= \elo(R_0)+\elo(R_1,\dots,R_{k-1})\\
            &\leq \elo(R_0)+\elo(\widetilde{R})\\
            &\leq \omega(\card(X))(\dm(x_1,y)+\dm(x_2,y))+\omega(\card(\widetilde{X}))\left(\sum_{x\in \widetilde{f}^{-1}(y)}\dm(x,y)\right)\\
            &\leq \omega(\card(X))\left(\sum_{x\in f^{-1}(y)}\dm(x,y)\right)\\
            &\leq \elo(R).
        \end{align*}
        Which concludes the proof in this case.
        \item if $y\in f^{-1}(y)$, pick some $x_0\not= y\in f^{-1}(y)$, and define $\widetilde{X}=X\setminus\{x_0\}$. The proof is analogous to the previous case, with the first step being a simple merge between $x_0$ and $y$.
    \end{itemize}
\end{proof}
We are now ready to prove \cref{theo:MBS}
\begin{proof}[Proof of \cref{theo:MBS}]
    By combining \cref{cor:Combinatorial_Chain_Graph} and \cref{lem:Surjections_Are_Merges}, we see that any combinatorial chain between $X$ and $Y$, can be shortened by passing to a chain of the form $(R_0,\dots,R_{k-1})$, where $R_i$ is either a merge, a bijection or a splitting. Then we can apply \cref{cor:all merges first} to obtain $0\leq m\leq l$ and a shorter combinatorial chain $(P_0,\dots,P_{l-1})$ such that $P_i$ is a merge for all $0\leq i \leq m$, and a splitting or a bijection for all $m\leq i \leq l-1$. By symmetry of merge and splittings, we can apply \cref{cor:multiple_Merge_Bijections_Exchange} to $(P_m,\dots,P_{l-1})$ to get a shorter chain $(Q_m,\dots,Q_{l'-1})$ such that $Q_m$ is a bijection and $Q_i$ is a splitting for all $m<i\leq l'-1$. By concatenating we get a MBS-chain $(P_0,\dots,P_{m-1},Q_m,\dots,Q_{l'-1})$ between $X$ and $Y$. We then conclude by applying \cref{lem:Simple_Merges} to all $P_i$, $0\leq i \leq m-1$ and to all $Q_i$, $m<i\leq l'-1$ to get a shorter simple MBS-chain between $X$ and $Y$.
\end{proof}

\subsection{Existence of geodesic chains}
\label{sec:Geodesic_chains}
Finally, in this section, we use \cref{theo:MBS} to prove that if $M$ satisfies the Heine-Borel property, then the infimum in \cref{def:weighted-distance} is computed over a compact set and is thus attained for a given MBS-chain. This is the subject of the following theorem.

\begin{theo}\label{theo:Geodesic_Chain}
    Assume that $M$ is a metric space, satisfying the Heine-Borel property (i.e. in which bounded and closed sets are compact). Then, for any configurations $X,Y\in \Ran(M)$, and any weight $\omega$, there exists a combinatorial chain $(R_0,\dots,R_{k-1})$ between $X$ and $Y$ such that
    \begin{equation*}
        \elo(R_0,\dots,R_{k-1})=\dom(X,Y)
    \end{equation*}
    Furthermore, the chain $(R_0,\dots,R_{k-1})$ may be assumed to be a simple MBS-chain, in which case, we call it a \textbf{geodesic chain} between $X$ and $Y$.
\end{theo}

\begin{proof}
    By \cref{theo:MBS}, the infimum defining $\dom(X,Y)$ can be computed over simple MBS-chains, thus it is enough to show that there exists a simple MBS-chain between $X$ and $Y$ with minimal length. Now observe that if $(R_0,\dots,R_{k-1})$ is a simple MBS-chain between $X$ and $Y$, then there exists some $0\leq m\leq k-1$, such that $\card(R_i)=\card(R_{i+1})+1$ for all $i\leq m-1$, and $\card(R_i)=\card(R_{i-1})+1$ for all $m<i\leq k-1$. In particular, the value of $k$ determines $\card(R_i)$ for all $0\leq i\leq k-1$. We deduce that for any simple MBS-chain between $X$ and $Y$, we must have $k\leq \card(X)+\card(Y)-1$. This gives
    \begin{equation*}
        \dom(X,Y)=\min_{1\leq k\leq \card(X)+\card(Y)-1}\{\inf\{\elo(R_0,\dots,R_{k-1})\}\}
    \end{equation*}
    Thus it is enough to show that for any fixed $k\geq 1$, there is a simple MBS-chain $(R_0,\dots,R_{k-1})$ between $X$ and $Y$ of minimal length among simple MBS-chains with $k$ relations.

    Given $k\geq 1$, there is a unique $k$-tuple of integer $(n_0,\dots,n_{k-1})$ satisfying $\card(R_i)=n_i, \forall i$, for any simple MBS-chain $(R_0,\dots,R_{k-1})$ between $X$ and $Y$. Now, by \cref{lem:Sequence_tuples}, such a chain can be encoded as the data of a $2k$-tuple $(\underline{x_0},\dots,\underline{x_{2k-1}})$ satisfying conditions~\ref{item:Correct_Cardinal},\ref{item:Boundary} and~\ref{item:Matching}, and its length can be recovered as
    \begin{equation*}
        \elo(R_0,\dots,R_{k-1})=\sum_{i=0}^{k-1}\omega(n_i)\dsum(\underline{x_{2i}},\underline{x_{2i+1}}).
    \end{equation*}
Thus, if we define $M^{(n_0,\dots,n_{k-1})}=\prod_{i=0}^{k-1}(M^{n_i}\times M^{n_i})$, we are left with proving that the function $(\underline{x_0},\dots,\underline{x_{2k-1}})\mapsto\sum_{i=0}^{k-1}\omega(n_i)\dsum(\underline{x_{2i}},\underline{x_{2i+1}})$, considered on the subset of $M^{(n_0,\dots,n_{k-1})}$ of $2k$-tuples $(\underline{x_0},\dots,\underline{x_{2k-1}})$ satisfying conditions~\ref{item:Boundary} and~\ref{item:Matching}, admits a minimum. Now, let $\epsilon >0$ and set $r=\dom(X,Y)+\epsilon$. Then, there must exist a simple MBS-chain of $\omega$-length less than $r$ between $X$ and $Y$. We may assume that there exists one with $k$ terms, otherwise there is nothing to prove. But now, by \cref{lem:Sequences_In_Balls} if $(R_0,\dots,R_{k-1})$ is such a simple MBS-chain, the corresponding $2k$-tuple $(\underline{x_0},\dots,\underline{x_{2k-1}})$ is such that
    \begin{equation*}
        (\underline{x_0},\dots,\underline{x_{2k-1}})\in K=\prod_{i=0}^{k-1}\left((\overline{B}_M(X,r))^{n_i}\times (\overline{B}_M(X,r))^{n_i}\right)
    \end{equation*}
    Finally, observe that $\overline{B}_M(X,r)=\cup_{x\in X}\overline{B}_M(x,r)$ is an union of closed balls. Since we assumed that $M$ satisfies the Heine-Borel property, $\overline{B}_M(X,r)$ is thus a finite union of compact subsets of $M$, thus it is compact, and $K$ is compact, as a product of compact. Now, observe that the subset of $K$ of $2k$-tuple satisfying conditions~\ref{item:Boundary} and~\ref{item:Matching}, call it $K'$, is a closed subset (it is the locus of solutions of equations specified by continuous maps), hence it is also compact. Finally, since the function $(\underline{x_0},\dots,\underline{x_{2k-1}})\mapsto \sum_{i=0}^{k-1}\omega(n_i)\dsum(\underline{x_{2i}},\underline{x_{2i+1}})$ is continuous, and bounded below by zero, it must reach a minimum on the compact $K'$. By the converse of \cref{lem:Sequence_tuples}, this produces the geodesic chain we are looking for.
\end{proof}
The proof of the previous theorem proceeds by exhibiting a compact set over which we compute the minimum. Since $M$ satisfies the Heine-Borel property, we construct this compact set from closed balls in $M$. We thus translate our optimization problem over the set of configuration chains into an another optimization problem over a set of tuples in a product of $M$.
\begin{defin}
    Let $n\geq 1$, and $\underline{x}=(x^{(1)},\dots,x^{(n)})$ and $\underline{y}=(y^{(1)},\dots,y^{(n)})$ be two points in $M^n$. Write $\pi_n(\underline{x})=X$ and $\pi_n(\underline{y})=Y$, and define the relation $R_{\underline{x},\underline{y}}\subset X\times Y$ as follows. For $(x,y)\in X\times Y$,
    \begin{equation*}
        xR_{\underline{x},\underline{y}}y\Leftrightarrow \exists i,\ x=x^{(i)} \text{ and }y=y^{(i)}
    \end{equation*}
\end{defin}

\begin{lem}\label{lem:Relations_lift}
    Let $X,Y\in \Ran(M)$ be two configurations and $R$ be a surjective relation between $X$ and $Y$. Write $n=\card(R)$. Then, there exists $\underline{x},\underline{y}\in M^n$ such that \begin{itemize}
        \item $\pi_n(\underline{x})=X$,
        \item $\pi_n(\underline{y})=Y$,
        \item $R_{\underline{x},\underline{y}}=R$
    \end{itemize} 
    In that case, we have
    \begin{equation*}
        \elo(R)=\omega(n)\dsum(\underline{x},\underline{y})=\omega(n)\sum_{i=1}^n\dm(x^{(i)},y^{(i)})
    \end{equation*}
\end{lem}

\begin{proof}
    Fix some bijection $\phi\colon \{1,\dots,n\}\to R\subset X\times Y$ and write $\phi(i)=(x^{(i)},y^{(i)})$. Surjectivity of $R$ guarantees that $\pi_n(\underline{x})=X$ and $\pi_n(\underline{y})=Y$. Finally, the last equality follows from the definition of $\elo$.
\end{proof}

\begin{lem}\label{lem:Sequence_tuples}
    Let $X,Y\in \Ran(M)$ be two configurations. If $(R_0,\dots,R_{k-1})$ is a combinatorial chain between $X$ and $Y$, write $n_i=\card(R_i)$, for $0\leq i\leq k-1$. There exists a $2k$-tuple $(\underline{x_0},\dots,\underline{x_{2k-1}})$ such that
    \begin{enumerate}
        \item \label{item:Correct_Cardinal} $\underline{x_{2i}},\underline{x_{2i+1}}\in M^{n_i}$, for all $0\leq i\leq k-1$
        \item\label{item:Boundary} $\pi_{n_0}(\underline{x_{0}})=X$ and $\pi_{n_{k-1}}(\underline{x_{2k-1}})=Y$
        \item\label{item:Matching} $\pi_{n_i}(\underline{x_{2i}})=\pi_{n_{i-1}}(\underline{x_{2i-1}})$, $1\leq i\leq k-1$
        \item\label{item:Induce_Relations} $R_i=R_{\underline{x_{2i}},\underline{x_{2i+1}}}$, $0\leq i\leq k-1$.
        \item\label{item:Compute_Distance} $\elo(R_0,\dots,R_{k-1})=\sum_{i=0}^{k-1}\omega(n_i)\dsum(\underline{x_{2i}},\underline{x_{2i+1}})$
    \end{enumerate}
    Conversely, given two tuples $(n_0,\dots,n_{k-1})$ and $(\underline{x_0},\dots,\underline{x_{2k-1}})$ which verifies conditions~\ref{item:Correct_Cardinal}, \ref{item:Boundary} and~\ref{item:Matching}, there exists a chain $(R_0,\dots,R_{k-1})$ between $X$ and $Y$ which satisfied~\ref{item:Induce_Relations} and~\ref{item:Compute_Distance}.
\end{lem}
\begin{proof}
    Given a chain $((X_0,\dots,X_k),(R_0,\dots,R_{k-1}))$, we can apply \cref{lem:Relations_lift} to each of the $R_i$ to produce a $2k$-tuple $(\underline{x_0},\dots,\underline{x_{2k-1}})$. By construction, this tuple satisfies conditions~\ref{item:Correct_Cardinal}, \ref{item:Boundary} and~\ref{item:Induce_Relations}. Furthermore, $\pi_{n_i}(\underline{x_{2i}})=X_i=\pi_{n_{i-1}}(\underline{x_{2i-1}})$, thus condition~\ref{item:Matching} is also satisfied. Finally, condition~\ref{item:Compute_Distance} also follows from \cref{lem:Relations_lift}.

    For the converse, observe that condition~\ref{item:Induce_Relations} unambiguously characterizes the relations $R_i$, which are surjective by construction. On the other hand, condition~\ref{item:Matching} guarantees that $(R_0,\dots, R_{k-1})$ forms a combinatorial chain and condition~\ref{item:Boundary} implies that it is a chain between $X$ and $Y$. Finally, condition~\ref{item:Compute_Distance} is simply a direct computation following from \cref{lem:Relations_lift}.
\end{proof}

Finally, we exhibit an upper bound for the distance of intermediate configurations to one of the end of a combinatorial chain. This will be necessary to prove that we can compute the distance over a closed and bounded set in a product of $M$, hence a compact set.
\begin{lem}\label{lem:Sequences_In_Balls}
    Let $X,Y\in \Ran(M)$ be two configurations, and $(R_0,\dots,R_{k-1})$ a combinatorial chain between $X$ and $Y$, with $\card(R_i)=n_i$. Then, if $(\underline{x_0},\dots,\underline{x_{2k-1}})$ is a $2k$-tuple obtained from \cref{lem:Sequence_tuples}, we have, for all $0\leq i\leq k-1$
    \begin{equation*}
        \underline{x_{2i}},\underline{x_{2i+1}}\in \left(\overline{B}_M(X,\elo(R_0,\dots,R_{k-1}))\right)^{n_i},
    \end{equation*}
    Where, for $X\subset M$ and $r\geq 0$, 
    \begin{equation*}
        \overline{B}_M(X,r)=\{x\in M\mid \dm(x,X)\leq r\}
    \end{equation*}
\end{lem}

\begin{proof}
    Since, for all $1\leq i\leq k-1$, every coordinate of $\underline{x_{2i}}$ is equal to one of the coordinates of $\underline{x_{2i-1}}$, it is enough to prove the claim for the terms $\underline{x_0}$ and $\underline{x_{2i+1}}$, $0\leq i\leq k-1$. Since $\pi_{n_0}(\underline{x_0})=X$, the claim is clear for $\underline{x_0}$. Now, it remains to be proven that for any $0\leq i\leq k-1$, and $1\leq j\leq n_i$, $\dm\left(x_{2i+1}^{(j)},X\right)\leq \elo(R_0,\dots,R_{k-1})$. 
    Now, write $X_l=\pi_{n_l}(\underline{x_{2l}})$, for $0\leq l\leq k-1$. Since the relations $R_l$ are surjective, there must exist $y_0,\dots,y_i$, such that $y_l\in X_l$, for $0\leq l\leq i$, and $y_i=x_{2i+1}^{(j)}$, and such that $y_lR_ly_{l+1}$, for all $0\leq l\leq i-1$. But then, we have
    \begin{align*}
        \dm\left(x_{2i+1}^{(j)},X\right)&\leq \dm(y_i,y_0)\\
        &\leq \sum_{l=0}^{i-1}\dm(y_l,y_{l+1})\\
        &\leq \sum_{l=0}^{i-1}\left(\omega(n_l)\sum_{xR_ly}\dm(x,y)\right)\\
        &\leq \elo(R_0,\dots,R_{i-1})\\
        &\leq \elo(R_0,\dots,R_{k-1})
    \end{align*}
\end{proof}

\section{Length structure}
\label{sec:PathMetricsOnRan}
The goal of this section is to show that the weighted distances on $\Ran(M)$ that we introduced in \cref{sec:Combinatorial_Distance} are in fact length metrics (see \cref{sec:PathMetricSpaces}), provided $M$ is a length space itself. This is the content of \cref{theo:PathDistanceEqualCombDistance}. This allows for an interpretation of the distances in $\Ran(M)$ as being induced by a set of geodesics, for which we provide an explicit description in \cref{sec:GeodesicPaths}.

In the following $M$ is a length space (see \cref{sec:PathMetricSpaces}) with length structure $\cP$ and $\ell$, and $\omega$ is a weight.

\subsection{Paths and lifts}
\label{sec:Lifts}
The goal of this section is to define a length structure on $\Ran(M)$. One may wonder what a continuous path on $\Ran(M)$ may be depending on which topology is chosen. Indeed, the set of continuous path on $\Ran(M)$ are distinct if one considers the Hausdorff topology, or the final topology, or even one of the topologies induced by the weighted distances. However, paths in $\Ran(M)$ that we consider to be admissible will be continuous for all of those topologies (see \cref{rem:Chemins_Continus_Hausdorff_Faible}), thus the ambiguity on what continuous means will be alleviated immediately.

\begin{exa}
    Before specifying to which paths in $\Ran(M)$ we will associate a length, let us first give a few examples of what continuous paths in $\Ran(M)$ may look like. Consider $M=\R$, in that case, any path can be plotted as the graph of a multi-valued function, as displayed on \cref{fig:exemples_chemins}. Given a path $\gamma\colon [a,b]\to \R$, we also plot the integer-valued function $t\mapsto\card(\gamma(t))$. Note that, even over a closed interval, path of unbounded cardinality may be continuous for the Hausdorff topology, however they will be excluded from the definition of admissible paths. All paths pictured here will in fact be admissible, for the usual length structure on $\R$, induced by its structure of Riemannian manifold. 
\end{exa}

    \begin{figure}[htp]
        \centering
        \input{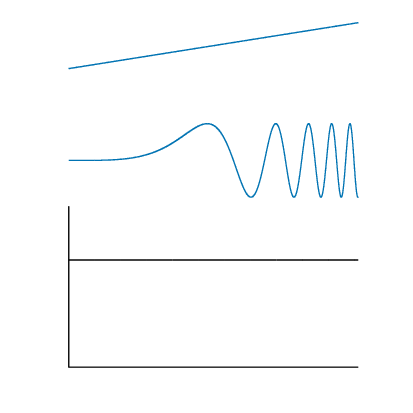}
        \input{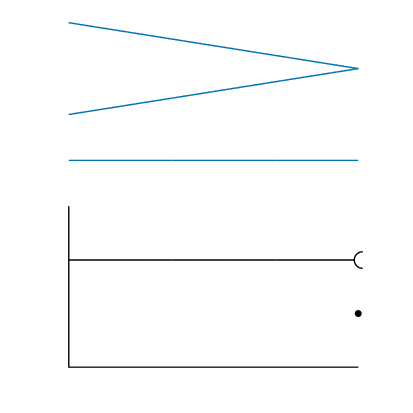}
        \input{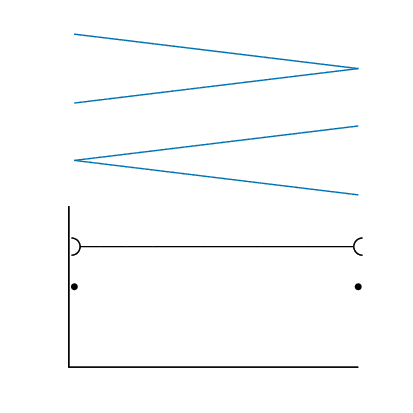}
        \input{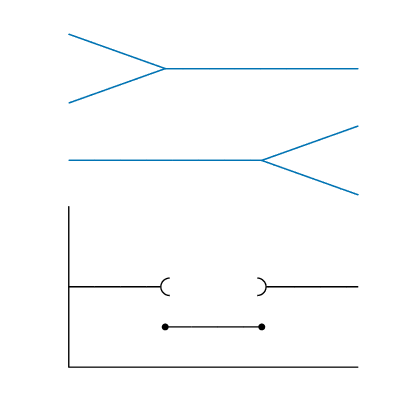}
        \input{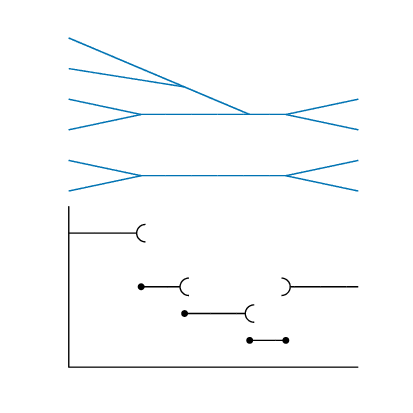}
        \input{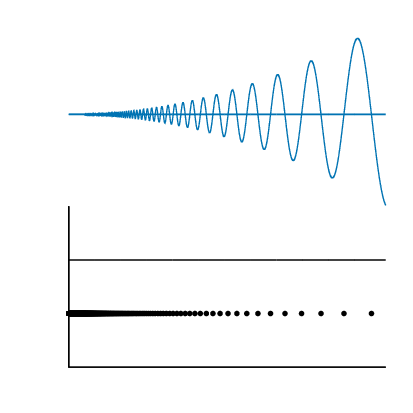}
        \caption{Examples of continuous paths in $\Ran(\R)$ and their cardinality.}
        \label{fig:exemples_chemins}
    \end{figure}

\label{sec:Chemin_Relevement_C1_Morceaux}
\begin{defin}\label{def:relevements}
    Let $a\leq b\in \R$, and $\Gamma\colon [a,b]\to \Ran(M)$ a continuous path. A lift of $\Gamma$ is given by
    \begin{itemize}
        \item a subdivision of $[a,b]$, $a=t_0<\dots < t_k=b$, with $k\geq 0$,
        \item for all $0\leq i\leq k-1$, a continuous path, $\underline{\gamma_i}\colon [t_i,t_{i+1}]\to M^{n_i}$, with $n_i\geq 1$, such that $\pi_{n_i}(\underline{\gamma_i}(t))=\Gamma(t)$ for all $t\in [t_i,t_{i+1}]$. 
    \end{itemize} 
    Such a lift is written $((t_0,\dots,t_k),(\underline{\gamma_0},\dots,\underline{\gamma_{k-1}}))$, or just $(\underline{\gamma_0},\dots,\underline{\gamma_{k-1}})$. This lift is said to be admissible if all the $\underline{\gamma_i}$ are admissible paths in $ M^{n_i}$ (see \cref{prop:Length_structure_Mn}).
\end{defin}

\begin{defin}\label{def:Chemin_C1_Par_Morceaux}
    Let $\Gamma\colon [a,b]\to \Ran(M)$. We say that $\Gamma$ is admissible, if it is continuous and admits an admissible lift.
\end{defin}

\begin{rem}
    \label{rem:Chemins_Continus_Hausdorff_Faible}
    The assumption of continuity for an admissible path is redundant, though we keep it for clarity. Indeed, if $\Gamma\colon [a,b]\to \Ran(M)$ is a function, not necessarily continuous, and if there exits a lift of $\Gamma$, $((t_0,\dots,t_k),(\underline{\gamma_0},\dots,\underline{\gamma_{k-1}}))$ such that every $\underline{\gamma_i}\colon [t_i,t_{i+1}]\to M^{n_i}$, $0\leq i\leq k-1$ is continuous, then, by construction, 
    $\Gamma\colon [a,b]\to \Ran(M)$ is continuous for the final topology and thus for the Hausdorff topology. Indeed, continuity of $\Gamma$ on the intervals $[t_i,t_{i+1}]$ is clear, and continuity at $t_i$ follows from the condition the $\pi_{n_i}(\underline{\gamma_i}(t_i))=\pi_{n_{i-1}}(\underline{\gamma_{i-1}}(t_i))$, which is the definition of $(\underline{\gamma_0},\dots, \underline{\gamma_{k-1}})$ being a lift.
    
    In particular, the set of continuous paths in $\Ran(M)$ which admit a lift (admissible or otherwise), does not depend on whether we consider the final topology or the Hausdorff topology. Note however that this is no longer the case if we drop the hypothesis of admitting a lift. Indeed, there exist continuous paths whose cardinality is unbounded both for the Hausdorff topology and for weighted topologies.
    Let us show how results that we will prove latter allow one to construct such a path. Consider the sequence of unbounded cardinality, $(X_n)_{n\in \N}$ which converges to $\{0\}$ in $\Ran([0,1])$ for $\dom$,
    introduced in \cref{exa:config_convergente}. Up to extracting a subsequence, we may assume that $\sum_{n=0}^{\infty}\dom(X_n,X_{n+1})<\infty$.
    Since this is a sequence in $\Ran([0,1])$, \cref{cor:existence geodesic path} applies, and up to reparametrization,
    we may find continuous paths between $X_n$ and $X_{n+1}$, $\Gamma_n\colon [1-\frac{1}{2^n},1-\frac{1}{2^{n+1}})\to \Ran(M)$
    whose length will be $\dom(X_n,X_{n+1})$.
    Concatenating those paths gives a continuous map $\Gamma\colon [0,1)\to \Ran(M)$ of unbounded cardinality, but the convergence of the sequence $(X_n)_{n\in \N}$ to $\{0\}$ (together with the fact that $\sum_{k=n}^{\infty}\elo(\Gamma_k)$ tends to $0$) implies that we can extend this map by continuity by defining $\Gamma(1)=\{0\}$.
    This is an example of the desired form for the weighted topology, but since it refines the Hausdorff topology (\cref{prop:tom raffine tH}) it is also an example for the Hausdorff topology. Such a path cannot be continuous for the final topology however, since if it were, we could deduce that $(X_n)$ converges to $\{0\}$ for the final topology, which is impossible by \cref{prop:unbounded cardinal diverges in final}. 
    
    Furthermore, path of the above form do not admit any lift (admissible or otherwise). Since we are only interested in admissible paths here, which admit a lift by definition, we may thus unambiguously speak of their continuity without specifying the choice of topology on $\Ran(M)$.
    
    The previous comment also extends to the topologies induced by the weighted distances. In fact, observe that, by definition, admissible paths in $\Ran(M)$ must factor through $M^n$ for some large enough $n$, and this factorization must in fact be continuous. Now, observe that the projection $M^n\to \Ran(M)$ is continuous for every topology we considered on $\Ran(M)$. Thus, admissibility of a path reduces to the existence of a continuous (global) lift to $M^n$ for some $n$. The reason we did not define admissibility this way from the start is that the above definition of a lift is better suited for the definition of the length of an admissible path. 
\end{rem}

\begin{defin}\label{def:longueur_rel}
    Let $\Gamma\colon [a,b]\to \Ran(M)$, and $((t_0,\dots,t_k),(\underline{\gamma_0},\dots,\underline{\gamma_{k-1}}))$ be an admissible lift of $\Gamma$. The length of the lift relative to $\omega$ is defined as
    \begin{equation*}
        \elo(\underline{\gamma_0},\dots,\underline{\gamma_{k-1}})=\sum_{i=0}^{k-1} \omega(n_i)\ell^{n_i}(\underline{\gamma_i})
    \end{equation*}
\end{defin}

\begin{defin}\label{def:longueur_chemin}
    Let $\Gamma\colon [a,b]\to \Ran(M)$ be an admissible path. Its length relative to $\omega$ is defined as
    \begin{equation*}
        \elo(\Gamma)=\inf\{\elo(\underline{\gamma_0},\dots,\underline{\gamma_{k-1}})\}
    \end{equation*}
    where the infimum is taken over the admissible lifts of $\Gamma$.
\end{defin}

\begin{rem}\label{rem:subdivision_concatenations_longueurs}
    Observe that the length of a lift is conserved when refining the subdivision. In particular, the length of a concatenation of lifts is the sum of the lengths of each concatenated piece. Similarly the length of a concatenation of paths is the sum of the paths length.
\end{rem}

We can now show that the weighted distance $\dom$ is in fact a length metric, for the length structure where $\cP$ is the set of paths admitting an admissible lift, and the length is $\elo$.
\begin{theo}\label{theo:PathDistanceEqualCombDistance}    
    Let $X,Y\in \Ran(M)$. The weighted distance, relative to $\omega$, between $X$ and $Y$ is equal to
    \begin{equation*}
        \dom(X,Y)=\inf\{\elo(\Gamma)\mid \Gamma\colon [a,b]\to \Ran(M), \Gamma(a)=X,\Gamma(b)=Y\}
    \end{equation*}
    where the infimum is computed over all admissible paths from $X$ to $Y$.
\end{theo}

\begin{proof}
    Let $X,Y\in\Ran(M)$, and let us write 
    \begin{equation*}
        I=\{\Gamma \colon [a,b]\to\Ran(M)\mid a<b\in\R\text{ and $\Gamma$ admissible path from $X$ to $Y$}\}.
    \end{equation*} 
    
    Let us first show that $\inf_{\Gamma\in I}\{\elo(\Gamma)\}\geq\dom(X,Y)$. Let $\epsilon>0$ and $\Gamma\in I$ such that 
    \begin{align*}
        \elo(\Gamma)\leq \inf_{\Lambda\in I}\{\elo(\Lambda)\}+\frac{\epsilon}{2}.
    \end{align*} 
    Let $((a=t_0,\dots,t_k=b),(\underline{\gamma_0},\dots,\underline{\gamma_{k-1}}))$ be an admissible lift of $\Gamma$ such that
     \begin{align*}
        \elo(\underline{\gamma_0},\dots,\underline{\gamma_{k-1}})\leq \elo(\Gamma) + \frac{\epsilon}{2}.
    \end{align*}
    For all $0\leq i\leq k-1$, we write explicitly
    \begin{align*}
        \underline{\gamma_i}\colon[t_i,t_{i+1}]&\to M^{n_i}\\
        t&\mapsto (\gamma_i^{(1)}(t),\dots,\gamma_i^{(n_i)}(t))
    \end{align*}
   
    Let us define $((X_0,\dots,X_k),(R_0,\dots,R_{k-1}))$ as 
    \begin{align*}
        &X_i=\Gamma(t_i), \quad \text{for all } 0\leq i\leq k\\ 
        &R_i=\left\{ \left( \gamma_i^{(j)}(t_i),\gamma_i^{(j)}(t_{i+1}) \right)\mid 1\leq j\leq n_i\right\}, \quad \text{for all } 0\leq i\leq k-1.
    \end{align*}
    One easily checks that, since $(\underline{\gamma_0},\dots,\underline{\gamma_{k-1}})$ is a lift of $\Gamma$, this defines a combinatorial chain from $X$ to $Y$. Let $0\leq i\leq k-1$. For any $1\leq j\leq n_i$ we have
    \begin{align*}
        \dm\left(\gamma_i^{(j)}(t_i),\gamma_i^{(j)}(t_{i+1})\right)&=\inf\{\ell(\lambda) \} \leq \ell(\gamma_i^{(j)}),
    \end{align*}
    where is the infimum is taken over all admissible paths in $M$ from $\gamma_i^{(j)}(t_i)$ to $\gamma_i^{(j)}(t_{i+1})$. Hence
    \begin{align*}
        \elo(R_i)&=\omega(n_i)\sum_{j=1}^{n_i} \dm\left(\gamma_i^{(j)}(t_i),\gamma_i^{(j)}(t_{i+1})\right)\leq \omega(n_i) \sum_{j=1}^{n_i}\ell(\gamma_i^{(j)}) =\elo(\underline{\gamma_i}).
    \end{align*}
    We finally get
    \begin{align*}
        \dom(X,Y)\leq \elo(R_0,\dots,R_{k-1})\leq \elo(\underline{\gamma_0},\dots,\underline{\gamma_{k-1}})\leq \inf_{\Lambda\in I}\left\{\elo(\Lambda)\right\} + \epsilon.
    \end{align*}
    Since this is true for any $\epsilon>0$, we have the desired inequality.

    Let us now show that $\dom(X,Y)\geq \inf_{\Lambda\in I}\left\{\elo(\Lambda)\right\}$. Let $\epsilon>0$. By definition, there exists $((X=X_0,\dots,X_k=Y),(R_0,\dots,R_{k-1}))$ a combinatorial chain from $X$ to $Y$ such that
    \begin{align*}
        \elo(R_0,\dots,R_{k-1})\leq \dom(X,Y) + \frac{\epsilon}{2}.
    \end{align*}
    Let $0\leq i\leq k-1$, write $n_i=\card(R_i)$ and let us fix an order on $R_i=\{(x^{(j)},y^{(j)})\}_{1\leq j\leq n_i}$. We then choose an admissible path in $M^{n_i}$, $\underline{\gamma_i}\colon [i,i+1]\to M^{n_i}$ from $\underline{x}=(x^{(1)},\dots,x^{(n_i)})$ to $\underline{y}=(y^{(1)},\dots,y^{(n_i)})$ which we may assume to satisfy
    \begin{align*}
\ell^{n_i}(\underline{\gamma_i}) \leq \dsum(\underline{x},\underline{y})+ \frac{\epsilon}{2k\omega(n_i)}.
    \end{align*}
    Then, we directly have
    \begin{align*}
        \elo(\underline{\gamma_i})=\omega(n_i)\ell^{n_i}(\underline{\gamma_i}) \leq \frac{\epsilon}{2k} + \omega(n_i)\sum_{j=1}^{n_i} \dm(x^{(j)},y^{(j)}) = \frac{\epsilon}{2k} + \elo(R_i).
    \end{align*}
    Moreover, since $R_i$ is surjective, we also have $\pi_{n_i}(\underline{\gamma_i}(i))=X_i$ and $\pi_{n_i}(\underline{\gamma_i}(i+1))=X_{i+1}$, so that $((0,\dots,k),(\underline{\gamma_0},\dots,\underline{\gamma_{k-1}}))$ defines an admissible lift of the path $\Gamma$ defined as
    \begin{align*}
        \Gamma\colon [0,k]&\to\Ran(M)\\
        t&\mapsto \pi_{n_i}(\underline{\gamma_i}(t)) \quad \text{if } i\leq t \leq i+1,
    \end{align*}
    which is by construction an admissible path from $X$ to $Y$.
    We finally have
    \begin{align*}
        \inf_{\Lambda\in I}\left\{\elo(\Lambda)\right\}\leq \elo(\Gamma)\leq \elo(\underline{\gamma_0},\dots,\underline{\gamma_{k-1}}))\leq \elo(R_0,\dots,R_{k-1}) + \frac{\epsilon}{2} \leq \dom(X,Y)+\epsilon.
    \end{align*}
    Since this is true for any $\epsilon>0$, this concludes this argument, and this proof.
\end{proof}

\begin{corol}
\label{cor:ran length space}
 The metric space $(\Ran(M),\dom)$ is a length space.
\end{corol}

\begin{proof}
    We have checked all but the last two conditions in \cref{def:Length_structure}. The continuity assumption follows from that of the length structure on $M$. And by \cref{rem:checking_length_structure}, the last assumption follows from the fact that $d^{\elo}$ is indeed a distance refining the Hausdorff topology.
\end{proof}

\subsection{Geodesic paths}
\label{sec:GeodesicPaths}
In this section we show that the infimum in \cref{theo:PathDistanceEqualCombDistance} can be taken over a smaller class of path that have a piecewise constant cardinality, and for which the infimum in \cref{def:longueur_chemin} is realized (that is, they admit a lift with the correct length). We call these paths crenelated paths. Then we show that we can restrict even further to paths whose cardinality first decreases and then increases. By similarity with MBS-chains, we call these paths MBS-paths. Finally, we use these results to show that, provided geodesic chains exist in $\Ran(M)$, and geodesics between any two points of $M$ exist, then the infimum in \cref{theo:PathDistanceEqualCombDistance} is realized by a MBS-path.

We start by showing that path with a constant cardinality always admit a lift of minimal dimension.
\begin{lem}\label{lem:Relevement_continu_card_constant}
Let $n\geq 1$, $I$ an interval (closed, open or semi-open), and $\Gamma\colon I\to \Ran(M)$ a continuous path, such that $\card(\Gamma(t))=n$, for all $t\in I$. Then there exists a continuous lift of $\Gamma$, $\underline{\gamma}\colon I\to M^n$. Moreover, if $\underline{\lambda}\colon I\to M^m$ is such that $\pi_m(\underline{\lambda}(t))=\Gamma(t)$, $\forall t\in I$, then there exists a surjection $f\colon \{1,\dots,m\}\to \{1,\dots,n\}$ such that $\lambda=f\cdot\gamma$.
\end{lem}

\begin{proof}
The path $\Gamma$ has constant cardinality, so that its image is a subset of $\Conf_n(M)$. Since $\pi_n\colon \Conford_n(M)\to\Conf_n(M)$ is a covering, there exists a lift of $\Gamma$, $\underline{\gamma}=(\gamma^{(1)},\dots,\gamma^{(n)})\colon I\to \Conford_n(M)\subset M^n$.

Let $\underline{\lambda}=(\lambda^{(1)},\dots,\lambda^{(m)})\colon I\to M^m$ such that $\pi_m(\underline{\lambda}(t))=\Gamma(t)$, $\forall t\in I$. Since the cardinality of $\Gamma(t)$ is equal to $n$ for all $t\in I$, we must have $m\geq n$.

Let us fix $t\in I$. Since $\pi_m(\underline{\lambda}(t))=\Gamma(t)$, for all $i\in\{1,\dots,m\}$, $\lambda^{(i)}(t)\in\Gamma(t)$ and since $\Gamma(t)=\pi_n(\underline{\gamma}(t))$ and $\card(\Gamma(t))=n$, there exists a unique $j_i\in\{1,\dots,n\}$ such that $\lambda^{(i)}(t)=\gamma^{(j_i)}(t)$.
We can thus define a map $f\colon\{1,\dots,m\}\to\{1,\dots,n\}$ such that $f(i)=j_i$, which will automatically be surjective. Then, $\lambda(t)=f\cdot\gamma(t)$.

We then define $\Omega=\{s\in I\mid \lambda(s)=f\cdot\gamma(s)\}$. Let $s\in I$, we have $s\in \Omega$ if and only if,
     \begin{enumerate}
         \item for all $i\in \{1,\dots,m\}$, $\lambda^{(i)}(s)=\gamma^{(f(i))}(s)$,
         \item for all $i\in\{1,\dots,m\}$, $j\in\{1,\dots,n\}$, $j\neq f(i) \Rightarrow \lambda^{(i)}(s)\neq\gamma^{(j)}(s)$.
     \end{enumerate}

Note that both conditions are actually equivalent. Indeed, if $s\in I$ satisfies the first condition, then, as the $\gamma^{(j)}(s)$ are pairwise distinct, $s$ also satisfies the second condition. And if $s\in I$ satisfies the second condition then, for all $i\in\{1,\dots,m\}$, since $\lambda^{(i)}(s)\in \pi_n(\underline{\gamma}(s))$, there exists $j\in\{1,\dots,n\}$ such that $\lambda^{(i)}(s)=\gamma^{(j)}(s)$, which implies $\lambda^{(i)}(s)=\gamma^{(f(i))}(s)$.

Let us consider the following maps:
     \begin{align*}
         g\colon I&\to \R_{\geq 0}\\
         s&\mapsto \sum_{i=1}^m \dm(\lambda^{(i)}(s),\gamma^{(f(i))}(s))\\
         h\colon I&\to \R_{\geq 0}\\
         s&\mapsto \prod_{i=1}^m\prod_{\substack{j=1\\ j\neq f(i)}}^n \dm(\lambda^{(i)}(s),\gamma^{(j)}(s))
     \end{align*}
The maps $g$ and $h$ are continuous. Moreover, using the previous remarks, we have $\Omega=g^{-1}(0)=h^{-1}((0,+\infty))$. We deduce that $\Omega$ is both open and closed. Since $t\in \Omega$ by definition, and since $\Omega\subset I$ which is connected, we have $\Omega=I$.
\end{proof}

Since we want to consider paths that have a piecewise constant cardinality, we want to extend the previous result to path whose cardinality is only constant over an open interval.
\begin{prop}\label{prop:relevement_unique_permutation_pres}
    Let $n\geq 0$, $a<b$ and $\Gamma\colon [a,b]\to \Ran(M)$ an admissible path such that $\card(\Gamma(t))=n$ for all $t\in (a,b)$. Then there exists an admissible lift of $\Gamma$, $\underline{\gamma}\colon [a,b]\to M^n$. Moreover, if $\underline{\lambda}\colon [a,b]\to M^n$ is another admissible lift of $\Gamma$, then there exists a permutation $\phi\in\Sigma_n$ such that $\underline{\lambda}=\phi\cdot\underline{\gamma}$. 
\end{prop}

\begin{proof}
We start by showing unicity. Let $\underline{\lambda},\underline{\mu}\colon [a,b]\to M^n$ two admissible lifts of $\Gamma$. By \cref{lem:Relevement_continu_card_constant}, there exists a continuous lift of $\Gamma_{|(a,b)}$, $\underline{\gamma}\colon (a,b)\to M^n$. 
Moreover, since $\underline{\lambda}_{|(a,b)}$ and $\underline{\mu}_{|(a,b)}$ are two continuous lifts of $\Gamma_{|(a,b)}$, there exist two bijections $\phi,\psi\colon \{1,\dots,n\}\to\{1,\dots,n\}$ such that 
\begin{equation*}
\underline{\lambda}_{|(a,b)}=\phi\cdot \underline{\gamma} \quad \text{ and, }\quad \underline{\mu}_{|(a,b)}=\psi\cdot\underline{\gamma}
\end{equation*}
We thus get $\underline{\lambda}_{|(a,b)}=(\phi\psi^{-1})\cdot \underline{\mu}_{|(a,b)}$. By continuity of $\underline{\lambda},\underline{\mu}$, and of the group action of $\Sigma_n$ on $M^n$, we can extend the equality on the closed interval $[a,b]$, which is the desired result.

For the existence property, since $\Gamma$ is an admissible path, there exists an admissible lift, $((a=t_0<\dots<t_k=b),(\underline{\gamma_0},\dots,\underline{\gamma_{k-1}}))$, where for any $0\leq i\leq k-1$, $\underline{\gamma_i}\colon [t_i,t_{i+1}]\to M^{n_i}$ is an admissible lift of $\Gamma_{|[t_i,t_{i+1}]}$. Now let us construct the desired lift on each subset $[t_i,t_{i+1}]$. Let $0\leq i\leq k-1$. 

Observe that, for any $1\leq i\leq k-2$, the restriction of $\Gamma$ to the interval $[t_i,t_{i+1}]$ is of constant cardinality, $n$. Thus, by \cref{lem:Relevement_continu_card_constant}, we may find lifts $\underline{\lambda_i}\colon [t_i,t_{i+1}]\to M^n$, for all $1\leq i\leq k-2$.
On the other hand, the restriction of $\Gamma$ to the interval $(t_0,t_1]$ is of constant cardinality, $n$, thus by the same lemma, we may find a lift $\lambda_0\colon (t_0,t_1]\to M^n$, together with a surjective map $f\colon \{1,\dots,n_0\}\to \{1,\dots,n\}$ such that $\gamma_0^{(i)}(t)=\lambda_0^{f(i)}(t)$ for all $t\in (t_0,t_1]$. Continuity of $\underline{\lambda_0}$, together with the fact that $M$ is Hausdorff implies that if $i,i'\in \{0,\dots,n_0\}$ are such that $f(i)=f(i')$, then $\lambda_0^{(i)}=\lambda_0^{(i')}$ on the entire closed interval $[t_0,t_1]$. In particular, we may unambiguously define the continuous lift $\underline{\lambda_0}\colon [t_0,t_1]\to M^n$ via $\lambda_0^{(j)}(t)=\gamma_0^{(i)}(t)$ for all $t\in [t_0,t_1]$ for any $i$ such that $f(i)=j$. This will still be a lift of the restriction of $\Gamma$ to $[t_0,t_1]$ by surjectivity of $f$, and it will be admissible since each of its component is a component of $\underline{\gamma_0}$ which was assumed to be admissible.
The symmetric construction yields a lift $\underline{\lambda_{k-1}}\colon [t_{k-1},t_k]\to M^n$. Then, up to choosing a permutation for each subdivided lift, we may concatenate the $\underline{\lambda_l}$ to yield a global lift $\underline{\lambda}\colon [t_0,t_k]\to M^n$.

\end{proof}

Then, we show that the lift of minimal dimension is also minimal in length.
\begin{prop}\label{prop:Chemin_card_constant_relevement}
    Let $n\geq 0$, $a<b\in \R$, $\Gamma\colon [a,b]\to \Ran(M)$ an admissible path such that $\card(\Gamma(t))=n$ for all $t\in (a,b)$. Then there exists an admissible lift of $\Gamma$, $\underline{\gamma}\colon [a,b]\to M^n$. In addition, any such lift satisfies
    \begin{equation*}
        \elo(\Gamma)=\elo(\underline{\gamma}).
    \end{equation*}
\end{prop}

\begin{proof}
The existence is given by \cref{prop:relevement_unique_permutation_pres}.
The same proposition, ensures that any such pair of lifts $\underline{\gamma},\underline{\lambda}\colon [a,b]\to M^n$ there exists a permutation $\phi\in \Sigma_n$ such that $\underline{\gamma}=\phi\cdot\underline{\lambda}$.
Since $\Sigma_n$ acts isometrically on $M^n$, we get $\elo(\underline{\gamma})=\elo(\underline{\lambda})$. 

Now let $((t_0,\dots,t_k),(\underline{\mu_0},\dots,\underline{\mu_{k-1}}))$ an admissible  lift of $\Gamma$.
Let us show that $\elo(\underline{\mu_0},\dots,\underline{\mu_{k-1}})\geq \elo(\underline{\gamma})$. By \cref{rem:subdivision_concatenations_longueurs}, we just have to show that $\elo(\underline{\mu_i})\geq \elo(\underline{\gamma}_{[t_i,t_{i+1}]})$ for all $0\leq i\leq k-1$.
For some given $i$, we have the lift $\underline{\mu_i}\colon [t_i,t_{i+1}]\to M^{n_i}$. 
Using \cref{lem:Relevement_continu_card_constant}, and extending by continuity, we get the existence of a surjection $f\colon \{1,\dots,n_i\}\to \{1,\dots,n\}$ such that $\underline{\mu_i}=f\cdot \underline{\gamma}_{|[t_i,t_{i+1}]}$.
We thus have :
\begin{align*}
\elo(\underline{\gamma}_{|[t_i,t_{i+1}]})&=\omega(n)\sum_{j=1}^{n}\ell(\gamma^{(j)}_{|[t_i,t_{i+1}})\\
&\leq\omega(n_i)\sum_{k=1}^{n_i}\ell(\gamma^{(f(k))}_{|[t_i,t_{i+1}]})  \\
&\leq \omega(n_i)\sum_{k=1}^{n_i}\ell(\mu_i^{(k)})\\
&\leq\elo(\underline{\mu_i})
\end{align*}
which concludes this proof.
\end{proof}

We can now concatenate paths whose cardinality is constant over an open interval to obtain a path whose cardinality is piecewise constant. This is the object of the following definition.

\begin{defin}\label{def:Chemin_Crenele}
Let $\Gamma\colon [a,b]\to \Ran(M)$ an admissible path. We say that $\Gamma$ is \textbf{crenelated} if there exists a subdivision $a=t_0<\dots<t_k=b$, and integers $n_0,\dots,n_{k-1}$ such that, for all $0\leq i\leq k-1$, and for all $t\in(t_i,t_{i+1})$, $\card(\Gamma(t))=n_i$.
\end{defin}

We can thus concatenate minimal lifts to obtain a minimal lift of crenelated paths.
\begin{prop}\label{prop:longueur_realisee}
Let $\Gamma\colon[a,b]\to\Ran(M)$ a crenelated path. The infimum in  \cref{def:longueur_chemin} is realized. 
    
Moreover, if $a=t_0<\dots<t_k=b$ is a subdivision of $[a,b]$ such that $\card(\Gamma(t))=n_i$ for all  $t\in (t_i,t_{i+1})$ and $i\in \{0,\dots,k-1\}$, then any lift of the form $((t_0,\dots,t_k),(\underline{\gamma_0},\dots,\underline{\gamma_{k-1}}))$, with $\underline{\gamma_i}\colon [t_i,t_{i+1}]\to M^{n_i}$, for $0\leq i\leq k-1$, realizes the infimum.
\end{prop}

\begin{proof}
By hypothesis, there exists $a=t_0<\dots<t_k=b$, a subdivision of $[a,b]$, and $(n_i)_{i\in \{1,\dots,k-1\}}$ integers, such that for all $i\in\{0,\dots,k-1\}$, and for all $t\in (t_i,t_{i+1})$, $\card(\Gamma(t))=n_i$. 
Now, by definition of $\elo$ on admissible paths in $\Ran(M)$, we have the equality
\begin{equation*}
    \elo(\Gamma)=\sum_{i=0}^{k-1}\elo(\Gamma_{|[t_i,t_{i+1}]}).
\end{equation*}
By \cref{prop:Chemin_card_constant_relevement} the length $\elo(\Gamma_{|[t_i,t_{i+1}]})$ is realized by an (unique up to permutation) admissible lift $\underline{\gamma_i}\colon [t_i,t_{i+1}]\to M^{n_i}$. Concatenating the $\underline{\gamma_i}$ gives the desired lift.
\end{proof}

Now that we know how to compute the length of admissible paths, we want to restrict the set of paths over which we compute the infimum in \cref{theo:PathDistanceEqualCombDistance}. As in \cref{sec:Geodesic_chains}, we will use this result to show that the infimum can, under some hypotheses, be realized. First we define properly the path equivalent of MBS-chains.
\begin{defin}
    Let $\Gamma\colon[a,b]\to\Ran(M)$ be a path in $\Ran(M)$. It is called a MBS-path if it is crenelated and if there exists $a\leq c\leq b$ such that $\card(\Gamma)_{|[a,c]}$ is non increasing and $\card(\Gamma)_{|[c,b]}$ is non decreasing.
\end{defin}
Then we show that we can restrict the computation of the infimum in \cref{theo:PathDistanceEqualCombDistance} over MBS-paths.
\begin{prop}
    \label{prop: inf sur chemin crenele}
    Let $X,Y\in\Ran(M)$. The infimum in \cref{theo:PathDistanceEqualCombDistance} can be computed over all MBS-paths. 
\end{prop}

The proof relies on the construction, from a combinatorial chain, of a crenelated path with a shorter length (this relies on \cref{lem:AtomicMBS,lem:AtomicToPath}). To control the cardinality of the constructed path, it is easier to work with the following definition.
\begin{defin}
    Let $X,Y\in\Ran(M)$ and $R$ a relation between $X$ and $Y$. We say that $R$ is \textbf{atomic} if it is surjective and if there exists at most a unique $(x,y)\in R$ such that $x\neq y$. We say that $R$ is an atomic relation from $x$ to $y$. 
\end{defin}

\begin{rem}
    If $R$ is an atomic relation, then $R$ is either (i) the graph of the identity, or (ii) the graph of a map that is the identity everywhere except in one point, which is either a bijection or a simple merge, or (iii) $R^\mathrm{op}$ is the graph of a map that is the identity everywhere except in one point, so that $R$ is either a bijection or a simple split.
    
    Furthermore, observe that if $R$ is an atomic relation from $x$ to $y$, then the $\omega$-length of $R$ is equal to $\elo(R)=\omega(\card(R))\dm(x,y)$.
\end{rem}
    
\begin{proof}
    Let $\epsilon>0$. By \cref{theo:MBS} there exists a simple MBS-chain $((X=X_0,\dots,X_k=Y),(R_0,\dots,R_{k-1}))$ from $X$ to $Y$ such that $\elo(R_0,\dots,R_{k-1})\leq \dom(X,Y)+\frac{\epsilon}{2}$.
    By \cref{lem:AtomicMBS}, we may replace the chain by a chain $(A_0,\dots,A_{l-1})$ with the same length or shorter, but such that for some $0\leq j\leq l$, $A_i$ is either an atomic merge or an atomic bijection for all $i<j$ and either an atomic split or an atomic bijection for all $i\geq j$. By \cref{lem:AtomicToPath}, each of those atomic relation can be realized by an admissible path $\Gamma_i$, with $\elo(\Gamma_i)\leq \elo(A_i)+\frac{\epsilon}{2l}$. For convenience, assume that $\Gamma_i$ is parametrized by the interval $[i,i+1]$. We know that for $0\leq i\leq j$, $\card(\Gamma_j(t))$ is non increasing on $[i,i+1]$, thus the concatenation of the $\Gamma_i$ for $0\leq i\leq j$ gives a map $[0,j+1]\to \Ran(M)$ whose cardinality is non increasing. Similarly, the concatenation of the $\Gamma_i$ for $j<i\leq l$ gives a map $[j+1,l]\to \Ran(M)$ whose cardinality is non decreasing. Thus, the overall concatenation $\Gamma\colon [0,l]\to \Ran(X)$ is a MBS-path. Furthermore, we have
    \begin{align*}
        \elo(\Gamma)&=\sum_{i=0}^{l-1}\elo(\Gamma_i)\\
        &\leq \sum_{i=0}^{l-1}\left(\elo(A_i)+\frac{\epsilon}{2l}\right)\\
        &\leq \frac{\epsilon}{2}+\elo(A_0,\dots,A_{l-1})\\
        &\leq \frac{\epsilon}{2}+\elo(R_0,\dots,R_{k-1})\\
        &\leq \dom(X,Y)+\epsilon
    \end{align*}
    Thus, there exists MBS-path in $\Ran(M)$ between $X$ and $Y$ of $\omega$-length arbitrary close to $\dom(X,Y)$.
\end{proof}

We need to decompose the relations of MBS-chains into atomic relations, without increasing their length. This is object of the following lemma. 
\begin{lem}\label{lem:AtomicMBS}
Let $X,Y\in \Ran(M)$ and $(R_0,\dots,R_{k-1})$ be a MBS-chain between $X$ and $Y$. Then, there exists another combinatorial chain, $(A_0,\dots,A_{l-1})$, between $X$ and $Y$, together with some $0\leq j\leq l$ such that:
\begin{itemize}
    \item $\elo(A_0,\dots,A_{l-1})\leq \elo(R_0,\dots,R_{k-1})$,
    \item $A_i$ is either an atomic merge or an atomic bijection, for all $i<j$,
    \item $A_i$ is either an atomic split or an atomic bijection, for all $i\geq j$
\end{itemize}
\end{lem}

\begin{proof}
     We will prove the result in two parts. First, we will show that any simple merge $R_i$ can be decomposed, without increasing its length, into an atomic bijection followed by an atomic merge. This will imply the corresponding statement for splits. Secondly, we will show that any bijection $R$ between configurations $Z$ and $W$ must satisfy one of the following two properties
     \begin{itemize}
         \item $R$ can be decomposed, without increasing its length into a sequence of atomic bijections.
         \item There exists a MBS-chain between $Z$ and $W$, $(Q_0,\dots,Q_{l-1})$, with $l\geq 2$ and which is shorter than $R$.
     \end{itemize}
     Then, to obtain the statement, first consider the bijection step in $(R_0,\dots,R_{k-1})$. Either it can be decomposed into atomic bijections, or it can be shortened by replacing it by some MBS-chain $(Q_0,\dots,Q_{l-1})$. In the latter case, since $l\geq 2$, the bijection step in $(Q_0,\dots,Q_{l-1})$ must be of strictly smaller cardinality than that in $(R_0,\dots,R_{k-1})$. Since any bijection of cardinality $1$ is an atomic bijection, the previous construction always yields, in finitely many steps, a MBS-chain $(R'_0,\dots,R'_{k-1})$ whose bijection step decomposes into atomic bijections. Then, decomposing all merges (resp. splits) into atomic merges (resp. atomic splits) and atomic bijections gives the desired result.
     
     Now, for the first step, assume that $R$ is a simple merge between $X$ and $Y$. Then there exists $x\neq x'\in X$ and $y\in Y$ such that $(x,y),(x',y)\in R$. If either $x=y$ or $x'=y$ then $R$ is already atomic. Otherwise, $y$ is not in $X$ (since if it were, the merge would not be simple) and we can decompose $R$ into two atomic relations as follows. Let us define $X'=X\setminus\{x\}\cup\{y\}$ and $P=R\setminus\{(x',y)\}\cup \{(x',x')\}$ and $P'=R\setminus\{(x,y)\}$. Then $P$ is an atomic bijective relation from $X$ to $X'$, and $P'$ is an atomic merge from $X'$ to $Y$. Furthermore, if $n=\card(R)$, then $\elo(R)=\omega(n)(d(x,y)+d(x',y))$ and on the other hand, $\elo(P)=\omega(n)d(x,y)$ and $\elo(P')=\omega(n)d(x',y)$, which gives $\elo(R)=\elo(P,P')$.
     
    For the second step, let $R$ be a bijection between $X$ and $Y$, and write $\phi\colon X\to Y$ for the corresponding map. Let us first assume that for all $x\in X\cap Y$, $\phi(x)=x$. Let us fix an order on the elements in $X\setminus(X\cap Y)=\{x_1,\dots,x_n\}$. For $0\leq i\leq n$, let $X_i= (X\cap Y)\cup \{\phi(x_1),\dots,\phi(x_i),x_{i+1},\dots,x_n\}$, and define $\phi_i\colon X_i\to X_{i+1}$ by $\phi_i(x_{i+1})=\phi(x_{i+1})$ and $\phi_i(z)=z$ for all $z\in X_i\setminus\{x_{i+1}\}$. Then, by construction, $X=X_0$, $Y=X_n$, and each of the $X_i$ is of the same cardinality as $X$, and each of the $\phi_i$ is a bijection, whose graph we will denote $R_i$. Furthermore, $\elo(R_i)=\omega(\card(R))d(x_{i+1},\phi(x_{i+1}))$ and hence, $\elo(R)=\elo(R_0,\dots,R_{n-1})$, which completes the proof in that case.
    
    Now, assume that there exists some $x\in X\cap Y$ such that $\phi(x)\not=x$. Let $\psi\colon Y\to X$ denote the inverse of $\phi$, and consider the sets $X_1=X\setminus\{\psi(x)\}$ and $X_2=Y\setminus \{\phi(x)\}$, together with $X_0=X$ and $X_3=Y$. We define $R_0\subset X_0\times X_1$ as the merge of $x$ and $\psi(x)$ at $x$, and $R_2\subset X_2\times X_3$ as the splitting of $x$ into $x$ and $\phi(x)$. Finally, define $R_1\subset X_1\times X_2$ as follows. For $(y,z)\in X_1\times X_2$ we have $(y,z)\in R_1$ if either $(y,z)\in R$ or $(y,z)=(x,x)$. Note that $R_1$ is still the graph of a bijection. Thus, we have decomposed $R$ into the MBS-chain $(R_0,R_1,R_2)$. Furthermore, setting $n=\card(R)$, we have
    \begin{align*}
        \elo(R_0,R_1,R_2)&=\elo(R_0)+\elo(R_1)+\elo(R_2)\\
        &=\omega(n)d(\psi(x),x)+\omega(n-1)\left(\sum_{x\in X\setminus \{x,\psi(x)\}}d(x,\phi(x))\right)+\omega(n)d(x,\phi(x))\\
        &\leq \omega(n)\sum_{x\in X}d(x,\phi(x))=\elo(R),
    \end{align*}
    which completes the proof.
\end{proof}

Finally, this last technical lemma allows to construct, from an atomic relation, a crenelated path whose length is as close as wanted from the length of the relation, and with well-behaved cardinality. The technicality resides in avoided crossings between distinct components of the constructed path. The proof is illustrated in \cref{fig:path without crossing}. 
\begin{lem}
   \label{lem:AtomicToPath}
   Let $X,Y\in\Ran(M)$ and let $R$ be an atomic relation from $X$ to $Y$. For any $\epsilon > 0$, there exists a crenelated path $\Gamma\colon[a,b]\to\Ran(M)$ from $X$ to $Y$ such that $\elo(\Gamma)\leq \elo(R) + \epsilon$, and such that either
   \begin{itemize}
       \item The cardinality of $\Gamma$ is constant, if $R$ is a bijection.
       \item The cardinality of $\Gamma$ is non increasing, if $R$ is a simple merge.
       \item The cardinality of $\Gamma$ is non decreasing, if $R$ is a simple split.
   \end{itemize}
\end{lem}

\begin{proof}
    Let us fix $\epsilon>0$ and let us denote $n=\card(R)$. Let us start with the case where $R$ is an atomic bijection. If $R$ is the graph of the identity, then any constant path works. Otherwise, there exists a unique $(x,y)\in R$ such that $x\neq y$. Let us fix an order $X=\{x,x_1,\dots,x_{n-1}\}$, and $Y=\{y,x_1,\dots,x_{n-1}\}$.

    There exists an admissible path $\gamma\colon[0,1]\to M$ from $x$ to $y$ such that $\ell(\gamma)\leq \dm(x,y)+\frac{\epsilon}{\omega(n)}$. Note that we cannot simply define our desired path $\Gamma$ as 
    \begin{align*}
        [0,1]&\to \Ran(M)\\
        t&\mapsto X\setminus\{x\}\cup \{\gamma(t)\}
    \end{align*}
    because there could be "crossings", i.e. times $t\in[0,1]$ such that there exists $1\leq i\leq n-1$ such that $\gamma(t)=x_i$. Note however that we can chose $\gamma$ such that, for any $z\in X\cup Y$, there exists at most one $t\in[0,1]$ such that $\gamma(t)=z$. Indeed, if there is more that one, they form a closed set of $[0,1]$, so that we can take their minimum $t_\mathrm{min}=\min\{t\in [0,1]\mid \gamma(t)=z\}$ and similarly for the maximum $t_\mathrm{max}$, and define a shorter, continuous path, with a single crossing with $z$.
    \begin{align*}
       [0,1]&\to M\\
       t&\mapsto 
       \begin{cases}
          \gamma(2t_\mathrm{min}t)  & \text{if } 0\leq t\leq \frac{1}{2} \\
          \gamma\Big(2(1-t_\mathrm{max})(t-1) + 1\Big) & \text{if }\frac{1}{2} \leq t\leq 1 
        \end{cases}.
    \end{align*}   
    By repeating this process for every $z\in X\cup Y$, we obtain a path $\gamma$ which crosses each of the $z\in X\cup Y$ at most one time. 
    
    Let $m$ be the number of "crossings", i.e. the number of times $t\in(0,1)$ such that there exists $j\in\{1,\dots,n-1\}$ such that $\gamma(t)=x_j$. Up to a reordering of the $x_j$, we can suppose that for all $1\leq j \leq m$ there exists a unique $t_j\in(0,1)$ such that $\gamma(t_j)=x_j$, and that $0< t_m \leq \dots \leq t_1< 1 $. 
    If $m=0$, then we simply define
    \begin{align*}
        \Gamma \colon [0,1]&\to \Ran(M)\\
        t&\mapsto \{\gamma(t),x_1,\dots,x_{n-1}\}
    \end{align*}
    which verifies that $\card(\Gamma(t))=n$ for all $t\in[0,1]$. Moreover, we have an admissible lift of $\Gamma$ of which we know the length, so that
    \begin{align*}
        \elo(\Gamma)\leq \omega(n)\ell(\gamma) \leq \omega(n)\dm(x,y)+\epsilon = \elo(R) +\epsilon.
    \end{align*}  
    Otherwise, we define
    \begin{align*}
        \Gamma_0\colon[0,1-t_1]&\to\Ran(M)\\
        t&\mapsto \{x,\gamma(t+t_1),x_2,\dots,x_{n-1}\}\\
        \Gamma_m\colon[1-t_m,1]&\to\Ran(M)\\
        t &\mapsto \{y,x_1,\dots,x_{m-1},\gamma(t+t_m-1),x_{m+1},\dots,x_{n-1}\}
    \end{align*}
    and, for all $1\leq i\leq m-1$,
    \begin{align*}
        \Gamma_i\colon[1-t_i,1-t_{i+1}] & \to \Ran(M)\\
        t&\mapsto \{x,y,x_1,\dots,x_{i-1},\gamma(t+t_i+t_{i+1}-1),x_{i+2},\dots,x_{n-1}\}
    \end{align*}
    One can indeed check that for any $0\leq i\leq m$, the path $\Gamma_i$ has a constant cardinality. Furthermore, $\Gamma_0$ gives a path between $X$ and $\{x,y,x_2,\dots,x_{n-1}\}$, $\Gamma_m$ gives a path between $\{x,y,x_1,\dots,x_{m-1},x_{m+1},\dots,x_{n-1}\}$ and $Y$, and $\Gamma_i$ gives a path between $\{x,y,x_1,\dots,\widehat{x_i},x_{i+1},\dots,x_{n-1}\}$ and $\{x,y,x_1,\dots,x_i,\widehat{x_{i+1}},\dots,x_{n-1}\}$, so that they can be concatenated into a path $\Gamma$ between $X$ and $Y$ of constant cardinality. Furthermore, each of the $\Gamma_i$ is admissible by construction, and thus so is $\Gamma$, and we have: 
    \begin{align*}
        \elo(\Gamma)=\sum_{i=0}^{m}\elo(\Gamma_i)=\omega(n)\ell(\gamma)\leq \elo(R)+\epsilon.
    \end{align*}
    
    Let us now treat the case where $R$ is an atomic merge. There exists a unique $(x,y)\in R$ such that $x\neq y$, but this time we have $y\in X$, so that $Y\subset X$ and $\card(Y)=n-1$. As before, there exists an admissible path $\gamma\colon[0,1]\to M$ from $x$ to $y$ such that $\ell(\gamma)\leq \dm(x,y)+\frac{\epsilon}{\omega(n)}$. Furthermore, we may repeat the previous construction to ensure that, for any $z\in X$, there exists at most one $t\in[0,1]$ such that $\gamma(t)=z$. Order the elements of $X$ so that $X=\{x,y,x_1,\dots,x_{n-2}\}$, and let $m$ be the number of crossings occurring in $(0,1)$. If $m=0$, then we may simply define
        \begin{align*}
        \Gamma \colon [0,1]&\to \Ran(M)\\
        t&\mapsto \{\gamma(t),y,x_1,\dots,x_{n-2}\}
    \end{align*}
        which is a crenelated path, and satisfies $\card(\Gamma(t))=n$ for all $t\in[0,1)$ and $\card(\Gamma(1))=n-1$, which is indeed non increasing. Moreover, we have as before
    \begin{align*}
        \elo(\Gamma)\leq \omega(n)\ell(\gamma) = \elo(R) +\epsilon.
    \end{align*}
    
    Otherwise, $m>0$, and let $0<t_m<\dots<t_1<1$ be such that $\gamma(t_i)\in X$. Note that we must have $\gamma(t_i)\not= x,y$ by construction, and we may once more reorder the elements in $X$ such that $\gamma(t_i)=x_i$ for $1\leq i\leq m$.

   Once more, let us define
    \begin{align*}
        \Gamma_0\colon[0,1-t_1]&\to\Ran(M)\\
        t&\mapsto \{x,y,\gamma(t+t_1),x_2,\dots,x_{n-2}\}
    \end{align*}
    which verifies that $\card(\Gamma_0(t))=n$ for all $t\in[0,1-t_1)$ and $\card(\Gamma(1-t_1))=n-1$. Furthermore, it is admissible by construction, and $\elo(\Gamma_0)\leq \omega(n)\ell(\gamma_{|[t_1,1]})$. Now we can define $X_1=\Gamma_0(1-t_1)=\{x,y,x_2,\dots,x_{n-2}\}$ and $\phi\colon X_1\to Y$ the atomic bijection sending $x$ to $x_1$. The restriction $\gamma_{|[0,t_1]}$ provides an admissible path from $x$ to $x_1$, thus, by the first part of this proof, there exists an admissible path $\Gamma_1$ from $X_1$ to $Y$ which satisfies $\elo(\Gamma_1)\leq \omega(n-1)\ell(\gamma_{|[0,t_1]})$. Concatenating $\Gamma_0$ and $\Gamma_1$ provides an admissible path $\Gamma\colon [0,1]\to \Ran(M)$ between $X$ and $Y$, such that $\card(\Gamma(t))$ is $n$ if $t\in [0,1-t_1)$ and $n-1$ if $t\in [1-t_1,1]$.
    Furthermore, we have
    \begin{align*}
        \elo(\Gamma)&=\elo(\Gamma_0)+\elo(\Gamma_1)\\
        &\leq \omega(n)\ell(\gamma_{|[t_1,1]})+\omega(n-1)\ell(\gamma_{|[0,t_1]})\\
        &\leq \omega(n)\ell(\gamma)\\
        &\leq \omega(n)d(x,y)+\epsilon\\
        &=\elo(R)+\epsilon,
    \end{align*}
    which concludes the proof in the case of an atomic merge. The case of an atomic split follows by symmetry.
\end{proof}

\begin{figure}
\centering
\begin{tikzpicture}[xscale=5.5, domain=0:1]
\draw[->] (-0.1,0) -- (1.1,0) node[below]{$t$} ;
\node (10) at (1,0) {};
\draw (10.north) -- (10.south) node[below] {1};
\node (0) at (0,0) {};
\draw (0.north) -- (0.south) node[below] {0};
\node [color=T-D-S2,left] at (0,0.3) {$\gamma(t)$};
\draw[color=T-D-S4, domain=0:0.275,line width=1.5pt]   plot (\x,{4*sin(\x r *0.5*pi)^2+0.3})    node[right] {};
\draw[color=T-D-S3, domain=0.275:0.452,line width=1.5pt]   plot (\x,{4*sin(\x r *0.5*pi)^2+0.3})    node[right] {};
\draw[color=T-D-S2, domain=0.452:0.823,line width=1.5pt]   plot (\x,{4*sin(\x r *0.5*pi)^2+0.3})    node[right] {};
\draw[color=T-D-S1, domain=0.823:1,line width=1.5pt]   plot (\x,{4*sin(\x r *0.5*pi)^2+0.3})    node[right, black] {$y$};
\draw (0,1) node [left] {$x_3$} -- (1,1) node [right] {$x_3$};
\draw (0,2) node [left] {$x_2$} -- (1,2) node [right] {$x_2$};
\draw (0,4) node [left] {$x_1$} -- (1,4) node [right] {$x_1$};
\draw (0.275,1) node(x1) [circle,fill=T-D-S4,inner sep=2pt] {};
\draw (0.452,2) node(x2) [circle,fill=T-D-S3,inner sep=2pt] {};
\draw (0.823,4) node(x3) [circle,fill=T-D-S2,inner sep=2pt] {};
\draw [dotted] (x1) -- (0.275,0) node [below] {$t_3$};
\draw [dotted] (x2) -- (0.452,0) node [below] {$t_2$};
\draw [dotted] (x3) -- (0.823,0) node [below] {$t_1$};
\node [below,yshift=-0.3cm,white] at (0,0) {$1-t_2$};
\node[above left] at (0,5) {$X$};
\node[above right] at (1,5) {$Y$};
\end{tikzpicture}
\hfill
\begin{tikzpicture}[xscale=5.5]
\draw[->] (-0.1,0) -- (1.1,0) node[below]{$t$};
\node (10) at (1,0) {};
\draw (10.north) -- (10.south) node[below] {1};
\node (0) at (0,0) {};
\draw (0.north) -- (0.south) node[below] {0};
\draw (0,0.3) node [left] {$x$} -- (0.725,0.3);
\draw (0,1) node [left] {$x_3$} -- (0.548,1);
\draw (0,2) node [left] {$x_2$} -- (0.177,2);
\node [left] at (0,4) {$x_1$};
\draw (0.725,2) -- (1,2) node [right]{$x_2$};
\draw (0.548,4) -- (1,4) node [right]{$x_1$};
\draw (0.177,4.3) -- (1,4.3) node [right] {$y$};
\draw[color=T-D-S1, domain=0:0.177,line width=1.5pt]   plot (\x,{4*sin((\x + 0.823)r *0.5*pi)^2+0.3})   ;
\draw[color=T-D-S2, domain=0.177:0.548,line width=1.5pt]   plot (\x,{4*sin((\x + 0.275)r *0.5*pi)^2+0.3})    ;
\draw[color=T-D-S3, domain=0.548:0.725,line width=1.5pt]   plot (\x,{4*sin((\x - 0.273)r *0.5*pi)^2+0.3})    ;
\draw[color=T-D-S4, domain=0.725:1,line width=1.5pt]   plot (\x,{4*sin((\x - 0.725)r *0.5*pi)^2+0.3})  node [right,black] {$x_3$}  ;
\draw [dotted] (0.177,4.3) -- (0.177,0) node [below] {$1-t_1$};
\draw [dotted] (0.548,4) -- (0.548,0) node [below,yshift=-0.3cm] {$1-t_2$};
\draw [dotted] (0.725,2) -- (0.725,0) node [below] {$1-t_3$};
\node[above,rotate=90,yshift=0.5cm] at (0,2.15) {$\Gamma(t)$};
\node[above left] at (0,5) {$X$};
\node[above right] at (1,5) {$Y$};
\end{tikzpicture}
\\\vspace{1cm}
\begin{tikzpicture}[xscale=5.5, domain=0:1]
\draw[->] (-0.1,0) -- (1.1,0) node[below]{$t$} ;
\node (10) at (1,0) {};
\draw (10.north) -- (10.south) node[below] {1};
\node (0) at (0,0) {};
\draw (0.north) -- (0.south) node[below] {0};
\node [color=T-D-S2,left] at (0,0.2) {$\gamma(t)$};
\draw[color=T-D-S4, domain=0:0.9,line width=1.5pt]   plot (\x,{0.016*((10*(\x-0.5))^3 + (10*(\x-0.5))^2 + (10*(\x-0.5)))+1.82} ) ;
\draw[color=T-D-S1, domain=0.9:1,line width=1.5pt]   plot (\x,{0.016*((10*(\x-0.5))^3 + (10*(\x-0.5))^2 + (10*(\x-0.5)))+1.82} ) ;
\draw (0,1.484) node [left] {$x_3$} -- (1,1.484) node [right] {$x_3$};
\draw (0,1.814) node [left] {$x_2$} -- (1,1.814) node [right] {$x_2$};
\draw (0,3.164) node [left] {$x_1$} -- (1,3.164) node [right] {$x_1$};
\draw (0,4.3) node [left] {$y$} -- (1,4.3)node [right] {$y$};
\draw (0.2,1.484) node(x1) [circle,fill=T-D-S3,inner sep=2pt] {};
\draw (0.45,1.814)node(x2) [circle,fill=T-D-S3,inner sep=2pt] {};
\draw (0.9,3.164) node(x3) [circle,fill=T-D-S3,inner sep=2pt] {};
\draw (1,4.3) node(x4) [circle,fill=T-D-S1,inner sep=2pt] {};
\draw [dotted] (x1) -- (0.2,0) node [below] {$t_3$};
\draw [dotted] (x2) -- (0.45,0) node [below] {$t_2$};
\draw [dotted] (x3) -- (0.9,0) node [below] {$t_1$};
\node[above left] at (0,5) {$X$};
\node[above right] at (1,5) {$Y$};
\end{tikzpicture}
\hfill
\begin{tikzpicture}[xscale=5.5, domain=0:1]
\draw[->] (-0.1,0) -- (1.1,0) node[below]{$t$} ;
\node (10) at (1,0) {};
\draw (10.north) -- (10.south) node[below] {1};
\node (0) at (0,0) {};
\draw (0.north) -- (0.south) node[below] {0};
\draw[color=T-D-S4, domain=0.1:1,line width=1.5pt]  plot (\x,{0.016*((10*(\x-0.6))^3 + (10*(\x-0.6))^2 + (10*(\x-0.6)))+1.82} ) node[right,black]{$x_1$} ;
\draw[color=T-D-S1, domain=0:0.1,line width=1.5pt]   plot (\x,{0.016*((10*(\x+0.4))^3 + (10*(\x+0.4))^2 + (10*(\x+0.4)))+1.82} ) ;
\draw (0,0.14) node [left] {$x$} -- (0.1,0.14);
\draw (0,1.484) node [left] {$x_3$} -- (1,1.484) node [right] {$x_3$};
\draw (0,1.814) node [left] {$x_2$} -- (1,1.814) node [right] {$x_2$};
\draw (0,4.3) node [left] {$y$} -- (1,4.3) node [right] {$y$};
\draw (0.3,1.484) node(x1) [circle,fill=T-D-S3,inner sep=2pt] {};
\draw (0.55,1.814)node(x2) [circle,fill=T-D-S3,inner sep=2pt] {};
\draw [dotted] (0.1,4.3) -- (0.1,0) node [below,xshift=0.4cm] {$1-t_1$};
\node[left] at (0,3.164) {$x_1$};
\node[above left] at (0,5) {$X$};
\node[above right] at (1,5) {$Y$};
\end{tikzpicture}
\caption{Sketch of the proof of \cref{lem:AtomicToPath} in the case $M=\R$. The upper panels show the case of an atomic bijection from $X$ to $Y$ in $\Ran(\R)$. The upper left panel shows the path $\gamma\colon[0,1]\to\R$ from $x$ to $y$ that crosses other points in $X$. The upper right panel shows the path $\Gamma$ from $X$ to $Y$ in $\Ran(\R)$ constructed from $\gamma$. Its cardinality is constant and equal to 4 on $[0,1]$. The lower panels show the case of an atomic merge from $X$ to $Y$. The lower left panel shows the path $\gamma$ from $x$ to $y$ that crosses other points in $X$. The lower right panel shows the first step in constructing a path $\Gamma\colon [0,1]\to \Ran(M)$, with $\Gamma_0\colon[0,1-t_1]\to\Ran(\R)$ from $X$ to $X_1$, and  $\Gamma_1$ constructed by applying the bijection case to the remainder of the path.}
\label{fig:path without crossing}
\end{figure}
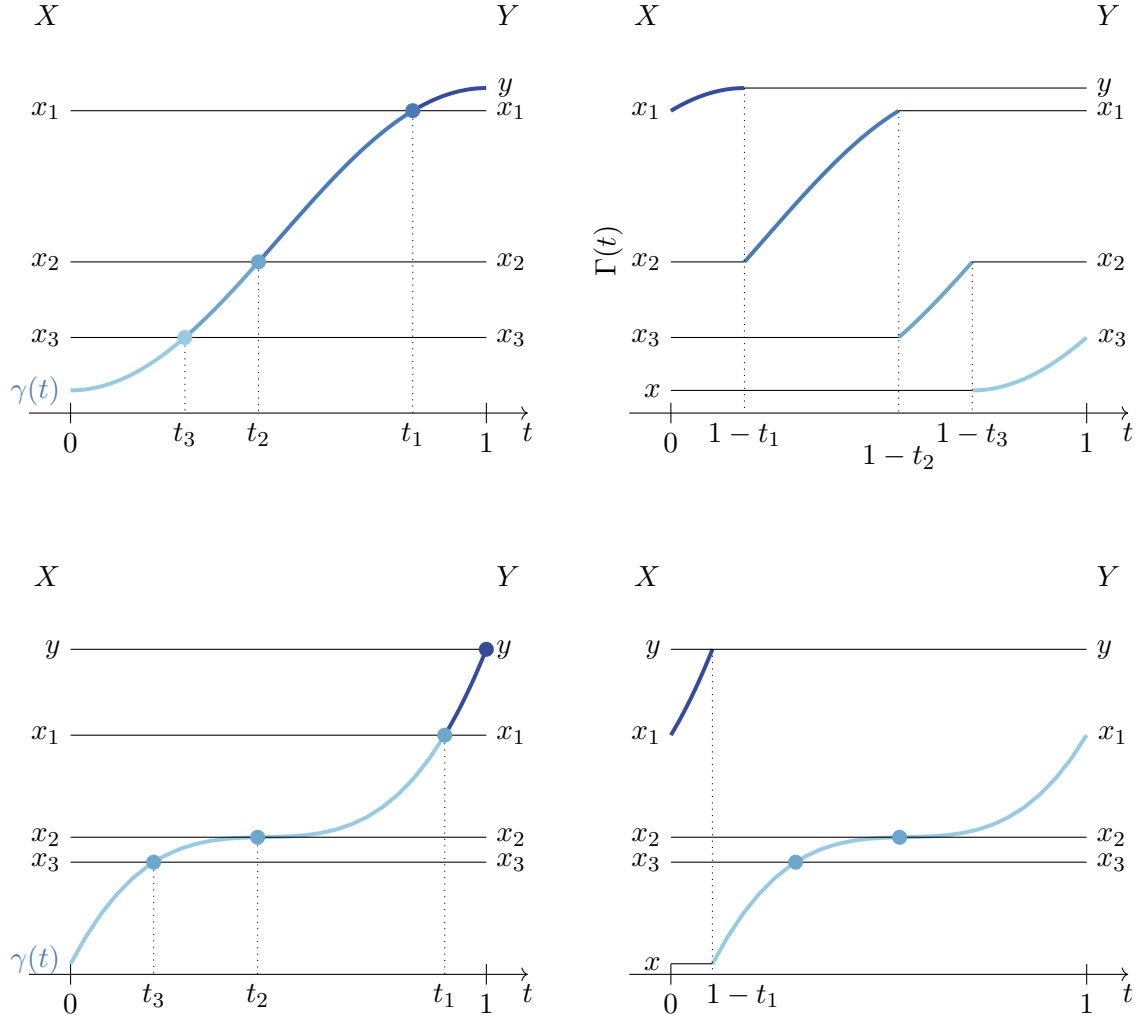

\begin{theo}
\label{cor:existence geodesic path}
    Let $M$ be a length space satisfying the Heine-Borel property, and such that any for any two points $x,y\in M$ there exists an admissible path $\gamma$ between $x$ and $y$ which realizes the distance between $x$ and $y$.
    
    Then, for any $X,Y\in \Ran(M)$ and any weight $\omega$, there exists a MBS-path between $X$ and $Y$, $\Gamma\colon [a,b]\to \Ran(M)$, such that $\elo(\Gamma)=\dom(X,Y)$.
\end{theo}

\begin{proof}
    Since $M$ satisfies the Heine-Borel property, \cref{theo:Geodesic_Chain} gives the existence of a simple MBS chain $(R_0,\dots,R_{k-1})$ between $X$ and $Y$ such that $\elo(R_0,\dots,R_{k-1})=\dom(X,Y)$. By \cref{lem:AtomicMBS}, we may consider instead a chain $(P_0,\dots,P_{l-1})$ with the same length (since it can't possibly be shorter) but such that the $P_i$ are all atomic merges, bijections or splits. Now, for each $P_i$, consider the unique pair $x_i\not=y_i$ such that $(x_i,y_i)\in P_i$. By assumption, there exists an admissible path in $M$, $\gamma_i$ between $x_i$ and $y_i$ whose length is exactly $d(x_i,y_i)$. Observe that providing such a path in the proof of \cref{lem:AtomicToPath} yields a path in $\Ran(M)$ whose length is exactly $\elo(P_i)$. Concatenation of those paths thus yields the desired path in $\Ran(M)$.
\end{proof}

\begin{rem}
 The hypothesis in \cref{cor:existence geodesic path} are certainly not sharp. Consider for example the space $M=(0,1)$ and $N=[0,1]$, both with the expected metric and length structure, and some weight $\omega$. Geodesics exist in both $M$ and $N$, yet $M$ does not satisfy the Heine-Borel property, hence \cref{cor:existence geodesic path} does not apply to it. On the other hand, if $X,Y\subset M$ are two configurations, then applying \cref{cor:existence geodesic path} to $N$ gives a geodesic path $\Gamma\colon [0,1]\to \Ran(N)$ between $X$ and $Y$. It is clear that such a geodesic is in fact contained in $\Ran(M)$, and in turn that it is a geodesic in $\Ran(M)$, since distances in $\Ran(M)$ are bounded below by distances in $\Ran(N)$. This shows that geodesics exist in $\Ran(M)$ despite $M$ not satisfying the Heine-Borel property. This hypothesis is certainly already too strong in \cref{theo:Geodesic_Chain}, which we use in the proof of \cref{cor:existence geodesic path}, though in the former it is unclear what weaker hypothesis could take its place. In \cref{cor:existence geodesic path} however, we conjecture that it is already enough for geodesics to exist in the underlying space $M$.
\end{rem}

\section{Examples and computations}
\label{sec:Examples_Computations}
This section aims to be less abstract and more practical. We illustrate the definitions and results of the previous sections on some examples, and with some drawings. Finding the actual infimum in the definition of the weighted distance may be a non-trivial optimization problem in some cases. We thus present simple cases, for which we can compute the distance and find the geodesic chain/path that realises it, or at least give a combinatorial chain that reaches a reasonable upper bound of the distance. Note that, even if identifying a geodesic chain in $\Ran(M)$ may be quite intuitive in some case, showing that it is indeed geodesic often turns out to be a bit more difficult.

In the following sections, we will assess the topological properties of $(\Ran(M),\dom)$ based on the convergence properties of sequences in $\Ran(M)$. We thus spend some time in this section to define some illustrative sequences in $\Ran(M)$, and assess their convergence properties with respect to various topologies. In particular, \cref{exa:config_convergente} will allow to distinguish between distinct weighted topologies, and \cref{exa:Cantor} shows that $(\Ran(M),\dom)$ is not complete, as a metric space.

Eventually, we discuss the topological space $(\Ran(M),\tom)$ equipped with the topology induced by $\dom$, seen as an invariant of $M$. In particular, we show that, unlike the Hausdorff or final topology, the weighted topology does not allow to lift arbitrary continuous maps $f\colon (M,d_M)\to (N,d_N)$ to continuous maps between the corresponding Ran spaces. Therefore $(\Ran(M),\tom)$ is not necessarily an invariant of the homeomorphism type of $M$, but a finer invariant. Besides, this also shows that $(\Ran(-),\tom)$ is not functorial on Top, the category of topological spaces and continuous maps. Nevertheless, we show in \cref{prop:Lipschitz_Map_lift} that a locally Lipschitz applications $f\colon (M,d_M)\to (N,d_N)$ lifts to a locally Lipschitz application $\Ran(f)\colon (\Ran(M),\tom)\to (\Ran(N),\tom)$, showing that $(\Ran(-),\tom)$ does define a functor on the category of metric spaces and locally Lipschitz maps. To prove this, we define the notion of locality, that simplifies the computation of distances in the neighborhoods of configurations of cardinality $n\geq 2$. This notion of locality will be very useful in \cref{sec:conicity}.

\subsection{Examples}
\label{sec:Examples_Examples}
We first give examples where we can compute the $\omega$-distance and find the associated geodesic chain/path, or a combinatorial chain that gives a reasonable upper bound of the distance. Through these practical computations and drawings, we hope to give the reader an intuitive picture of the important properties of the space  $(\Ran(M),\tom)$, and of the distance $\dom$.

We first illustrate, on a very simple example for $M=\R^2$, that the geodesic chain that realises the distance may depend on the chosen weight. 
\begin{exa}
\label{exa:peace_and_love}
\begin{figure}[ht]
    \centering
    \begin{tikzpicture}[xscale=4,yscale=3]
    \draw [->] (-0.2,0) -- (1.2,0) ;
    \draw [->] (0,-1.2) -- (0,1.2) ;
    \draw (0,0) node(O)[below left] {0};
    \node (10) at (1,0) {};
    \draw (10.north) -- (10.south) node[below] {1};
    \node (01) at (0,1) {};
    \draw (01.east) -- (01.west) node[left] {1};
    \node (0minus1) at (0,-1) {};
    \draw (0minus1.east) -- (0minus1.west) node[left] {-1};
    \draw (0,0) node(X)[circle,fill=T-D-S11,inner sep=3pt]{};
    \draw (X.north west) node[above left,T-D-S11] {$X$};
    \draw (1,1) node(Y1)[circle,fill,inner sep=3pt]{};
    \draw (Y1.north east) node[above right] {$Y$};
    \draw (1,-1) node(Y2)[circle,fill,inner sep=3pt]{};
    \draw (Y2.north east) node[above right] {$Y$};
    \draw (0.42265,0) node(Z1) [circle,fill=T-D-S5,inner sep=3pt] {};
    \draw(Z1.north) node[above,blue] {$Z_1$};
    \draw (0.7418,0) node(Z2) [circle,fill=T-D-S3,inner sep=3pt] {};
    \draw(Z2.north) node[above,blue] {$Z_2$};
    \draw (0.949937,0) node(Z10) [circle,fill=T-D-S1,inner sep=3pt] {};
    \draw(Z10.north east) node[above right,blue] {$Z_{10}$};
    \draw [dotted,T-D-S5,line width=1pt] (Y1) -- (Z1) -- (Y2);
    \draw [dotted,T-D-S5,line width=1pt] (Z1) -- (X);
    \draw [dotted,T-D-S3,line width=1pt] (Y1) -- (Z2) -- (Y2);
    \draw [dotted,T-D-S3,line width=1pt] (Z2) -- (Z1) -- (X);
    \draw [dotted,T-D-S1,line width=1pt] (Y1) -- (Z10) -- (Y2);
    \draw [dotted,T-D-S1,line width=1pt] (Z10) -- (Z2) -- (Z1) -- (X);
    \end{tikzpicture}
    \caption{Sketch of \cref{exa:peace_and_love}. Configurations $X=\{(0,0)\}$ is in red, configuration $Y=\{(1,1),(1,-1)\}$ is in black, and the intermediate configuration $Z_{\omega(2)}$ is in blue for three values of $\omega(2)=1$, $2$ and $10$. Dotted lines indicate the corresponding geodesic paths between $Y$ and $X$.}
    \label{fig:peace and love}
\end{figure}
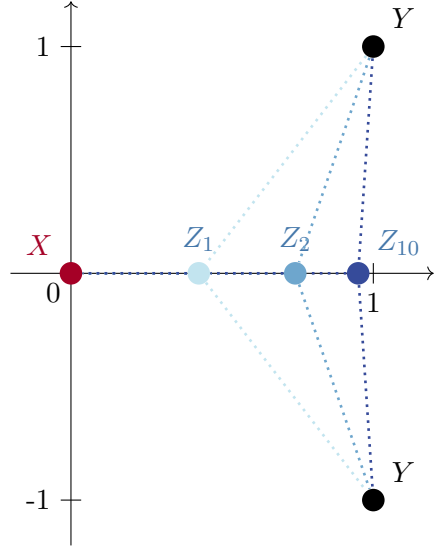
This example illustrates that, given two configurations $X,Y\in\Ran(M)$, the geodesic chain between $X$ and $Y$, if it exists, may depend on the weight $\omega$ chosen to compute the distance. Let us illustrate this for $M=\R^2$, and $X=\{(0,0\}$, $Y=\{(1,1),(1,-1)\}$. By \cref{theo:Geodesic_Chain}, for any weight $\omega$, the distance $\dom(X,Y)$ can be computed as the minimum over the $\omega$-length of MBS-chains between $Y$ and $X$, i.e. a merge $R_\omega$ followed by a, possibly constant, bijection $P_\omega$: $((Y,Z_\omega,X),(R_\omega,P_\omega))$.
Writing $Z_\omega=(z_1,z_2)$, we get the following (recall that $\omega(1)=1$ by definition).
\begin{align*}
    \elo(R_\omega,P_\omega)=\omega(2)\left[\sqrt{(1-z_1)^2+(1-z_2)^2}+\sqrt{(1-z_1)^2+(1+z_2)^2} \right] + \sqrt{z_1^2+z_2^2}
\end{align*}
First note that we can take $Z_\omega$ in the convex hull of $X\cup Y$. Second, note that $\elo(R_\omega,P_\omega)$ is even in $z_2$, and thus has an extremum when $z_2=0$. It is easy to see that it is indeed a minimum (one can for example compute the partial second derivative with respect to $z_2$ and check that it is non negative). We can thus restrict ourselves to the case $z_1\in[0,1]$ and $z_2=0$ to get
\begin{align*}
    \elo(R_\omega,P_\omega)=2\omega(2)\sqrt{(1-z_1)^2+1} + z_1
\end{align*}
We finally get a minimum when 
\begin{align*}
z_1=1-\frac{1}{\sqrt{4\omega(2)^2-1}}
\end{align*}

This example is illustrated on \cref{fig:peace and love}, where the intermediate configuration $Z_\omega$ is displayed for $\omega(2)=1,2,10$. We see that, as $\omega(2)$ increases, the intermediate singleton $Z_\omega$ gets closer and closer to $(1,0)$.

This example is also useful to better understand the importance of the Heine-Borel hypothesis in \cref{theo:Geodesic_Chain,cor:existence geodesic path}. Indeed, if, for a given $\omega$, we remove the point of $Z_\omega$ from $\R^2$, then $N=\R^2\setminus Z_\omega$ is still locally compact, but it does not satisfy the Heine-Borel property. And we directly see that no geodesic chain exist in $\Ran(N)$ between $X$ and $Y$ for $\dom$, although one exists for any other weight $\chi$ such that $\chi(2)\neq\omega(2)$. Similarly, no geodesic path exist in $\Ran(N)$ between $X$ and $Y$ for $\dom$, although one exists for any other weight $\chi$ such that $\chi(2)<\omega(2)$. Now if we remove the entire $\R_+^*$ axis, then $N=\R^2\setminus(0,1]\times\{0\}$ is no longer locally compact nor does it satisfy Heine-Borel, and this time, no geodesic chain nor path exist between $X$ and $Y$ for any choice of weight. A similar case is obtained for the locally compact space $\R^2\setminus Z_1$, with $Z_1$ the singleton corresponding to the minimum weight $\omega(2)=1$.
\end{exa}

We then illustrate that, even in one of the most simple examples one could think of, i.e. $M=[0,1]$, it is not always straightforward to compute the $\omega$-distance between configurations.
\begin{exa}
\label{exa:aligned_points}
Let $\omega$ be a weight, and consider $M=[0,1]$, equipped with the distance given by $\dm(x,y)=|x-y|$. Let us estimate the distance of a configuration $X\in\Ran(M)$ to the singleton $X_0=\{x_0=0\}$. Let us define $n=\card(X)$, and $X=\{x_1<\dots < x_n\}$. 

Let us first consider the MBS-chain, $(R_{n},\dots,R_1)$, where $R_i$ merges $\{x_{i},x_{i-1}\}$ at $x_{i-1}$ for any $2\leq i \leq n$, and $R_1=\{(x_1,0)\}$ is a, possibly constant, bijection.
This allows us to give an upper bound of $\dom(X,\{0\})$, as
    \begin{align*}
        \dom(X,\{0\})&\leq \elo(R_{n},\dots,R_1)\\
        &=\sum_{i=1}^{n}\omega(i)\elo(R_i)\\
        &=\sum_{i=1}^{n}\omega(i)\dm(x_i,x_{i-1})\\
        &\leq \omega(n)\dm(x_n,x_0) = \omega(n)x_n.
    \end{align*}

In the case where $\omega$ is the constant weight equal to $1$, let us show that we actually have $\dom(X,\{0\})=x_n$. Indeed, by \cref{theo:Geodesic_Chain}, there exists a geodesic chain $((X=Y_0,\dots,Y_k=\{0\}),(P_0,\dots,P_{k-1}))$ from $X$ to $\{0\}$. By surjectivity of the $P_i$, for $0\leq i\leq k-1$, there must exist a sequence of $y_i\in Y_i$, such that $y_0=x_n$, $y_k=0$ and for $0\leq i\leq k-1$, $y_iP_iy_{i+1}$. We thus have
    \begin{align*}
        \dom(X,\{0\})=\elo(P_0,\dots,P_{k-1})\geq \sum_{i=0}^{k-1} \omega(\card(P_i))\dm(y_i,y_{i+1})\geq \dm(x_{n},0)=x_n
    \end{align*}
So that the MBS-chain $(R_{n},\dots,R_1)$ defined above is actually a geodesic chain.

This does not hold in the general case. Consider the case $X=\{0,0.1,1\}$. The previously defined MBS-chain gives
    \begin{align*}
        \dom(X,\{0\})&\leq \elo(R_{3},R_2,R_1)\\
        &=0.9\omega(3)+0.1\omega(2).
    \end{align*}
    with $R_1$ being the identity.
Now let us define the chain $(Q_1,Q_2,Q_3)$ where $Q_1$ merges $\{0,0.1\}$ at 0.1, $Q_2$ merges $\{1,0.1\}$ at 0.1, and $Q_3$ is the bijection sending 0.1 to 0. 
\begin{align*}
    \elo(Q_1,Q_2,Q_3)=0.1\omega(3)+0.9\omega(2)+0.1.
\end{align*}
Note that, as soon as $\omega(3)> \omega(2)+0.125$, then $\elo(Q_1,Q_2,Q_3)<\elo(R_3,R_2,R_1)$, and the latter is not a geodesic chain anymore.
\end{exa}

Nevertheless, if the points of the configuration $X\in\Ran([0,1])$ are such that their pairwise distances are non increasing, i.e. if they get closer to each other as they get further away from $0$, as sketched on \cref{fig:points_alignes}, then we can compute the distance to the singleton $\{0\}$. This is the subject of the following example.

\begin{exa}\label{exa:Telescoping_points_arbitrary_weight}
Consider the previous example, in the particular case where $x_2-x_1\geq x_3-x_2\geq \dots \geq x_n-x_{n-1}$.  We will show that the chain $(R_n,\dots,R_1)$ previously defined is indeed a geodesic chain. So that we have
\begin{align*}
    \dom(X,\{0\})=\sum_{k=1}^n \omega(k)d(x_{k-1},x_k).
\end{align*}

By \cref{theo:Geodesic_Chain} there exists a geodesic chain $((X=Y_n,\dots,Y_0=\{0\}),(P_n,\dots,P_1))$ from $X$ to $\{0\}$, where $P_k$ is a simple merge for any $2\leq k\leq n$ and $P_1$ is a possibly constant bijection. Let us first show that, for any $0\leq m \leq n-1$ we have 
\begin{align*}
    \sum_{k=n-m}^n \ell(R_k) \leq \sum_{k=n-m}^n \ell(P_k).
\end{align*}
The case $m=n-1$ corresponds to the constant weight equal to $1$, and was covered in the previous example. Let $0\leq m < n-1$. Since the $P_k$ are all merges for $n-m\leq k\leq n$, they all are graphs of surjective maps $p_k\colon Y_k\to Y_{k-1}$. Let us denote $p=p_{n-m}\circ \dots \circ p_n$, and remark that $\{p^{-1}(y)\mid y\in Y_{n-m-1}\}$ forms a partition of $X$. For $y\in Y_{n-m-1}$ let us define $x_\mathrm{min}(y)=\min \{p^{-1}(y)\}$ and $x_\mathrm{max}(y)=\max\{p^{-1}(y)\}$, and observe that if $y\in X$ and $y$ has not been merged with any other points of $X$ after $m$ step, then $x_{\mathrm{max}}(y)=x_{\mathrm{min}}(y)=y$.
Now, by definition, we have
\begin{align*}
    \sum_{j=n-m}^n \ell(P_j) &= \sum_{j=n-m}^n \sum_{x \in Y_j} \dm(x,p_j(x)) \\
    &= \sum_{x\in Y_n=X} \dm(x,p_n(x))+ \sum_{x\in Y_{n-1}} \dm(x,p_{n-1}(x))+\dots+ \sum_{x\in Y_{n-m}} \dm(x,p_{n-m}(x)).
\end{align*}
Let us construct a very coarse bound on this sum. Consider the elements in $Y_n$ of the form $x_\mathrm{max}(y)$ and $x_\mathrm{min}(y)$, for some $y\in Y_{n-m-1}$. If $x_\mathrm{max}(y)=x_\mathrm{min}(y)$, then nothing has merged at $y$ yet, and $y$ does not contribute in any way to the sum.
Otherwise, we may consider the sequences $x_\mathrm{max}(y)=z_n,\dots, z_{m-n-1}=y$ and $x_\mathrm{min}(y)=w_n,\dots,w_{m-n-1}=y$, where $p_j(z_j)=z_{j-1}$ and $p_j(w_j)=w_{j-1}$, for $n\geq j\geq n-m$. Observe that, if $z_j\not = w_j$, then both $d(z_j,z_{j-1})$ and  $d(w_j,w_{j-1})$ appear in the sum. Furthermore, for some $j$, at the latest for $j=n-m-1$, we must have $z_j=w_j$. 
Thus, by the triangular inequality, those terms must sum to something bigger than $d(x_\mathrm{max}(y),x_\mathrm{min}(y))$. Observing that those sequences are disjoint for distinct values of $y\in Y_{n-m-1}$ gives
\begin{equation*}
    \sum_{j=n-m}^n\ell(P_j)\geq \sum_{y\in Y_{n-m-1}}d(x_\mathrm{max}(y),x_\mathrm{min}(y))
\end{equation*}

Now, by assumption, each of the $P_j$ is a simple merge, and thus $\card(Y_{n-m-1})=n-m-1$. Thus, the partition of $X$, $\{p^{-1}(y)\mid y\in n-m-1\}$ must satisfy :
\begin{equation*}
    \sum_{y\in Y_{n-m-1}}\left(\card(p^{-1}(y))-1\right)=n-(n-m-1)=m+1
\end{equation*}
Now, given some $y\in Y_{n-m-1}$ whose preimage under $p$ is not reduced to a singleton, we can consider the order on $p^{-1}(y)=\{x_{i_1}<\dots<x_{i_{k_y}}\}$ with $x_{i_1}=x_{\mathrm{min}}(y)$ and $x_{i_{k_y}}=x_{\mathrm{max}}(y)$. Observe that
\begin{equation*}
    d(x_\mathrm{max}(y),x_\mathrm{min}(y))=\sum_{j=1}^{k_y-1}(x_{i_{j+1}}-x_{i_{j}})
\end{equation*}
and that exactly $\card(p^{-1}(y))-1$ term appear in this sum. Furthermore, if $y'\not=y\in Y_{n-m-1}$, then $d(x_\mathrm{max}(y'),x_\mathrm{min}(y'))$ is equal to a similar sum, but with no $x_i$ in common. We can summarize this as follows. Consider $S$ the set of pairs $(x,x')$ where $x,x'\in p^{-1}(y)$ for some $y\in Y_{n-m-1}$, and $x$ is the successor of $x'$ in $p^{-1}(y)$. By the previous observations, $\card(S)=m+1$, each $x\in X$ appears at most once as the first coordinate in a pair in $S$, and
\begin{equation*}
    \sum_{j=n-m}^n\ell(P_j)\geq \sum_{(x,x')\in S}d(x,x').
\end{equation*}
Now, let $T=\{(x_{j},x_{j-1})\mid n\geq j\geq n-m\}$. We have, by construction
\begin{equation*}
    \sum_{j=n-m}^n\ell(R_j)=\sum_{(x,x')\in T}d(x,x').
\end{equation*}
To conclude the first part of this proof, define orders on $S$ and $T$, via $(x,x')\preceq (z,z') \Leftrightarrow x\geq z$. Then, listing the pair in order, the $k$-th pair in $S$, $(x,x')$ must satisfy $x\leq x_{n-k+1}$, and $x'<x$. Thus, by hypothesis on $X$, we must have $d(x,x')\geq d(x_{n-k+1},x_{n-k})$. And finally, the $k$-th pair in $S$ is further apart than the $k$-th pair in $T$, which gives the desired inequality :
\begin{equation*}
    \sum_{j=n-m}^n\ell(P_j)\geq \sum_{j=n-m}^n\ell(R_j).
\end{equation*}
We can now rearrange a sum:
\begin{align*}
   \elo(R_n,\dots,R_1) = & \sum_{k=1}^n \omega(k) \ell(R_k)\\
    = & \sum_{k=1}^n \left[\omega(1)+\sum_{m=2}^k \left(\omega(m)-\omega(m-1)\right) \right] \ell(R_k) \\
    = & \sum_{k=2}^n \left[\sum_{m=2}^k \left(\omega(m)-\omega(m-1)\right) \right] \ell(R_k) +\sum_{k=1}^n\omega(1)\ell(R_k)\\
    =& \sum_{m=2}^n\left[ \left(\omega(m)-\omega(m-1)\right) \sum_{k=m}^n \ell(R_k)\right] + \omega(1)\sum_{k=1}^n \ell(R_k)\\ 
    \leq & \sum_{m=2}^n\left[ \left(\omega(m)-\omega(m-1)\right) \sum_{k=m}^n \ell(P_k)\right] + \omega(1)\sum_{k=1}^n \ell(P_k)\\
    = & \elo(P_n,\dots,P_1).
\end{align*}
and finally $\elo(R_n,\dots,R_1)\leq \elo(P_n,\dots,P_1)$ thus $(R_n,\dots,R_1)$ is indeed a geodesic chain, and $\dom(X,\{0\})=\sum_{k=1}^n \omega(k)\dm(x_{k-1},x_k)$.
\end{exa}

We now have everything to define a sequence of configurations of increasing cardinality, that converges to a singleton for a given weight, but that may not converge for other, larger, weights. This is done in the following example and will be useful later as it will allow to distinguish between weighted topologies associated to different weights.
\begin{exa}
\label{exa:config_convergente}
    
    Let $\omega$ be a weight, and let $M=[0,1]$. Let us fix a sequence of positive real numbers which converges to $0$, $(u_n)_{n\in \N}$, and let us define, for $n\in \N^*$, $1\leq k\leq n$, and $1\leq i\leq n$ the sequences 
    \begin{align*}
        \epsilon_{n,k}&=\frac{u_n}{n\omega(k)}\\
        x_{n,i}&=\sum_{k=1}^i\epsilon_{n,k}
    \end{align*}
    and consider the sequence of configurations in $M$,
    \begin{equation*}
        X_n=\left\{x_{n,i}\mid 1\leq i\leq n\right\}
    \end{equation*}
    A sketch is given on \cref{fig:points_alignes}. 

    Observe that we have, for any $n\geq1$, $x_{n,2}-x_{n,1}\geq x_{n,3}- x_{n,2} \geq \dots\geq x_{n,n}-x_{n,n-1}$. We can thus use \cref{exa:Telescoping_points_arbitrary_weight} to show that 
    \begin{align*}
        \dom(X_n,\{0\})&=\sum_{i=1}^n\omega(i)d(x_{n,i-1},x_{n,i})\\
        &=\sum_{i=1}^n\omega(i)\frac{u_n}{n\omega(i)}\\
        &=u_n
    \end{align*}
    which converges to zero. Thus, the sequence $X_n$ converges to $\{0\}$ for the topology $\tom$. Note however that $\card(X_n)=n$ diverges to infinity. 
    
    Finally, remark that the sequence $(X_n)$ may not converge for other weighted topologies. For example if we have an increasing sequence $(v_n)_{n\in\N}$ in $\R$ such that $(u_nv_n)_{n\in\N}$ does not go to zero, and if we define the weight $\chi(n)=v_n\omega(n)$ then $\dchi(X_n,\{0\})$ does not go to zero, so the sequence $(X_n)$ does not converge towards $\{0\}$ for the topology $\tchi$. This will be useful in the following section, when we compare the different weighted topologies.
\end{exa}

    \begin{figure}[ht]
    \centering
    \begin{tikzpicture}
    \path (-0.5,0) node (0) [circle,inner sep=2pt] {} 
    (0,0) node(lim) [rectangle,fill=red,inner sep=4pt] {}
    (7,0) node(1) [circle,fill,inner sep=2pt] {}
    (6.7,0) node(2) [circle,fill,inner sep=2pt] {}
    (6.3,0) node(3) [circle,fill,inner sep=2pt] {}
    (5.5,0) node(4) [circle,fill,inner sep=2pt] {}
    (3.5,0) node(5) [circle,fill,inner sep=2pt] {};
    \draw [dotted] (0.center) -- +(8,0);
    \foreach \i in {1,...,5}{
      \node (x\i) [below=0.3 of \i.center] {};
      \draw (x\i.south) -- (x\i.north);
    }
    \node (xlim) [below=0.3 of lim.center] {};
    \draw (xlim.south) -- (xlim.north) ;
    \draw [<->] (xlim.center) -- (x5.center) node[midway,below] {$\epsilon_{5,1}$};
    \draw [<->] (x5.center) -- (x4.center) node[midway,below] {$\epsilon_{5,2}$};
    \draw [<->] (x2.center) -- (x1.center) node[midway,below] {$\epsilon_{5,5}$};
    \end{tikzpicture}
    \caption{Sketch of two configurations of \cref{exa:config_convergente}, $\{0\}$ as a red square $\textcolor{red}{\blacksquare}$ and $X_5$ as black dots \faCircle.}
    \label{fig:points_alignes}
\end{figure}
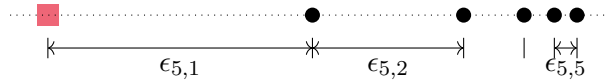

This example illustrates that, similarly to the Hausdorff topology and unlike the final topology, the weighted topology may have sequences that converge even though their cardinality goes to infinity. The subtlety is that, as we will see in \cref{prop:convergence_cardinal_borne_tpi}, for any such sequence, we can always construct a weighted topology for which it will not converge. This will ensure that the only sequences that converge for all the weighted topology have a bounded cardinality. In particular, the weighted topologies collectively mimic the behavior of the final topology. However, since the data of converging sequences is not sufficient to characterize a topology in general, some more work will be needed to show that the final topology is the limit of the weighted topologies.

Another similarity between the weighted and Hausdorff topology is that the metric space $(\Ran(M),\dom)$ is not complete. This is illustrated in the following example by defining a Cauchy sequence for $\dom$ that converges for $\dH$ to a Cantor set. Consequently, it cannot converge for $\dom$ in $\Ran(M)$. The convergence and limits of Cauchy sequences for various weights will be more thoroughly studied in \cref{sec:completeness}. In particular we will show later that the completion of $(\Ran(M),\dom)$, as a set, may be identified with a subset of $\Comp(M)$.
\begin{exa}
\label{exa:Cantor}

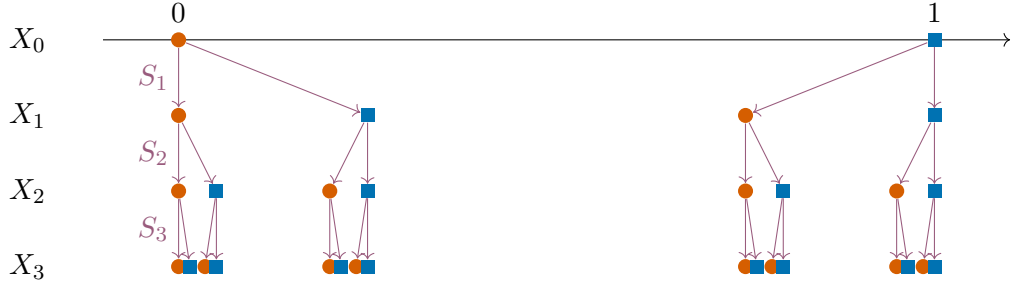
\begin{figure}[ht]
    \centering
    \begin{tikzpicture}[xscale=10]
    \draw[->] (-0.1,0) -- (1.1,0) ;
    \node[circle,fill=OI6,inner sep=2pt] (0) at (0,0) {} ;
    \node[rectangle,fill=OI5,inner sep=2.5pt] (1) at (1,0) {} ;
    \node[above] at (0.north) {0};
    \node[above] at (1.north) {1};
    \node at (-0.2,0) {$X_0$};
    \node at (-0.2,-1) {$X_1$};
    \node[circle,fill=OI6,inner sep=2pt] (L10) at (0,-1) {} ;
    \node[circle,fill=OI6,inner sep=2pt] (L11) at (0.75,-1) {} ;
    \node[rectangle,fill=OI5,inner sep=2.5pt] (R10) at (0.25,-1) {} ;
    \node[rectangle,fill=OI5,inner sep=2.5pt] (R11) at (1,-1) {} ;
    \node at (-0.2,-2) {$X_2$};
    \node[circle,fill=OI6,inner sep=2pt] (L20) at (0,-2) {} ;
    \node[circle,fill=OI6,inner sep=2pt] (L21) at (0.2,-2) {} ;
    \node[circle,fill=OI6,inner sep=2pt] (L22) at (0.75,-2) {} ;
    \node[circle,fill=OI6,inner sep=2pt] (L23) at (0.95,-2) {} ;
    \node[rectangle,fill=OI5,inner sep=2.5pt] (R20) at (0.05,-2) {} ;
    \node[rectangle,fill=OI5,inner sep=2.5pt] (R21) at (0.25,-2) {} ;
    \node[rectangle,fill=OI5,inner sep=2.5pt] (R22) at (0.8,-2) {} ;
    \node[rectangle,fill=OI5,inner sep=2.5pt] (R23) at (1,-2) {} ;
    \node at (-0.2,-3) {$X_3$};
    \node[circle,fill=OI6,inner sep=2pt] (L30) at (0,-3) {} ;
    \node[circle,fill=OI6,inner sep=2pt] (L31) at (0.035,-3) {} ;
    \node[circle,fill=OI6,inner sep=2pt] (L32) at (0.2,-3) {} ;
    \node[circle,fill=OI6,inner sep=2pt] (L33) at (0.235,-3) {} ;
    \node[circle,fill=OI6,inner sep=2pt] (L34) at (0.75,-3) {} ;
    \node[circle,fill=OI6,inner sep=2pt] (L35) at (0.785,-3) {} ;
    \node[circle,fill=OI6,inner sep=2pt] (L36) at (0.95,-3) {} ;
    \node[circle,fill=OI6,inner sep=2pt] (L37) at (0.985,-3) {} ;
    \node[rectangle,fill=OI5,inner sep=2.5pt] (R30) at (0.015,-3) {} ;
    \node[rectangle,fill=OI5,inner sep=2.5pt] (R31) at (0.05,-3) {} ;
    \node[rectangle,fill=OI5,inner sep=2.5pt] (R32) at (0.215,-3) {} ;
    \node[rectangle,fill=OI5,inner sep=2.5pt] (R33) at (0.25,-3) {} ;
    \node[rectangle,fill=OI5,inner sep=2.5pt] (R34) at (0.765,-3) {} ;
    \node[rectangle,fill=OI5,inner sep=2.5pt] (R35) at (0.8,-3) {} ;
    \node[rectangle,fill=OI5,inner sep=2.5pt] (R36) at (0.965,-3) {} ;
    \node[rectangle,fill=OI5,inner sep=2.5pt] (R37) at (1,-3) {} ;
    \draw[->,OI7!75!black] (0) -- (L10) node[midway,left,OI7!75!black]{$S_1$};
    \draw[->,OI7!75!black] (0) -- (R10);
    \draw[->,OI7!75!black] (1) -- (L11);
    \draw[->,OI7!75!black] (1) -- (R11);
    \draw[->,OI7!75!black] (L10) -- (L20) node[midway,left,OI7!75!black]{$S_2$};
    \draw[->,OI7!75!black] (L10) -- (R20);
    \draw[->,OI7!75!black] (R10) -- (L21);
    \draw[->,OI7!75!black] (R10) -- (R21);
    \draw[->,OI7!75!black] (L11) -- (L22);
    \draw[->,OI7!75!black] (L11) -- (R22);
    \draw[->,OI7!75!black] (R11) -- (L23);
    \draw[->,OI7!75!black] (R11) -- (R23);
    \draw[->,OI7!75!black] (L20) -- (L30) node[midway,left,OI7!75!black]{$S_3$};
    \draw[->,OI7!75!black] (L20) -- (R30);
    \draw[->,OI7!75!black] (R20) -- (L31);
    \draw[->,OI7!75!black] (R20) -- (R31);
    \draw[->,OI7!75!black] (L21) -- (L32);
    \draw[->,OI7!75!black] (L21) -- (R32);
    \draw[->,OI7!75!black] (R21) -- (L33);
    \draw[->,OI7!75!black] (R21) -- (R33);
    \draw[->,OI7!75!black] (L22) -- (L34);
    \draw[->,OI7!75!black] (L22) -- (R34);
    \draw[->,OI7!75!black] (R22) -- (L35);
    \draw[->,OI7!75!black] (R22) -- (R35);
    \draw[->,OI7!75!black] (L23) -- (L36);
    \draw[->,OI7!75!black] (L23) -- (R36);
    \draw[->,OI7!75!black] (R23) -- (L37);
    \draw[->,OI7!75!black] (R23) -- (R37);
    \end{tikzpicture}
    \caption{Sketch of the Cauchy sequence in \cref{exa:Cantor}. The "left" sequence $L_n$ are drawn as red dots \textcolor{OI6}{\faCircle}, the "right" sequence $R_n$ as blue squares $\textcolor{OI5}{\blacksquare}$, and the split $S_n$ from $X_{n-1}$ to $X_n$ as purple arrows $\textcolor{OI7!75!black}{\to}$.}
    \label{fig:Cantor}
\end{figure}

Here we consider $M=[0,1]$ and we give an example of a Cauchy sequence in $(\Ran(M),\tom)$ that does not converge. Note that this proves that $(\Ran(M),\tom)$ is not complete. 

Let $\omega$ be a weight. We will define our Cauchy sequence as the union of two "left" and "right" sequences. Note that, in the case where $\omega$ is the constant weight equal to one, then these sequences actually correspond to the left and right boundaries of the intervals used to define the Cantor set. Let us define $L_0=\{0\}$, $R_0=\{1\}$, and, for any $n\geq 1$
\begin{align*}
L_n&=L_{n-1}\bigcup \left( R_{n-1}- \frac{1}{3^n\omega(2^{n+1})}\right)\\
R_n&=R_{n-1}\bigcup \left( L_{n-1}+ \frac{1}{3^n\omega(2^{n+1})}\right).
\end{align*}
Finally, for any $n\geq 0$, let us define $X_n=L_n\cup R_n$. Note that $\card(X_n)=2^{n+1}$. Let us first show that it is indeed a Cauchy sequence. Let $n\geq 1$, then there exists a split $S_n$ from $X_{n-1}$ to $X_n$ that send every point in $X_{n-1}$ to its closest point in $X_n$. This corresponds to splitting each point in $L_{n-1}$ into itself plus its copy shifted by $\frac{1}{3^n\omega(2^{n+1})}$, and symmetrically for $R_{n-1}$. This reads
\begin{align*}
    S_n=\left\{\left(x,x+\frac{1}{3^n\omega(2^{n+1})}\right)\mid x\in L_{n-1}\right\}\bigcup \left\{\left(x,x-\frac{1}{3^n\omega(2^{n+1})}\right)\mid x\in R_{n-1}\right\},
\end{align*}
where we omitted the identity terms. We thus have
\begin{align*}
    \dom(X_{n-1},X_n)\leq \elo(S_n)=\omega(2^{n+1})\sum_{(x,y)\in S_n}\dm(x,y)=\omega(2^{n+1}) \frac{2^{n}}{3^n\omega(2^{n+1})}=\left(\frac{2}{3}\right)^n.
\end{align*}
Since this is geometric, $X_n$ is indeed a Cauchy sequence for $\dom$.

Let us now show that this sequence cannot converge in $(\Ran(M),\tom)$. Let us suppose that $X\in\Ran(M)$ is a limit of $X_n$, and let us denote $m=\card(X)$, and let $N\geq \log_2(m)$. Since $\card(X_N)=2^{N+1}> m$, there must exist $x_N\in X_N$ such that $x_N\not\in X$. Now, for any $n\geq N$, we have $X_N\subset X_n$, so that we also have $x_N\in X_n$. So that $\dom(X_n,X)\geq\dH(X_n,X)\geq \dm(x_N,X)$, which cannot converge to zero. So that $X_n$ cannot converge to $X$.

Note that all the previously defined sequences are also Cauchy sequences for the Hausdorff distance. In that case, we can show that $X_n$ converges in the set of compact subset of $M$, for the Hausdorff distance.
\end{exa}

\begin{rem}
\label{rem:generalization of examples}
   We can generalize \cref{exa:aligned_points,exa:Telescoping_points_arbitrary_weight,exa:config_convergente,exa:Cantor} to the case where $M$ is a length space such that there exist $x_0\neq x_1\in X$ and an admissible path $\gamma\colon[a,b]\to M$ such that $\gamma(a)=x_0$, $\gamma(b)=x_1$, and $\ell(\gamma)=\dm(x_0,x_1)=L$. We will say that such a space is \textbf{equipped with a geodesic}. Up to a rescaling of the length structure on $M$ we can suppose that $L=1$. Then, remark that, for any $0\leq t\leq 1$ there exists a unique $x_t\in\gamma([a,b])$ such that $\dm(x,x_0)=t$. Let us write $f\colon[0,1]\to M$ the application that assigns $x_t$ to $t\in[0,1]$. Finally, we just need to apply $f$ to the previous examples.
\end{rem}

Finally this last example illustrates why $(\Ran(-),\tom)$ is not functorial on Top, and why it is not, in general, an invariant of the homeomorphism type of $M$. This constitutes a remarkable distinction with the Hausdorff and final topology.  
\begin{exa}
\label{exa:racine pas relevable sur ran omega}
Let $M=\R^+$, and let $f\colon\R^+\to\R^+$ be the square root function. Then $f$ is indeed an homeomorphism, but it is not locally Lipschitz at zero. Let us show that there exists a weight $\omega$ such that the map
\begin{align*}
    \Ran(f)\colon (\Ran(M),\tom)&\to(\Ran(M),\tom)\\
    X&\mapsto f(X)
\end{align*}
is not continuous at $\{0\}$.

Let us define the sequence of configurations $(X_n)$ in $\Ran(M)$, where for any $n\in\N$ $X_n=\{ku_n\mid 1\leq k\leq n\}$, $u_n=\e^{-3n/2}$. And let us define the weight $\omega$ as $\omega(n)=\e^n$ for any $n\in\N$. Remark that, for any $n\in\N$, we can apply \cref{exa:Telescoping_points_arbitrary_weight} to get
\begin{align*}
    \dom(X_n,\{0\})&=\sum_{k=1}^{n}\omega(k)u_n\\
    &\leq \omega(n) n u_n =n\e^{-n/2}.
\end{align*}
Therefore, $X_n$ converges to $\{0\}$ in $(\Ran(M),\tom)$.

Now, remark that we can also apply \cref{exa:Telescoping_points_arbitrary_weight} to compute the distance from $f(X_n)$ to $f(\{0\})=\{0\}$ to get
\begin{align*}
    \dom(f(X_n),\{0\})&= \sum_{k=1}^n \omega(k)\left( \sqrt{ku_n}-\sqrt{(k-1)u_n}\right)\\
    &\geq \omega(n)\sqrt{u_n}\left( \sqrt{n}-\sqrt{(n-1)}\right)= \omega(n)\frac{\sqrt{u_n}}{\sqrt{n}+\sqrt{n-1}}
    &\geq \omega(n)\frac{\sqrt{u_n}}{2\sqrt{n}}=\frac{\e^{n/4}}{2\sqrt{n}},
\end{align*}
so that $f(X_n)$ does not converge to $f(\{0\})$. Therefore, $\Ran(f)$ is not continuous at $\{0\}$.
\end{exa}

\subsection{Local computations}
\label{subsec:locality}
In this section we define the notion of locality, that is very handy to make calculations close to configurations of cardinality greater that 1. The idea is to identify the size of neighborhoods $U$ of $X\in\Ran(M)$ under which the configurations $Y\in U$ can be considered "localized" around the points of $X$, i.e. every point in $Y$ is close to a unique point in $X$ and every point in $X$ is close to at least a point in $Y$. In the following, we will call \textbf{clusters} the elements of the induced partition on $Y$. 

We also extend the notion of locality to combinatorial chains by asking that the chain only sends points localized around a given $x\in X$ to points localized around the same $x\in X$. This will be useful to lift local properties of $M$ and of applications on $M$ to local properties of $\Ran(M)$ and of applications on $\Ran(M)$.

Throughout this section, we fix $(M,d)$ a metric space.

Let us state a useful lemma, that, together with the following corollary, give the intuition why any small enough neighborhood of a given configuration $X\in\Ran(M)$ only contains configurations of cardinality greater than $M$.

\begin{lem}
    \label{lem:minorant distance}
    Let $X,Y\in \Ran(M)$. If $\card(X)>\card(Y)$ then 
    \begin{equation*}
        \dom(X,Y)\geq \omega(\card(X))\min_{x\neq z\in X}\{\dm(x,z)\}.
    \end{equation*}
\end{lem}
\begin{proof}
Let $\epsilon>0$. There exists a MBS-chain $(R_0,\dots,R_{k-1})$ from $X$ to $Y$ such that $\elo(R_0,\dots,R_{k-1})\leq \dom(X,Y)+\epsilon$. Since $\card(X)>\card(Y)$, then $R_0$ must be a simple merge between two points $x\neq w\in X$. We thus have
\begin{align*}
    \dom(X,Y)&\geq \elo(R_0,\dots,R_{k-1}) -\epsilon \\
    &\geq \elo(R_0) - \epsilon \\
    &\geq\omega(\card(X))\dm(x,w) -\epsilon \\
    &\geq \omega(\card(X))\min_{x\neq z\in X}\{\dm(x,z)\}-\epsilon
\end{align*}
Since this is true for any $\epsilon>0$, we have the desired result.
\end{proof}

This lemma allows to derive the following corollary, which is a concrete way of interpreting \cref{cor:truncations closed}.
\begin{corol}\label{cor:neighborhood_min_cardinal}
    Let $X\in \Ran(M)$ be a configuration, with $\card(X)=n$. Then any small enough neighborhood of $X$ in $\Ran(M)$, is contained in $\Ran_{\geq n}(M)$.
\end{corol}

\begin{proof}
   Let $r=\omega(\card(X))\min_{x\neq y\in X}\{\dm(x,y)\}.$ By \cref{lem:minorant distance}, any $Y\in \Bom(X,r)$ must have cardinality at least $n$.
\end{proof}

The notion of locality actually derives from the Hausdorff distance. It will therefore be useful to compare the weighted distance to the Hausdorff distance.
\begin{prop}
\label{prop:dH leq dom}
    Let $\omega\colon \N^*\to [0,\infty)$ be a weight, and $X,Y\in\Ran(M)$. Then $\dH(X,Y)\leq\dom(X,Y)$. In particular, for any $r>0$, $\Bom(X,r)\subset\BH(X,r)$.
\end{prop}
\begin{proof}
    There exists $(x,y)\in X\times Y$ such that $\dH(X,Y)=\dm(x,y)$. By symmetry, we may assume that $d(x,y)=d(x,Y)$, so that we have $d(x,y')\geq d(x,y)$ for any $y'\in Y$. Let $\epsilon>0$. There exists a combinatorial chain $((X=X_0,\dots,X_k=Y),(R_0,\dots,R_{k-1}))$ from $X$ to $Y$ such that $\elo(R_0,\dots,R_{k-1})\leq \dom(X,Y)+\epsilon$. By surjectivity of $R_i$ for all $0\leq i \leq k-1$, there must exist a sequence $(x=x_0,\dots,x_k)$ such that $x_i\in X_i$ for any $0\leq i\leq k$ and $x_i R_i x_{i+1}$ for any $0\leq i \leq k-1$. We thus have
    \begin{align*}
        \dom(X,Y)+\epsilon&\geq \elo(R_0,\dots,R_{k-1})\\
        &\geq \sum_{i=0}^{k-1} \omega(\card(R_i)) \dm(x_i,x_{i+1})\\
        &\geq \dm(x,x_k)\geq\dH(X,Y)
    \end{align*}
    Since this is true for any $\epsilon>0$, we finally get $\dH(X,Y)\leq\dom(X,Y)$.
\end{proof}

We arrive at the core notion of this section. Intuitively, we use the notion of locality to lift local properties of $M$ onto $\Ran(M)$. We start by estimating the size of a neighborhood $X\in U$ we need to ensure that the distance between configurations in $U$ can be computed from combinatorial chains that stay sufficiently close to $X$. 
\begin{lem}\label{lem:Omega_distance_2epsilon_ball}
    Let $X\in \Ran(M)$, $\epsilon>0$, and $\omega$ be a weight. Then, for any $Y,Z\in \Bom(X,\epsilon)$, the distance $\dom(Y,Z)$ can be computed from chains in $\Bom(X,2\epsilon)$. In particular it can be computed from chains in $\BH(X,2\epsilon)$.
\end{lem}

\begin{proof}
    This is an easy application of the triangular inequality. Since $\dom(Y,Z)\leq \dom(Y,X)+\dom(Y,Z)<2\epsilon$, there must be some combinatorial chain $(R_0,\dots,R_n)$ between $Y$ and $Z$ whose $\omega$-length is strictly bounded by $2\epsilon$. Assume that some intermediary configuration, $Y_k$ lies outside of $\Bom(X,2\epsilon)$. Then, 
    \begin{align*}
        2\epsilon &>\elo(R_0,\dots,R_k)+\elo(R_{k+1},\dots,R_n)\\
        &\geq \dom(Y,Y_k)+\dom(Y_k,Z)\\
        &>2\epsilon
    \end{align*}
    Hence, the distance $\dom(Y,Z)$ can be computed as the infimum of the $\omega$-length of chains between $Y$ and $Z$ which lie entirely in $\Bom(X,2\epsilon)$. Finally, we get the special case by observing that, by \cref{prop:dH leq dom} $\Bom(X,2\epsilon)\subset \BH(X,2\epsilon)$
\end{proof}

We can deduce the following corollary that will be useful in \cref{sec:conicity}.
\begin{corol}\label{corol:Opens_induce_open_embeddings}
    Let $(M,d)$ be a metric space, and $V\subset M$ be an open subspace. Then, the inclusion $V\to M$ induces a map, $\Ran(V)\to \Ran(M)$, which is continuous for the $\omega$ topologies, for any weight $\omega$. Furthermore, it is an open embedding.
\end{corol}

\begin{proof}
 We extend the inclusion $V\to M$ to a map between Ran spaces in the obvious way. That is, we consider the map sending a configuration $X\subset V$ to the configuration $X\subset M$. Now, to show that it is continuous, consider $X\in \Ran(V)$ a configuration. Since $V$ was assumed to be open, there exists some $\epsilon>0$ such that we have an inclusion $\cup_{x\in X}B(x,2\epsilon)\subset V$. This means that $\BH(X,2\epsilon)\subset \Ran(V)$. But now, by \cref{lem:Omega_distance_2epsilon_ball}, we know that $\Bom_V(X,\epsilon)=\Bom_M(X,\epsilon)$, and that distances in both coincide. Thus, the inclusion induces an isometry when restricted on $\Bom_V(X,\epsilon)$. In particular, it is continuous at $X$.
\end{proof}

To estimate if the points in $Y$ are localized around the points in $X$ we need an estimate of the minimal distance between distinct points in $X$. We call it the \textbf{merging radius}.
\begin{defin}
\label{def:merging radius}
    We define the \textbf{merging radius} function as
    \begin{align*}
        \merg\colon \Ran(M)&\to (0,+\infty]\\
        X &\mapsto
            \begin{cases}
            \frac{1}{2}\min_{x\neq y\in X}\{\dm(x,y)\}  \quad &\text{if }\card(X)>1\\
            1 &\text{if }\card(X)=1
        \end{cases}.
    \end{align*}
\end{defin}

\begin{rem}
    Note that the merging radius function has discontinuities around merges. Note also that the value chosen for cardinality one is arbitrary and doesn't affect any of the following proofs.
\end{rem}

We can now define more precisely the locality of configuration, relations, and combinatorial chains.
\begin{defin}
    Let $X\in\Ran(M)$ and $0<r\leq\merg(X)$. We say that a configuration $Y\in\Ran(M)$ is \textbf{$\BH(X,r)$-local} if $Y\subset \bigcup_{x\in X} B(x,r)$ and $Y\cap  B(x,r)\neq\emptyset$ for all $x\in X$.
    
    If $Y,Z\in\Ran(M)$ are two $\BH(X,r)$-local configuration, we say that a relation $R\subset X \times Y $ between $Y$ and $Z$ is $\BH(X,r)$-local if $yRz$ implies that there exists $x\in X$ such that $y,z\in B(x,r)$. 
    
    We say that a combinatorial chain $((X_0,\dots,X_k),(R_0,\dots,R_{k-1}))$ in $M$ is $\BH(X,r)$-local if $X_i$ is $\BH(X,r)$-local for any $0\leq i\leq k$ and if $R_i$ is $\BH(X,r)$-local for any $0\leq i\leq k-1$.
\end{defin}

\begin{rem}
\label{rem:card local}
        Let $X\in\Ran(M)$ and $0<r\leq\merg(X)$. A configuration $Y\in\Ran(M)$ is $\BH(X,r)$-local if and only if $Y\in\BH(X,r)$. Moreover, if $Y\in\Ran(M)$ is $\BH(X,r)$-local then $\card(Y)\geq\card(X)$.
\end{rem}

The previous \cref{lem:Omega_distance_2epsilon_ball} gives a condition that ensures that we can compute the length from chains whose configurations are local. In the following lemma we extend this result to the relations of the chains, to ensures that we can compute the length from chains that are local. Note that the hypothesis is slightly stronger.

\begin{lem}
\label{lem:boules omega locales}
Let $X\in\Ran(M)$, $0<r\leq\frac{\merg(X)}{2}$ and $\omega$ a weight. For any $Y,Z\in\Bom(X,r)$, the distance $\dom(Y,Z)$ can be computed over the MBS-chains from $Y$ to $Z$ that are $\BH(X,2r)$-local.
\end{lem}
\begin{proof}
    Let $Y,Z\in\Bom(X,r)$. First, remark that $\dom(Y,Z)\leq \dom(Y,X)+\dom(X,Z)<2r=\merg(X)$. Let $0<\epsilon<2r-\dom(Y,Z)$. By \cref{lem:Omega_distance_2epsilon_ball} there exists a MBS chain $((Y_0,\dots,Y_k),(R_0,\dots,R_{k-1}))$ from $Y$ to $Z$ such that $\elo(R_0,\dots,R_{k-1})\leq\dom(Y,Z) + \epsilon$ and such that $Y_0,\dots,Y_k\in\Bom(X,2r)\subset\BH(X,2r)$. 

    Let $i\in\{0,\dots,k-1\}$, and $(y_i,y_{i+1})\in R_i$. Since $((Y_0,\dots,Y_k),(R_0,\dots,R_{k-1}))$ is a combinatorial chain from $Y$ to $Z$, there exists $(y_0,\dots,y_k)\in Y_0\times\dots\times Y_k$ such that $y_j R_j y_{j+1}$ for all $j\in\{0,\dots,k-1\}$. In addition, for all $j\in\{0,\dots,k\}$, since $Y_j$ is $\BH(X,2r)$-local, there exists a unique $x_j\in X$ such that $\dm(y_j,x_j)=\min_{x\in X}\{\dm(y_j,x)\}\leq \dH(Y_j,X)\leq\dom(Y_j,X)<2r$. If $\card(X)=1$, then $x_j=x_0$ for all $j\in\{0,\dots,k\}$. Otherwise 
    \begin{align*}
        \dm(x_0,x_j)&\leq \dm(x_0,y_0) + \dm(y_0,y_j) + \dm(y_j,x_j)\\
        &\leq \dom(X,Y) + \elo(R_0,\dots,R_{j-1}) + \dom(Y_j,X)\\
        &\leq \dom(X,Y) + \elo(R_0,\dots,R_{j-1}) + \dom(Y_j,Z) + \dom(Z,X)\\
        &< 2 \merg(X)=\min_{z\neq z'\in X}\{d(z,z')\}.
    \end{align*}
    and thus $x_j=x_0$ for all $j\in\{0,\dots,k\}$. In particular $x_i=x_{i+1}$, so that $(y_i,y_{i+1})\in \left(B(x_i,2r)\right)^2$ and $R_i$ is $\BH(X,2r)$-local. The MBS chain $((Y_0,\dots,Y_k),(R_0,\dots,R_{k-1}))$ is thus $\BH(X,2r)$-local. We can thus compute the infimum over the $\BH(X,2r)$-local MBS-chains.
\end{proof}

Finally, the following lemma will be useful to show that we can lift a locally Lipschitz map on $M$ to a locally Lipschitz map on $\Ran(M)$. It will also be useful in \cref{sec:conicity} to compute the distance between configurations that are transformed under a homothety or a scaling on $\Ran(M)$.
\begin{lem}
\label{lem:local scaling}
    Let $X\in\Ran(M)$, $0<r<\merg(X)$, and let $\phi\colon \bigsqcup_{x\in X} B(x,r)\to M$ be a map such that there exists $s\in[0,+\infty)$ such that $\dm(\phi(y),\phi(z))\leq s\dm(y,z)$ for all $(y,z)\in\bigsqcup_{x\in X} B(x,r)^2$. Let $Y,Z\in\Ran(M)$ be two $\BH(X,r)$-local configurations.
    
    If the $\omega$-distance between $Y$ and $Z$ can be computed over $\BH(X,r)$-local combinatorial chains from $Y$ to $Z$ then $\dom(\phi(Y),\phi(Z))\leq s\dom(Y,Z)$.
\end{lem}
\begin{proof}
    Let $\epsilon>0$ and let $((Y_0,\dots,Y_k),(R_0,\dots,R_{k-1}))$ be a $\BH(X,r)$-local combinatorial chain from $Y$ to $Z$ such that $\elo(R_0,\dots,R_{k-1})\leq\dom(Y,Z) +\epsilon$. Let us denote $P_i=\phi^2(R_i)$ for all $0\leq i\leq k-1$, and let show that $((\phi(Y_0),\dots,\phi(Y_k)),(P_0,\dots,P_{k-1}))$ is a combinatorial chain from $\phi(Y)$ to $\phi(Z)$. Let $i\in\{0,\dots,k-1\}$, and let $z\in\phi(Y_i)$, there exists $y\in Y_i$ such that $\phi(y)=z$. Since $R_i$ is surjective, there exists $y'\in Y_{i+1}$ such that $y R_i y'$, and thus $z P_i \phi(y')$. Similarly, we show that for any $z\in \phi(Y_{i+1})$, there exists $z'\in \phi(Y_i)$ such that $z' P_i z$. This proves that $P_i$ is surjective, and that $((\phi(Y_0),\dots,\phi(Y_k)),(P_0,\dots,P_{k-1}))$ is indeed a combinatorial chain from $\phi(Y)$ to $\phi(Z)$. We can compute its length:
    \begin{align*}
        \elo(P_0,\dots,P_{k-1})&=\sum_{i=0}^{k-1} \omega(\card(P_i))\sum_{(z,z')\in P_i} \dm(z,z')\\
        &\leq \sum_{i=0}^{k-1} \omega(\card(R_i))\sum_{(y,y')\in R_i} \dm(\phi(y),\phi(y')) \\
        &\leq \sum_{i=0}^{k-1} \omega(\card(R_i))\sum_{(y,y')\in R_i} s\dm(y,y')\\
        &\leq s \elo(R_0,\dots,R_{k-1}).
    \end{align*}
    Finally 
    $$\dom(\phi(X),\phi(Y))\leq \elo(P_0,\dots,P_{k-1})\leq s\elo(R_0,\dots,R_{k-1})\leq s\dom(X,Y) + s\epsilon$$ 
    Since this is true for any $\epsilon>0$, we have $\dom(\phi(X),\phi(Y))\leq s\dom(X,Y)$. 
\end{proof}

\subsection{Functoriality and invariance}
\label{subsec:Functoriality}
We can now use the previous examples and computations to show some general properties of the weighted topologies. In particular \cref{exa:racine pas relevable sur ran omega} illustrates that the lift $\Ran(f)$ of a continuous map $f$ is not always continuous. From this, we deduce two important differences with the Hausdorff and final topologies. Firstly the weighted topologies do not, in general, induce a functor on the category of metric spaces and continuous maps. Nevertheless, if one considers the subcategory of metric spaces and locally Lipschitz maps, $\mathcal{C}$, then for any weight, $\omega$, we do have a functor $\Ran^\omega\colon\mathcal{C}\to \mathcal{C}$. Secondly, the homeomorphism type of $\Ran^\omega(M)$ is not necessarily an invariant of the homeomorphism type of $M$, however it is still invariant in a more rigid sense, see \cref{rem:Invariance_locally_bi_Lipschitz}.

\begin{prop}\label{prop:Lipschitz_Map_lift}
    Let $(M,d_M)$ and $(N,d_N)$ be two metric spaces and $f\colon M\to N$ a locally Lipschitz map. Then for any $\omega$, the lifted map 
    \begin{align*}
        \Ran(f)\colon \Ran(M)&\to \Ran(N)\\
        X&\mapsto f(X)
    \end{align*}
    is locally Lipschitz for the distances induced by $\omega$ on $\Ran(M)$ and $\Ran(N)$.
\end{prop}
\begin{proof}
    Let $X\in\Ran(M)$. For any $x\in X$ there exists $\epsilon_x>0$ and $p_x$ such that for any $y,z\in B_M(x,\epsilon_x)$, $d_N(f(y),f(z))\leq p_x d_M(y,z)$. Let us define $\epsilon=\min \left(\{\epsilon_x \mid x\in X\}\cup \{\merg(X)\}\right)$, and $p=\max\{p_x\mid x\in X\}$. Then $f_{|\bigcup_{x\in X} B(x,\epsilon)}$ is such that for any $(y,z)\in \bigcup_{x\in X} \left(B(x,\epsilon)^2\right)$, we have $d_N(f(y),f(z))\leq p d_M(y,z)$. 
    
    Let $Y,Z\in\Bom(X,\frac{\epsilon}{2})$. By \cref{lem:boules omega locales}, the distance $\dom(Y,Z)$ can be computed over $\BH(X,\epsilon)$-local MBS-chains, so that we can apply \cref{lem:local scaling} to directly get that $\dom(f(Y),f(Z))\leq p\dom(Y,Z)$.
\end{proof}

\begin{corol}\label{corol:Ran_omega_functor}
Let $\mathcal{C}$ be the category whose objects are metric spaces, and whose morphisms are locally Lipschitz maps, then, for any weight $\omega$ the following functor is well defined
\begin{align*}
    \Ran^\omega\colon \mathcal{C} &\to \mathcal{C}\\
    (M,d)&\mapsto (\Ran(M),\dom)\\
    f\colon M\to N &\mapsto \Ran(f)\colon \Ran^\omega(M) \to \Ran^\omega(N)
\end{align*}
\end{corol}

\begin{rem}\label{rem:Invariance_locally_bi_Lipschitz}
    The previous corollary implies, in particular, that if $(M,d)$ and $(N,d')$ are two metric spaces related by a locally bi-Lipschitz homeomorphism, then $\Ran^{\omega}(M,d)$ and $\Ran^{\omega}(N,d')$ are related by a locally bi-Lipschitz homeomorphism.    
\end{rem}

\section{Distances and topologies}
\label{sec:Distances_Topologies}
Now that we defined and characterized new metrics on $\Ran(M)$, we spend some time, in this section, to compare the induced topologies between themselves, and to the previously known Hausdorff and final topologies. In particular, we show that there are uncountably many distinct weighted topologies on $\Ran(M)$ (\cref{rem:Uncountably_many_topologies}), that they all refine the Hausdorff topology and are all refined by the final topology \cref{prop:tom raffine tH,tpi rafine tom}. If the case where the metric space $M$ is locally compact, we also show that the weighted balls $\Bom(X,\epsilon)$ form a basis of the final topology, i.e. that $(\Ran(M),\tpi)$ is the limit in $\Top$ of the $(\Ran(M),\tom)$ \cref{theo:topologie limite}. Finally, we use these tools to give new interpretations of some known properties of the final topology.

\subsection{Weighted topologies}
\label{subsec:comparing_weighted_topologies}
We first compare the weighted topologies.

\begin{defin}\label{defin:Functor_Topology}
    Let $\W$ be the preorder whose objects are weights, $\omega\colon \N^*\to [1,\infty)$ and where there is a relation $\omega\preceq \chi$ if and only if, there exists some $N\in \N$ such that for all $n\geq N$, $\omega(n)\leq \chi(n)$. We also write $\W$ for the categorification of the preorder $\W$, and define the functor
    \begin{align*}
        \Ran^{-}(M)\colon \W^{\op}&\to \Top\\
        \omega&\mapsto (\Ran(M),\tau^{\omega})\\
        \omega\preceq \chi &\mapsto \Ran(\Id)\colon (\Ran(M),\tom)\to (\Ran(M),\tau^{\chi})
    \end{align*}
\end{defin}

\begin{prop}\label{prop:Smaller_Weights_continuous}
    If $\omega,\chi\colon \N^* \to [1,\infty)$ are two weights such that $\omega\preceq \chi$, then the identity induces a continuous map
    \begin{equation*}
        (\Ran(M),\tau^{\chi})\to (\Ran(M),\tau^{\omega}).
    \end{equation*}
    In particular, the functor $\Ran^{-}(M)$ is well-defined.
\end{prop}

The proof will follow immediately from the following lemma
\begin{lem}\label{lem:Topology_Weights_Bounded}
    Let $\omega,\chi\colon \N^*\to [1,\infty)$ be two weights. If the sequence $\frac{\omega(n)}{\chi(n)}$ is bounded, then $\Id\colon(\Ran,\tchi)\to(\Ran,\tom)$ is Lipschitz and {\tchi} is finer than {\tom}. In particular, if $\omega(n)=\chi(n)$ for all but finitely many $n$, then the topologies $\tom$ and $\tchi$ coincide.
\end{lem}

\begin{proof}
    Let $\omega,\chi\colon \N^*\to [1,\infty)$ be two weights such that the sequence $\frac{\omega(n)}{\chi(n)}$ is bounded. There exists $m> 0$ such that, for all $n\in \N^*$, $\omega(n)\leq m \chi(n)$ Now, let $X,Y\in \Ran(M)$ be two configurations, $\epsilon >0$, and $(R_1,\dots,R_n)$ be a chain whose $\chi$-length is at most $d^{\chi}(X,Y)+\epsilon$.
    Then, we compute
    \begin{align*}
        \dom(X,Y)&\leq \elo(R_1,\dots,R_n)\\
        &=\sum_{k=1}^n\omega(\card(R_k))\ell(R_k)\\
        &\leq \sum_{k=1}^nm\chi(\card(R_k))\ell(R_k)\\
        &=m\ell^{\chi}(R_1,\dots,R_n)\\
        &\leq md^{\chi}(X,Y)+m\epsilon
    \end{align*}
    Since this holds for any $\epsilon >0$, we must have $\dom(X,Y)\leq m d^{\chi}(X,Y)$, which proves that $\Id\colon(\Ran(M),\chi)\to(\Ran(M),\tom)$ is Lipschitz and that {\tchi} is finer than {\tom}.
    
    If $\omega(n)=\chi(n)$ for all but finitely many $n$ then we can define $m_1=\max\{\frac{\omega(n)}{\chi(n)}\mid n\in \N^*\}$ and $m_2=\max\{\frac{\chi(n)}{\omega(n)}\mid n\in \N^*\}$. Thus both $\frac{\omega(n)}{\chi(n)}$ and $\frac{\chi(n)}{\omega(n)}$ are bounded, so that {\tchi} is finer than {\tom} and {\tom} is finer than {\tchi}, i.e. they coincide.
\end{proof}

\begin{proof}[Proof of \cref{prop:Smaller_Weights_continuous}]
    Let $\omega$ and $\chi$ be two weights, and assume that $\omega\preceq \chi$. Then the sequence $\frac{\omega(n)}{\chi(n)}$ is bounded so that, by \cref{lem:Topology_Weights_Bounded}, $\tchi$ refines $\tom$ and $\Ran(\Id)$ is indeed continuous.\end{proof}

\begin{prop}\label{prop:omega_strictement_plus_fine}
    Let $M$ be a length space equipped with a geodesic. Let $\omega,\chi$ two weights. If the sequence $(v_n)_{n\in\N^*}$ defined as
    \begin{align*}
    v_n = \frac{\chi(n)}{\omega(n)}
    \end{align*}
    goes to infinity, then {\tchi} is stricly finer than {\tom}. In particular, for any weight $\omega$, there exists a weight $\chi$ such that {\tchi} is strictly finer than \tom.
    \end{prop}
    
    \begin{proof}
    Since $\frac{\chi(n)}{\omega(n)}$ goes to infinity, then $\frac{\omega(n)}{\chi(n)}$ is bounded. Thus, by \cref{lem:Topology_Weights_Bounded}, $\tchi$ is finer that $\tom$. Let us show that the two topologies differ.
    For this, let us define the sequence $u_n=\frac{\omega(n)}{\chi(n)}$ which converges to zero.  By \cref{rem:generalization of examples}, we can generalize \cref{exa:config_convergente} to define a sequence $X_n$ of configurations in $\Ran(M)$ that converges towards $\{x_0\}$ for the topology $\tom$, using the sequence $u_n$.
     
    Besides, as shown in \cref{exa:config_convergente}, we have, for any $n\geq 1$,
    \begin{align*}
    \dchi(X_n,\{x_0\})= \sum_{i=1}^n \chi(i)d(x_{i-1},x_i)= u_n \sum_{i=1}^n \frac{\chi(i)}{\omega(i)} \geq u_n \frac{\chi(n)}{\omega(n)}\geq 1
    \end{align*}
    which does not go to zero. Therefore, the sequence $X_n$ does not converge towards $\{x_0\}$ for the $\tchi$ topology, so that these two topologies do not coincide.
    \end{proof}
    
    \begin{rem}\label{rem:Uncountably_many_topologies}
    There exists an uncountable infinity of distinct topologies (we can for example take the geometric progressions of common ratio $q\geq 1$).
    \end{rem}

We also show that sequences of unbounded cardinality cannot converge for all weights.
\begin{prop}
\label{prop:unbounded card does not converge omega}
    Let $(M,d)$ be a metric space and let $(X_k)_{k\in \N}$ be a sequence of elements of $\Ran(M)$. Assume that the sequence of cardinalities $(\card(X_k))_{k\in \N}$ is unbounded, then there exists some weight $\omega$ such that the sequence $(X_k)_{k\in \N}$ is not Cauchy for $\dom$.
\end{prop}

\begin{proof}
    Up to extracting a subsequence, we may assume that $(\card(X_k))_{k\in \N}$ is strictly increasing, and write $n_k$ for the corresponding sequence of cardinalities. Now, for $k\in \N$, let $d_k=\merg(X_k)$. By \cref{theo:MBS}, for any weight $\omega$ and any $k>0$, the distance $\dom(X_k,X_{k-1})$ can be computed from MBS-chains from $X_k$ to $X_{k-1}$. The first step must be a merge, since $n_{k}>n_{k-1}$, and thus the length of the first step is at least $\omega(n_k)d_k$. Now, we see that defining $\omega$ such that for all $k$, $\omega(n_k)\geq \frac{1}{d_k}$ gives a distance for which the sequence $(X_k)$ verifies $\dom(X_k,X_{k-1})\geq 1$ for any $k>0$, which cannot be Cauchy.
\end{proof}

\subsection{Comparison to classical topologies}
\label{subsec:Final_limite}
We now compare the weighted topologies to the previously known Hausdorff and final topologies. We first show that the weighted topologies lie in between the Hausdorff and the final topology.

\begin{prop}
\label{prop:tom raffine tH}
    Let $\omega$ be a weight. The topology $\tom$ is finer that $\tH$.
\end{prop}
\begin{proof}
    This a direct consequence of \cref{prop:dH leq dom}.
\end{proof}

\begin{prop}
\label{tpi rafine tom}
    Let $\omega$ be a weight. The topology $\tpi$ is finer that $\tom$. 
\end{prop}
\begin{proof}
    Let $X\in\Ran(M)$ and $\epsilon>0$ and let us show that $\Bom(X,\epsilon)$ is open in {\tpi}. Let us fix $n\in\N^*$, and let us show that $\pi_n^{-1}(\Bom(X,\epsilon))$ is an open set of $M^n$. Let $\underline{y}=(y^{(1)},\dots,y^{(n)})\in\pi_n^{-1}(\Bom(X,\epsilon))$, $Y=\pi_n(\underline{y})$, and let us define $\delta=\frac{\epsilon-\dom(X,Y)}{\omega(n)}$. Let $\underline{z}=(z^{(1)},\dots,z^{(n)})\in M^n$ such that $\dsum(\underline{y},\underline{z})<\delta$. We have
    \begin{align*}
        \dom(\pi_n(\underline{z}),X)&\leq \dom(\pi_n(\underline{z}),Y)+\dom(Y,X)\\
        &\leq \sum_{j=1}^n \omega(n)\dm(z^{(j)},y^{(j)}) +\dom(Y,X)\\
        &\leq \omega(n) \dsum(\underline{z},\underline{y}) + \dom(Y,X)\\
        &< \omega(n)\delta  + \dom(Y,X)\\
        & <\epsilon.
    \end{align*}
    So that $\pi_n(\underline{z})\in \Bom(X,\epsilon)$ and thus $\Bsum(\underline{y},\delta)\subset \pi_n^{-1}(\Bom(X,\epsilon))$. This proves that $\pi_n^{-1}(\Bom(X,\epsilon))$ is indeed open for any $n\geq 1$ and thus that $\Bom(X,\epsilon)$ is open in \tpi.
\end{proof}

We now show that in the case where $M$ is locally compact then the final topology is the limit of the weighted topologies.
\begin{theo}
\label{theo:topologie limite}
    If the metric space $(M,d)$ is locally compact, then the functor $\Ran^{-}(M)$ defined in \cref{defin:Functor_Topology} admits $(\Ran(M),\tpi)$ as a limit.
\end{theo}

\begin{rem}
The result of \cref{theo:topologie limite} actually induces a uniformity on $(\Ran(M),\tpi)$. We detail this point of view in the following section.
\end{rem}

\begin{proof}
    Given that all maps in the diagram $\Ran^{-}(M)$ are underlied by identities, one needs to show that any map $f\colon N\to \Ran(M)$ is continuous for the final topology if and only if it is continuous for $\tom$ for all $\omega$. One direction is clear, since the map $(\Ran(M),\tpi)\to (\Ran(M),\tom)$ is always continuous. For the reciprocal, it will be sufficient to show that the collection of open balls for all weights form a basis for the final topology.
    
    To show this, let us first assume that $M$ is compact. Let $U\subset \Ran(M)$ be an open for the final topology, and $X\in U$. We need to exhibit a weight $\omega$ and some $\epsilon> 0$ such that $\Bom(X,\epsilon)\subset U$. We will proceed inductively. Given a weight $\omega$, $\epsilon >0$ and $k\geq 1$, let us denote $\Bom_{\leq k}(X,\epsilon)$ for the intersection $\Bom(X,\epsilon)\cap \Ran_{\leq k}(M)$. Observe that, to define $\Bom_{\leq k}(X,\epsilon)$ one needs only know the value of $\omega$ up to $k$. In particular, we may construct $\omega$ inductively, and prove that $\Bom_{\leq k}(X,\epsilon)\subset U$ before constructing all of $\omega$. We will proceed as follows. 

    Let $n=\card(X)$. The case $n=1$ is somewhat of a special case, so we leave it for the end. Assume $n\geq 2$, and set $\epsilon =\min_{x\not= y\in X}\{d(x,y)\}$. Then, if $Y\in \Ran(M)$ is such that $\dom(X,Y)<\epsilon$, for some weight $\omega$, then we must have $\card(Y)\geq n$ by \cref{lem:minorant distance}. Thus, we see that whatever the weight $\omega$ may be, we will have $\Bom_{\leq k}(X,\epsilon)=\emptyset$, for any $k<n$. For convenience, we may thus define $\omega(n-1)=\dots = \omega(1)=1$. 
    Now, we will construct subsets $K_k\subset M^k$, together with $\omega(k)$, such that, for all $k\geq n$.
    \begin{itemize}
        \item $K_k$ is compact
        \item $\Bom_{\leq k}(X,\epsilon)\subset \pi_k(K_k)\subset U$
        \item $\pi_k(K_k)\subset \pi_{k+1}(K_{k+1})$.
    \end{itemize}
    Which will give the claim, since the inclusion $\Bom_{\leq k}(X,\epsilon)\subset U$ for all $k$ gives $\Bom(X,\epsilon)\subset U$.

    For $k=n$, let $\underline{x}\in M^n$ be such that $\pi_n(\underline{x})=X$. Since $U$ is open in the final topology, $\pi_{n}^{-1}(U)$ must be an open subset of $M^n$, containing $\underline{x}$. 
    We will denote it $U_n$ from now on for simplicity. Thus, there exists some $\epsilon_n>0$ such that $\overline{\Bsum}(\underline{x},\epsilon_n)\subset U_n$. 
    Set $K_n=\overline{\Bsum}(\underline{x},\epsilon_n)$. It is indeed compact since it is a closed subset of the compact space $M^n$ (recall that we assumed $M$ to be compact for now).
    Set $\omega(n)=\max\{\frac{\epsilon}{\epsilon_n},1\}$, and let $Y\in \Bom_{\leq n}(X,\epsilon)$. We have already seen that we must have $\card(Y)=n$, and the distance between $X$ and $Y$ can be computed through bijections only. But since there are finitely many of those, there must be some bijection $\phi\colon X\to Y$ such that
    \begin{equation*}
        \dom(X,Y)=\omega(n)\sum_{x\in X}d(x,\phi(x))<\epsilon
    \end{equation*}
    From which we deduce that 
    \begin{equation*}
        \sum_{x\in X}d(x,\phi(x))<\frac{\epsilon}{\omega(n)}\leq \epsilon_n
    \end{equation*}
    On the other hand, from $\phi$, we deduce a preferred lift of $Y$ to $M^n$: $\underline{y}=(\phi(x^{(1)}),\dots,\phi(x^{(n)}))$, and thus we have
    \begin{equation*}
        \dsum(\underline{x},\underline{y})=\sum_{x\in X}d(x,\phi(x))\leq \epsilon_n
    \end{equation*}
    This proves that $Y=\pi_n(\underline{y})\in \pi_n(K_n)$, and thus that $\Bom_{\leq n}(X,\epsilon)\subset \pi_n(K_n)$.

    Now, assume that the $K_i$ and $\omega(i)$ have all been constructed up to some value $k$, and let us extend the construction to $k+1$. First, observe that as the image of a compact, $\pi_k(K_k)$ is compact, thus closed in $\Ran(M)$ for any weighted topology. In particular, $C=\pi_{k+1}^{-1}(\pi_k(K_k))$ must be a closed subset of $M^{k+1}$ thus a compact subspace, since $M^{k+1}$ is compact by assumption. Let us now consider the following subspaces of $M^{k+1}$ for $i\geq 1$
    \begin{equation*}
        C_i=\{\underline{y}\in M^{n+1}\mid \dsum(\underline{y},C)\leq \frac{1}{i}\}
    \end{equation*}
    By construction, the $C_i$ are closed, thus compact, and we have $\cap_{i\geq 1} C_i=C$. Furthermore, $C\subset U_{k+1}$ since by hypothesis, $\pi_k(K_k)\subset U$. Thus, there must exist some $i\geq 1$ such that $C_i\subset U_{k+1}$. We set $K_{k+1}=C_i$ and $\omega(k+1)=\max\{i\epsilon,\omega(k)\}$. $K_{k+1}$ is compact, and satisfies $\pi_{k+1}(K_{k+1})\subset U$ by construction. Furthermore, we have $\pi_{k+1}^{-1}(\pi_k(K_k))\subset K_{k+1}$, thus $\pi_k(K_k)\subset \pi_{k+1}(K_{k+1})$. Finally, let $Y\in \Bom_{\leq k+1}(X,\epsilon)$. If $\card(Y)\leq k$, then $Y\in \pi_k(K_k)\subset \pi_{k+1}(K_{k+1})$. Otherwise, $\card(Y)=k+1$. Consider a MBS-chain between $Y$ and $X$, $(R_0,\dots,R_l)$, such that $\elo(R_0,\dots,R_l)<\epsilon$. Since $\card(Y)>\card(X)$, the first step must be a merge, giving an intermediate configuration $Z$, with $\card(Z)\leq k$. We thus have
    \begin{align*}
        \dom(Y,X)\leq \dom(Y,Z)+\dom(Z,X)<\epsilon
    \end{align*}
    In particular, $\dom(Z,X)<\epsilon$ and thus $Z\in \Bom_{\leq k}(X,\epsilon)\subset \pi_k(K_k)$. On the other hand, if we chose a lift $\underline{y}\in M^{k+1}$ of $Y$ and let $f\colon Y\to Z$ stand for the map underlying the merge $R_0$, we may take $\underline{z}=(f(y^{(1)}),\dots,f(y^{(k+1)}))$ as a lift for $Z$. This gives 
    \begin{align*}
        \elo(R_0)=\omega(k+1)\sum_{y\in Y}d(y,f(y))=\omega(k+1)\dsum(\underline{y},\underline{z})<\epsilon
    \end{align*}
    From which we deduce that 
    \begin{equation*}
        \dsum(\underline{y},\underline{z})< \frac{\epsilon}{\omega(k+1)}
    \end{equation*}
    Since we already know that $Z\in \pi_k(K_k)$ this gives $\underline{y}\in K_{k+1}$ and thus $Y\in \pi_{k+1}(K_{k+1})$, which proves that $\Bom_{\leq k+1}(X,\epsilon)\subset \pi_{k+1}(K_{k+1})$.
    This completes the induction step, and we have thus constructed $\omega$ and $\epsilon$ such that $\Bom(X,\epsilon)\subset U$ in the case where the ambient space $M$ is compact and $\card(X)\geq 2$.

    If $\card(X)=1$, the first step in the induction needs to be changed. But in that case, there is some $x\in M$ such that $X=\{x\}$. We thus have, $U_1=\pi_1^{-1}(U)\subset M$ is an open subset containing $x$, thus there exists some $\epsilon>0$ such that $B(x,\epsilon)\subset \overline{B}(x,\epsilon)\subset U_1$, and we may take $K_1=\overline{B}(x,\epsilon)$. The rest of the proof goes through unchanged.

    Finally, if $M$ is not compact, but only locally compact. Let $X\in U\subset \Ran(M)$ be a point in an open subset for the final topology. Then, $X=\{x^{(1)},\dots,x^{(n)}\}$ for some points $x^{(i)}\in X$. Since $M$ is locally compact, each of those points admits a compact neighborhood $K_i\subset M$ for $1\leq i\leq n$. Fix in addition, for each $1\leq i\leq n$ some open neighborhood of $x^{(i)}$, $V_i$ such that $V_i\subset K$, and such that $V_i=B(x^{(i)},\epsilon_i)$, for some $\epsilon_i>0$. Now consider the compact subset $K=\cup_i K_i\subset M$, and the subset $V=\cup_iV_i\subset K$ which is open, both in $K$ and in $M$, and let $U'=U\cap \Ran(V)\subset \Ran(M)$. It is an open subset of $\Ran(M)$ since $\Ran(V)$ is one (indeed, it is an open subset of $\Ran(M)$ for the Hausdorff topology). 

    Now, we need to distinguish \textit{a priori} between open balls computed in $\Ran(M)$ and open balls computed in $\Ran(K)$, since the distance between two configurations depends on the ambient space. Let us denote $\Bom_M(X,\epsilon)$ and $\Bom_K(X,\epsilon)$ for the balls computed in $\Ran(M)$ and $\Ran(K)$ respectively. Consider $\epsilon=\frac{1}{2}\min\{\epsilon_i\mid 1\leq i\leq n\}$. Now, observe that the Hausdorff distances coincide, whether they are computed in $\Ran(M)$ or in $\Ran(K)$, and that by construction, $\BH(X,2\epsilon)\subset \Ran(K)$. Let us show the equality $\Bom_M(X,\epsilon)=\Bom_K(X,\epsilon)$. If $Y\in \Bom_K(X,\epsilon)$, we have $Y\in \Ran(M)$, and $\dom_M(X,Y)\leq\dom_K(X,Y)<\epsilon$, which gives the inclusion $\Bom_K(X,\epsilon)\subset \Bom_M(X,\epsilon)$. On the other hand, if $Z\in \Bom_M(X,\epsilon)$ , then $Z\in \BH(X,2\epsilon)\subset \Ran(K)$. Furthermore, $\dom_M(X,Z)$ can be computed in $\BH(X,2\epsilon)\subset \Ran(K)$, by \cref{lem:Omega_distance_2epsilon_ball}, thus we have $\dom_M(X,Z)=\dom_K(X,Z)$, and the inclusion $\Bom_M(X,\epsilon)\subset \Bom_K(X,\epsilon)$.
    
    We can now apply the already proven compact case to $U'\subset \Ran(K)$. There exists $\epsilon'>0$, and we may assume $\epsilon'<\epsilon$, and $\omega$ such that $\Bom_K(X,\epsilon')\subset U'$. By the above discussion, we have $\Bom_M(X,\epsilon')=\Bom_K(X,\epsilon')$, thus we have
    \begin{equation*}
        \Bom_M(X,\epsilon')\subset U'\subset U
    \end{equation*}
    Which is precisely what we set out to prove.
\end{proof}

We can now use the weighted distances and topologies to re-show some properties of the final topology. In particular, we first re-show that only the sequence of bounded cardinality can converge.

\begin{prop}\label{prop:convergence_cardinal_borne_tpi}
    Let $M$ be a metric space. Let $(X_n)_{n\in\N}$ a sequence in $\Ran(M)$ such that $\card(X_n)$ is unbounded. Then $(X_n)_{n\in\N}$ does not converge in $(\Ran(M),\tpi)$.
\end{prop}

\begin{proof}
    By \cref{prop:unbounded card does not converge omega}, there exists a weight $\omega$ such that $(X_n)$ does not converge for \tom. Since \tpi is finer that \tom, the sequence $(X_n)$ does not converge for \tpi.
\end{proof}

Finally, \cref{theo:topologie limite} gives an interesting interpretation for the observation that the topology $\tpi$ is not first countable, at least when the underlying space $M$ is suitably nice: given an arbitrary countable collection of open neighborhoods $(U_i)$, a diagonalization argument allows one to construct neighborhoods which contain none of the $U_i$.
\begin{prop}\label{cor:base_denombrable_voisinage}
    Let $M$ be a locally compact length space equipped with a geodesic, then $(\Ran(M),\tpi)$ is not first countable.
\end{prop}

\begin{proof}
Let $X_0=\{x_0\}\in\Ran(M)$ be a one point configuration in $\Ran(M)$, and $(U_i)_{i\in \N}$ a countable collection of neighborhoods of $X_0$. Let us show that it cannot constitute a basis of neighborhoods.

By \cref{theo:topologie limite}, for any $i\in \N$, there exist a weight $\omega_i$ and $\epsilon_i>0$ such that $B^{\omega_i}(X_0,\epsilon_i)\subset U_i$. We then define a weight $\chi$ as
    \begin{equation*}
        \chi(k)=\max\{{\omega_i(k)k\mid i\leq k}\},
    \end{equation*}
for all $k\geq 1$. By construction, for all $i\in \N$, and for all $k\geq i$, we have 
    \begin{equation*}
        \frac{\chi(k)}{\omega_i(k)}\geq k
    \end{equation*}
so that, by \cref{prop:omega_strictement_plus_fine}, the topology $\tchi$ is strictly finer that $\tau^{\omega_i}$, for all $i\in \N$. We deduce that for any $\epsilon >0$, we have for all $i\in \N$:
    \begin{equation*}
        B^{\omega_i}(X_0,\epsilon_i)\not\subset B^{\chi}(X_0,\epsilon).
    \end{equation*}
    Indeed, using \cref{exa:config_convergente}, we can construct points arbitrarily close to $X_0$ for $d^{\omega_i}$ which are arbitrarily far away from $X_0$ for $d^{\chi}$. 
We conclude that $(U_i)_{i\in \N}$ is not a basis of neighborhoods of $X_0$.

\end{proof}

\section{Completeness of the final topology}
\label{sec:completeness}
The previous \cref{theo:topologie limite} is worth interpreting from the perspective of uniform spaces \cite{Bourbaki-topology,Guenard-Lelievre-uniform-spaces}: it directly allows to construct a uniformity on $(\Ran(M),\tpi)$ by defining entourages from the weighted distances (see \cref{def:entourages Ran}). Therefore, even if $(\Ran(M),\tpi)$ is not metrizable, we can still meaningfully discuss its completeness. However, the notion is a bit more subtle for uniform spaces since it does not only require Cauchy sequences to converge, but also Cauchy filters. The convergence of Cauchy sequences is actually a direct consequence of \cref{prop:unbounded card does not converge omega}:
\begin{prop}\label{prop:cauchy_tal_tpi}
    Let $M$ be a complete metric space, and let $(X_n)_{n\in\N}$ be a sequence in $\Ran(M)$. If $(X_n)$ is a Cauchy sequence for all distances $\dom$, then $(X_n)$ converges in \tpi.
\end{prop}
\begin{proof}
    By \cref{prop:unbounded card does not converge omega}, if $(X_n)_{n\in\N}$ is a Cauchy sequence for all the distances $\dom$, then there exists $N\geq 1$ such that $\card(X_n)\leq N$ for all $n\in\N$. 
    By \cref{prop:dH leq dom}, $(X_n)$ is a Cauchy sequence for the Hausdorff distance.
    By \cref{prop:hausdorff_complete_troncation} the Hausdorff topology is complete on the truncation $\Ran_{\leq N}(M)$, so that $(X_n)$ converges in $\Ran_{\leq N}(M)$ for \tH. Besides, by \cref{prop:topologie_troncation}, {\tpi} and {\tH} coincide on the truncations, so that $(X_n)$ converges for {\tpi}.
\end{proof}
On the other hand, convergence of Cauchy filters requires some more work. Our proof relies on \cite[Corollary of Prop. 10, Chap II.3.5]{Bourbaki-topology} and shows that we can obtain $\Ran(M)$ as a limit of complete uniform spaces. This requires a few definitions and lemmas to better characterize Cauchy sequences and their limits for the weighted topologies. 

We start by extending the notion of combinatorial chain to define chains with an infinite number of relations and a finite length. This allows to extend the weighted distance $\dom$ to the completion $\mathcal{R}_\omega$ of $(\Ran(M),\dom)$ and thus evaluate the distance between elements of $\mathcal{R}_\omega$ not only abstractly but from actual chains that we can manipulate. This then allows us to show in \cref{prop:same limit cauchy} that if two sequences are Cauchy for $\dom$ and converge, for the Hausdorff distance, to the same compact subspace $K\subset M$, then their $\omega$-distance goes to zero. We can thus interpret $\mathcal{R}_\omega$ as a subset of $\Comp(M)$, and if $\omega\preceq\chi$ then we can interpret $\mathcal{R}_\chi$ as a subset of $\mathcal{R}_\omega$. It therefore makes sense to consider the intersection $\bigcap_{\omega\in\W}\mathcal{R}_\omega\subset \Comp(M)$. Moreover we also show in \cref{prop:limits_all_omega_finite} that any element in this intersection must actually have a finite cardinality, so that $\bigcap_{\omega\in\W}\mathcal{R}_\omega=\Ran(M)$. 

Finally, we show that the uniformity on $\Ran(M)$ obtained from \cref{theo:topologie limite}, is the same as the one obtained on $\bigcap_{\omega\in\W}\mathcal{R}_\omega$ by taking the limit of the $\mathcal{R}_\omega$ in the category of uniform spaces and uniformly continuous maps, which allows us to conclude.

For the entirety of \cref{sec:completeness}, we assume that $M$ is a complete and locally compact metric space.

\subsection{Infinite Chains in $\Ran(M)$}
\label{subsec:Infinite_Chains}
We start by defining infinite chains and use them to give upper bounds for the $\omega$-distance between their end points, or finite subsets of their end points, in $\Ran(M)$.

\begin{defin}
    Let $(X_k)_{k\in \N}$ be an infinite sequence of elements of $\Ran(M)$, and $(R_k)_{k\in \N}$ a sequence such that for all $k$, $R_k$ is a surjective relation from $X_k$ to $X_{k+1}$. We say that $((X_k)_{k\in \N},(R_k)_{k\in \N})$ is an \textbf{infinite $\omega$-chain} if
    \begin{equation*}
        \sum_{k=0}^{\infty}\elo(R_k)<\infty
    \end{equation*}
    
    If, in addition, there exists a compact subset $X_\infty\subset M$ such that $\dH(X_k,X_\infty)\to 0$, then we say that $((X_k)_{k\in \N},(R_k)_{k\in \N})$ is an infinite $\omega$-chain \textbf{between $X_0$ and $X_{\infty}$}.
    
    In any case, we define the $\omega$-length of the chain by
    \begin{equation*}
        \elo((R_k)_{k\in \N})=\sum_{k=0}^\infty \elo(R_k)
    \end{equation*}
\end{defin}

\begin{rem}
    We show in the following proposition that all infinite $\omega$-chains are actually infinite $\omega$-chains between some $X_0\in\Ran(M)$ and some $X_\infty \in \Comp(M)$. We also give an explicit construction of the limit $X_\infty$ from sequences extracted from the infinite chain.
\end{rem}

\begin{prop}\label{prop:Infinite_Chain_to_Hausdorff_Limit}
    Let $((X_k)_{k\in \N},(R_k)_{k\in \N})$ be an infinite $\omega$-chain. Let $S$ denote the set of sequences $(x_k)_{k\in \N}$ such that $x_k\in X_k$ and $x_kR_kx_{k+1}$ for all $k$, and let us define
    \begin{equation*}
        X_{\infty}=\{\lim(x_k)\mid (x_k)_{k\in \N}\in S\}.
    \end{equation*}
    Then
    \begin{itemize}
        \item $(X_k)_{k\in \N}$ is a Cauchy sequence for $\dom$
        \item all the sequences in $S$ are Cauchy sequences
        \item $((X_k)_{k\in \N},(R_k)_{k\in \N})$ is an infinite $\omega$-chain between $X_0$ and $X_{\infty}$
    \end{itemize}
\end{prop}

\begin{proof}
    For the first claim, observe that for any $n>k$, $(R_k,\dots,R_{n-1})$ is a combinatorial chain, in the sense of \cref{def:combinatorial chain}, between $X_k$ and $X_n$, and we have $\dom(X_k,X_n)\leq \elo(R_k,\dots,R_n)\leq \elo((R_i)_{i\geq k})$, the latter tends to $0$ as $k$ goes to infinity, which proves that $(X_k)_{k\in \N}$ is Cauchy for $\dom$.
    
    For the second claim, first note that if $(x_k)_{k\in \N}$ is a sequence in $S$, and $k<n$, then since $x_iR_ix_{i+1}$ for any $k\leq i<n$, we have
    \begin{align*}
        d(x_k,x_n)&\leq \sum_{i=k}^{n-1}d(x_i,x_{i+1})\\
        &\leq \sum_{i=k}^{n-1}\elo(R_i)\\
        &\leq \sum_{i=k}^{\infty}\elo(R_i)
    \end{align*}
    Since this last terms converges to $0$ as $k$ grows, $(x_k)$ is Cauchy, and
    thus, all elements of $S$ are Cauchy sequences in $M$.
    
    Lastly, for the third claim, first remark that $M$ was assumed to be complete, we thus may define $X_\infty=\{\lim(x_k)\mid (x_k)_{k\in \N}\in S\}\subset M$. As a set of limits, $X_\infty$ is automatically a complete and thus closed subspace of $M$, as can be deduced by diagonalizing sequences of sequences. Moreover, recalling the proof of \cref{prop:compact sets completion of ran hausdorff}, we see that $X_\infty\subset K$ where $K\subset M$ is the compact subset containing all limits of all Cauchy sequences $(x_k)_{k\in \N}$ with $x_k\in X_k$ for all $k\in \N$. Since we know $K$ to be compact, $X_\infty$ must also be compact.
    
    Now, let us prove that $\dH(X_k,X_\infty)$ converges to $0$, by showing that
    \begin{equation*}
        \dH(X_k,X_\infty)\leq \sum_{i\geq k}\elo(R_i)
    \end{equation*}
    Indeed, consider a sequence $(x_i)_{i\geq k}$ with $x_iR_ix_{i+1}$. Such a sequence must be Cauchy, and thus admits a limit, $x_\infty\in X_\infty$, and by the triangular inequality, we have
    \begin{align*}
        d(x_k,x_{\infty})&\leq \sum_{i=k}^{\infty}d(x_i,x_{i+1})\\
        &\leq \sum_{i=k}^{\infty}\elo(R_i)
    \end{align*}
    Since every element of $X_k$ can be completed into such a sequence, and every element in $X_\infty$ is reached as the limit of such a sequence, this gives the desired inequality, and concludes this proof.
\end{proof}

We now show that, if the end points of the infinite chain have a finite cardinality, then their distance minors the length of the chain. This shows that we can extend the set over which the infimum is taken in \cref{def:weighted-distance} to infinite chains.

\begin{lem}\label{lem:Infinite_Chain_Between_config}
 Let $(R_k)_{k\in \N}$ be an infinite chain between $X\in \Ran(M)$ and $X_{\infty}$. Assume that $X_{\infty}\in \Ran(M)$, then
 \begin{equation*}
     \dom(X,X_{\infty})\leq \elo((R_k)_{k\in \N})
 \end{equation*}
\end{lem}

\begin{proof}
    Let us fix some ordering of the points of $X_{\infty}=\{y^{(1)},\dots,y^{(n)}\}$. Then, by \cref{prop:Infinite_Chain_to_Hausdorff_Limit}, there exists $n$ sequences $(x^{(1)}_k)_{k\in \N},\dots,(x^{(n)}_k)_{k\in \N}$ such that $x^{(i)}_k\in X_k$ and $x^{(i)}_kR_kx^{(i)}_{k+1}$ for all $k\geq 0$ and all $1\leq i\leq n$ and such that $x^{(i)}_k$ converges to $y^{(i)}$ for all $i$. Now, let us construct a chain between $X_0$ and $X_{\infty}$ as follows. First, observe that since for any $i\not=j$ the sequence $x^{(i)}_k$ and $x^{(j)}_k$ are Cauchy sequences converging to distinct limits, there must be some $k_{i,j}$ such that for all $k'>k_{i,j}$ $x^{(i)}_{k'}\not= x^{(j)}_{k'}$. Take $k_0$ to be the max of all such $k_{i,j}$. 
    We already have a (finite) chain from $X=X_0$ to $X_{k_0}$, together with $n$ distinct elements in $X_{k_0}$, given by the $x^{(i)}_{k_0}$. 
    Let $m\geq n$ be the cardinality of $X_{k_0}$, and let us order its elements $\{x^{(1)}_{k_0},\dots,x^{(n)}_{k_0},\dots,x^{(m)}_{k_0}\}$. 
    We may choose Cauchy sequences $(x^{(i)}_k)_{k\geq k_0}$ for all $n<i\leq m$ which converge to some point $x^{(i)}_{\infty}$ in $X_{\infty}$ and satisfy $x^{(i)}_kR_kx^{(i)}_{k+1}$ for all $k\geq k_0$. 
    For now, assume that for all $k\geq k_0$, $\{x^{(1)}_{k},\dots,x^{(n)}_{k},\dots,x^{(m)}_{k}\}$ is of cardinality $m$. Then, for any $k\geq k_0$, we have
    \begin{equation*}
        \elo(R_k)\geq \sum_{i=1}^m d(x^{(i)}_k,x^{(i)}_{k+1}),
    \end{equation*}
    and in particular $\card(R_k)\geq m$.
    Now, define the surjective map $f\colon X_{k_0}\to X_{\infty}$ sending $x^{(i)}_{k_0}$ to the limit of the sequence $x^{(i)}_{k}$ for all $1\leq i\leq m$. Its graph gives a surjective relation $R'$ and we have a chain $(R_0,\dots,R_{k_0-1},R')$ between $X_0$ and $X_{\infty}$ which satisfies
    \begin{align*}
        \elo(R_0,\dots,R_{k_0-1},R')&=\elo(R_0,\dots,R_{k_0-1})+\omega(m)\sum_{i=1}^m d(x^{(i)}_{k_0},x^{(i)}_{\infty})\\
        &\leq \elo(R_0,\dots,R_{k_0-1})+\omega(m)\sum_{i=1}^m\left(\sum_{k=k_0}^{\infty}d(x^{(i)}_k,x^{(i)}_{k+1})\right)\\
        &\leq \elo(R_0,\dots,R_{k_0-1})+\sum_{k=k_0}^\infty\elo(R_k)\\
        &\leq \elo((R_k)_{k\in \N})
    \end{align*}
    which completes the special case.
    
    Now, if there exists some $k>k_0$ such that $\{x^{(1)}_{k},\dots,x^{(n)}_{k},\dots,x^{(m)}_{k}\}$ is of cardinality strictly less than $m$, let $k_1$ be the smallest such $k$. For all $k_0\leq k<k_1$, define $Y_k=\{x^{(i)}_k\mid 1\leq i\leq m\}$. 
    Consider the partition of $\{1,\dots,m\}$ defined by the equivalence relation $i\sim j\Leftrightarrow x^{(i)}_{k_1}=x^{(j)}_{k_1}$. This partition is non-trivial, by definition of $k_1$ and for any $i<j$ with $i\sim j$, we must have $j>n$ by definition of $k_0$. Now, let $I_1\subset \{1,\dots,m\}$ be the strict subset containing only the least element in each set of the partition. 
    By the previous remark, $\{1,\dots,n\}\subset I_1$.
    Let $Y_{k_1}=\{x^{(i)}_{k_1}\mid i\in I_1\}$. Now, either the sequence of sets $\{x^{(i)}_k\mid i\in I_1\}$ is of constant cardinality, for $k\geq k_1$, or we may iterate the process finitely many times to obtain some $k_{\infty}\in \N$, $\{1,\dots,n\}\subset I_{\infty}\subset \{1,\dots,m\}$, and such that the sequence $\{x^{(i)}_k\mid i\in I_{\infty}\}$ has constant cardinality for all $k\geq k_{\infty}$. 
    In either case, to get a finite chain between $X_0$ and $Y_{k_\infty}$, first consider the chain $(R_0,\dots,R_{k_0-1})$ from $X_0$ to $Y_{k_0}=X_{k_0}$.
    Then consider the sequence of bijections, followed by a single merge between $Y_{k_j}$ and $Y_{k_{j+1}}$ given by the maps sending $x^{(i)}_k$ to $x^{(i)}_{k+1}$ for all $k_j\leq k\leq k_{j+1}-1$ concatenating each of those gives a chain $(P_{k_0},\dots, P_{k_{\infty}-1})$. The key observation is that, by construction $Y_k\subset X_k$ for all $k_0\leq k\leq k_{\infty}$ and $P_k\subset R_k$ for all $k_0\leq k\leq k_{\infty}-1$. 
    Thus, by \cref{lem:included relation shorter}, we get
    \begin{align*}
        \dom(X_0,Y_{k_{\infty}})&\leq \elo(R_0,\dots,R_{k_0-1},P_{k_0},\dots,P_{k_{\infty}-1})\\
        &\leq\elo(R_0,\dots,R_{k_0-1})+\elo(R_{k_0},\dots,R_{k_{\infty}-1})
    \end{align*}
    Finally, the argument of the special case still applies to give a chain $R'$ between $Y_{k_{\infty}}$ and $X_{\infty}$ whose length is bounded by $\sum_{k\geq k_{\infty}}^{\infty}\elo(R_k)$, which completes the proof.
\end{proof}

Now, given an infinite $\omega$-chain between $X_0$ and $X_{\infty}$, if $X_{\infty}$ is not in $\Ran(M)$ we cannot directly compute the distance $\dom(X_0,X_\infty)$. Nevertheless, the distance $\dom$ can formally be extended to the completion $\mathcal{R}_\omega$ of $(\Ran(M),\dom)$, and we will see in the following that it actually makes sense to think of the end-point $X_\infty$ as the limit, for $\dom$, of the infinite $\omega$-chain. 

First we can compute the distance between the starting point $X_0$ and some finite subsets $Y\subset X_\infty$ of the end-point, and show that it minors the length of the chain. 
\begin{lem}
\label{lem:distance to subset of limit}
 Let $((X_k)_{k\in \N},(R_k)_{k\in \N})$ be an infinite $\omega$-chain between $X=X_0\in \Ran(M)$ and $X_{\infty}$. For each $x\in X$, choose a sequence $(u^{(x)}_k)_{k\in\N}$ such that $u^{(x)}_0=x$, $u^{(x)}_k\in X_k$ and $u^{(x)}_k R_k u^{(x)}_{k+1}$ for all $k\geq 0$. Then for any $Y\in \Ran(M)$ such that
 \begin{itemize}
     \item $\{\lim(u^{(x)}_k) \mid x\in X_0\}\subset Y$
     \item $Y\subset X_\infty$
 \end{itemize}
 we have
 \begin{equation*}
     \dom(X,Y)\leq \elo((R_k)_{k\in \N})
 \end{equation*}
\end{lem}

\begin{proof}
  Let us fix some ordering of the points of $Y=\{y^{(1)},\dots,y^{(n)}\}$. Then, by \cref{prop:Infinite_Chain_to_Hausdorff_Limit}, there exists $n$ sequences $(x^{(1)}_k)_{k\in \N},\dots,(x^{(n)}_k)_{k\in \N}$ such that $x^{(i)}_k\in X_k$ and $x^{(i)}_kR_kx^{(i)}_{k+1}$ for all $k\geq 0$ and all $1\leq i\leq n$ and such that $x^{(i)}_k$ converges to $y^{(i)}$ for all $i$. 
Now, let $Y_k=\{x^{(i)}_k\mid 1\leq i\leq n\}\subset X_k$, and $Z_k=\{u^{(x)}_k\mid x\in X_0\}\cup Y_k$. Observe that by construction, $Z_0=X_0$, and $Z_k\subset X_k$. Furthermore, for $k\geq 0$, we define $Q_k\subset Z_k\times Z_{k+1}$ as $Q_k=\{(u^{(x)}_k,u^{(x)}_{k+1})\mid x\in X_0\}\cup \{(x^{(i)}_k,x^{(i)}_{k+1})\mid 1\leq i\leq n\}$. Observe that $Q_k\subset R_k$, by construction, and $Q_k$ is indeed a surjective relation between $Z_k$ and $Z_{k+1}$, thus, by \cref{lem:included relation shorter}, $\elo(Q_k)\leq \elo(R_k)$. Let us show that $(Q_k)_{k\in \N}$ is an infinite chain from $X$ to $Y$. First of all $\elo((Q_k)_{k\in \N})<\infty$ by the previous observation. Let $Z_{\infty}\subset M$ be defined as in \cref{prop:Infinite_Chain_to_Hausdorff_Limit}. For any $k\in \N$, $x^{(i)}_kQ_kx^{(i)}_{k+1}$, for all $i$, and thus, $Y\subset Z_{\infty}$. 
On the other hand, assume that $(v_k)_{k\in \N}$ is a Cauchy sequence in $M$ such that $v_k\in Z_k$ and $v_kQ_kv_{k+1}$ for all $k\in \N$. 
Then since $Z_k$ is of bounded cardinality, there must either be some $x\in X$ or some $1\leq i\leq n$ such that some subsequence of $v_k$ is equal to some subsequence of either $(u_k^{(x)})_{k\in \N}$ or $(x^{(i)}_k)_{k\in \N}$. 
In either case, we can conclude that $\lim(v_k)\in Y$. Thus $Z_{\infty}=Y$, and $(Q_k)_{k\in \N}$ is an infinite chain between $X$ and $Y$. The inclusion $Q_k\subset R_k$, together with \cref{lem:Infinite_Chain_Between_config} gives
\begin{equation*}
    \dom(X,Y)\leq \elo((Q_k)_{k\in \N})\leq\elo((R_k)_{k\in \N})
\end{equation*}  
\end{proof}

We can now show that if two infinite $\omega$-chains have the same end-point then we can minor the sum of their length by the distance of their starting points.

\begin{lem}
\label{lem:distance majoré par longueur chaine infinie}
    Let $X,Y\in \Ran(X)$ and $K\subset M$ be a compact, and $(R_k)_{k\in \N}$ and $(P_k)_{k\in \N}$ be two infinite chains from $X$ to $K$ and from $Y$ to $K$ respectively. Then,
    \begin{equation*}
        \dom(X,Y)\leq \elo((R_k)_{k\in \N})+\elo((P_k)_{k\in \N}).
    \end{equation*}
\end{lem}

\begin{proof}
   Fix orderings on $X$ and $Y$, $X=\{x^{(1)}_0,\dots,x^{(n)}_0\}$ and $Y=\{y^{(1)}_0,\dots,y^{(m)}_0\}$ and choose sequences $x^{(i)}_k\in X_k$ and $y^{(j)}_k\in Y_k$, such that $x^{(i)}_kR_kx^{(i)}_{k+1}$ and $y^{(j)}_kP_ky^{(j)}_{k+1}$, for all $k\geq 0$, $1\leq i\leq n$ and $1\leq j\leq m$. As observed in \cref{prop:Infinite_Chain_to_Hausdorff_Limit} those sequence admit limits in $K$ that we denote $x^{(i)}_{\infty}$ and $y^{(j)}_{\infty}$ respectively. Let $Z=\{x^{(i)}_{\infty}\mid 1\leq i\leq n\}\cup \{y^{(j)}_{\infty}\mid 1\leq j\leq m\}\subset K$. Then, $Z$ together with either infinite chain satisfies the hypothesis of \cref{lem:distance to subset of limit}, thus we  have
   \begin{align*}
       \dom(X,Y)\leq \dom(X,Z)+\dom(Z,Y)\leq \elo((R_k)_{k\in \N})+\elo((P_k)_{k\in \N})
   \end{align*}
\end{proof}

\subsection{Cauchy sequences in $\Ran(M)$ and their limits}
\label{subsec:Cauchy_Ran_omega}
We now use infinite $\omega$-chains to characterize the behavior of Cauchy sequences in $(\Ran(M),\dom)$. Specifically, we use it to give a precise meaning to the idea that cauchy sequences in $(\Ran(M),\dom)$  converge to compact subsets of $M$. This allows us to interpret the completion $\mathcal{R}_\omega$ of $(\Ran(M),\dom)$ as a subset of $\Comp(M)$. Then, we show that if a compact $K\in \Comp(M)$ belongs to all of the $\mathcal{R}_\omega$, it must have a finite cardinality.

\begin{defin}
Let $(X_n)_{n\in \N}$ be a sequence in $\Ran(M)$, and $K\subset M$ be a compact. If $\dH(X_n,K)$ tends to $0$ as $n$ grows, we say that the sequence $X_n$ converges to $K$ for the Hausdorff topology. If, in addition $(X_n)$ is a Cauchy-sequence for $\dom$, for some weight $\omega$, we say that $(X_n)$ \textbf{converges to} $K$ \textbf{for} $\dom$.
\end{defin}

\begin{rem}\label{rem:Cauchy_omega_Converging_To_Compact}
    By \cref{prop:Infinite_Chain_to_Hausdorff_Limit}, if $((X_n)_{n\in \N}),(R_n)_{n\in \N})$ is an infinite $\omega$-chain between $X_0\in \Ran(M)$ and $K\subset M$, then $(X_n)$ converges to $K$ for $\dom$.
    
    Furthermore, if $(X_n)_{n\in \N}$ is a Cauchy sequence for $\dom$, it must be one for $\dH$, and thus it converges to some compact $K\subset M$ by \cref{prop:compact sets completion of ran hausdorff} for both the Hausdorff topology and $\dom$.
\end{rem}

The next proposition shows that the terminology "$(X_n)$ converges to $K$ for $\dom$" is reasonable.

\begin{prop}
\label{prop:same limit cauchy}
Let $(X_n)_{n\in\N}$ and $(Y_n)_{n\in\N}$ be two sequences in $\Ran(M)$ that converge to the same compact subset $K\subset M$ for $\dom$, then the distance $\dom(X_n,Y_n)$ goes to zero as $n$ grows.
\end{prop}

\begin{proof}
    Since the sequences $(X_n)_{n\in \N}$ and $(Y_n)_{n\in \N}$ are Cauchy, we can extract two subsequences $(X_{\varphi(n)})_{n\in \N}$ and $(Y_{\varphi(n)})_{n\in \N}$ such that we have both $\sum \dom(X_{\varphi(n)},X_{\varphi(n+1)})<\infty$ and $\sum\dom(Y_{\varphi(n)},Y_{\varphi(n+1)})<\infty$. Now, by construction, the quantities
    \begin{equation*}
        \sum_{n=k}^{\infty}\dom(X_{\varphi(n)},X_{\varphi(n+1)})\text{ and } \sum_{n=k}^{\infty}\dom(Y_{\varphi(n)},Y_{\varphi(n+1)})
    \end{equation*}
    both tend to $0$ as $k$ goes to infinity. Now, let $k\geq 0$. For any $n\geq k$, we may construct a chain $(R_0,\dots, R_l)$ between $X_{\varphi(n)}$ and $X_{\varphi(n+1)}$,
    whose length can be chosen to be arbitrarily close to $\dom(X_{\varphi(n)},X_{\varphi(n+1)})$.
    Concatenating each of those, for $n\geq k$ gives an infinite $\omega$-chain between $X_{\varphi(k)}$ and $K$ whose length can be assumed to be arbitrarily close to $\sum_{n\geq k}\dom(X_{\varphi(n)},X_{\varphi(n+1)})$.
    Proceeding similarly for $(Y_{\varphi(n)})_{n\geq k}$, we can apply \cref{lem:distance majoré par longueur chaine infinie}, to deduce that, for any $k\geq 0$, 
    \begin{equation*}
        \dom(X_{\varphi(k)},Y_{\varphi(k)})\leq \sum_{n=k}^{\infty}\dom(X_{\varphi(k)},X_{\varphi(k+1)})+\sum_{n=k}^{\infty}\dom(Y_{\varphi(k)},Y_{\varphi(k+1)})
    \end{equation*}
    Since the right hand sides tends to $0$ as $k$ grows, we have that $\dom(X_{\varphi(k)},Y_{\varphi(k)})$ does too. But now, for any $k\geq 0$,
    \begin{equation*}
        \dom(X_k,Y_k)\leq \dom(X_k,X_{\varphi(k)})+\dom(X_{\varphi(k)},Y_{\varphi(k)})+\dom(Y_{\varphi(k)},Y_k).
    \end{equation*}
    Since $(X_n)_{n\in \N}$ and $(Y_n)_{n\in \N}$ are Cauchy sequences, the first and last terms in the right hand side tends to $0$ too as $k$ grows, which gives that $\dom(X_k,Y_k)\to 0$.
\end{proof}

We can identify the limits of Cauchy sequences in $(\Ran(M),\dom)$ with compact subsets of $M$. So that it makes sense to consider the intersection of those, for varying $\omega$, in $\Comp(M)$, i.e. the compact sets $K$ such that, for any weights $\omega$, we can find a sequence in $\Ran$ that converges to $K$. The following couple lemmas show that such compact sets must have a finite cardinality, allowing to identify this intersection, as a set, with $\Ran(M)$.

\begin{lem}\label{lem:Subsets_limits_are_limits}
Let $K\subset M$ be a compact, and $L\subset K$ a compact subspace. If there exists a sequence $(X_n)_{n\in \N}$ which converges to $K$ for $\dom$, then there exists a sequence $(Y_n)_{n\in \N}$ which converges to $L$ for $\dom$. \end{lem}

\begin{proof}
    Once more, up to extracting a subsequence, we may assume that for all $n\geq 0$, $\dom(X_n,X_{n+1})<\frac{1}{2^{n+1}}$. Now, for each $n\geq 0$ there exists a chain $(R^n_0,\dots,R^n_{k_n})$ between $X_n$ and $X_{n+1}$ whose $\omega$ length may be assumed to be smaller than $\dom(X_n,X_{n+1})+\frac{1}{2^{n+1}}$, in particular,  
    \begin{equation*}
        \sum_{n=0}^\infty\sum_{i=0}^{k_n}\elo(R^n_i)<\sum_{n=0}^{\infty} \left(\dom(X_n,X_{n+1})+\frac{1}{2^{n+1}}\right)\leq 2
    \end{equation*}
    so that $(R_0^{0},R_1^0,\dots,R_{k_0}^0,R^1_0,\dots)$ forms an infinite $\omega$-chain, between $X_0$ and $K$. For convenience, we relabel the sequence of relations $(P_n)_{n\in \N}$ and the intermediate configurations $(Z_n)_{n\in \N}$.
    Now by \cref{prop:Infinite_Chain_to_Hausdorff_Limit}, for each $y\in L$, we may chose a sequence $(x_n^{(y)})_{n\in \N}$ where $x^{(y)}_n\in Z_n$, and $x^{(y)}_n P_{n}x^{(y)}_{n+1}$ for all $n\geq 0$, and such that $x_n^{(y)}$ converges to $y$.  For $n\geq 0$, define $Y_n=\{x^{(y)}_n\mid y\in L\}$ and $Q_n=\{(x^{(y)}_n,x^{(y)}_{n+1})\mid y\in L\}$. By construction $(Q_n)_{n\in \N}$ forms an infinite $\omega$-chain, since $Q_n\subset P_n$ for all $n\geq 0$, and by \cref{lem:included relation shorter,prop:Infinite_Chain_to_Hausdorff_Limit}, it forms an infinite $\omega$-chain from $Y_0$ to $Y_{\infty}=\{\lim x_k\mid (x_k)_{k\in\N} , x_k\in Y_k \text{ and } x_kQ_k x_{k+1} \forall k\geq 0\}$. By construction $L\subset Y_{\infty}$. 
    To prove the other inclusion, let $(z_n)_{n\in \N}$ be a sequence such that $z_n\in Y_n$, and $z_nQ_nz_{n+1}$ for all $n\geq 0$. By \cref{prop:Infinite_Chain_to_Hausdorff_Limit}, such a sequence must be Cauchy, thus it converges to some limit $z_{\infty}\in K$. On the other hand, for every $n\in \N$, $z_n=x^{(y)}_n$ for some $y\in L$,
    thus we may define a sequence $y_n$ such that $z_n=x^{(y_n)}_n$. Since $L$ is compact, up to extracting a subsequence, we may assume that $y_n$ converges to some $y_{\infty}\in L$. But now, we have
    \begin{align*}
    d(z_{\infty},y_{\infty})&\leq d(z_\infty,z_n)+d(z_n,y_n)+d(y_n,y_{\infty})\\
    &\leq d(z_{\infty},z_n)+\sum_{k=n}^{\infty}d(x^{(y_n)}_k,x^{(y_n)}_{k+1})+d(y_n,y_{\infty})\\
    &\leq d(z_{\infty},z_n)+\sum_{k=n}^{\infty}\elo(Q_k)+d(y_n,y_{\infty})
    \end{align*}
    Since all three terms on the right hand side converge to zero, we must have $z_{\infty}=y_{\infty}$ and $z_n$ does converge to a limit in $L$. Thus, $(Q_n)_{n\in \N}$ is indeed an infinite $\omega$-chain between $Y_0$ and 
    $L$, and the sequence $(Y_n)$ converges to $L$ for $\dom$.
\end{proof}

\begin{prop}\label{prop:limits_all_omega_finite}
    Let $K\subset M$ be a compact subset. Assume that for all weight $\omega$, there exists some sequence in $\Ran(M)$, $(X_n^{\omega})_{n\in \N}$ such that $(X^{\omega}_n)$ converges to $K$ for $\dom$, then $K$ is finite. 
\end{prop}

\begin{proof}
    Let $K\subset M$ be an infinite compact subspace. Then there exists a sequence of pairwise distinct points in $K$, $(u_n)_{n\in \N}$. Since $K$ is compact, up to extracting a subsequence, we may assume that $(u_n)$ converges to some point $u_{-1}\in K$ so that $\{u_n\mid n\in \N\}\cup \{u_{-1}\}\subset K$. Now, for all $n\geq 1$, define
    \begin{equation*}
        \epsilon_n=\min\{\dm(x,y)\mid x\not=y\in \{u_{-1},u_0,\dots,u_n\}\}
    \end{equation*}
    Then $\epsilon_n$ is non-increasing, and converges to $0$ since $(u_n)$ must be a Cauchy sequence. Let us consider the weight $\omega(n)=\max(1/\epsilon_n,1)$, and show that no sequence in $\Ran(M)$ converges to $K$ for $\dom$.
    By \cref{lem:Subsets_limits_are_limits}, if there exists a sequence converging to $K$ there also exists one converging to $\{u_n\mid -1\leq n\}$ thus we may assume without loss of generality that $K=\{u_n\mid -1\leq n\}$.
    By contradiction, assume that such a sequence exits.
    As we have seen several times already, we may transform it into an infinite $\omega$-chain $((X_n)_{n\in \N},((R_n)_{n\in \N}))$ between $X_0$ and $K$ whose $\omega$-length is less than $1/2$. Now, let $N=\card(X_0)$, and enumerate the elements of $X_0$, $\{x^{(1)},\dots,x^{(N)}\}$,
    consider sequences $x^{(i)}_n$ where $x^{(i)}_0=x^{(i)}$, and for all $n\geq 0$ $x^{(i)}_n\in X_n$ and $x^{(i)}_nR_nx^{(i)}_{n+1}$. By \cref{prop:Infinite_Chain_to_Hausdorff_Limit}, those sequences exist and converge to limits, $x^{(i)}_{\infty}$ in $K$. Now, let $I=\max\{k\mid \exists i,\ x^{(i)}_{\infty}=u_k\}$ and $J=\max\{I,N\}$. Let $Y_{\infty}=\{u_k\mid -1\leq k\leq J\}$. Once more, by \cref{prop:Infinite_Chain_to_Hausdorff_Limit}, for each $-1\leq k\leq J$, there exists a sequence $(y^{(k)}_n)_{n\in \N}$ such that $y^{(k)}_n\in X_n$ and $y^{(k)}_nR_ny^{(k)}_{n+1}$ for all $n\geq 0$ and such that $y^{(k)}_n$ converges to $u_k$. Now for $n\geq 0$, define
    \begin{equation*}
        Y_n=\{y^{(k)}_n\mid -1\leq k\leq J\}\cup\{x^{(i)}_n\mid 1\leq i\leq N\}
    \end{equation*}
    \begin{equation*}
        P_n=\{(y^{(k)}_n,y^{(k)}_{n+1})\mid -1\leq k\leq J\}\cup \{(x^{(i)}_n,x^{(i)}_{n+1})\mid 1\leq i\leq N\}
    \end{equation*}
    Then, $P_n\subset R_n$ for all $n\geq 0$, and thus, by \cref{lem:included relation shorter}, $(P_n)_{n\in \N}$ is an infinite $\omega$-chain. Furthermore, by construction $Y_0=X_0$, and $(Y_n)$ converges to $Y_{\infty}$. Thus, by \cref{lem:Infinite_Chain_Between_config}, we must have
    \begin{equation*}
        \dom(X_0,Y)\leq \elo((P_n)_{n\in \N})\leq \elo((R_n)_{n\in \N})\leq \frac{1}{2}
    \end{equation*}
    On the other hand, $\card(Y)>\card(X_0)$, thus by \cref{lem:minorant distance} $\dom(Y,X_0)\geq \omega(J+2)\epsilon_J\geq 1$, which gives the desired contradiction.
\end{proof}

\subsection{Uniform structure and completeness}
\label{subsec:Completeness_proof}
We are now equipped to show that $(\Ran(M),\tpi)$ can be interpreted as the limit of the completion of the $(\Ran(M),\dom)$ in the category of uniform spaces. 

Throughout, $\Uni$ denotes the category of uniform spaces and uniformly continuous maps \cite{Bourbaki-topology}.

Given a space $U$ and some (pseudo)-metric on $U$, $d$, we will simply write $(U,d)$ for the uniformity induced on $U$ by the (pseudo)-metric $d$. Explicitly, a basis of this uniformity is given by the set $\{\{(x,y)\mid d(x,y)<\epsilon\}\mid\epsilon>0\}$.

In \cref{theo:topologie limite}, we showed that $(\Ran(M),\tpi)$ is the limit of a diagram $\Ran^{-}(M)\colon \W\to\Top$. This diagram takes values in metric spaces and uniformly continuous maps, and can thus be interpreted as a diagram $\Ran^{-}(M)\colon \W\to\Uni$ instead. Doing so equips $(\Ran(M),\tpi)$ with a uniformity, which we now describe explicitly.

\begin{defin}
\label{def:entourages Ran}
    Let $\omega$ be a weight and $\epsilon>0$, define the entourage $V^{\omega}_{\epsilon}$ in $\Ran(M)$, as follows.
        \begin{equation*}
        V^{\omega}_{\epsilon}=\{(X,Y)\mid \dom(X,Y)<\epsilon\}\subset \Ran(M)\times \Ran(M)
    \end{equation*}
    And define the following basis of uniformity on $\Ran(M)$. 
    \begin{equation*}
        \B=\{V^{\omega}_{\epsilon}\mid \omega\in \W, \epsilon>0\}
    \end{equation*}
\end{defin}

\begin{prop}
\label{prop:Uniformity_Ran_finale}
    The basis $\B$ is a basis of the uniformity of the limit of $\Ran^{-}(M)\colon \W\to \Uni$.
\end{prop}

\begin{proof}
    This follows immediately from \cite[chapter II.2.7 proposition 10]{Bourbaki-topology}.
\end{proof}

We now define the completion functor $\mathcal{R}_{-}(M)$ whose value on a weight $\omega$ is the completion of $(\Ran(M),\dom)$. Since the completion of a metric space can be seen as a quotient of its space of Cauchy sequences, it is also useful to define the corresponding functor $\mathcal{C}_{-}(M)$.

\begin{defin}
\label{def:completion functor}
 For any weight $\omega\in\W$ we define the uniform space $(\mathcal{C}_\omega(M),\dom)$ of sequences in $\Ran(M)$ that are Cauchy for $\dom$, with the uniform structure induced by the pseudometric $\dom((X_n),(Y_n))=\lim(\dom(X_n,Y_n))$. We define the functor
\begin{align*}
    \mathcal{C}_{-}(M)\colon\W^{\op}&\to \Uni\\
        \omega&\mapsto (\mathcal{C}_\omega(M),\dom)
\end{align*}

We also define $(\mathcal{R}_\omega(M),\dom)$ the completion of the uniform space $(\Ran(M),\dom)$, and the functor
\begin{align*}
    \mathcal{R}_{-}(M)\colon\W^{\op}&\to \Uni\\
        \omega&\mapsto (\mathcal{R}_\omega(M),\dom)
\end{align*}
\end{defin}

\begin{rem}
    The space $(\mathcal{R}_\omega(M),\dom)$ can be seen as the quotient of $(\mathcal{C}_\omega(M),\dom)$ by the equivalence relation $(X_n)\sim_\omega (Y_n)\Leftrightarrow \dom(X_n,Y_n)\to 0 \Leftrightarrow \dom((X_n)_{n\in \N},(Y_n)_{n\in \N})=0$. The associated quotient map $\pi_\omega\colon(\mathcal{C}_\omega(M),\dom)\to(\mathcal{R}_\omega(M),\dom)$ is uniformly continuous. This construction gives a complete uniform space precisely because $(\Ran(M),\dom)$ was a metric space. For uniform spaces in general, it is not enough to add limits of Cauchy sequences and one needs to consider limits of Cauchy filters. 
\end{rem}

\begin{prop}\label{prop:Smaller_Weights_uniformly_continuous}
    If $\omega,\chi\colon \N^* \to [1,\infty)$ are two weights such that $\omega\preceq \chi$, then the identity $\Id\colon(\Ran(M),\tchi)\to(\Ran(M),\tom)$ induces a uniformly continuous inclusion
    \begin{equation*}
        (\mathcal{C}_\chi(M),\dchi)\hookrightarrow (\mathcal{C}_\omega(M),\dom).
    \end{equation*}
    and a uniformly continuous injection
    \begin{equation*}
        (\mathcal{R}_\chi(M),\dchi)\to (\mathcal{R}_\omega(M),\dom).
    \end{equation*}
    In particular, the functors and $\mathcal{C}_{-}(M)$ and $\mathcal{R}_{-}(M)$ are well-defined.
\end{prop}
\begin{proof}
By \cref{prop:Smaller_Weights_continuous} the identity $\Id\colon(\Ran(M),\tchi)\to(\Ran(M),\tom)$ is Lipschitz so that it indeed induces a uniformly continuous inclusion $(\mathcal{C}_\chi(M),\dchi)\hookrightarrow (\mathcal{C}_\omega(M),\dom)$.

In addition, the composition    
\begin{align*}
    (\Ran(M),\dchi)\xrightarrow{\Id}(\Ran(M),\dom)\hookrightarrow (\mathcal{R}_\omega(M),\dom)
\end{align*}
is a uniformly continuous map from $(\Ran(M),\dchi)$ to a complete uniform space. The universal property of the completion $(\mathcal{R}_\chi(M),\dchi)$ gives a unique uniformly continuous map $f\colon (\mathcal{R}_\chi(M),\dchi)\to (\mathcal{R}_\omega(M),\dom) $ such that the following diagram commutes
\begin{equation*}
    \begin{tikzcd}
    (\Ran(M),\dchi) \arrow{r}{\Id} \arrow[hook]{d} &(\Ran(M),\dom) \arrow[hook]{d}\\
    (\mathcal{R}_\chi(M),\dchi) \arrow{r}{f} &(\mathcal{R}_\omega(M),\dom)
\end{tikzcd}
\end{equation*}
Remark that $f$ also makes the following diagram commutes.
\begin{equation*}
    \begin{tikzcd}
    (\mathcal{C}_\chi(M),\dchi) \arrow[hook]{r} \arrow{d}{\pi_\chi} &(\mathcal{C}_\omega(M),\dom) \arrow{d}{\pi_\omega}\\
    (\mathcal{R}_\chi(M),\dchi) \arrow{r}{f} &(\mathcal{R}_\omega(M),\dom)
\end{tikzcd}
\end{equation*}
To show that $f$ is injective, we thus need to show that if two sequences $(X_n)_{n\in\N},(Y_n)_{n\in\N}\in \mathcal{C}_\chi(M)$ and are such that $\dom(X_n,Y_n)$ goes to zero, then $\dchi(X_n,Y_n)$ also goes to zero. Let us take two such sequences. By \cref{prop:dH leq dom} the two sequences are also Cauchy for $\dH$, and $\dH(X_n,Y_n) \leq \dom(X_n,Y_n)$ goes to zero. By \cref{prop:compact sets completion of ran hausdorff} $X_n$ and $Y_n$ converge to the same compact subset $K\in \Comp(M)$ for $\dH$. And finally, by \cref{prop:same limit cauchy}, $\dchi(X_n,Y_n)$ goes to zero, i.e. $(X_n)_{n\in\N}$ and $(Y_n)_{n\in\N}$ correspond to the same element in $\mathcal{R}_\chi(M)$.
\end{proof}

We can finally combine our previous results to show that $(\Ran(M),\tpi)$, seen as a uniform space, can be seen as the limit of the completion of the $(\Ran(M),\dom)$.

\begin{prop}\label{prop:Final_limit_completion}
    The limit of the functor $\mathcal{R}_{-}(M)\colon \W^{\op}\to \Uni$ is $\Ran(M)$ equipped with the uniformity generated by $\B$.
\end{prop}
\begin{proof}

    Let us write $\mathcal{R}_\infty(M)$ for the limit of the functor $\mathcal{R}_{-}(M)$. 
    First, observe that the compositions
    \begin{equation*}
        (\Ran(M),\B)\to (\Ran(M),\dom)\to \mathcal{R}_{\omega}(M),
    \end{equation*}
    for all $\omega$, produce a unique map to the limit $(\Ran(M),\B)\to \mathcal{R}_{\infty}(M)$ which we will prove to be an isomorphism. 
    
    Let $\omega$ be a weight, and $(X_n)_{n\in \N}$ be a Cauchy sequence for $\dom$. Then, by \cref{rem:Cauchy_omega_Converging_To_Compact}, $X_n$ must converge to some compact $K\subset M$ giving a well-defined map of sets, $\mathcal{C}_{\omega}(M)\to \Comp(M)$. By \cref{prop:same limit cauchy}, this map descends to the quotient to give an injective map $\mathcal{R}_{\omega}(M)\hookrightarrow \Comp(M)$. Since, by \cref{prop:Smaller_Weights_uniformly_continuous}, any map $\chi\to \omega$ in $\W^{\op}$ induces an injective map $\mathcal{R}_{\chi}(M)\to\mathcal{R}_{\omega}(M)$, which commute with the inclusions into $\Comp(M)$, we may freely identify the underlying sets of $\mathcal{R}_{\omega}(M)$ with subsets of $\Comp(M)$, corresponding to compact subsets of $M$ appearing as the limit of some Cauchy sequence for $\dom$. After such an identification, we see that the underlying set of $\mathcal{R}_{\infty}(M)$ is given by
    \begin{equation*}
        \mathcal{R}_{\infty}(M)=\bigcap\limits_{\omega\in \W}\mathcal{R}_{\omega}(M)\subset\Comp(M)
    \end{equation*}
    By \cref{prop:limits_all_omega_finite}, this intersection reduces to finite subset of $M$, or in other words, to $\Ran(M)$, and thus the map $(\Ran(M),\B)\to \mathcal{R}_{\infty}(M)$ is a bijection. Now, by \cite[chapter II.2.7 proposition 10]{Bourbaki-topology}, a basis of the uniformity on $\mathcal{R}_{\infty}(M)$ is given by the set
    \begin{equation*}
        \B'=\{U^{\omega}_{\epsilon}\cap (\Ran(M))^2\mid \omega\in \W, \epsilon >0\}
    \end{equation*}
    where, for any $\omega\in \W$, $\{U^{\omega}_{\epsilon}\mid \epsilon >0\}$ is the basis of uniformity generated by the metric $\dom$. That is, 
    \begin{equation*}
        U^{\omega}_{\epsilon}=\{(X,Y)\mid \dom(X,Y)<\epsilon\}\subset (\mathcal{R}_{\omega}(M))^2.
    \end{equation*}
    But now, observe that the distance $\dom$ on $\mathcal{R}_{\omega}(M)$ restricts to $\dom$ on $\Ran(M)$ (and the abuse in notation does not in fact induce contradictions). In particular, we have $U^{\omega}_{\epsilon}\cap (\Ran(M))^2=V^{\omega}_{\epsilon}$, and thus $\B=\B'$.
\end{proof}

We can now conclude.
\begin{corol}
\label{cor:ran finale complet}
The uniform space $(\Ran(M),\B)$ is a complete uniform space.
\end{corol}
\begin{proof}
    By \cref{prop:Final_limit_completion}, $(\Ran(M),\B)$ is the inverse limit of the functor $\mathcal{R}_{-}(M)\colon \W^{\op}\to \Uni$. Since $\mathcal{R}_{-}(M)$ takes value into complete Hausdorff uniform spaces, and since $\W$ is a directed set, by \cite[Corollary of Prop. 10, Chap II.3.5]{Bourbaki-topology} $(\Ran(M),\B)$ is a complete uniform space.
\end{proof}

\begin{rem}
    There is another common point of view on uniform spaces, for which uniform spaces are spaces equipped with a collection of continuous pseudo-metrics, which jointly induce the topology. Given such a collection of pseudo-metrics on a space $U$, $(d_i)_{i\in I}$, a basis of uniformity is then given by the set
    \begin{equation*}
        \{\{(x,y)\mid d_i(x,y)<\epsilon\}\mid i\in I, \epsilon>0\}.
    \end{equation*}
    It can be shown (see for example the proof of \cite[theorem 16 p 35]{Guenard-Lelievre-uniform-spaces}) that any uniformity on a Hausdorff space can in fact be specified through a family of pseudo-metrics. In particular, the fact that the uniformity we obtain on $(\Ran(M),\tpi)$ is induced by a collection of (pseudo)-metrics is not surprising, though the fact that those pseudo-metrics all happen to be metrics deserves mention.
\end{rem}

\section{Conicality}
\label{sec:conicity}
This section is dedicated to showing that if $M$ is a Riemannian manifold then $\Ran(M)$ with the weighted topology is conically stratified. Since this is a local property on $\Ran(M)$, the section is written in the case where $M$ is a Euclidean space, and the generalization to manifolds is done only at the very end (see \cref{Corol:Ran_Riemannian_Conical}). We first recall the definition of a conically stratified space, and give a brief summary of the proof. Sections~\ref{subsec:locality_Conical} and~\ref{subsec:scaling} are devoted to introducing technical tools. Specifically, we define a generalization of homotheties to $\Ran(M)$ that we call $\Shift$ and show that it is continuous. This function is in turn used in section~\ref{subsec:conicity_Euclidean} to construct a stratified homeomorphism between neighborhoods in $\Ran(M)$ and a stratified cone which proves that $\Ran(M)$ is a conically stratified space when $M$ is either a Euclidean space (\cref{Corol:Ran_Euclidean_Conical}) or a Riemannian manifold (\cref{Corol:Ran_Riemannian_Conical}). The very last subsection shows that while $\Ran(M)$ equipped with the final topology is not conically stratified, it satisfies a close analogue to this property.

\subsection{Preliminaries}
\label{subsec:Prelim_conical}
We start by recalling the definition of a conically stratified space.

\begin{defin}
\label{def:Stratified_Cone}
    Let $\varphi\colon X\to P$ be a stratified space. The \textbf{cone} on $X$, denoted $c(X)$ is a space stratified over $P\cup\{-\infty\}$, where $-\infty < p $ for any $p\in P$. Its underlying space is defined as $X\times [0,1)/_{X\times\{0\}}$ equipped with the \textbf{teardrop} topology defined from the following basis $\{ \pi(U)\mid U\in \mathrm{open}(X\times (0,1) )  \} \cup \{ \pi(X\times [0,\epsilon)) \mid 0<\epsilon < 1 \}$, where $\pi\colon X\times [0,1)\to c(X)$ is the quotient map. The stratification on the cone is defined as
    \begin{align*}
        c(\varphi)\colon c(X)&\to P\cup\{-\infty\}\\
 [(x,t)]&\mapsto 
 \begin{cases}
 \varphi(x)  \quad  &\text{if } t>0\\
 -\infty &\text{if } t=0
 \end{cases}        
    \end{align*}
\end{defin}

\begin{defin}
    Let $P$ be a poset, and $p\in P$. We define the following subsets of $P$:
    \begin{align*}
        &P_{>p}=\{q\in P \mid q > p \} \\
        &P_{\geq p}=\{q\in P \mid q \geq p \}
    \end{align*}
    Remark that $P_{>p}\cup \{-\infty\}$ is isomorphic to $P_{\geq p}$.
\end{defin}

\begin{defin}
    Let $\varphi\colon X\to P$ be a stratified space. We say that it is \textbf{conically stratified} if, for any $x\in X$, $p=\varphi(x)$, there exists
\begin{itemize}
    \item a neighborhood $U\subset X$ of $x$. 
    \item an open neighborhood of $x$ in $\varphi^{-1}(p)$,  $E\subset \varphi^{-1}(p)$
    \item a stratified space $\varphi_L\colon L \to P_{>p}$
    \item a stratified homeomorphism $E\times c(L)\xrightarrow{\simeq} U$, i.e. an homeomorphism such that the following diagram commutes
\end{itemize}
    \begin{equation*}
        \begin{tikzcd}
            E\times c(L)
            \arrow{r}{\simeq}
            \arrow{d}{c(\varphi_L)}
            & U
            \arrow{d}{\varphi_{|U}}
            \\
            P_{\geq p}
            \arrow[hook]{r}
            & P
        \end{tikzcd}
    \end{equation*}
\end{defin}

\begin{rem}
    In this section, we restrict to the case where $M$ is a manifold. Specifically, $M$ will be a Euclidean space for most of the section, and finally, a Riemannian manifold in \cref{Corol:Ran_Riemannian_Conical}. In both cases, $\Ran_{=n}(M)$ is a manifold of dimension $n$ times that of $M$ and so each point admits a neighborhood isomorphic to $\R^k$ for some $k$. We aim to prove that in both those cases $\Ran(M)$, equipped with the weighted topology $\tom$, is a conically stratified space. This amounts to exhibiting, for every points $X\in \Ran(M)$ specific neighborhoods in both $\Ran(M)$ and $\Ran_{=n}(M)$, where $\card(X)=n$. By the previous observation, we may think of $E$ as a finite-dimensional vector space, instead of as an explicit neighborhood in $\Ran_{=n}(M)$. Furthermore, in the case where $M=\R^p$, this vector space will be homeomorphic to $(\R^p)^n$, but equipped with a specific norm $\Nsum{-}$ (see \cref{subsec:locality_Conical}), instead of the usual one, which means that $E$ will be homeomorphic to a Euclidean space, though not isometric to one. We will still use the terminology of Euclidean space in the following.
\end{rem}

\begin{rem}
    Readers familiar with pseudo-manifolds may wonder, in view of the previous remark, why we did not state the definition of a pseudo-manifold instead, since we will be considering conically stratified spaces with manifold strata. Observe that we aim to prove that $\Ran(M)\to \N$ is conically stratified, and notably $\N$ does not have a maximal element. In particular, there is no regular stratum in $\Ran(M)$ whose closure would recover all of $\Ran(M)$ and thus it is not, in fact, a pseudo-manifold. However, proving that $\Ran(M)$ is conically stratified implies that so is $\Ran_{\leq n}(M)$ for all $n\in \N$. From which we can immediately recover the fact that the latter is in fact a pseudo-manifold (see for example \cite[Definition 2.4.1]{FriedmanBook}).
\end{rem}

\subsection{Sketch of the proof}
\label{subsec:sketch}
Since this section is long and technical, we give here a brief summary of the proof of the main result, with a few drawings, for the case where $M=\R^p$ is a Euclidean space.

We want to construct, for any $X\in\Ran(M)$, $n=\card(X)$
\begin{enumerate}
    \item\label{it:neigborhood} an open neighborhood $U\subset\Ran(M)$ of $X$. In fact, by \cref{cor:neighborhood_min_cardinal}, we may chose $U$ such that $U\subset \Ran_{\geq n}(M)$.
    \item\label{it:cone} a cone $c(L)$ with the adequate stratification
    \item\label{it:euclidean} a Euclidean space $E$ of dimension $\dim(M^n)$
    \item\label{it:homeo} a stratified homeomorphism $E\times c(L)\xrightarrow{\simeq} U$.
\end{enumerate} 
\begin{figure}[htb]
    \centering
    \begin{tikzpicture}[scale=1.4]
    \tikzset{triangle/.style={shape=isosceles triangle,shape border rotate=#1,isosceles triangle stretches=true,minimum width=9pt,minimum height=7.8pt,inner sep=0pt,fill=blue}}
    \tikzset{empty triangle/.style={shape=isosceles triangle,shape border rotate=#1,isosceles triangle stretches=true,minimum width=7pt,minimum height=6.06pt,inner sep=0pt,draw=blue, thick}}
    \node (x1) at (0,0) [circle,fill=red,inner sep=3pt] {};
    \node at (x1.south) [red, right] {$x^{(1)}$};
    \node (x2) at ({2*sqrt(2)},{2*sqrt(2)}) [circle,fill=red,inner sep=3pt] {};
    \node at (x2.south) [red, right] {$x^{(2)}$};
    \node (x3) at (4.95,0.707) [circle,fill=red,inner sep=3pt] {};
    \node at (x3.south) [red, right] {$x^{(3)}$};
    \draw[thick,dotted](x1) circle (1.5);
    \draw[thick,dotted](x2) circle (1.5);
    \draw[thick,dotted](x3) circle (1.5);
    \node (y11) at ($(x1)+(0.5,0.5)$) [triangle=90] {};
    \node (y12) at ($(x1)+(-0.2,0.4)$) [triangle=90] {};
    \node (y13) at ($(x1)+(0.1,-0.6)$) [triangle=90] {};
    \node (y21) at ($(x2)+(-0.4,0.4)$) [triangle=180] {};
    \node (y31) at ($(x3)+(0.05,0.6)$) [triangle=-90] {};
    \node (y32) at ($(x3)+(0.3,0.3)$) [triangle=-90] {};
    \node at (y12) [below left, blue] {$Y^{(1)}$};
    \node at (y21) [above right, blue] {$Y^{(2)}$};
    \node at (y31) [above left, blue] {$Y^{(3)}$};
    \node (z11) at ($(x1)+2*(y11)-2*(x1)$) [rectangle,inner sep=3pt,fill=violet] {};
    \draw[->,thick,dotted] (y11) -- (z11);
    \node (z12) at ($(x1)+2*(y12)-2*(x1)$) [rectangle,inner sep=3pt,fill=violet] {};
    \draw[->,thick,dotted] (y12) -- (z12);
    \node (z13) at ($(x1)+2*(y13)-2*(x1)$) [rectangle,inner sep=3pt,fill=violet] {};
    \draw[->,thick,dotted] (y13) -- (z13);
    \node (z21) at ($(x2)+2*(y21)-2*(x2)$) [rectangle,inner sep=3pt,fill=violet] {};
    \draw[->,thick,dotted] (y21) -- (z21);
    \node (z31) at ($(x3)+2*(y31)-2*(x3)$) [rectangle,inner sep=3pt,fill=violet] {};
    \draw[->,thick,dotted] (y31) -- (z31);
    \node (z32) at ($(x3)+2*(y32)-2*(x3)$) [rectangle,inner sep=3pt,fill=violet] {};
    \draw[->,thick,dotted] (y32) -- (z32);
    \end{tikzpicture}
    \caption{Sketch of the clustering and scaling functions for $M=\R^2$ equipped with the Euclidean distance. The reference configuration $X=\{x^{(1)},x^{(2)},x^{(3)}\}$ is displayed as red dots \textcolor{red}{\faCircle} for a cardinality $n=3$. The Hausdorff sphere $\BH(X,\merg(X))$ is sketched as dotted circles around the points of $X$. A local configuration $Y\in\BH(X,\merg(X))$ is displayed as blue triangles $\textcolor{blue}{\blacktriangle}$, the orientation of the triangles is different for each of the three clusters $(Y^{(1)},Y^{(2)},Y^{(3)})=\ClustX(Y)$, and the scaled configuration $Z=\ScaleX(2,Y)$ is displayed as violet squares $\textcolor{violet}{\blacksquare}$.}
    \label{fig:scaleX}
\end{figure}
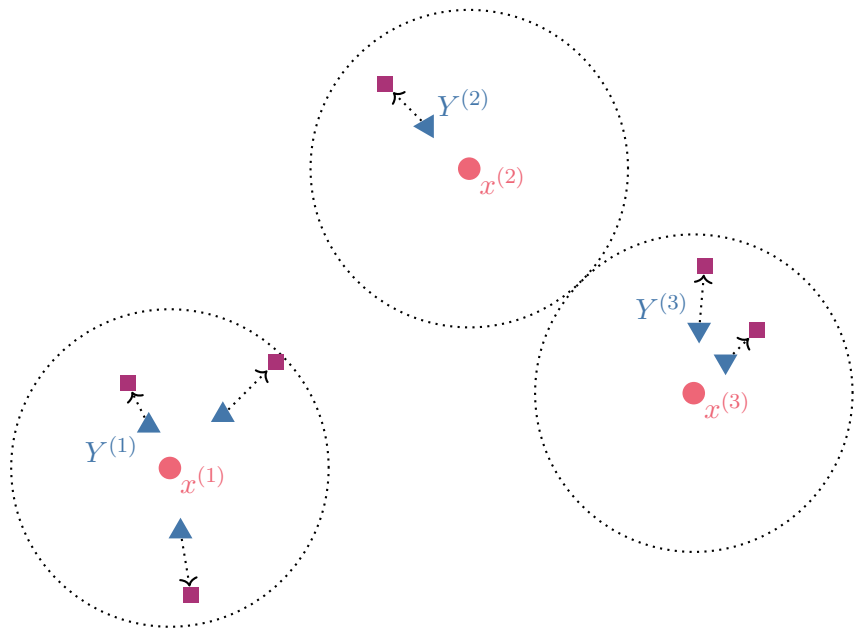

The construction revolves around the central idea of \textbf{locality} which we introduced in \cref{subsec:locality} and refine in \cref{subsec:locality_Conical}: when we work in a sufficiently small neighborhood of $X$, then every local configuration $Y$ is composed of \textbf{clusters} of points localized around the points in $X$, and any geodesic chain between two sufficiently close local configurations is also local, i.e. it does not mix the clusters.
We introduce in \cref{def:clustering} a clustering function $\ClustX$ that gives the clusters of a local configuration. We can then apply some affine transformations to those clusters, using the corresponding points in $X$ as centers. 

Using this idea, we introduce in \cref{def:scaling} the scaling of local configurations around a fixed configuration $X$. This is sketched on \cref{fig:scaleX}, where we fixed a configuration $X$ of cardinality $n=3$, and show a local configuration $Y$, its three clusters $(Y^{(1)},Y^{(2)},Y^{(3)})=\ClustX(Y)$ and the rescaled configuration $Z=\ScaleX(2,Y)$. Observe that such a $Z$ may not always be local anymore and could potentially have a cardinality smaller than that of $Y$, if the scaling factor is too big. This is why we define the notion of bijective shift system in \cref{def:shift system}, for which we prove that the scaling function does not mix clusters (see \cref{prop:shifting equivalence}).

We then show in \cref{prop:homeo cone to w-ball} that this scaling function directly gives a homeomorphism from the cone on a sphere $\Som(X,\epsilon)$ in $(\Ran(M),\tom)$ to a ball $\Bom(X,\epsilon)$ in $(\Ran(M),\tom)$. This homeomorphism is not stratified however, since the cone contains a one point stratum, its vertex, whereas the ball $\Bom(X,\epsilon)$ in $\Ran(M)$ contains many points of the same cardinality as $X$, its center.

 We thus define a subspace of the ball, $C_{\underline{x}}$, whose elements are such that each of their clusters is centered around a point of $X$. This naturally excludes all configurations of cardinality exactly $\card(X)$ except $X$ itself.  Since the teardrop topology allows to identify the cone on a subspace with a subspace of the cone, we then obtain a stratified homeomorphism between the cone on $C_{\underline{x}}\cap \Som(X,\epsilon)$ and $C_{\underline{x}}$ (\cref{prop:strat homeo radial cone ball W})
  Note that, for the above line of reasoning to work, one needs the definition of the "center" of a cluster to be continuous and thus, one cannot simply use the iso-barycenter. We therefore define in \cref{rem:why midpoint,def:midpoint} a more suitable notion of center that  we call \textbf{midpoint}.

\begin{figure}[hp]
    \centering
    \begin{tikzpicture}[scale=1.35]
    \tikzset{triangle/.style={shape=isosceles triangle,shape border rotate=#1,isosceles triangle stretches=true,minimum width=9pt,minimum height=7.8pt,inner sep=0pt,fill=blue}}
    \tikzset{empty triangle/.style={shape=isosceles triangle,shape border rotate=#1,isosceles triangle stretches=true,minimum width=7pt,minimum height=6.06pt,inner sep=0pt,draw=blue, thick}}
    \node (x1) at (0,0) [circle,fill=red,inner sep=3pt] {};
    \node at (x1.north) [red, above right] {$x^{(1)}$};
    \node (x2) at ({4*cos(37)},{4*sin(37)}) [circle,fill=red,inner sep=3pt] {};
    \node at (x2.south) [red, right] {$x^{(2)}$};
    \draw[thick, dotted](x1) circle (1);
    \draw[thick, dotted](x2) circle (1);

    \node (x1up) at ({-2*cos(37)-1.5},6.1) [circle,fill=red,inner sep=3pt] {};
    \node at (x1up.north) [red, above right] {$x^{(1)}$};
    \node (x2up) at ($(x1up)+ ({4*cos(37)},{4*sin(37)})$) [circle,fill=red,inner sep=3pt] {};
    \node at (x2up.south) [red, right] {$x^{(2)}$};
    \draw[thick, dotted](x1up) circle (1);
    \draw[thick, dotted](x2up) circle (1);
    \node (y1) at ($(x1up)+(0.1,-0.4)$) [rectangle,fill=violet,inner sep=3pt] {};
    \node (y2) at ($(x2up)+(-0.4,0.4)$) [rectangle,fill=violet,inner sep=3pt] {};
    \node at (y1) [violet,below right] {$y^{(1)}$};
    \node at (y2) [violet,above right] {$y^{(2)}$};

    \node (g1) at ($(x1up)+ ({2.75+4*cos(37)},{2*sin(37)})$) [rectangle,fill=violet,inner sep=3pt] {};
    \node at (g1.south) [violet, right] {$g^{(1)}$};
    \node (g2) at ($(g1)+(3,0)$) [rectangle,fill=violet,inner sep=3pt] {};
    \node at (g2.south) [violet, right] {$g^{(2)}$};
    \draw[thick, dotted](g1) circle (0.75);
    \draw[thick, dotted](g2) circle (0.75);
    \node (z11) at ($(g1)+(0.35,0.55)$) [triangle=90] {};
    \node (z12) at ($(g1)+(-0.35,0.35)$)[triangle=90] {};
    \node (z13) at ($(g1)+(-0.05,-0.55)$)[triangle=90] {};
    \node (z21) at ($(g2)+(-0.2,-0.2)$)[triangle=90] {};
    \node (z22) at ($(g2)+(0.2,0.2)$)[triangle=90] {};
    
    \node at (z11) [above right, blue] {$Z^{(1)}$};
    \node at (z22) [above left, blue] {$Z^{(2)}$};

    \node (y11) at ($(x1)+(y1)-(x1up)+0.5*(z11)-0.5*(g1)$) [triangle=90] {};
    \node (y12) at ($(x1)+(y1)-(x1up)+0.5*(z12)-0.5*(g1)$) [triangle=90] {};
    \node (y13) at ($(x1)+(y1)-(x1up)+0.5*(z13)-0.5*(g1)$) [triangle=90] {};
    \node (y21) at ($(x2)+(y2)-(x2up)+0.5*(z21)-0.5*(g2)$) [triangle=90] {};
    \node (y22) at ($(x2)+(y2)-(x2up)+0.5*(z22)-0.5*(g2)$) [triangle=90] {};

    \draw [dotted,rounded corners] ($(x1up)+(-1.25,-1.25)$) -| ($(x2up)+(1.25,1.25)$) node[pos=0.75] (ul) {} -- ++({-4*cos(37)-2.5},0) node [midway, above] {$E=\Bsum((x^{(1)},x^{(2)}),\theta)$} -- cycle;
    \draw [dotted,rounded corners] ($(g1)+(-1,{-1.25-2*sin(37)})$) -| ($(g2)+(1,{1.25+2*sin(37)})$) -- ++({-5},0) node [midway, above] {$V=\Bom(G,1)\cap C_{\underline{g}}\simeq c(L)$} -- cycle node[midway] (ur){};
    \draw [dotted,rounded corners] ($(x1)+(-1.25,-1.25)$) -| ($(x2)+(1.25,1.25)$) -- ++({-4*cos(37)-2.5},0) node (U) [midway, above] {$U=f(E\times V)$} -- cycle;
    
   \node at ($0.5*(ul)+0.5*(ur)$) {$^\times$};
   \draw [->] ($(U)+(0,0.8)$) -- (U.north) node[midway,right] (f) {};
   \node at (f) [rotate=-90,right=-4pt,anchor=south] {$\simeq$};
   \node at (f.east) [right] {$f$};
    \end{tikzpicture}
    \caption{Sketch of the homeomorphism $f$ defined in \cref{prop:homeo strat} for $M=\R^2$. The reference configuration $X=\{x^{(1)},x^{(2)}\}$ is displayed as red dots \textcolor{red}{\faCircle}. The Euclidean $E\subset M^2$ is displayed at the top left. 
    A point $(y^{(1)},y^{(2)})\in E$ is displayed as violet squares $\textcolor{violet}{\blacksquare}$. 
    The space $V=\Bom(G,1)\cap C_{\underline{g}}$ is homeomorphic to a stratified cone $c(L)$ and is displayed at the top right, where the configuration $G=\{g^{(1)},g^{(2)}\}$ is displayed as violet squares $\textcolor{violet}{\blacksquare}$, and a configuration $Z=Z^{(1)}\cup Z^{(2)}\in V$ is displayed as blue triangles {$\textcolor{blue}{\blacktriangle}$}. 
    The neighborhood $X\in U\subset \Ran(M)$ is displayed at the bottom. The image $f(\underline{y},Z)$ is displayed as blue triangles $\textcolor{blue}{\blacktriangle}$, its two clusters are centered around $(y^{(1)},y^{(2)})$.}
    \label{fig:homeo cone strat}
\end{figure}
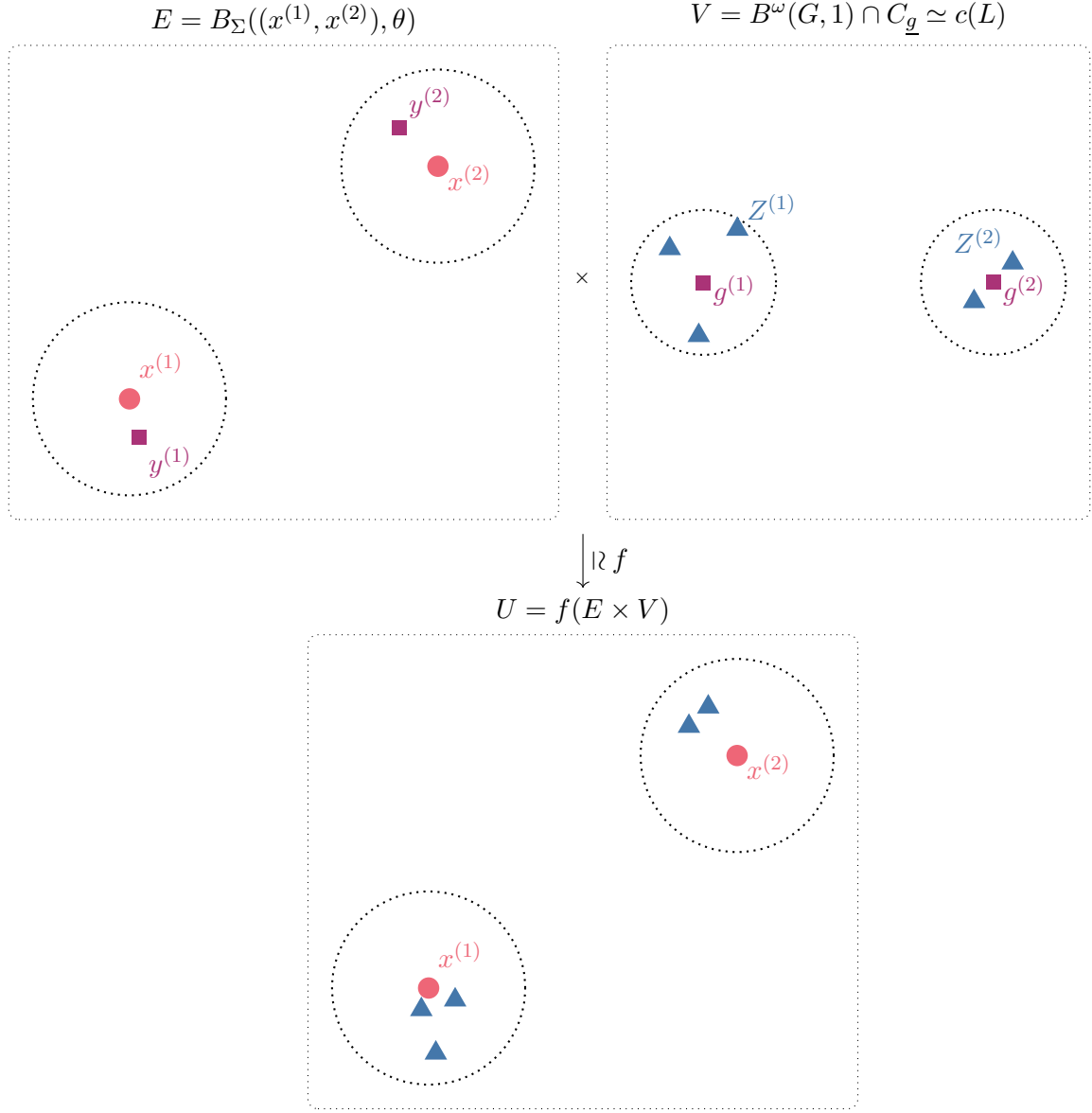

At this stage we are still not finished because  $C_{\underline{x}}$  is not a neighborhood of $X$ in $\Ran(M)$ (in particular, it is not an open subset of $\Ran(M)$). However, we proceed with the proof as follows. 
\begin{itemize}
    \item We introduce a generic configuration of cardinality $n$,  $G\subset \R^p$ so that the subspace $C_{\underline{g}}\subset \Bom(G,\epsilon)$ is a cone, by the aforementioned Proposition. This is the needed cone. \cref{it:cone} 
    \item  $E$ will be a ball in $M^n$ around $\underline{x}$, of adequately small radius. \cref{it:euclidean}
    \item  We define the needed homeomorphism in \cref{prop:homeo strat} in two steps:
    \begin{itemize}
        \item We introduce a map $f\colon E\times C_{\underline{g}}\to \Ran(M)$, show that it is open, and define $U$ to be its image \cref{it:neigborhood}.
        \item We show that the (co)-restriction $f\colon E\times C_{\underline{g}}\to U$ is a stratified homeomorphism \cref{it:homeo}.
    \end{itemize}
\end{itemize}
 The map $f^{-1}$ can be understood as returning two things. The midpoints of a collection of clusters, giving a point in $M^n$ close to $\underline{x}$, i.e. an element of $E$. And the local positions of the clusters around their midpoints, which can be encoded generically as positions around the configuration $G$, whose midpoints are $G$, i.e. elements of $C_{\underline{g}}$. This is sketched in \cref{fig:homeo cone strat}.

\subsection{Locality}
\label{subsec:locality_Conical}
From now on, we fix $p>0$ an integer, and set $M=\R^p$, equipped with the Euclidean distance:
\begin{align*}
    d\colon M\times M & \to \R_+\\
    (x,y)&\mapsto \norm{x-y} .
\end{align*}
For any $n\in\N$ we equip $M^n$ with the norm
\begin{align*}
    \Nsum{\cdot}\colon M^n & \to \R_+\\
    \underline{x}=(x^{(1)},\dots,x^{(n)})&\mapsto \sum_{i=1}^n \norm{x^{(i)}},
\end{align*}
so that the distance $\dsum$ on $M^n$ is inherited from $\Nsum{\cdot}$ as
\begin{align*}
    \dsum\colon M^n\times M^n & \to \R_+\\
    (\underline{x},\underline{y})&\mapsto \Nsum{\underline{x}-\underline{y}} .
\end{align*}

We fix $\omega$ a weight. Throughout this section $\Ran(M)$ will always be equipped with the $\tom$ topology, unless explicitly stated. We will denote continuity with respect to the topology $\tom$ as $\omega$-continuity for short. Note that, when there is no ambiguity, we will also call $\omega$-continuous any continuous maps defined on a product of space, in which every occurrence of a Ran space is equipped with the topology $\tom$.

In this section and the next, we set-up the technical results that will be needed to prove both \cref{Corol:Ran_Riemannian_Conical} and \cref{Ran faible conique}. Since the latter deals with the final topology, we need to be slightly more careful with our technical lemmas. In particular, in this section, we prove a slightly stronger version of \cref{lem:boules omega locales}, to extend its result from the $\omega$ ball onto the Hausdorff ball. This is not strictly needed to prove \cref{Corol:Ran_Riemannian_Conical}, but in order to investigate the final topology, we need to work on a neighborhood that is common to all weighted topologies. This role will be played by the Hausdorff ball.

\begin{lem}
\label{lem:boule hausdorff locale}
    Let $X\in \Ran(M)$. For any $r<\merg(X)$, there exists $\epsilon\geq \merg(X)-r >0$ such that for all $Y,Z\in \BH(X,r)$, if $\dom(Y,Z)<\epsilon$, then any geodesic chain from $Y$ to $Z$ is $\BH(X,r)$-local.
\end{lem}

\begin{proof}
    Let $\epsilon=\merg(X)-r$. Since $Y$ and $Z$ are in the open ball $\BH(X,r)$, we may find $r'<r$ such that $Y$ and $Z$ are contained in the closed ball $\overline{B}_H(X,r')$. We will prove that any geodesic chain between $Y$ and $Z$ is contained in $\overline{B}_H(X,r')$. Consider a minimal combinatorial chain $((Y_0,\dots,Y_n),(R_0,\dots,R_{n-1}))$ between $Y$ and $Z$. By \cref{theo:MBS}, we may assume that it is a sequence of simple merges, a bijection and simple splittings. Let $R\subset Y\times Y'$ be the composition of all merges, $B\subset Y'\times Z'$ be the bijection and $S\subset Z'\times Z$ the composition of all splittings, such that $(R,B,S)$ is still a combinatorial chain between $Y$ and $Z$, and consider the following equivalence relation on $Y$. 
    \begin{equation*}
        y\sim y'\Leftrightarrow \exists z\in Y', \ yRz\text{ and } y'Rz
    \end{equation*}
    Now, let $y,y'\in Y$. There must be $x,x'\in X$ such that $y\in \overline{B}(x,r')$ and $y'\in \overline{B}(x',r')$. Now observe that if $y\sim y'$, we must have $x=x'$. Indeed, if it were not the case, we would have $d(y,y')\geq d(x,x')-2r'\geq 2\merg(X)-2r>\epsilon$, which would imply $\dom(Y,Y')>\epsilon$. However, by construction $\dom(Y,Z)=\dom(Y,Y')+\dom(Y',Z')+\dom(Z',Z)<\epsilon$. A symmetric reasoning on $S$ gives that no point in $Z'$ splits into 2 different clusters in $Z$, and we may show similarly that $B$ exchange points in $Y'$ and $Z'$ living in the neighborhood of the same points of $X$. Thus, we may write a partition
    \begin{equation*}
        R_i=\coprod_{x\in X}R_i^x\subset \coprod Y_i ^x\times Y_{i+1}^x
    \end{equation*}
    Now, for $x\in X$, consider the map $\varphi_x\colon M\to \overline{B}(x,r')$, defined by
    \begin{align*}
        \varphi_x\colon M&\to \overline{B}(x,r')\\
        y&\mapsto \left\{
        \begin{array}{cc}
           y  & \text{ if $y\in \overline{B}(x,r')$} \\
           x+r'^2\frac{y-x}{d(x,y)^2}  & \text{ otherwise} 
        \end{array}\right.
    \end{align*}
    By \cref{lem:Inversion_phi}, the application $\varphi_x$ contracts distances, and takes values in $\overline{B}(x,r')$, furthermore, it leaves $\overline{B}(x,r')$ unchanged. Thus, we may define for $0\leq i\leq n$
    \begin{equation*}
        Y'_i=\coprod_{x\in X}\varphi_x(Y_i^x)
    \end{equation*}
    \begin{equation*}
        R'_i=\coprod_{x\in X}\varphi_x(R_i^x)\subset Y'_i\times Y'_{i+1}
    \end{equation*}
    Observe that $Y'_0=Y_0=Y$ and $Y'_n=Y_n=Z$, since $Y,Z\in \overline{B}_H(X,r')$, and each $\varphi_x$ restrict to the identity on each $\overline{B}(x,r')$, for $x\in X$. Furthermore, the new chain $(R'_0,\dots,R'_{n-1})$ is $\overline{B}_H(X,r')$-local, and since $\varphi_x$ contracts distances, it is shorter than $(R_0,\dots,R_{n-1})$. However, since we assumed that $(R_0,\dots,R_{n-1})$ was a geodesic chain, we must have equality between the length of the two chains, which is only possible if $(R_0,\dots,R_{n-1})$ was already $\overline{B}_H(X,r')$-local. 
\end{proof}

\begin{lem}\label{lem:Inversion_phi}
    Let $x\in M$ and $r>0$, the map
    \begin{align*}
        \varphi_x\colon M&\to M\\
        y&\mapsto \left\{
        \begin{array}{cc}
             y& \text{ if $d(x,y)\leq r$} \\
             x+r^2\frac{y-x}{d(x,y)^2}& \text{if $d(x,y)> r$}
        \end{array}
        \right.
    \end{align*}
    satisfies
    \begin{itemize}
        \item $\varphi_x$ restricts to the identity on the closed ball $\overline{B}(x,r)$.
        \item $\varphi_x$ takes values in the closed ball $\overline{B}(x,r)$
        \item for all $y,z\in M$, $d(y,z)\geq d(\varphi_x(y),\varphi_x(z))$. With equality if and only if $y,z\in \overline{B}(x,r)$
    \end{itemize}
\end{lem}

\begin{proof}
    Without loss of generality, we may assume that $x=0$, and consider the sphere $S(0,r)=\{v\in M\mid d(0,v)=r\}$. We will simply write $\varphi$ for $\varphi_0$. Then, any point $y\in M\setminus \{0\}$ can be written uniquely as $y=\lambda v$, with $\lambda\in \R_+$ and $v\in S(0,r)$. The map $\varphi$ then sends $\lambda v$ to either $\lambda v$, if $\lambda\leq 1$ or $\frac{1}{\lambda}v$ if $\lambda \geq 1$. This shows the first two points. For the last part, it is clear that $\varphi_0$ preserves distances in the closed ball, since it preserves all points. Now, let $y\in \overline{B}(0,r)$ and $z\in M\setminus\overline{B}(0,r)$, and consider $y'$ the orthogonal projection of $y$ on the line $(0,z)$. By construction, the triangles $(y,y',z)$ and $(y,y',\varphi(z))$ are both rectangle at $y'$ thus it is enough to show that $d(y',\varphi(z))<d(y',z)$. We have, by construction $z=\lambda v$ for some $\lambda>1$ and $v\in S(0,r)$, $\varphi(z)=\frac{1}{\lambda}v$ and $y'=\lambda'v$, with $\lambda'\in [-1,1]$. If $\lambda'\leq \frac{1}{\lambda}$, $\varphi(z)$ is in between $y'$ and $z$, and thus there is nothing to prove. Otherwise, We compute
    \begin{align*}
        d(z,y')-d(\varphi(z),y')&=\left(\lambda-\lambda'\right)r-\left(\lambda'-\frac{1}{\lambda}\right)r\\
        &=\frac{r}{\lambda}\left(\lambda^2-2\lambda\lambda'+1\right)\\
        &\geq \frac{r}{\lambda}\left(\lambda^2-2\lambda+1\right)\\
        &\geq \frac{r}{\lambda}(\lambda-1)^2>0
    \end{align*}
    which concludes this case.

    If $y,z\in M\setminus\overline{B}(0,r)$, we first consider two particular cases.
    \begin{itemize}
        \item if $y=\lambda_1v$ and $z=\lambda_2v$, with $v\in S(0,r)$, $\lambda_1>1$ and $|\lambda_2|\geq\lambda_1>1$. (This covers all cases where $0,y$ and $z$ are collinear, up to exchanging $y$ and $z$ and replacing $v$ by $-v$). Then 
        \begin{align*}
            d(\varphi(y),\varphi(z))&=\left|\frac{r}{\lambda_1}-\frac{r}{\lambda_2}\right|\\
            &=\frac{|r(\lambda_2-\lambda_1)|}{|\lambda_1\lambda_2|}\\
            &\leq r|\lambda_2-\lambda_1|=d(y,z)
        \end{align*}
        Where the equality is obtained only if $y=z$.
        \item if $y=\lambda v_1$, and $z=\lambda v_2$, with $\lambda>1$ and $v_1,v_2\in S(0,r)$. Then, $\varphi(z)=\frac{z}{\lambda^2}$ and $\varphi(y)=\frac{y}{\lambda^2}$ and thus, $\varphi(z)-\varphi(y)=\frac{z-y}{\lambda^2}$, which implies that $d(\varphi(z),\varphi(y))\leq d(y,z)$ with equality if and only if $y=z$.
    \end{itemize}
    Now, for the general case, consider $y=\lambda_1v_1$ and $z=\lambda_2v_2$. We may assume, $\lambda_1\not=\lambda_2$ and that they are both positive, and that $v_1\not = \pm v_2$, since those cases have been addressed. By symmetry, consider $\lambda_1<\lambda_2$, and consider $y'=\lambda_1v_2$ and $z'=\lambda_2v_1$. Now, observe that $(y,y',z,z')$ is an isosceles trapezoid, and that $d(y,z)$ corresponds to the length of a diagonal. Furthermore, $(\varphi(y),\varphi(y'),\varphi(z),\varphi(z'))$ is also an isosceles trapezoid, and each of its side is strictly shorter than the corresponding side in $(y,y',z,z')$, by the two previous special cases. Furthermore, the sides of the two trapezoid are parallel, directed by $v_1$, $v_2$ and $v_2-v_1$ respectively. This implies that $d(\varphi(y),\varphi(z))<d(y,z)$.
    
    This proof is represented \cref{fig:Inversion_into_Ball}
\end{proof}

\begin{figure}
    \centering
    \begin{tikzpicture}
    \draw (0,0) circle (2.3cm);
    \filldraw (0,0) circle (2pt);
     \node at (0,-0.5) {$0$};
    \draw(-3,5)--(0,0);
    \filldraw (-2.4,4) circle (2pt);
    \node at (-2.1,4.2) {$z$};
    \filldraw (-0.6,1) circle (2pt);
    \node at(-1,0.6) {$\varphi(z)$};
    \filldraw (0,2) circle (2pt);
    \node at (0.3,2) {$y$};
    \draw (-0.6,1)--(0,2)--(-2.4,4);
    \filldraw (-0.88,1.47) circle (2pt);
    \node at (-1.3,1.5) {$y'$};
    \draw[dashed](-0.88,1.47)--(0,2);

    \draw[shift={(7,0)}] (0,0) circle (2.3cm);
    \filldraw[shift={(7,0)}] (0,0) circle (2pt);
     \node[shift={(7,0)}] at (0,-0.5) {$0$};
     \draw[shift={(7,0)}] (0,0)--(-4,4);
     \draw[shift={(7,0)}] (0,0)--(4,4);
     \filldraw[shift={(7,0)}] (-2.6,2.6) circle (2pt);
     \node[shift={(7,0)}] at (-2.9,2.4) {$y$};
     \filldraw[shift={(7,0)}] (4,4) circle (2pt);
     \node[shift={(7,0)}] at (4,4.5) {$z$};
     \draw[shift={(7,0)}](-2.6,2.6)--(4,4);
     \draw[shift={(7,0)}](-2.6,2.6)--(2.6,2.6);
     \filldraw[shift={(7,0)}] (2.6,2.6) circle (2pt);
     \node[shift={(7,0)}] at (2.8,2.4) {$y'$};
     \draw[shift={(7,0)}](-4,4)--(4,4);
     \filldraw[shift={(7,0)}] (-4,4) circle (2pt);
     \node[shift={(7,0)}] at (-4,4.5) {$z'$}; 

     \filldraw[shift={(7,0)}] (1.3,1.3) circle (2pt);
     \node[shift={(7,0)}] at (1,1.6) {$\varphi(y')$};
     \filldraw[shift={(7,0)}] (-1.3,1.3) circle (2pt);
    \node[shift={(7,0)}] at (-1,1.6) {$\varphi(y)$};
    \draw[shift={(7,0)}] (-1.3,1.3)--(1.3,1.3);
    \draw[shift={(7,0)}](-0.8,0.8)--(0.8,0.8);
    \filldraw[shift={(7,0)}] (-0.8,0.8) circle (2pt);
    \filldraw[shift={(7,0)}] (0.8,0.8) circle (2pt);
    \node[shift={(7,0)}] at (-1.3,0.6) {$\varphi(z')$};
    \node[shift={(7,0)}] at (1.3,0.6) {$\varphi(z)$};
    \draw[shift={(7,0)}] (-1.3,1.3)--(0.8,0.8);
    \end{tikzpicture}
    \caption{An illustration of the proof of \cref{lem:Inversion_phi}}
    \label{fig:Inversion_into_Ball}
\end{figure}
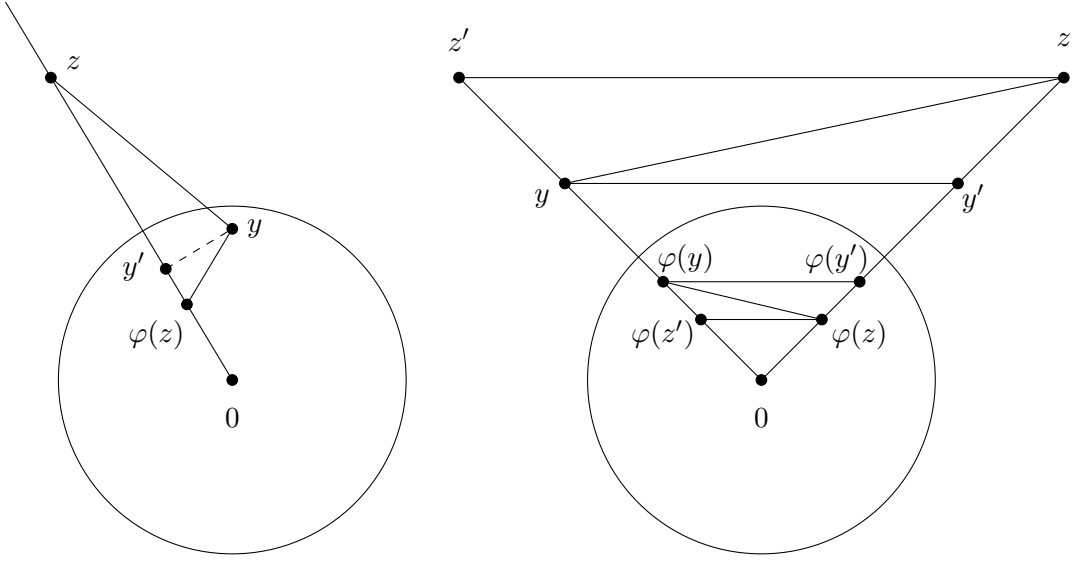

\subsection{Scaling}
\label{subsec:scaling}
In this section, we use the notion of locality to define the equivalent of a homothety on $\Ran(M)$. We will use two slightly different points of view of this notion: (i) in \cref{def:shifting function} we define an application called $\Shift$ that has four variables: the "centers" of the antecedent and image configurations, the scaling factor, and the configuration in $\Ran(M)$ to which the homothety is applied; (ii) alternatively, in \cref{def:shift system,def:shift system scaling} we define an application on $M$ indexed by a set of parameter $\Sigma$ called a \textbf{shift system}. The two point of view are closely related, as outlined in \cref{prop:shifting equivalence} but they serve different purposes. The first point of view is better suited to describe the continuity of the homothety, and the second is better suited to describe bijectivity and inverses. Those are the two properties that we will need to build our homeomorphism. Moreover, the second point of view also enables us to apply \cref{lem:boule hausdorff locale} directly. This ultimately justifies this somewhat double-definition. Note that, what we call a generalized homothety is in fact the composition of a homothety with a translation.

We first generalize some elementary transformations on $M$ onto $\Ran(M)$.
\begin{defin}
    We define translation and scaling function as
    \begin{align*}
        +\colon  \Ran(M)\times M&\to \Ran(M)\\
        (Y,x)&\mapsto Y+x=\{y+x|y\in Y\}\\
        -\colon \Ran(M)\times M &\to \Ran(M)\\
        (Y,x)&\mapsto Y-x=\{y-x|y\in Y\}\\
        \scale\colon \R^+\times \Ran(M)&\to \Ran(M)\\
        (s,X)&\mapsto s\times X=\{sx|x\in X\}
    \end{align*}
    and, for any $n\in\N^*$
    \begin{align*}
        +\colon \left(\Ran(M)\right)^n\times M^n&\to \left(\Ran(M)\right)^n\\
        (\underline{Y},\underline{x})&\mapsto \underline{Y}+\underline{x}=\left(Y^{(1)}+x^{(1)},\dots, Y^{(n)}+x^{(n)}\right)\\
        -\colon \left(\Ran(M)\right)^n \times M^n&\to \left(\Ran(M)\right)^n\\
        (\underline{Y},\underline{x})&\mapsto \underline{Y}-\underline{x}=\left( Y^{(1)}-x^{(1)},\dots,Y^{(n)}-x^{(n)}\right)\\
        \scale\colon \R^+\times \left(\Ran(M)\right)^n&\to \left(\Ran(M)\right)^n\\
        (s,\underline{Y})&\mapsto s\times \underline{Y} =\left(s\times Y^{(1)},\dots,s \times Y^{(n)}\right)
    \end{align*}
\end{defin}

Then we define the \textbf{clustering} function, that gives a partition of a configuration with respect to a reference configuration $X\in\Ran(M)$. We also define the inverse application that gives the union of a family of clusters.
\begin{defin}
\label{def:clustering}
    Let $X\in\Ran(M)$ be a configuration in $\Ran(M)$, $n=\card(X)$, and $\underline{x}=(x^{(1)},\dots,x^{(n)})\in\pi_n^{-1}(\{X\})$. The $\underline{x}$-clustering function is defined as
    \begin{align*}
        \ClustX\colon \BH(X,\merg(X))&\to \prod_{i=1}^n \Ran(B(x^{(i)},\merg(X)))\\
        Y&\mapsto \left(Y \cap B(x^{(i)},\merg(X)) \right)_{1\leq i\leq n}
    \end{align*}
    We also note, for any $\underline{Y}=(Y^{(1)},\dots,Y^{(n)})\in\left(\Ran(M)\right)^n$, 
    \begin{align*}
        &\bigcup \underline{Y}=\bigcup_{i=1}^n Y^{(i)}\\
        &\bigsqcup \underline{Y}=\bigsqcup_{i=1}^n Y^{(i)} \quad \text{ if }  Y^{(i)}\cap Y^{(j)}=\emptyset, \forall i\neq j\in\{1,\dots,n\}.
    \end{align*}
\end{defin}

\begin{rem}
    For any $Y\in\BH(X,\merg(X))$, we have $Y=\bigsqcup\ClustX(Y)$.
\end{rem}

Then we show that this clustering function is continuous.
\begin{prop}
\label{prop: clust continuous}
Let $X\in\Ran(M)$ be a configuration in $\Ran(M)$, $n=\card(X)$, and $\underline{x}=(x^{(1)},\dots,x^{(n)})\in\pi_n^{-1}(\{X\})$. The $\underline{x}$-clustering function is $\omega$-continuous on $\BHc(X,\frac{\merg(X)}{2})$.
\end{prop}

\begin{proof}
Let $(Y_k)_{k\in\N}\in\BHc(X,\frac{\merg(X)}{2})$, be a sequence that converges to $Y\in\BHc(X,\frac{\merg(X)}{2})$ for the distance $\dom$, and let $\left(Y^{(1)},\dots,Y^{(n)}\right)=\ClustX(Y)$ and $\left(Y_k^{(1)},\dots,Y_k^{(n)}\right)=\ClustX(Y_k)$ for all $k\geq 0$, and let us show that $\ClustX(Y_k)\to\ClustX(Y)$, i.e. that $Y_k^{(j)}\to Y^{(j)}$ for all $1\leq j \leq n$. By \cref{lem:boule hausdorff locale} there exists $\epsilon>0$ such that for all $Z\in\BHc(X,\frac{\merg(X)}{2})\subset\BH(X,\frac{3\merg(X)}{4})$, if $\dom(Y,Z)<\epsilon$ then any geodesic chain from $Y$ to $Z$ is $\BH(X,\frac{3\merg(X)}{4})$-local, and thus $\BH(X,\merg(X))$-local. But since $Y_k\to Y$, there exists $K\geq 0$ such that for any $k\geq K$, $\dom(Y_k,Y)<\epsilon$. Let $k\geq K$, then, by \cref{theo:Geodesic_Chain}, there exists a geodesic chain $((Y_k=Z_0,\dots,Z_l=Y),(R_0,\dots,R_{l-1}))$ from $Y_k$ to $Y$, which will necessarily be $\BH(X,\merg(X))$-local.

Let $1\leq j\leq n$, and let $Z_m^{(j)}=Z_m\cap B(x^{(j)},\merg(X))$ for all $0\leq m\leq l$ and $R_m^{(j)}=R_m\cap B^2(x^{(j)},\merg(X))$ for all $0\leq m\leq l-1$. Then $((Z_0^{(j)},\dots,Z_l^{(j)}),(R_0^{(j)},\dots,R_{l-1}^{(j)}))$ is a combinatorial chain from $Y_k^{(j)}$ to $Y^{(j)}$. We thus have
\begin{align*}
    \dom(Y_k^{(j)},Y^{(j)})\leq \elo(R_0^{(j)},\dots,R_{l-1}^{(j)})\leq \elo(R_0,\dots,R_{l-1})=\dom(Y_k,Y)
\end{align*}
 which converges to zero, so that indeed $\ClustX(Y_k)\to\ClustX(Y)$. Since $(\Ran(M),\tom)$ is metric, this concludes this proof.
\end{proof}

We now have everything to define the generalized homothety on $\Ran(M)$ from both perspectives.
\begin{defin}
\label{def:shifting function}
    Let $n\in\N^*$, and $A_n=\{(\underline{x},\underline{y},s,Z)\in M^n\times M^n \times \R^+ \times \Ran(M) \mid \card(\pi_n(\underline{x}))=n, Z\in \BH(\pi_n(\underline{x}),\merg(\pi_n(\underline{x})))\}$. We define the $n$-\textbf{shifting function} as 
    \begin{align*}
        \Shift_n\colon A_n &\to\Ran(M)\\
        (\underline{x},\underline{y},s,Z) & \mapsto \bigcup \left ( s\scale (\ClustX(Z) - \underline{x}) + \underline{y}  \right)
    \end{align*}
    In the following, the index $n$ will often be omitted.
\end{defin}

\begin{defin}
\label{def:shift system}
    Let $n\in\N^*$. A $n$-\textbf{shift system} is a quadruplet $\Sigma=(\underline{x},\underline{y},s,r)\in M^n\times M^n \times \R^+ \times \R^+$ such that $\card(\pi_n(\underline{x}))=n$, and $0<r\leq\merg(\pi_n(\underline{x})).$

    A $n$-shift system  $\Sigma=(\underline{x},\underline{y},s,r)$ is \textbf{bijective} if $\card(\pi_n(\underline{y}))=n$ and $0<s\leq\frac{\merg(\pi_n(\underline{y}))}{r}$.

    The \textbf{dual system} of a bijective $n$-shift system $\Sigma=(\underline{x},\underline{y},s,r)$ is $\overline{\Sigma}=(\underline{y},\underline{x},\frac{1}{s},sr)$.
\end{defin}

\begin{rem}
We will often drop the $n$ when there is no ambiguity.
\end{rem}

\begin{defin}
\label{def:shift system scaling}
Let $n\in\N^*$ and $\Sigma=(\underline{x},\underline{y},s,r)$ be a $n$-shift system. The \textbf{$\Sigma$-shifting function} is 
\begin{align*}
\phi_{\Sigma}\colon \bigsqcup_{i=1}^n B(x^{(i)},r) &\to \bigcup_{i=1}^n B(y^{(i)},sr)\\
z & \mapsto s(z-x^{(i)})+y^{(i)} \quad\text{ if } z\in B(x^{(i)},r).
\end{align*}
\end{defin}

\begin{rem}
This function is well defined. Indeed, since $r\leq \merg(\pi_n(\underline{x}))$, the $B(x^{(i)},r)$ are pairwise disjoint, and since $\dm$ is induced by the Euclidean norm $\dm(\phi_\Sigma(z),y^{(i)})=s\dm(z,x^{(i)})<sr$ for any $z\in B(x^{(i)},r)$, and any $i\in\{1,\dots,n\}$.
\end{rem}

These two points of view are actually equivalent, and we detail the relation from one to the other in the following proposition.
\begin{prop}
\label{prop:shifting equivalence}
    Let $n\in\N^*$, let $\Sigma=(\underline{x},\underline{y},s,r)$ be a $n$-shift system, and let $Z\in \BH(\pi_n(\underline{x}),r)$
    \begin{enumerate}
        \item\label{it:shifting equivalence} $\Shift(\underline{x},\underline{y},s,Z)=\{\phi_\Sigma(z)\mid z\in Z\}=\bm{\phi}_\Sigma(Z)$.
        \item \label{it:shifting bijection} If $\Sigma$ is a bijective shift system, then
        \begin{enumerate}
            \item\label{it:shifting inverse} $\phi_\Sigma$ is a bijection of inverse $\phi_{\overline{\Sigma}}$
            \item\label{it:shift conserves cardinal} $\card(\Shift(\underline{x},\underline{y},s,Z))=\card(Z).$
            \item\label{it:shifting identity} $\Shift\left(\underline{y},\underline{x},\frac{1}{s},\Shift(\underline{x},\underline{y},s,Z)\right)=Z.$
        \end{enumerate}
    \end{enumerate}
\end{prop}
\begin{proof}
Let $X=\pi_n(\underline{x})$. By unpacking \cref{def:shifting function}, we have
\begin{align*}
    \Shift(\underline{x},\underline{y},s,Z)=\bigcup_{i=1}^n \{s(z-x^{(i)})+ y^{(i)} \mid z\in Z\cap B(x^{(i)},\merg(X))\}.
\end{align*}
Since $Z\in\BH(X,r)$, and $r\leq\merg(X)$, for any $z\in Z$, there is a unique $i\in\{1,\dots,n\}$ such that $\dm(z,x^{(i)})<\merg(X)$, and $\dm(z,x^{(i)})=\min_{x\in X}\{\dm(z,x)\}\leq \dH(Z,X)<r$. So that $Z\cap B(x^{(i)},\merg(X))=Z\cap B(x^{(i)},r)$ and $Z=\bigsqcup_{i=1}^n Z\cap B(x^{(i)},r)$. We thus have
\begin{align*}
    \Shift(\underline{x},\underline{y},s,Z)=\bigcup_{i=1}^n \{\phi_\Sigma(z) \mid z\in Z\cap B(x^{(i)},r)\}=\{\phi_\Sigma(z)\mid z\in Z\},
\end{align*}
which proves \cref{it:shifting equivalence}.

Let us suppose now that $\Sigma$ is bijective. In that case $\card(\pi_n(\underline{y}))=n$ so that the $y^{(i)}$ are pairwise distinct, for $1\leq i\leq n$. Moreover $sr\leq \merg(\pi_n(\underline{y}))$, so that the balls $B(y^{(i)},sr)$ are pairwise disjoint, for $1\leq i\leq n$. Moreover, for any $i\in\{1,\dots,n\}$, the restriction $\phi_\Sigma|_{B(x^{(i)},r)}^{B(y^{(i)},sr)}$ is a bijection, so that $\phi_\Sigma$ is a bijection of inverse
\begin{align*}
\phi_{\overline{\Sigma}}\colon \bigsqcup_{i=1}^n B(y^{(i)},sr) &\to \bigsqcup_{i=1}^n B(x^{(i)},r)\\
z & \mapsto \frac{1}{s}(z-y^{(i)})+x^{(i)} \quad\text{ if } z\in B(y^{(i)},sr).
\end{align*}
By the previous argument, we thus have 
$$\card(\Shift(\underline{x},\underline{y},s,Z))=\card(\bm{\phi}_\Sigma(Z))=\card(Z),$$
and
$$\Shift\left(\underline{y},\underline{x},\frac{1}{s},\Shift(\underline{x},\underline{y},s,Z)\right)=\bm{\phi}_{\overline{\Sigma}}(\bm{\phi}_\Sigma(Z))=Z.$$
\end{proof}

\begin{rem}
\label{rem:shifting conditions}
Let $(\underline{x},\underline{y},s,Z)\in A_n$. If $\card(\pi_n(\underline{y}))=n$ and $0<s\dH(\pi_n(\underline{x}),Z)< \merg(\pi_n(\underline{y}))$, then, we can set $r=\min\{\merg(\pi_n(\underline{x})),\frac{\merg(\pi_n(\underline{y}))}{s}\}$, so that $(\underline{x},\underline{y},s,r)$ is a bijective system such that $\dH(\pi_n(\underline{x}),Z)<r$ and we can apply \cref{prop:shifting equivalence}. 
\end{rem}

We now take advantage of both points of view to show that the shifting function is continuous. We will need the following upper bounds.
\begin{prop}
\label{prop: shifting upper bounds}
    Let $n\in\N^*$, and $(\underline{x},\underline{y},s,Z)\in A_n$, and $X=\pi_n(\underline{x})$.
    \begin{enumerate}
        \item\label{it: x upper bound} For any $\underline{\hat{x}}\in M^n$ such that $(\underline{\hat{x}},\underline{y},s,Z)\in A_n$ and $\Nsum{\underline{x}-\underline{\hat{x}}}\leq\merg(X)-\dH(X,Z)$, 
        \begin{align*}
            \dom(\Shift(\underline{x},\underline{y},s,Z),\Shift(\underline{\hat{x}},\underline{y},s,Z))\leq s\Nsum{\underline{x}-\underline{\hat{x}}}\card(Z)\omega(\card(Z))
        \end{align*}
        \item\label{it: y upper bound} For any $\underline{\hat{y}}\in M^n$, 
        \begin{align*}
            \dom(\Shift(\underline{x},\underline{y},s,Z),\Shift(\underline{x},\underline{\hat{y}},s,Z))\leq \Nsum{\underline{y}-\underline{\hat{y}}}\card(Z)\omega(\card(Z))
        \end{align*}
        \item\label{it: s upper bound} For any $\hat{s}\geq 0$, 
        \begin{align*}
            \dom(\Shift(\underline{x},\underline{y},s,Z),\Shift(\underline{x},\underline{y},\hat{s},Z))\leq \abs{s-\hat{s}}\card(Z)\omega(\card(Z)) \dom(X,Z)
        \end{align*}
        \item\label{it: Z upper bound} For any $\dH(X,Z)<r<\merg(X)$, there exists $\epsilon\geq \merg(X)-r >0$ such that for any $\hat{Z}\in\BH(X,r)$, if $\dom(Z,\hat{Z})<\epsilon$ then
        \begin{equation*}
            \dom(\Shift(\underline{x},\underline{y},s,Z),\Shift(\underline{x},\underline{y},s,\hat{Z}))\leq s \dom(Z,\hat{Z})
        \end{equation*}
    \end{enumerate}
\end{prop}
\begin{proof}
    Let $\underline{\hat{x}}\in M^n$ such that $(\underline{\hat{x}},\underline{y},s,Z)\in A_n$ and $\Nsum{\underline{x}-\underline{\hat{x}}}\leq\merg(X)-\dH(X,Z)$, and let $\hat{X}=\pi_n(\hat{x})$. Let us define the shift systems $\Sigma=(\underline{x},\underline{y},s,\merg(X))$ and $\hat{\Sigma}=(\underline{\hat{x}},\underline{y},s,\merg(\hat{X}))$. Then, by \cref{prop:shifting equivalence} \cref{it:shifting equivalence}, we have
    \begin{align*}
        \Shift(\underline{x},\underline{y},s,Z)&=\bm{\phi}_\Sigma(Z)\\
        \Shift(\underline{\hat{x}},\underline{y},s,Z)&=\bm{\phi}_{\hat{\Sigma}}(Z).
    \end{align*}
    We can thus define the following surjective relation between $\Shift(\underline{x},\underline{y},s,Z)$ and $\Shift(\underline{\hat{x}},\underline{y},s,Z)$:
    \begin{align*}
        R=\{(\phi_\Sigma(z),\phi_{\hat{\Sigma}}(z))\mid z\in Z\}.
    \end{align*}
    We thus have
    \begin{align*}
        \dom(\Shift(\underline{x},\underline{y},s,Z),\Shift(\underline{\hat{x}},\underline{y},s,Z))&\leq \elo(R)\leq \omega(\card(Z))\sum_{z\in Z} \dm(\phi_\Sigma(z),\phi_{\hat{\Sigma}}(z)).
    \end{align*}
    For any $z\in Z$, there is a unique $i\in\{1,\dots,n\}$ such that $\dm(z,x^{(i)})\leq \dH(X,Z) <\merg(X)$, and a unique $j\in\{1,\dots,n\}$ such that $\dm(z,\hat{x}^{(j)})=\min_{k\in\{1,\dots,n\}}\{\dm(z,\hat{x}^{(k)})\}<\merg(\hat{X})$. Let us take $1\leq k\leq n$, such that $k\neq i$, then
    \begin{align*}
        \dm(z,\hat{x}^{(k)})&\geq \dm(z,x^{(k)}) - \dm(x^{(k)},\hat{x}^{(k)})\\
        &\geq \dm(x^{(i)},x^{(k)}) - \dm(z,x^{(i)}) - \dm(x^{(k)},\hat{x}^{(k)}) \\
        &\geq 2\merg(X) - \dH(X,Z) - \Nsum{\underline{x}-\underline{\hat{x}}}\\
        &\geq \merg(X)\\
        &\geq \dH(X,Z) + \Nsum{\underline{x}-\underline{\hat{x}}}\\
        &\geq \dm(z,x^{(i)}) + \dm(x^{(i)},\hat{x}^{(i)})\\
        &\geq \dm(z,\hat{x}^{(i)}).
    \end{align*}
    Therefore $\dm(z,\hat{x}^{(i)})=\min_{k\in\{1,\dots,n\}}\{\dm(z,\hat{x}^{(k)})\}$, and thus $i=j$.
    We thus have
    \begin{align*}
        \dm(\phi_\Sigma(z),\phi_{\hat{\Sigma}}(z))=\dm(s(z-x^{(i)})+y^{(i)},s(z-\hat{x}^{(i)})+y^{(i)})=s\dm(x^{(i)},\hat{x}^{(i)})\leq s\Nsum{\underline{x}-\underline{\hat{x}}},
    \end{align*}
    and finally
    \begin{align*}
        \dom(\Shift(\underline{x},\underline{y},s,Z),\Shift(\underline{\hat{x}},\underline{y},s,Z))&\leq \omega(\card(Z))\sum_{z\in Z} s\Nsum{\underline{x}-\underline{\hat{x}}} \\
        &\leq s\Nsum{\underline{x}-\underline{\hat{x}}} \omega(\card(Z)) \card(Z), 
    \end{align*}
    which yields \cref{it: x upper bound}.

    Let $\underline{\hat{y}}\in M^n$. Then again, let us define the shift system $\hat{\Sigma}=(\underline{x},\underline{\hat{y}},s,\merg(X))$, which, by \cref{prop:shifting equivalence} \cref{it:shifting equivalence} satisfies $\Shift(\underline{x},\underline{\hat{y}},s,Z)=\bm{\phi}_{\hat{\Sigma}}(Z)$, and the relation
    \begin{align*}
        R=\{(\phi_\Sigma(z),\phi_{\hat{\Sigma}}(z))\mid z\in Z\}.
    \end{align*}
    So that we have
    \begin{align*}
        \dom(\Shift(\underline{x},\underline{y},s,Z),\Shift(\underline{x},\underline{\hat{y}},s,Z))&\leq \omega(\card(Z)) \sum_{z\in Z} \dm(\phi_\Sigma(z),\phi_{\hat{\Sigma}}(z)).
    \end{align*}
    For any $z\in Z$, there exists a unique $i\in\{1,\dots,n\}$ such that $z\in B(x^{(i)},\merg(X))$. So we have 
    \begin{align*}
        \dm(\phi_\Sigma(z),\phi_{\hat{\Sigma}}(z))=\dm(s(z-x^{(i)})+y^{(i)},s(z-x^{(i)})+\hat{y}^{(i)})=\dm(y^{(i)},\hat{y}^{(i)})\leq \Nsum{\underline{y}-\underline{\hat{y}}}.
    \end{align*}
    So that finally
    \begin{align*}
        \dom(\Shift(\underline{x},\underline{y},s,Z),\Shift(\underline{x},\underline{\hat{y}},s,Z))&\leq \omega(\card(Z)) \sum_{z\in Z}\Nsum{\underline{y}-\underline{\hat{y}}}\\
        &\leq \Nsum{\underline{y}-\underline{\hat{y}}} \omega(\card(Z))\card(Z),
    \end{align*}
    which gives \cref{it: y upper bound}.

    Let $\hat{s}\geq 0$. Similarly, we can define the shift system $\hat{\Sigma}=(\underline{x},\underline{y},\hat{s},\merg(X))$, which, by \cref{prop:shifting equivalence} \cref{it:shifting equivalence} satisfies $\Shift(\underline{x},\underline{y},\hat{s},Z)=\bm{\phi}_{\hat{\Sigma}}(Z)$, and the relation
    \begin{align*}
        R=\{(\phi_\Sigma(z),\phi_{\hat{\Sigma}}(z))\mid z\in Z\}.
    \end{align*}
    So that we have
    \begin{align*}
        \dom(\Shift(\underline{x},\underline{y},s,Z),\Shift(\underline{x},\underline{y},\hat{s},Z))&\leq \omega(\card(Z)) \sum_{z\in Z} \dm(\phi_\Sigma(z),\phi_{\hat{\Sigma}}(z)).
    \end{align*}
    For any $z\in Z$, there exists a unique $i\in\{1,\dots,n\}$ such that $z\in B(x^{(i)},\merg(X))$. So we have 
    \begin{align*}
        \dm(\phi_\Sigma(z),\phi_{\hat{\Sigma}}(z))=\dm(s(z-x^{(i)})+y^{(i)},\hat{s}(z-x^{(i)})+y^{(i)})=\abs{s-\hat{s}}\dm(z,x^{(i)})\leq \abs{s-\hat{s}}\dH(X,Z).
    \end{align*}
    So that finally
    \begin{align*}
        \dom(\Shift(\underline{x},\underline{y},s,Z),\Shift(\underline{x},\underline{y},\hat{s},Z))&\leq \omega(\card(Z)) \sum_{z\in Z}\abs{s-\hat{s}}\dH(X,Z)\\
        &\leq \abs{s-\hat{s}} \omega(\card(Z))\card(Z)\dH(X,Z),
    \end{align*}
    which gives \cref{it: s upper bound}.

    Finally, let $\dH(X,Z)<r<\merg(X)$ and let us define $\epsilon$ as the one given by \cref{lem:boule hausdorff locale}. Then, for any $\hat{Z}\in \BH(X,r)$, if $\dom(Z,\hat{Z})<\epsilon$ then, by \cref{theo:Geodesic_Chain} and \cref{lem:boule hausdorff locale} there exists a $\BH(X,r)$-local geodesic chain from $Z$ to $\hat{Z}$. Moreover, by \cref{prop:shifting equivalence} \cref{it:shifting equivalence} we directly have $\Shift(\underline{x},\underline{y},s,\hat{Z})=\bm{\phi}_{\Sigma}(\hat{Z})$. And finally by \cref{lem:local scaling}, we have
    \begin{align*}
        \dom(\Shift(\underline{x},\underline{y},s,Z),\Shift(\underline{x},\underline{y},s,\hat{Z}))=\dom(\bm{\phi}_{\Sigma}(Z),\bm{\phi}_{\Sigma}(\hat{Z}))\leq s\dom(Z,\hat{Z}),
    \end{align*}
    which gives \cref{it: Z upper bound}.
\end{proof}

\begin{lem}
\label{lem:pi omega Lipschitz}
    Let $n\in\N$. The projection $\pi_n\colon M^n\to (\Ran(M),\dom)$ is $\omega(n)$-Lipschitz.
\end{lem}
\begin{proof}
    Let $\underline{x}=(x^{(1)},\dots,x^{(n)}),\underline{y}=(y^{(1)},\dots,y^{(n)})\in M^n$. And let $R=\{(x^{(i)},y^{(i)})\}_{1\leq i\leq n}$, which is a surjective relation from $X$ to $Y$. We thus have 
\begin{align*}
\dom(X,Y)\leq \elo(R)=\omega(n)\Nsum{\underline{x}-\underline{y}}.    
\end{align*}
\end{proof}

We can finally prove the continuity of the shifting function.
\begin{prop}
\label{prop:shift continous}
    Let $n\in\N^*$, and let us define $B_n=\{(\underline{x},\underline{y},s,Z)\in M^n\times M^n \times \R^+ \times \Ran(M) \mid \card(\pi_n(\underline{x}))=n, Z\in \BHc(\pi_n(\underline{x}),\frac{\merg(\pi_n(\underline{x}))}{2})\}$ The function $\Shift$ is $\omega$-continuous on $B_n$.
\end{prop}
\begin{proof}
Let $(\underline{x_k},\underline{y_k},s_k,Z_k)_{k\in\N}$ be a sequence in $B_n$ that converges to $(\underline{x},\underline{y},s,Z)$, and let us show that the image by $\Shift$ converges to $\Shift(\underline{x},\underline{y},s,Z)$. Let us write $X=\pi_n(\underline{x})$ and $X_k=\pi_n(\underline{x_k})$ for all $k\geq 0$.

First, let us remark that, for any $k\geq 0$, we have $\dH(X_k,X)\leq \Nsum{\underline{x}-\underline{x_k}}$. If $n=1$, then $\merg(X_k)=\merg(X)$. Otherwise, let $i\neq j\in\{1,\dots,n\}$. We have
\begin{align*}
    \dm(x_k^{(i)},x_k^{(j)})\geq \dm(x^{(i)},x^{(j)}) - \dm(x_k^{(i)},x^{(i)}) - \dm(x_k^{(j)},x^{(j)})\geq 2\merg(X)-\Nsum{\underline{x}-\underline{x_k}},
\end{align*}
so that, in both cases $\merg(X_k)\geq \merg(X)-\Nsum{\underline{x}-\underline{x_k}}$.

Since $(\underline{x_k})_{k\in\N}$ converges to $\underline{x}$ and $(Z_k)_{k\in\N}$ converges to $Z$, there exists $K\geq 0$ such that, for any $k\geq K$, we have $\Nsum{\underline{x}-\underline{x_k}} < \frac{\merg(X)-\dH(X,Z)}{2}$, $\dom(Z,Z_k) < \frac{\dH(X,Z)}{4}$ and $\merg(X_k)>\frac{\merg(X)}{2}$. Let $k\geq K$, then we have
\begin{align*}
    \dH(Z,X_k)&\leq \dH(Z,Z_k)+\dH(Z_k,X_k)\\
    &\leq\dom(Z,Z_k)+\dH(Z_k,X_k) &&\text{by \cref{prop:dH leq dom}}\\
    &< \frac{\dH(X,Z)}{4} + \frac{\merg(X_k)}{2} & &\text{since } (\underline{x_k},\underline{y_k},s_k,Z_k)\in B_n\\
    &<\frac{\merg(X)}{4} - \frac{\Nsum{\underline{x}-\underline{x_k}}}{2} + \frac{\merg(X_k)}{2} &\quad &\text{since } \Nsum{\underline{x}-\underline{x_k}} < \frac{\merg(X)-\dH(X,Z)}{2}\\
    &< \merg(X_k)-\frac{\merg(X)}{4},&\quad&\text{since } \Nsum{\underline{x}-\underline{x_k}} \geq \merg(X) - \merg(X_k) \\
    &<\merg(X_k)
\end{align*}
so that $(\underline{x_k},\underline{y_k},s_k,Z)\in A_n$. Moreover, let $r_k=\merg(X_k)-\frac{\merg(X)}{4}$, and $\epsilon_k\geq \merg(X_k) - r_k \geq \frac{\merg(X)}{4}$. Observe that $\frac{\merg(X_k)}{2}<r_k$. Thus, for any $\hat{Z}\in\BHc(X_k,\frac{\merg(X_k)}{2})\subset \BH(X_k,r_k)$, such that $\dom(Z,\hat{Z})<\frac{\merg(X)}{4}\leq \epsilon_k$, we may apply \cref{prop: shifting upper bounds} \cref{it: Z upper bound} to get
\begin{align*}
    \dom(\Shift(\underline{x_k},\underline{y_k},s_k,Z),\Shift(\underline{x_k},\underline{y_k},s_k,\hat{Z}))\leq s_k\dom(Z,\hat{Z})
\end{align*}
We can thus define $\epsilon=\frac{\merg(X)}{4}\leq \inf_{k\geq K}{\epsilon_k}$. Since $(Z_k)_{k\in\N}$ converges to $Z$, there exists $K'\geq K$ such that, for any $k\geq K'$, we have $\dom(Z_k,Z)<\epsilon$. 

Let $k\geq K'$, we have
\begin{align*}
    \dom(\Shift(\underline{x},\underline{y},s,Z),\Shift(\underline{x_k},\underline{y_k},s_k,Z_k))\leq \; & \dom(\Shift(\underline{x},\underline{y},s,Z),\Shift(\underline{x_k},\underline{y},s,Z))\\
    &+\dom(\Shift(\underline{x_k},\underline{y},s,Z),\Shift(\underline{x_k},\underline{y_k},s,Z))\\
    &+\dom(\Shift(\underline{x_k},\underline{y_k},s,Z),\Shift(\underline{x_k},\underline{y_k},s_k,Z))\\
    &+\dom(\Shift(\underline{x_k},\underline{y_k},s_k,Z),\Shift(\underline{x_k},\underline{y_k},s_k,Z_k)).
\end{align*}
Let us now show that we can apply \cref{prop: shifting upper bounds} to these four terms.

Since $\dH(Z,X_k)<\merg(X_k)$, we can apply \cref{prop: shifting upper bounds} \cref{it: y upper bound} and \cref{it: s upper bound} to the second and third term. Since, moreover, $\Nsum{\underline{x}-\underline{x_k}}<\merg(X)-\dH(X,Z)$,
we can apply \cref{prop: shifting upper bounds} \cref{it: x upper bound} to the first term.

Lastly, since $\dom(Z_k,Z)<\epsilon$, and $Z_k\in\BHc(X_k,\frac{\merg(X_k)}{2})$, by the above discussion, we can apply \cref{prop: shifting upper bounds} \cref{it: Z upper bound} to the fourth term. We get
\begin{align*}
    \dom(\Shift(\underline{x},\underline{y},s,Z),\Shift(\underline{x_k},\underline{y_k},s_k,Z_k))\leq \;& s\Nsum{\underline{x}-\underline{x_k}} \omega(\card(Z)) \card(Z) \\
    &+ \Nsum{\underline{y}-\underline{y_k}} \omega(\card(Z)) \card(Z)\\
    &+ \abs{s-s_k} \omega(\card(Z)) \card(Z)\dH(X_k,Z)\\
    &+ s_k\dom(Z,Z_k).
\end{align*}
We easily check that these four terms go to zero as $k$ goes to infinity, which proves that $\Shift(\underline{x_k},\underline{y_k},s_k,Z_k)$ converges to $\Shift(\underline{x},\underline{y},s,Z)$. Since $B_n$ is metrizable this concludes this proof.
\end{proof}

Finally, we define a restriction of the shifting function where the antecedent and image configurations have the same "center". In that case we only \textbf{scale} the configuration with respect to a reference configuration $X\in\Ran(M)$. We will use this in the following to define an homeomorphism between the cone on a sphere and a ball.
\begin{defin}
\label{def:scaling}
    Let $X\in\Ran(M)$ be a configuration in $\Ran(M)$, $n=\card(X)$, and $\underline{x}\in\pi_n^{-1}(\{X\})$. We define the \textbf{scaling with respect to} $\underline{x}$ as 
    \begin{align*}
        \Scalex\colon \R^+ \times \BH(X,\merg(X)) &\to\Ran(M)\\
        (s,Z) & \mapsto \Scalex(s,Z) = \Shift(\underline{x},\underline{x},s,Z)
    \end{align*}
\end{defin}

\begin{rem}
    The scaling with respect to $\underline{x}$ does not depend on the choice of $\underline{x}\in\pi_n^{-1}(X)$, we will thus call it the \textbf{scaling with respect to} $X$, written $\ScaleX$ in the following.
\end{rem}

Continuity of the scaling function directly follows from continuity of the shifting function.
\begin{prop}
\label{prop:scale continous}
    Let $X\in\Ran(M)$. The function $\ScaleX$ is $\omega$-continuous on $\R^+\times\BHc(X,\frac{\merg(X)}{2})$. 
\end{prop}
\begin{proof}
    The function $\ScaleX$ is a restriction of $\Shift$, which is continuous by \cref{prop:shift continous}.
\end{proof}

If the scaling factor is small enough, the scaling function behaves very well with the $\omega$-distance. This will be useful to show the bijectivity of the homeomorphism between the cone on the sphere and the ball.
\begin{prop}
\label{prop:scaling xlocal equality}
    Let $X\in\Ran(M)$, $Z\in\Bom(X,\frac{\merg(X)}{2})$ and $s\in\R^+$. If $s\dom(X,Z)<\frac{\merg(X)}{2}$ then
    \begin{align*}
    \dom(\ScaleX(s,Z),X)=s\dom(Z,X).
\end{align*}
\end{prop}
\begin{proof}
    Let us first show that
    \begin{enumerate}
        \item\label{it:scaling inequality} For any $X\in\Ran(M)$, and $Z\in\Bom(X,\frac{\merg(X)}{2})$, we have $$\dom(\ScaleX(s,Z),X)\leq s\dom(Z,X)$$.
    \end{enumerate}
    Let us define the shift system $\Sigma=(\underline{x},\underline{x},s,\merg(X))$. By \cref{prop:shifting equivalence} \cref{it:shifting equivalence} we have 
    \begin{align*}
        \ScaleX(s,Z)=\Shift(\underline{x},\underline{x},s,Z)=\bm{\phi}_{\Sigma}(Z).
    \end{align*}
    By \cref{theo:Geodesic_Chain} and \cref{lem:boules omega locales} there exists a $\BH(X,\merg(X))$-local geodesic chain from $Z$ to $X$. So that, by \cref{lem:local scaling}, we have
    \begin{align*}
        \dom(\ScaleX(s,Z),X)=\dom(\bm{\phi}_{\Sigma}(Z),\bm{\phi}_{\Sigma}(X))\leq s\dom(Z,X),
    \end{align*}
    which shows \cref{it:scaling inequality}. 
    
    Now first remark that if $s=0$ then $\ScaleX(s,Z)=X$. Now if $0<s\dom(X,Z)<\frac{\merg(X)}{2}$, by \cref{rem:shifting conditions}, since $\dH\leq \dom$, we can apply \cref{prop:shifting equivalence} \cref{it:shifting identity} to get
    \begin{align*}
        Z=\ScaleX\left(\frac{1}{s},\ScaleX(s,Z)\right).
    \end{align*}
    By applying \cref{it:scaling inequality} to $\ScaleX(s,Z)$ we thus get
    \begin{align*}
        s\dom(Z,X)=s\dom\left(\ScaleX\left(\frac{1}{s},\ScaleX(s,Z)\right),X\right)\leq \dom(\ScaleX(s,Z)),X).
    \end{align*}
    which concludes this proof.
\end{proof}

\subsection{Conicality}
\label{subsec:conicity_Euclidean}
We now use the shifting function to define an homeomorphism from the product of a Euclidean space and a stratified cone to an open neighborhood of a configuration in $(\Ran(M),\tom)$. First we show that the cone on the sphere is homeomorphic to the ball.

\begin{prop}
\label{prop:homeo cone to w-ball}
    Let $X\in\Ran(M)$ be a configuration in $\Ran(M)$ and let $0<\epsilon <\frac{\merg(X)}{2}$. The following map
    \begin{align*}
        f\colon c(\Som(X,\epsilon)) &\to \Bom(X,\epsilon)\\
        [(t,Y)]&\mapsto \ScaleX(t,Y)
    \end{align*}
    where $c(\Som(X,\epsilon))$ is the cone on the sphere $\Som(X,\epsilon)=\{Z\mid \dom(Z,X)=\epsilon\}$, equipped with the teardrop topology, is an homeomorphism.
\end{prop}

\begin{proof}
    Let $[(t,Y)]\in c(\Som(X,\epsilon))$. We have $\dom(X,Y)=\epsilon<\frac{\merg(X)}{2}$ and $t\dom(X,Y)=t\epsilon<\frac{\merg(X)}{2}$, so that we can apply \cref{prop:scaling xlocal equality} to get $\dom(f(t,Y),X)=t\dom(Y,X)=t\epsilon < \epsilon$, i.e. $f(t,Y)\in\Bom(X,\epsilon)$. Moreover, if $t=0$, we have $f(0,Y)=X$, so that $f$ is a well-defined map on the cone.

    Let $Z\in\Bom(X,\epsilon)\setminus\{X\}$ and $s=\frac{\epsilon}{\dom(X,Z)}$. Since $\dom(X,Z)<\epsilon<\frac{\merg(X)}{2}$ and $s\dom(X,Z)=\epsilon<\frac{\merg(X)}{2}$ we can apply \cref{prop:scaling xlocal equality} to get $\dom(\ScaleX(s,Z),X)=s\dom(Z,X)=\epsilon$. We can thus define
    \begin{align*}
        g\colon \Bom(X,\epsilon)&\to c(\Som(X,\epsilon))\\
        Z&\mapsto \begin{cases}
            \left[\left(\frac{\dom(X,Z)}{\epsilon},\ScaleX(\frac{\epsilon}{\dom(X,Z)},Z)\right)\right]  &\text{ if } Z\neq X\\
            [(0,Y)], \text{ where } Y \in \Som(X,\epsilon) \quad& \text{ if } Z= X
        \end{cases}.
    \end{align*}
    Let us show that $f$ is bijective, of inverse $g$.

    Let $Z\in\Bom(X,\epsilon)$. If $Z=X$, then $f\circ g(X)=f([(0,Y)])=X$. If $Z\neq X$ then let $[(t,Y)]=g(Z)$. Since $0<t< 1$, \cref{rem:shifting conditions} applies to the tuple $(\underline{x},\underline{x},t,Z)$, and we can thus apply \cref{prop:shifting equivalence} \cref{it:shifting identity} to get:
    \begin{align*}
        f\circ g(Z)=f\left(t,Y\right) = \Shift(\underline{x},\underline{x},t,\Shift(\underline{x},\underline{x},\frac{1}{t},Z)) =Z.
    \end{align*}
    We therefore have $f\circ g=\Id_{\Bom(X,\epsilon)}$.
    
    Let $[(t,Y)]\in c(\Som(X,\epsilon))$ and $[(t',Y')]=g\circ f(t,Y)$. We can apply \cref{prop:scaling xlocal equality} again to get 
    \begin{align*}
    t'=\frac{\dom(f(t,Y),X)}{\epsilon}=\frac{t\dom(Y,X)}{\epsilon}=t.    
    \end{align*}
    If $t=0$, then $[(0,Y)]=[(0,Y')]$. If $t>0$, then $\dom(f(t,Y),X)=t\dom(Y,X)=t\epsilon>0$, so that $f(t,Y)\neq X$ and we can apply \cref{prop:shifting equivalence} \cref{it:shifting identity} to get
    \begin{align*}
        Y'&=\ScaleX\left(\frac{\epsilon}{\dom(f(t,Y),X)},f(t,Y)\right)\\
        &=\ScaleX\left(\frac{1}{t},\ScaleX(t,Y)\right)\\
        &=\Shift(\underline{x},\underline{x},\frac{1}{t},\Shift(\underline{x},\underline{x},t,Y))\\
        &=Y.
    \end{align*}
    We therefore have $g\circ f=\Id_{c(\Som(X,\epsilon))}$, so that $f$ is bijective, of inverse $g$. 

    Let $U=c(\Som(X,\epsilon)\setminus\{[(0,Y)]\}$ and let $h$ denote the following homeomorphism
    \begin{align*}
        h\colon U &\to (0,1)\times \Som(X,\epsilon)\\
        [(t,Y)] & \mapsto (t,Y).
    \end{align*}
    The restriction $f_{|U}$ can be factorized as 
    \begin{align*}
        f_{|U} = \ScaleX \circ h.
    \end{align*}
    By \cref{prop:scale continous}, it is thus continuous as a composition of continuous functions. Since $U$ is an open set of $c(\Som(X,\epsilon))$, $f$ is continuous on $U$. Let us show that it is continuous at $[(0,Y)]$. Let $0<\theta$. Let $\zeta=\frac{\theta}{\epsilon}$, and $[(t,Y)]\in c(\Som(X,\epsilon))$ such that $t<\zeta$. Then $\dom(f(t,Y),X)=t\dom(Y,X)=t\epsilon<\zeta\epsilon=\theta$. We conclude that $f$ is continuous.

    Similarly, the restriction $g_{|\Bom(X,\epsilon)\setminus\{X\}}$ can be factorized. By defining the following map
    \begin{align*}
        \hat{g}\colon \Bom(X,\epsilon)\setminus\{X\} &\to (0,1)\times \Som(X,\epsilon)\\
        Z &\mapsto \left(\frac{\dom(X,Z)}{\epsilon},\ScaleX\left(\frac{\epsilon}{\dom(X,Z)},Z\right)\right),
    \end{align*}
    which is continuous by \cref{prop:scale continous}. We get $g_{|\Bom(X,\epsilon)\setminus\{X\}}=h^{-1}\circ \hat{g}$, which is also continuous. 
    Since $\Bom(X,\epsilon)\setminus\{X\}$ is an open set of $\Bom(X,\epsilon)$, the map $g$ is continuous on $\Bom(X,\epsilon)\setminus\{X\}$. Let us show that it is continuous at $X$. Let $0<\zeta<1$ and let us show that $g^{-1}\left(\{[(t,Y)]\in c\left(\Som(X,\epsilon)\right)\mid t<\zeta\}\right)$ is open. But $g^{-1}\left(\{[(t,Y)]\in c\left(\Som(X,\epsilon)\right)\mid t<\zeta\}\right)=f\left(\{[(t,Y)]\in c\left(\Som(X,\epsilon)\right)\mid t<\zeta\}\right)=\Bom(X,\zeta\epsilon)$ is indeed open. Therefore $g$ is continuous, which concludes this proof.
\end{proof}

\begin{rem}
\label{rem:why midpoint}
    Now it is important to remark that, even if the homeomorphism from the teardrop cone to the ball preserves the cardinality in some sense, it is not the stratified homeomorphism we are looking for. Indeed, we need to consider the stratification on the cone given in \cref{def:Stratified_Cone}. However, this stratification contains a single point stratum, the apex of the cone. The homeomorphism $f$ maps the apex to $X$, the center of the ball $\Bom(X,\epsilon)$, which does not lie alone in its stratum. On the contrary, there are many points in $\Bom(X,\epsilon)$ with the same cardinality as that of $X$, specifically, $\Bom(X,\epsilon)\cap \Ran_=n(M)$ is homeomorphic to a Euclidean space of dimension $p\card(X)$.
    
     We will thus consider another cone, which will be homeomorphic to a slice of the ball, in which $X$ is the only point with its cardinality. To achieve this, we will restrict our attention to configurations where each cluster, in the sense of \cref{def:clustering}, is "centered" around the corresponding point in $X$. 
     In this way, we will produce subobjects $B\subset\Bom(X,\epsilon)$ and $S\subset\Som(X,\epsilon)$, which will be such that $f$ restricts to a bijection from $c(S)$ to $B$. 
     Since the teardrop cone satisfies the property that cone of subspaces are subspaces of the cone, this will be enough to conclude that the restriction of $f$ to $c(S)$ and $B$ is in fact a homeomorphism.

    Now, we need to define an appropriate notion for the "center" of a configuration in $\Ran(M)$. We cannot naively use the iso-barycenter, since it is not continuous on $\Ran(M)$, as can be seen by considering merges. We thus define the \textbf{mid-point} of a configuration, and show that it is continuous.
\end{rem}

\begin{defin}
\label{def:midpoint}
    Recall that $M=\R^p$. The \textbf{mid-point} function is defined as
    \begin{align*}
        \Mid\colon \Ran(M)&\to M\\
        X&\mapsto \frac{1}{2}\left( \min_{x\in X}\{x_1\}+\max_{x\in X}\{x_1\},\dots, \min_{x\in X}\{x_p\}+\max_{x\in X}\{x_p\}\right).
    \end{align*}
    Let $n\in\N^*$. The mid-point function is extended to $n$-tuples as 
    \begin{align*}
        \Mid^n\colon \Ran^n(M)&\to M^n\\
        (X^{(1)},\dots,X^{(n)})&\mapsto \left( \Mid(X^{(1)}),\dots,\Mid(X^{(n)})\right).
    \end{align*}
    Let $\underline{x}\in M^n$ and $X=\pi_n(\underline{x})$. The $\underline{x}$-mid-point function is defined as
    \begin{align*}
        \MidX\colon \BH(X,\merg(X)) &\to M^n\\
        Y &\mapsto \Mid^n(\ClustX(Y)).
    \end{align*}
\end{defin}

\begin{rem}
    We have $\MidX(\pi_n(\underline{x}))=\underline{x}$.
\end{rem}

\begin{rem}
\label{rem:card midx}
    For any $Y\in\BH(X,\merg(X))$, $\card(\pi_n(\MidX(Y)))=n$ by \cref{rem:card local}.
\end{rem}

\begin{prop}
\label{prop:mid p-Lipschitz}
    The mid-point is $p$-Lipschitz for $\dH$ and thus for $\dom$. More specifically, it is continuous on $(\Ran(M),\tH)$ and on $(\Ran(M),\tom)$.
\end{prop}
\begin{proof}
    Let $X,Y\in\Ran(M)$. By the triangle inequality, we have
    \begin{align*}
        \dm(\Mid(X),\Mid(Y))&=\norm{\Mid(X)-\Mid(Y)}\\
        &\leq \frac{1}{2}\sum_{i=1}^p \abs{\min_{x\in X}\{x_i\}-\min_{y\in Y}\{y_i\}}+\abs{\max_{x\in X}\{x_i\}-\max_{y\in Y}\{y_i\}}.
    \end{align*}
    Let $i\in\{1,\dots,p\}$. There exists $\widetilde{x}\in X$ such that $\widetilde{x}_i=\min_{x \in X}\{x_i\}$ and $\widetilde{y}\in Y$ such that $\widetilde{y}_i=\min_{y\in Y}\{y_i\}$. Let us first suppose that $\widetilde{x}_i\leq \widetilde{y}_i$. There exists $z\in Y$ such that $\dm(\widetilde{x},z)=\dm(\widetilde{x},Y)$. We have
    \begin{align*}
        \abs{\min_{x\in X}\{x_i\}-\min_{y\in Y}\{y_i\}}&=\widetilde{y}_i-\widetilde{x}_i\\
        &\leq z_i-\widetilde{x}_i\\
        &\leq \abs{z_i-\widetilde{x}_i}\\
        &\leq \dm(z,\widetilde{x})\\
        &\leq \dm(\widetilde{x},Y)\\
        &\leq \dH(X,Y).
    \end{align*}
     By symmetry, we get the same results if $\widetilde{x}_i\geq \widetilde{y}_i$. Similarly we get 
     \begin{align*}
         \abs{\max_{x\in X}\{x_i\}-\max_{y\in Y}\{y_i\}}\leq \dH(X,Y).
     \end{align*}
     Finally, we have
     \begin{align*}
         \dm(\Mid(X),\Mid(Y))\leq \sum_{i=1}^p \dH(X,Y) = p \dH(X,Y)\leq  p \dom(X,Y).
     \end{align*}
\end{proof}
\begin{corol}
    \label{cor:midx continuous} Let $n\in\N^*$, $\underline{x}\in M^n$, $X=\pi_n(\underline{x})$. The function $\MidX$ is continuous on $\Bom(X,\frac{\merg(X)}{2})$.
\end{corol}
\begin{proof}
    By \cref{prop: clust continuous}, $\ClustX$ is continuous, thus $\MidX$ is continuous as a composition of continuous functions. 
\end{proof}

If the scaling factor is small enough, more precisely if the shift system is bijective, then the mid-point, in some sense, commutes with the shifting function. This is the subject of the following proposition.
\begin{prop}
    Let $n\in\N^*$, and let $\Sigma=(\underline{x},\underline{y},s,r)$ be a bijective $n$-shift system. Then for any $Z\in\BH(\pi_n(\underline{x}),r)$, we have
    \begin{align*}
        \Mid_{\underline{y}}(\bm{\phi}_\Sigma(Z))=\phi^n_\Sigma(\MidX(Z)).
    \end{align*}
\end{prop}
\begin{proof}
    Let $Z\in\BH(\pi_n(\underline{x}),r)$, and $(Z^{(1)},\dots,Z^{(n)})=\ClustX(Z)$. Then $Z=\bigsqcup_{i=1}^n Z^{(i)}$ and $\bm{\phi}_\Sigma(Z)=\bigcup_{i=1}^n \bm{\phi}_\Sigma(Z^{(i)})$. For all $1\leq i\leq n$, $\phi_\Sigma(Z^{(i)})\subset  B(y^{(i)},\merg(\pi_n(\underline{y}))) $, so that $\Clust_{\underline{y}}(\bm{\phi}_\Sigma(Z))=(\bm{\phi}_\Sigma(Z^{(1)}),\dots,\bm{\phi}_\Sigma(Z^{(n)}))$, and 
    \begin{align*}
        \Mid_{\underline{y}}(\bm{\phi}_\Sigma(Z))=\left( \Mid(\bm{\phi}_\Sigma(Z^{(1)})),\dots,\Mid(\bm{\phi}_\Sigma(Z^{(n)}))\right).
    \end{align*}
    Let $i\in\{1,\dots,n\}$, then, for any $z\in Z^{(i)}$, $\phi_\Sigma(z)=s(z-x^{(i)})+y^{(i)}$. Since addition with a constant term and scaling by a positive factor commute with taking the minimum (resp. the maximum) of a set, we get:
    \begin{align*}
        \Mid(\bm{\phi}_\Sigma(Z^{(i)}))=\phi_\Sigma(\Mid(Z^{(i)})).
    \end{align*}
    Therefore
    \begin{align*}
        \Mid_{\underline{y}}(\bm{\phi}_\Sigma(Z))=\left( \phi_\Sigma(\Mid(Z^{(1)})),\dots,\phi_\Sigma(\Mid(Z^{(n)}))\right)=\phi_\Sigma^n(\MidX(Z)).
    \end{align*}
\end{proof}

We deduce an important property of the mid-point that will be useful when defining inverses.
\begin{corol}
    \label{cor:shift mid commute}
    Let $n\in\N^*$ and $(\underline{x},\underline{y},s,Z)\in A_n$. If $\card(\pi_n(\underline{y}))=n$, $s\dom(Z,\pi_n(\underline{x}))< \merg(\pi_n(\underline{y}))$, and $\MidX(Z)=\underline{x}$, then
    \begin{align*}
        \Mid_{\underline{y}}(\Shift(\underline{x},\underline{y},s,Z))=\underline{y}.
    \end{align*}
\end{corol}

We can now define a stratified homeomorphism between a stratified cone and a subspace of the ball by restricting the previous (unstratified) homeomorphism to "centered" configurations, i.e. configurations whose mid-point is the center of the ball/sphere. 
\begin{prop}
\label{prop:strat homeo radial cone ball}
    Let $X\in\Ran(M)$ be a configuration, $\epsilon\in(0,\frac{\merg(X)}{2})$, $n=\card(X)$ and $\underline{x}\in\pi_n^{-1}(\{X\})$. Let $C_{\underline{x}}=\{Z\in\BH(X,\merg(X))\mid \MidX(Z)=\underline{x}\}$. The map 
    \begin{align*}
        f\colon c\left(\Som(X,\epsilon)\cap C_{\underline{x}}\right)&\to \Bom(X,\epsilon)\cap C_{\underline{x}}\\
        [(t,Y)] &\mapsto \ScaleX(t,Y)
    \end{align*}
    is a stratified homeomorphism for the stratification induced by the cardinality map.
\end{prop}
\begin{proof}
    Let us define the map $\tilde{f}$ as in \cref{prop:homeo cone to w-ball}:
    \begin{align*}
        \tilde{f}\colon c\left(\Som(X,\epsilon)\right)&\to \Bom(X,\epsilon)\\
        [(t,Y)] &\mapsto \ScaleX(t,Y),
    \end{align*}
    which is a homeomorphism. Let us define the following subset of the cone:
    \begin{align*}
    D=\left\{ [(t,Y)]\in c\left(\Som(X,\epsilon)\right) \mid Y\in C_{\underline{x}} \right\}.
    \end{align*}
    Note that, one of the key properties of the teardrop topology, is that, whenever we have a space and a subspace $L\subset L'$, the inclusion $c(L)\hookrightarrow c(L')$ is a homeomorphism onto its image. 
    Hence, the following map
    \begin{align*}
    h\colon c\left(\Som(X,\epsilon)\cap C_{\underline{x}}\right)&\to D\\
    [(t,Y)] & \mapsto [(t,Y)].
    \end{align*}
    is a homeomorphism.

    We then construct $f$ as suggested by the following diagram
    \begin{equation*}
        \begin{tikzcd}
            c(\Som(X,\epsilon))
            \arrow{r}{\simeq} 
            \arrow[phantom,bend right = 12]{r}{\tilde{f}}
            & \Bom(X,\epsilon)
            \\
            D
            \arrow[hookrightarrow]{u}{}
            \arrow[dashed]{r}{\simeq}
            & \Bom(X,\epsilon)\cap C_{\underline{x}}
            \arrow[hookrightarrow]{u}{} 
            \\
            c(\Som\cap C_{\underline{x}})
            \arrow{u}{\simeq}
            \arrow[phantom, bend right = 15]{u}{h}
            \arrow[dashed,swap]{ur}{f}
        \end{tikzcd}
    \end{equation*}

    By \cref{cor:shift mid commute}, for any $Y\in\overline{\Bom(X,\epsilon)}\cap C_{\underline{x}}$ and for any $t\leq\frac{\epsilon}{\dom(X,Y)}$, we have $\MidX(\ScaleX(t,Y))=\underline{x}$. Therefore, $\tilde{f}$ restricts to a homeomorphism from $D$ to $\Bom(X,\epsilon)\cap C_{\underline{x}}$. We can thus define the homeomorphism $f$ as the composition
    \begin{equation*}
        f=\tilde{f}|_{D}^{\Bom(X,\epsilon)\cap C_{\underline{x}}}\circ h.
    \end{equation*}
    One easily checks that this corresponds to the formula given in the statement.
    
    Let us now make explicit the stratification on the cone. Let $Y\in \Som(X,\epsilon)\cap C_{\underline{x}}$, and let us show that $\card(Y)\geq n+1$. Assume to the contrary that $\card(Y)\leq n$. Since $\dH(X,Y)\leq \dom(X,Y)=\epsilon < \frac{\merg{X}}{2}$, we must have $\card(Y)\geq n$ and thus $\card(Y)=n$. Furthermore, this immediately implies that for any $1\leq i\leq n$, $\card(Y\cap B(x^{(i)},\epsilon))=1$. But since, by assumption, $\Mid_{\underline{x}}(Y)=X$, we have that $Y\cap B(x^{(i)},\epsilon)=\{x^{(i)}\}$ and finally, we get that $Y=X$, which contradicts the fact that $\dom(X,Y)=\epsilon>0$. Thus,  $\card\colon \Som(X,\epsilon)\cap C_{\underline{x}}\to \N_{\geq 1}$ takes values into $\N_{\geq n+1}$ and so we can consider $\Som(X,\epsilon)\cap C_{\underline{x}}$ as naturally stratified over $\N_{\geq n+1}$. Its cone is then stratified over the same poset, with an added minimal element, which we can identify with $n$, so that $c\left(\Som(X,\epsilon)\cap C_{\underline{x}}\right)$ is stratified over the poset $\N_{\geq n}$ with $[(0,Y)]$ being sent to $n$ while every other point $[(t,Y)]$ is sent to $\card(Y)$. 
    
    We can now show that $f$ is a stratified homeomorphism. Let $[(t,Y)]\in c\left(\Som(X,\epsilon)\cap C_{\underline{x}}\right)$. If $t=0$, then $\card(f(t,Y))=\card(X)=n$. If $t>0$, by \cref{prop:shifting equivalence} \cref{it:shifting bijection}, $\card(f(t,Y))=\card(Y)$. Which concludes this proof.
\end{proof}

Note that we are not done yet, because the image of this homeomorphism is not open in $\Ran(M)$. We thus need to take the product with a Euclidean space that somewhat encodes the position of the mid-point of a configuration in $\Ran(M)$.

\begin{prop}
\label{prop:homeo strat}
    Let $X\in\Ran(M)$ be a configuration, and $n=\card(X)$ its cardinality. There exist $U$ an open neighborhood of $X$ in $(\Ran(M), \tom)$, E a Euclidean space of dimension $np$, $\varphi\colon L\to \{n+1,\dots\}$ a stratified space and $h\colon U \to E \times c(L)$ a stratified homeomorphism where the stratification on $U$ is given by the cardinality $\card_{|U}\colon U \to \{n,\dots\}$.
\end{prop}
\begin{proof}
    Let $X\in\Ran(M)$, $n=\card{X}$, $\epsilon\in(0,\frac{\merg(X)}{2})$ and $\underline{x}=(x^{(1)},\dots,x^{(n)})\in\pi_n^{-1}(X)$. By \cref{tpi rafine tom}, $\pi_n^{-1}\left(\Bom(X,\epsilon)\right)$ is an open set of $M^n$ containing $\underline{x}$, so that there exists $\theta\in(0,\epsilon]$ such that $\underline{x}\in \Bsum(\underline{x},\theta)\subset \pi_n^{-1}\left(\Bom(X,\epsilon)\right)$. Let $E=\Bsum(\underline{x},\theta)$, which is homeomorphic to a Euclidean space of dimension $np$. 
    
    Let $g^{(j)}=(4j,0,\dots,0)\in M$ for all $j\in\{1,\dots,n\}$ be $n$ points in $M$. Let $\underline{g}=(g^{(1)},\dots,g^{(n)})$, $G=\pi_n(\underline{g})$, and $V=\left\{Z\in\Bom(G,1)\mid \Mid_{\underline{g}}\left(Z\right)=\underline{g}\right\}$.

    We now define the following map 
    \begin{align*}
        f\colon E\times V &\to \Ran(M)\\
        (\underline{y},Z ) &\mapsto \Shift\left(\underline{g},\underline{y},\epsilon - \dom(\pi_n(\underline{y}),X),Z\right).
    \end{align*}
    Since $\merg(G)=2$, we have $V\subset\Bom(G,\frac{\merg(G)}{2})\subset \BHc(G,\frac{\merg(G)}{2})$, and $f$ is thus continuous by \cref{prop:shift continous}.

    We will now construct its inverse. Let us denote $U$ the image of $f$, and let us first show that $U\subset\Bom(X,\epsilon)$. Let $W\in U$, there exists $(\underline{y},Z)\in E\times V$ such that $f(\underline{y},Z)=W$. Let us denote $Y=\pi_n(\underline{y})$. Since $Z\in\Bom(G,\frac{\merg(G)}{2})$, we can apply \cref{prop: shifting upper bounds} \cref{it: Z upper bound} to get:
    \begin{align*}
        \dom\left(W,Y\right) &= \dom\left(\Shift\left(\underline{g},\underline{y},\epsilon - \dom(Y,X),Z\right),\Shift\left(\underline{g},\underline{y},\epsilon - \dom(Y,X),G\right)\right) \\
        &\leq (\epsilon - \dom\left(Y,X\right))\dom(Z,G) \\
        &< \epsilon - \dom\left(Y,X\right).
    \end{align*}
    and
    \begin{align*}
        \dom\left(W,X\right)  \leq \dom\left(W,Y\right) + \dom\left(Y,X\right)< \epsilon - \dom\left(Y,X\right) + \dom\left(Y,X\right) = \epsilon,
    \end{align*}
    so that indeed $U\subset\Bom(X,\epsilon)$. 

    Let us now compare the merging radii of $X$ and $Y$. Let $i\in\{1,\dots,n\}$. Since $Y$ is $\BH(X,\merg(X))$-local, there is a unique $j\in\{1,\dots,n\}$ such that $\dm(y^{(i)},x^{(j)})<\merg(X)$. But we also have
    \begin{align*}
        \dm(y^{(i)},x^{(i)})\leq \sum_{i=1}^n \dm(x^{(i)},y^{(i)}) =\Nsum{\underline{x}-\underline{y}}<\theta\leq\epsilon<\merg(X),
    \end{align*} 
    so that $j=i$, and $\dH(X,Y)\geq \max_{1\leq i\leq n} \{\dm(y^{(i)},x^{(i)})\}$. If $n=1$ then $\merg(X)=\merg(Y)$. Otherwise, let $i\neq j\in\{1,\dots,n\}$. We have
    \begin{align*}
        \dm(y^{(i)},y^{(j)})\geq \dm(x^{(i)},x^{(j)})- \dm(y^{(i)},x^{(i)}) - \dm(x^{(j)},y^{(j)})\geq 2\merg(X) - 2\dH(X,Y),
    \end{align*}
    so that 
    \begin{align}
        \label{eq:radX radY}
        \merg(Y)\geq\merg(X)-\dH(X,Y).
    \end{align}

    Let us now show that $\MidX(W)=\underline{y}$. Let $w\in W=f(\underline{y},Z)$. Since $\dom(W,X)<\epsilon<\merg(X)$, there is a unique $i\in\{1,\dots,n\}$ such that $w\in B(x^{(i)},\epsilon)$. Let us define the shift system $\Sigma=(\underline{g},\underline{y},(\epsilon - \dom(X,Y)),1)$. Since $Z\in\Bom(G,1)$, by \cref{prop:shifting equivalence}, $f(\underline{y},Z)=\bm{\phi}_{\Sigma}(Z)$. Therefore, there exists $j\in\{1,\dots,n\}$ and $z\in Z\cap B(g^{(j)},1)$ such that
    \begin{align*}
        w=y^{(j)} + (\epsilon - \dom(X,Y))(z-g^{(j)}).
    \end{align*}
    But then
    \begin{align*}
        \dm(w,y^{(j)})=(\epsilon - \dom(X,Y))\dm(z,g^{(j)})<\epsilon - \dom(X,Y)<\merg(X) - \dH(X,Y)\leq\merg(Y),
    \end{align*}
    and thus 
    \begin{align*}
        \dm(x^{(i)},y^{(j)})\leq \dm(x^{(i)},w)+\dm(w,y^{(j)})<2\epsilon<\merg(X),
    \end{align*}
    so that $j=i$. So that $\ClustX(W)=\Clust_{\underline{y}}(W)$,  and, by \cref{cor:shift mid commute}, 
    \begin{align*}
        \MidX(W)=\Mid_{\underline{y}}(W)=\underline{y}.
    \end{align*}
    Thus $\MidX(W)\in E$ so that $\card(\pi_n(\MidX(W)))=n$ and $\dom(\pi_n(\MidX(W)),X)<\epsilon$.
Moreover, we can apply \cref{prop: shifting upper bounds} \cref{it: Z upper bound} to get
    \begin{align*}
    \dom(W,Y)&=\dom(\Shift(\underline{g},\underline{y},\epsilon-\dom(X,Y),Z),\Shift(\underline{g},\underline{y},\epsilon-\dom(X,Y),G))\\
        &\leq (\epsilon-\dom(X,Y))\dom(Z,G)\\
        &< \frac{\merg(X)}{2} - \frac{\dH(X,Y)}{2} \quad \text{since }  \dom(X,Y)\geq \dH(X,Y) \text{ and } \dom(Z,G) <1\\
        &\leq \frac{\merg(Y)}{2} \quad \text{here we use \eqref{eq:radX radY}.}
    \end{align*}
    so that 
    \begin{align*}
        \left(\MidX(W),\underline{g},\frac{1}{\epsilon-\dom(\pi_n(\MidX(W)),X)},W\right)=\left(\underline{y},\underline{g},\frac{1}{\epsilon-\dom(Y,X)},W\right)\in B_n.
    \end{align*}
    We can therefore define the following map 
    \begin{align*}
    h\colon U &\to E\times V\\
            W &\mapsto \left(\MidX(W),\Shift\left(\MidX(W),\underline{g},\frac{1}{\epsilon-\dom(\pi_n(\MidX(W)),X)},W\right)\right),
    \end{align*}
    which is continuous as a composition of continuous functions by \cref{prop:shift continous,prop:mid p-Lipschitz}.

    Let us show that the image of $h$ indeed lies in $E\times V$ and that it defines an inverse of $f^{|U}$. We already showed that $\MidX(W)\in E$ for any $W\in U$. Let $(\underline{y},Z)\in E\times V$ and $Y=\pi_n(\underline{y})$. We already know that $\MidX(f(\underline{y},Z))=\underline{y}$. Moreover, as seen above $(\epsilon - \dom(X,Y))\dom(Z,G)<\epsilon - \dom(X,Y)<\merg(Y)$, so that by \cref{rem:shifting conditions} we can apply \cref{prop:shifting equivalence} \cref{it:shifting identity} to get
     \begin{align*}
         &\Shift\left(\MidX(f(\underline{y},Z)),\underline{g},\frac{1}{\epsilon-\dom(\pi_n(\MidX(f(\underline{y},Z))),X)},f(\underline{y},Z)\right)\\
         &=\Shift\left(\underline{y},\underline{g},\frac{1}{\epsilon-\dom(Y,X)},\Shift\left(\underline{g},\underline{y},\epsilon - \dom(Y,X),Z\right)\right)\\
         &=Z.
     \end{align*}
    This shows indeed that $\Ima(h)\subset E\times V$ and that $h\circ f=\Id_{E\times V}$. Injectivity of $h\circ f$ allows to conclude injectivity of $f$, so that $f$ is a bijection on its image, of inverse $h$. Therefore $f^{|U}$ is a homeomorphism.
    We can also apply \cref{prop:shifting equivalence} \cref{it:shift conserves cardinal} to get that
     \begin{align*}
         \card(f(\underline{y},Z))=\card\left(\Shift\left(\underline{g},\underline{y},\epsilon- \dom(Y,X),Z\right)\right)=\card(Z).
     \end{align*}
    We conclude that $f^{|U}$ is a stratified homeomorphism.

    Finally, let us show that $U$ is open in $\Ran(M)$. Let $\underline{y}_1\in E=\Bsum(\underline{x},\theta)$,
    $Y_1=\pi_n(\underline{y}_1)$, $Z_1\in V$, and  $W_1=f(\underline{y}_1,Z_1)\in U$,
    and let us find $N\subset U$ an open neighborhood of $W_1$ in $\Ran(M)$.
    As seen above, \cref{prop: shifting upper bounds} \cref{it: Z upper bound} gives:
    \begin{align*}
         \dom(W_1,Y_1)&=\dom\left(\Shift(\underline{g},\underline{y}_1,\epsilon-\dom(Y_1,X),Z_1),\Shift(\underline{g},\underline{y}_1,\epsilon-\dom(Y_1,X),G)\right)\\
         &\leq (\epsilon-\dom(Y_1,X)) \dom(W_1,G)\\
         &<\epsilon-\dom(Y_1,X).
    \end{align*}
     Since, by \cref{cor:midx continuous}, $\MidX$ is continuous on $\Bom(X,\epsilon)$,  we may choose $\zeta>0$ such that, for any $W_2\in \Bom(X,\epsilon)\cap\Bom(W_1,\zeta)$ , the following holds:
    \begin{equation}\label{eq:InegaliteMidChapeau}
        \Nsum{\MidX(W_1)-\MidX(W_2)}<\min\left\{ \theta - \Nsum{\underline{y}_1-\underline{x}}, \frac{\epsilon-\dom(X,Y_1)-\dom(Y_1,W_1)}{4\omega(n)}\right\}. 
    \end{equation}
 We then choose 
    \begin{align*}
        \delta=\min\left\{ \zeta, \frac{\epsilon-\dom(X,Y_1)-\dom(Y_1,W_1)}{2}\right\}.
    \end{align*}
    Let us define $N=\Bom(X,\epsilon)\cap\Bom(W_1,\delta)$ and let $W_2\in N$ and let us show that $W_2\in U$, i.e. let us find $(\underline{y}_2,Z_2)\in E\times V$ such that $f(\underline{y}_2,Z_2)=W_2$. Let us first define $\underline{y}_2=\MidX(W_2)$, and $Y_2=\pi_n(\underline{y}_2)$. We then have 
    \begin{align*}
        \Nsum{\underline{y}_2-\underline{x}}\leq \Nsum{\underline{y}_2-\underline{y}_1}+\Nsum{\underline{y}_1-\underline{x}}<\theta,
    \end{align*}
    so that $\underline{y}_2\in E$, and, by \cref{rem:card midx}, $\card(Y_2)=n$. Moreover, 
    \begin{align*}
        \dom(X,Y_2)&\leq \dom(X,Y_1) + \dom(Y_1,Y_2)&\\
        &\leq \dom(X,Y_1) + \omega(n)\Nsum{\underline{y}_1-\underline{y}_2}&\quad \text{by \cref{lem:pi omega Lipschitz}}\\
        &< \frac{\epsilon}{4}+\frac{3}{4}\dom(X,Y_1)-\dom(Y_1,W_1)&\quad \text{by \cref{eq:InegaliteMidChapeau}} \\
        &\leq\epsilon.&
    \end{align*}
    And we also have
\begin{align*}
    \dom(W_2,Y_2)&\leq \dom(W_2,W_1)+\dom(W_1,Y_1)+\dom(Y_1,Y_2)\\
    &< \delta + \dom(W_1,Y_1)+ 2 \dom(Y_1,Y_2) + \dom(Y_1,X) - \dom(Y_2,X)\\
    &<\frac{\epsilon-\dom(X,Y_1)-\dom(Y_1,W_1)}{2} + \dom(W_1,Y_1) + \frac{\epsilon-\dom(X,Y_1)-\dom(Y_1,W_1)}{2}\\
    &\phantom{<\frac{\epsilon-\dom(X,Y_1)-\dom(Y_1,W_1)}{2}}+ \dom(Y_1,X)-\dom(Y_2,X)\\
    &<\epsilon - \dom(Y_2,X)\\
    &<\merg(X) -\dH(Y_2,X)\\
    &\leq\merg(Y_2).
\end{align*}
Where we first used the definition of $N$ together with the triangle inequality applied to $\dom(Y_2,X)$. Then, on the third line, we used the definition of $\delta$, together with \eqref{eq:InegaliteMidChapeau}. The last line is \eqref{eq:radX radY} applied to $Y_2$, which holds because $\underline{y_2}\in E$.

We have thus checked that we can define the following : 
    \begin{align*}
        Z_2=\Shift\left(\underline{y}_2,\underline{g},\frac{1}{\epsilon-\dom(X,Y_2)},W_2\right).
    \end{align*}
Let us now show that $Z_2\in V$. We can apply \cref{prop: shifting upper bounds}\cref{it: Z upper bound} to get
     \begin{align*}
         \dom(Z_2,G)&=\dom\left(\Shift\left(\underline{y}_2,\underline{g},\frac{1}{\epsilon-\dom(X,Y_2)},W_2\right),\Shift\left(\underline{y}_2,\underline{g},\frac{1}{\epsilon-\dom(X,Y_2)},Y_2\right)\right)\\
         &\leq \frac{1}{\epsilon-\dom(X,Y_2)}\dom(W_2,Y_2)\\
         &< 1,
     \end{align*}
     and, by \cref{cor:shift mid commute},
     \begin{align*}
         \Mid_{\underline{g}}(Z_2)=\Mid_{\underline{g}}\left( \Shift\left(\MidX(W_2),\underline{g},\frac{1}{\epsilon-\dom(Y_2,X)},W_2\right) \right)=\underline{g},
     \end{align*}

Finally, by \cref{rem:shifting conditions} we can apply \cref{prop:shifting equivalence} \cref{it:shifting identity} to get
     \begin{align*}
         f(\underline{y}_2,Z_2)&=\Shift\left(\underline{g},\underline{y}_2,\epsilon - \dom(Y_2,X), \Shift\left(\underline{y}_2,\underline{g},\frac{1}{\epsilon-\dom(Y_2,X)},W_2\right) \right)=W_2.
     \end{align*}
Therefore $N \subset U$ so that $U$ is open. We finally conclude this proof using \cref{prop:strat homeo radial cone ball} to show that $V$ is indeed a cone with the appropriate stratification.

\end{proof}

\begin{corol}
\label{Corol:Ran_Euclidean_Conical}
    Let $M=\R^p$, and $\omega$ be a weight.
    Then $(\Ran(M), \tom)$ is conically stratified.
\end{corol}

\begin{corol}\label{Corol:Ran_Riemannian_Conical}
Let $(M,d)$ be a Riemannian manifold, and $\omega$ be a weight. Then $(\Ran(M),\tom)$ is conically stratified.
\end{corol}

\begin{proof}
    Write $p=\dim(M)$, and let $X\in \Ran(M)$. Order the elements of $X=\{x^{(1)},\dots,x^{(n)}\}$. Given some $1\leq i\leq n$, and any sufficiently small neighborhood of $x^{(i)}$ in $M$, $E_i$, there exists a bi-Lipschitz homeomorphism $\varphi_i\colon E_i\to B((3i,0,\dots,0),1)$ where $B((3i,0,\dots,0),1)$ is the unit ball in the Euclidean space $\R^p$. An elementary proof of this fact can be found in \cite{MathOverflowRiemann}.
    Furthermore, up to restricting the $\varphi_i$ to smaller neighborhoods, we may assume that the $E_i$ are disjoint. In fact, we will ask that each $E_i$ is a subset of a neighborhood of the form $B_M(x^{(i)},r)$ for some $r>0$ such that $d(x^{(i)},x^{(j)})>3r$ for all $1\leq i\not=j\leq n$. Here $B_M(x,r)$ denotes the ball in $M$ with center $x$ and radius $r$. 
    
    Let $E=\coprod E_i$, $B=\coprod B((3i,0,\dots,0),1)$ and denote $\varphi\colon E\to B$ restricting to $\varphi_i$ on $E_i$. Let us show that $\varphi$ is a bi-Lipschitz homeomorphism. First, note that for any $i$ and any $(y,z)\in E_i^2$, there exists $\alpha_i,\beta_i>0$,  such that $\alpha_id(\varphi(y),\varphi(z))\leq d(y,z)\leq \beta_i d(\varphi(y),\varphi(z))$, by hypothesis on the $\varphi_i$. 
    On the other hand, if $y\in E_i$ and $z\in E_j$ for $i\not=j$, then $r<d(y,z)<\diam(X)+2r$, by construction, while $1<d(\varphi(y),\varphi(z))<3n-1$. Let $\beta=\max\{\diam (X)+2r,\beta_i\mid 1\leq i\leq n\}$ and $\alpha=\min\{\frac{r}{3n-1},\alpha_i\mid 1\leq i\leq n\}$. We get 
    \begin{equation*}
        \alpha d(\varphi(y),\varphi(z))\leq d(y,z)\leq \beta d(\phi(y), \phi(z)), \ \forall (y,z)\in E^2.
    \end{equation*} 
    In particular, $\varphi$ is a bi-Lipschitz homeomorphism between $E$ and $B\subset \R^p$. 
    
    Now, the inclusions $E\subset M$ and $B\subset \R^p$ induce open embeddings $\Ran(E)\hookrightarrow \Ran(M)$ and $\Ran(B)\hookrightarrow \Ran(\R^p)$ for the topology $\tom$, by \cref{corol:Opens_induce_open_embeddings}. Since being conically stratified is a local property, then $\Ran(B)$ is conically stratified by \cref{Corol:Ran_Euclidean_Conical}. Furthermore, by \cref{rem:Invariance_locally_bi_Lipschitz}, $\Ran(E)$ and $\Ran(B)$ are homeomorphic in the stratified sense, so $\Ran(E)$ is also conically stratified, and it is a neighborhood of $X$ for the Hausdorff topology, by construction. Since the topology $\tom$ refines the Hausdorff topology, by \cref{prop:tom raffine tH}, then $E$ is also a neighborhood of $X$ for the topology $\tom$. Finally, we have showed that any $X\in \Ran(M)$ admits a neighborhood for the topology $\tom$ which is conically stratified, which proves the claim.
\end{proof}

\begin{rem}\label{rem:Lipschitz_manifold_Conical}
 Notice that in the previous proof, we didn't use explicitly the Riemmanian metric on $M$, but only the existence of local homeomorphism to the Euclidean space that are bi-Lipschitz. This suggests a more general construction. Indeed, assume that, instead of a specific Riemannian metric on $M$ one fixes an atlas, where for any charts $\varphi\colon U\to \R^p$ and $\psi\colon V\to \R^p$ the (partial) composition $\varphi\circ \psi^{-1}\colon \psi(V\cap U)\to \varphi(V\cap U)$ is bi-Lipschitz. This is enough to define a topology $\tom$ on $\Ran(M)$ as follows. For $X\in \Ran(M)$, pick disjoint charts around each $x\in X$, $(U_x,\phi_x)$, and let $E=\coprod_{x\in X}U_x$ be the disjoint union of the domains of those charts.
Up to translating their image, we may assume that we have a homeomorphism $\phi\colon E\to B$, where $B$ is an open subset of $\R^p$. We may then define a topology on $\Ran(E)$ by transporting that of $(\Ran(B),\tom)$ along the homeomorphism. The condition on the atlas will guarantee that the topology around $X$ does not depend on a choice of local charts, and since all such open neighborhoods cover $\Ran(M)$, we have indeed defined a topology on $\Ran(M)$. 

Note that this kind of atlas corresponds to the structure of a Lipschitz manifold, as defined by Sullivan in \cite{SullivanLipschitzManifold}, thus we may rephrase the above remark as saying that for any Lipschitz manifold, $M$, $\Ran(M)$ admits a topology $\tom$, which makes $\Ran(M)$ into a conically stratified space.

Furthermore, \cite[Corollary 3]{SullivanLipschitzManifold}, states that \textbf{any} topological manifold $M$, of dimension other than $4$ admits a \textbf{unique} Lipschitz structure, up to a bi-Lipschitz homeomorphism that is close to the identity. This means that given a topological manifold of dimension other than $4$, we may speak of the homeomorphism type of $(\Ran(M),\tom)$ without the need for further data. 

There are subtleties however. First, observe that if $(M,d)$ is a metric space, such that $M$ happens to be a topological manifold, this does not mean that there exists an atlas for $M$ composed of maps which are locally bi-Lipschitz with respect to $d$. In particular, the distance $\dom$, induced by $d$ may not induce the topology $\tom$ on $\Ran(M)$ defined in the previous paragraphs.

Secondly, while the homeomorphism type of $(\Ran(M),\tom)$ is well-defined knowing only that $M$ is a manifold, if we are equipped with two distinct Lipschitz atlases on $M$, they will provide two distinct topologies $\tom_1$ and $\tom_2$ on $\Ran(M)$. And while $(\Ran(M),\tom_1)$ and $(\Ran(M),\tom_2)$ will be homeomorphic stratified spaces, there is no reason to expect that the homeomorphism be given by the identity. Succinctly, while the homeomorphism type is well-defined, the topology may not be.
\end{rem}

\subsection{Naive conicality of the final topology}
\label{subsec:naive_conical}

We showed in the previous section that given a Riemannian manifold  $(M,d)$, the associated Ran space with the weighted topology, $(\Ran(M),\tom)$ is a conically stratified space for any weight $\omega$. A similar statement for the Hausdorff topology was proved by \v{C}epek and Lejay in \cite{CepekLejayExponential}. One may expect this statement to also hold for the final topology $\tpi$, but it is very easy to show that it cannot be true (see \cref{prop:Ran_final_Not_Conical}). We show however that $(\Ran(M),\tpi)$ satisfies a distinct, but somewhat similar property. Specifically,  we can still find an open neighborhood of any configuration that is stratified homeomorphic to the product of a Euclidean space and a stratified cone, except the cone is now equipped with the quotient topology instead of the teardrop topology.

\begin{prop}\label{prop:Ran_final_Not_Conical}
    Let $M$ be a manifold of dimension at least $1$. Then $(\Ran(M),\tpi)$ is not conically stratified.
\end{prop}

\begin{proof}
    Assume to the contrary that a stratified homeomorphism $f\colon U\to E\times c(L)$ exists, where $U\subset \Ran(M)$ is any non-empty subset, open for the topology $\tpi$. 
    Since $U$ is non-empty, there must be some $N\geq 1$ such that $\pi_N^{-1}(U)\not = \emptyset$. Let $n$ be the least such $N$. For any $k\geq n$, $\pi_k^{-1}(U)$ is an open subset of $M^k$ by assumption. But since $M^{k+1}\setminus M^k$ is dense in $M^{k+1}$ for all $k\geq n$, $U$ must contain a configuration of cardinality $k$ for all $k\geq n$.
    In particular, $c(L)$ must be stratified over $\N_{\geq n}$, and so $L$ must be stratified over $\N_{\geq n+1}$, with at least one point in each stratum. Let us denote by $\varphi\colon L\to \N_{\geq n+1}$ its stratification. Let $(X_k)_{k\in \N}$ be a sequence of elements in $L$ such that $\varphi(X_k)$ is strictly increasing. Now, consider $(t_k)_{k\in \N}$ some sequence in $(0,1)$ which decreases to $0$ as $k$ goes to infinity. By definition of the teardrop topology, the sequence $(t_k,X_k)$ must converge to the apex in $c(L)$. On the other hand, for $k\in \N$,  $f^{-1}(t_k,X_k)$ has cardinality $\varphi(X_k)$, thus the sequence of cardinalities $\card(f^{-1}(t_k,X_k))$ is unbounded, which implies, by \cref{prop:unbounded cardinal diverges in final}, that the sequence $(f^{-1}(t_k,X_k))_{k\in \N}$ does not converge in $(\Ran(M),\tpi)$.  This contradicts the assumption that $f$ is a homeomorphism.
\end{proof}

Now, to construct our homeomorphism, we start by showing that the shifting function is continuous for the final topology.

\begin{prop}
\label{prop:shift continuous tpi}
    Let $M=\R^p$, $n\in\N^*$ and let us write, as in the previous section $B_n=\{(\underline{x},\underline{y},s,Z)\in M^n\times M^n\times \R^+\times\Ran(M)\mid \card(\pi_n(\underline{x}))=n,Z\in \BHc(\pi_n(\underline{x}),\frac{\merg(\pi_n(\underline{x}))}{2})\}$. 
    The function $\Shift$ is continuous from $(B_n,\tpi)$ to $(\Ran(M),\tpi)$, where $(B_n,\tpi)$ is equipped with the product topology of $M^n\times M^n\times \R^+\times(\Ran(M),\tpi)$. 
\end{prop}
\begin{proof}
    By \cref{prop:shift continous}, for any weight $\omega$, the following composition is continuous
    \begin{align*}
        (B_n,\tpi)\overset{\Id}{\to} (B_n,\tom) \overset{\Shift}{\to} (\Ran(M),\tom).
    \end{align*}
    By \cref{theo:topologie limite}, we can directly conclude.
\end{proof}

The key technical step which distinguishes this proof from the one from \cref{subsec:conicity_Euclidean} is in the next proposition. In \cref{subsec:conicity_Euclidean}, we considered balls and spheres associated with a weight $\omega$, and considered continuity with respect to the topology $\tom$ associated to the same weight. Here, we are working with the final topology $\tpi$, which is not metrizable and so we do not have access to balls. On the other hand, any ball for any weight $\omega$ is an open set in $(\Ran(M),\tpi)$.  Since we still need to define ball-shaped subspaces of $\Ran(M)$, in order to produce our homeomorphisms, we will consider the balls associated to the \textbf{fixed} weight $\chi$, which we define to be constant equal to $1$. However, those subspaces of $\Ran(M)$ are not equipped with the topology induced from $d^{\chi}$ but instead are equipped with the restriction of $\tpi$.

Now, we know from \cref{prop:Ran_final_Not_Conical} that no open subset of $(\Ran(M),\tpi)$ can be conical, essentially because there are two many open subsets in the final topology. Thus, the maps appearing in \cref{prop:homeo cone to w-ball,prop:strat homeo radial cone ball} will still be open for the topology $\tpi$, but they will fail to be continuous. Replacing the topology on the cone $c(L)$ by the \textbf{quotient topology}, that is, the finest topology making $\pr\colon L\times [0,1)\to c(L)$ continuous, will alleviate this obstruction, which we show in the following proposition. From now on, we write $c_q(L)$ for the cone on $L$ equipped with the quotient topology.
\begin{prop}
\label{prop:strat homeo radial cone ball W}
    Let $M=\R^p$ and $X\in\Ran(M)$ be a configuration, $\epsilon\in(0,\frac{\merg(X)}{2})$, $n=\card(X)$ and $\underline{x}\in\pi_n^{-1}(\{X\})$. Let $C_{\underline{x}}=\{Z\in\BH(X,\merg(X))\mid \MidX(Z)=\underline{x}\}$. The map 
    \begin{align*}
        f\colon c_q\left(\left(\Schi(X,\epsilon),\tpi\right)\cap C_{\underline{x}}\right)&\to \left(\Bchi(X,\epsilon),\tpi\right)\cap C_{\underline{x}}\\
        [(t,Y)] &\mapsto \ScaleX(t,Y)
    \end{align*}
    is a stratified homeomorphism.
\end{prop}
\begin{proof}
    Let us denote $\pr\colon [0,1)\times \left(\left(\Schi(X,\epsilon),\tpi\right)\cap C_{\underline{x}}\right)\to c_q\left(\left(\Schi(X,\epsilon),\tpi\right)\cap C_{\underline{x}}\right) $  the quotient map. Then, by \cref{prop:shift continuous tpi}, the composition $f\circ \pr$ is continuous as a restriction of the $\Shift$ function, and so is $f$ by definition of the quotient topology.

    Moreover, by \cref{prop:strat homeo radial cone ball}, $f$ is a stratified bijection of inverse 
    \begin{align*}
        g\colon \left(\Bchi(X,\epsilon),\tpi\right)\cap C_{\underline{x}} &\to c_q\left(\left(\Schi(X,\epsilon),\tpi\right)\cap C_{\underline{x}}\right) \\
        Z&\mapsto 
        \begin{cases}
        \left[\left(\frac{\dchi(X,Z)}{\epsilon},\ScaleX\left(\frac{\epsilon}{\dchi(X,Z)},Z\right)\right)\right] & \text{if } Z\neq X\\
        [(0,Y)] \quad\text{where } Y\in \Schi(X,\epsilon) & \text{if } Z=X
        \end{cases}
    \end{align*}
    To show that $g$ is continuous, remark that $\left(\Bchi(X,\epsilon),\tpi\right)\cap C_{\underline{x}}$ is a clopen in $(\Ran(M),\tpi)$. We can thus use \cref{prop:contiunuite sur les troncations}, so that it is enough to show that the restrictions $g_{\leq N}$ of $g$ to the union of strata of cardinality $\leq N$ are continuous for all $N\geq 1$.
    Note that since $\Bchi(X,\epsilon)\subset\BH(X,\epsilon)\subset\Ran_{\geq n}(M)$, we can restrict ourselves to $N\geq n$. Let $N\geq n$. Then, by \cref{prop:shift continuous tpi}, $g_{\leq N}$ is continuous on $\left(\Bchi_{\leq N}(X,\epsilon)\cap C_{\underline{x}}\setminus\{X\},\tpi\right)$ as a restriction and composition of continuous functions. Let us show that it is continuous at $X$. Let $ U \subset c_q\left(\left(\Schi(X,\epsilon),\tpi\right)\cap C_{\underline{x}}\right)$ be an open neighborhood of the apex of the cone, and let us find $V$, an open neighborhood of $X$, in $g^{-1}_{\leq N}(U)\subset \left(\Bchi_{\leq N}(X,\epsilon),\tpi\right)\cap C_{\underline{x}}$. 
    
    Remark that $\Schi_{\leq N}(X,\epsilon)\cap C_{\underline{x}}$ is a closed subset of $(\Ran(M),\tpi)$, so that its preimage by $\pi_N$ is a closed subset of $M^N$. It is also bounded in $M^N$, so that it is compact.
    Since $\pi_N$ is surjective, we have $\Schi_{\leq N}(X,\epsilon)\cap C_{\underline{x}}=\pi_N(\pi_N^{-1}(\Schi_{\leq N}(X,\epsilon)\cap C_{\underline{x}}))$ which is compact as the image of a compact by $\pi_N$. 
    Therefore $\{0\}\times \left(\Schi_{\leq N}(X,\epsilon)\cap C_{\underline{x}}\right)\subset \pr^{-1}(U)$ is compact, 
    hence there exists $0<\delta_N\leq 1$ such that $[0,\delta_N)\times \left(\Schi_{\leq N}(X,\epsilon)\cap C_{\underline{x}}\right)\subset \pr^{-1}(U)$. 
    We can thus define $V=\Bchi_{\leq N}(X,\epsilon\delta_N)\cap C_{\underline{x}}$, which is an open subset of $\Bchi_{\leq N}(X,\epsilon)\cap C_{\underline{x}}$. We have indeed $X\in V$. 
    It remains to be showed that $V\subset g^{-1}_{\leq N}(U)$.
    Let $Z\in V$ and let us show that $Z\in g^{-1}_{\leq N}(U)$, i.e. that $g_{\leq N}(Z)=[(t,Y)]\in U$. There is nothing to check for $Y$, since we already know that $Y\in\Schi_{\leq N}(X,\epsilon)\cap C_{\underline{x}}$. Besides, $t=\frac{\dchi(X,Z)}{\epsilon}<\delta_N$, so that $g_{\leq N}(Z)=[(t,Y)]\in \pr\left([0,\delta_N)\times \left(\Schi_{\leq N}(X,\epsilon)\cap C_{\underline{x}}\right)\right)\subset U$. We conclude that $g_{\leq N}$ is continuous at $X$.
    
    Hence $g_{\leq N}$ is continuous for any $N\geq 1$, which, by \cref{prop:contiunuite sur les troncations} proves that $g$ is continuous. Therefore $f$ is indeed a stratified homeomorphism.
\end{proof}

The rest of the proof is mostly unchanged, so that we can finally conclude.
\begin{theo}
\label{Ran faible conique}
    Let $M=\R^p$ and $X\in\Ran(M)$ be a configuration, and $n=\card(X)$ its cardinality. There exist $U$ an open neighborhood of $X$ in $(\Ran(M), \tpi)$, E a Euclidean space of dimension $np$, $\varphi\colon L\to \{n+1,\dots\}$ a stratified space and $h\colon U \to E \times c_q(L)$ a stratified homeomorphism where $c_q(L)$ is the stratified cone on $L$  equipped with the quotient topology.
\end{theo}

\begin{proof}
    Recall, in the proof of \cref{prop:homeo strat}, we introduced a stratified homeomorphism (with an arbitrary weight $\omega$ instead of the minimal weight $\chi$), 
    \begin{align*}
        f\colon E\times V &\to U\\
        (\underline{y},Z ) &\mapsto \Shift\left(\underline{g},\underline{y},\epsilon - \dchi(\pi_n(\underline{y}),X),Z\right)
    \end{align*}
    of inverse
    \begin{align*}
    f^{-1}\colon U &\to E\times V\\
            W &\mapsto \left(\MidX(W),\Shift\left(\MidX(W),\underline{g},\frac{1}{\epsilon-\dchi(\pi_n(\MidX(W)),X)},W\right)\right).
    \end{align*}
    where $E=\Bsum(\underline{x},\theta)$, $V=\left\{Z\in B^{\chi}(G,1)\mid \Mid_{\underline{g}}\left(Z\right)=\underline{g}\right\}$, and $U\subset \Ran(M)$ was defined to be the image, which we proved to be open for the topology $\tchi$.
    This implies in particular that $U$ is an open subset for the topology $\tpi$, by \cref{tpi rafine tom}. In addition, as shown in the proof of \cref{prop:homeo strat}, $V\subset\BHc(G,\frac{\merg{G}}{2})$, so that for any $(\underline{y},Z )$ in $E\times V$, $\left(\underline{g},\underline{y},\epsilon - \dchi(\pi_n(\underline{y}),X),Z\right)$ lies in $B_n$. Now, since $\dchi$ is continuous for the topology $\tpi$, $f$ is a composition of the $\Shift$ function with continuous functions, hence by \cref{prop:shift continuous tpi}, $f$ is still continuous when both $V$ and $U$ are equipped with the topology $\tpi$. 
    Moreover, by \cref{cor:midx continuous}, $\MidX$ is continuous for $\tchi$ hence also for $\tpi$. Hence, by \cref{prop:shift continuous tpi} $f^{-1}$ is also continuous, and hence $f$ is indeed a stratified homeomorphism when $V$ and $U$ are equipped with the topology $\tpi$.

    We conclude with \cref{prop:strat homeo radial cone ball W} to show that $V$ is a stratified cone equipped with the quotient topology.
\end{proof}

Just as for the case of the weighted topologies, the generalization to the case of manifolds is immediate.

\begin{corol}\label{cor:RanFinalManifoldNaivementConique}
    Let $M$ be a manifold, then $(\Ran(M),\tpi)$ is \textbf{naively conical}. More explicitly, for any $X\in \Ran(M)$, with $\card(X)=n$, there exists $U\subset \Ran(M)$ a neighborhood of $X$, $E$ a Euclidean space, and $L\to \{n+1,\dots\}$ a stratified space together with a stratified homeomorphism $\varphi\colon U\to E\times c_q(L)$.
\end{corol}

\begin{proof}
    The proof is very similar to that of \cref{Corol:Ran_Riemannian_Conical}, so we only sketch it here. Given $X=\{x^{(1)},\dots,x^{(n)}\}\in \Ran(M)$, we may find disjoint open neighborhoods $x^{(i)}\in E_i\subset M$ together with homeomorphisms $\varphi\colon E_i\to B_i\subset \R^p$, which we may assemble to give a homeomorphism $E=\coprod E_i\to \coprod B_i=B\subset \R^p$. $\Ran(B)\subset \Ran(\R^p)$ is then an open subset for the final topology, thus, by \cref{Ran faible conique}, $\Ran(B)$ is naively conical. In turn, $\Ran(E)$ is naively conical and thus all points in $\Ran(M)$ admit a naively conical neighborhood.
\end{proof}

\begin{rem}\label{rem:naive_not_limit}
A plausible, but false, interpretation of the previous result would be the following. Since for all weights $\omega$, $(\Ran(M),\tom)$ is conically stratified, then for any $X\in \Ran(M)$, we must have a neighborhood of $X$ in $\Ran(M)$, $U$ such that $U\simeq E\times c(L)$ for some link $L$ where both $U$ and $L$ are equipped with the topology $\tom$. Then, when passing to the limit over $\omega\in \W$, somehow the topology on $c(L)$ must become the quotient topology. This interpretation is wrong in two places. First, upon inspection of the formulas, one notices that $L$, $U$ and $f$ all depend directly on $\omega$. Thus, there is no single link $L$ for which to compute the limit $(c(L),\tom)$.
But also, if one were to fix a subset $A\subset\Ran(M)$, and compute the limit of the teardrop cones $c(A,\tom)$ over $\omega\in \W$, one would get the teardrop cone $c(A,\tpi)$ (see the next proposition).

The phenomenon actually occurring has to do directly with the teardrop topology. Observe that, in the cone $c(L)$, the neighborhoods of the apex do not depend at all on the topology of $L$ but only on its underlying set. Thus, no common link (even with varying topology) can be found for all topologies $\tom$ since basis of neighborhoods vary when $\omega$ varies. And no link at all can be found such that $E\times c(L)$ embeds into $(\Ran(M),\tpi)$ since whatever $L$ could be, such an embedding would provide a countable basis of neighborhoods of some point in $M$, which we know do not exist by \cref{prop:final not metrizable}.

This sheds no light on why local homeomorphism with quotient cones exist however. Observe that, while we used the minimal weight $\chi$ to define the underlying sets of $U$ and $L$ and the map $f$, any other weight would have worked equally well. This phenomenon is not directly related to \cref{theo:topologie limite}, but rather to the definition of $(\Ran(M),\tpi)$ as the increasing union of the truncations $\Ran_{\leq n}(M)$. Each of those space has an unambiguously defined topology, and any of the map showing the local conicality can be restricted to the truncations. Note that, on the truncations, the links become compact, and thus quotient and teardrop topologies coincide. But then, passing to the colimit, that is, considering the increasing union $\cup_{n\geq 1} (c(L))_{\leq n}$, one gets the set $c(L)$ equipped with the quotient topology. Going through the above process, with conical neighborhoods in $(\Ran(M),\tchi)$ gives exactly the homeomorphism of \cref{Ran faible conique}.
\end{rem}

\begin{prop}
    Let $M$ be a locally compact metric space, and $A\subset \Ran(M)$. Then $\lim_\omega c(A,\tom)\simeq c(A,\tpi)$ where $c$ denotes the cone with the teardrop topology.
\end{prop}
\begin{proof}
    Let us first remark that, by \cref{tpi rafine tom} and by functoriality of the teardrop cone, $\Id\colon c(A,\tpi) \to c(A,\tom)$ is continuous for any weight $\omega$.

    Now let us show that $c(A,\tpi)$ satisfies the universal property. Let $f\colon E \to c(A,\tom)$ such that $f$ is continuous for any weight $\omega$, and let us show that $f\colon E \to c(A,\tpi)$ is continuous.

    First remark that $c(A,\tpi)\setminus \{[(0,Y)]\} \simeq (0,1)\times (A,\tpi) \simeq \lim_\omega (0,1)\times (A,\tom)\simeq \lim_{\omega}c(A,\tom)\setminus \{[(0,Y)]\}$. So that $f$ is continuous on $E\setminus f^{-1}(\{[(0,Y)]\})$

   Now, for $\epsilon\in (0,1)$,  let $U=\pi([0,\epsilon)\times A)\subset c(A,\tpi)$ and let us show that $f^{-1}(U)$ is open in $E$. Since $U$ is open in any $c(A,\tom)$ and $f\colon E \to c(A,\tom)$ is continuous, $f^{-1}(U)$ is indeed open in $E$. Thus $f\colon E \to c(A,\tpi)$ is continuous, and $c(A,\tpi)$ is indeed the limit.
\end{proof}

\bibliographystyle{alpha}
    \bibliography{biblio}

\end{document}